\documentclass[11pt]{article}
\usepackage[T1]{fontenc}
\usepackage{lmodern}

\usepackage{amsmath,amsthm,amssymb}
\usepackage[usenames,dvipsnames]{xcolor}
\usepackage{enumerate}
\usepackage{graphicx}
\usepackage[nospace,noadjust]{cite}
\usepackage{comment}
\usepackage{oands}
\usepackage{tikz}
\usepackage{changepage}
\usepackage{bbm}
\usepackage{mathtools}
\usepackage[margin=1in]{geometry}
\usepackage[pagewise,mathlines]{lineno}
\usepackage{appendix}
\usepackage{stmaryrd}
\usepackage{multicol}
\usepackage{microtype}
\usepackage[colorinlistoftodos]{todonotes}
\usepackage{dsfont}
\usepackage{mathrsfs}  
\usepackage{multirow}
\usepackage{graphicx,subfigure}
\usepackage{array}
\usepackage{bm}
\usepackage{tipa}
\usepackage{textcomp}
\newcolumntype{P}[1]{>{\centering\arraybackslash}p{#1}}
\usepackage{extarrows}
\usepackage{scalerel}
\usepackage[bottom]{footmisc}

\usepackage{longtable}
\usepackage{upgreek}

\usepackage[pdftitle={Supermultiplicativity},
  pdfauthor={Arka Adhikari, Jacopo Borga, Thomas Budzinski, William Da Silva, Delphin Sénizergues},
colorlinks=true,linkcolor=NavyBlue,urlcolor=RoyalBlue,citecolor=PineGreen,bookmarks=true,bookmarksopen=true,bookmarksopenlevel=2,unicode=true,linktocpage]{hyperref}

\usepackage[capitalise,noabbrev]{cleveref}

\makeatletter
\let\c@table\c@figure
\makeatother

\theoremstyle{plain}
\newtheorem{thm}{Theorem}[section]
\newtheorem{cor}[thm]{Corollary}
\newtheorem{lem}[thm]{Lemma}
\newtheorem{prop}[thm]{Proposition}
\newtheorem{conj}[thm]{Conjecture}

\def\@rst #1 #2other{#1}
\newcommand\MR[1]{\relax\ifhmode\unskip\spacefactor3000 \space\fi
  \MRhref{\expandafter\@rst #1 other}{#1}}
\newcommand{\MRhref}[2]{\href{http://www.ams.org/mathscinet-getitem?mr=#1}{MR#2}}

\newcommand{\dd}{\mathrm{d}}

\theoremstyle{definition}
\newtheorem{defn}[thm]{Definition}
\newtheorem{remark}[thm]{Remark}

\numberwithin{equation}{section}

\newcommand{\dsb}{\begin{adjustwidth}{2.5em}{0pt}
\begin{footnotesize}}
\newcommand{\dse}{\end{footnotesize}
\end{adjustwidth}}

\newcommand{\ssb}{\begin{adjustwidth}{2.5em}{0pt}}
\newcommand{\ssse}{\end{adjustwidth}}

\newcommand{\aryb}{\begin{eqnarray*}}
\newcommand{\arye}{\end{eqnarray*}}
\def\alb#1\ale{\begin{align*}#1\end{align*}}
\def\allb#1\alle{\begin{align}#1\end{align}}
\newcommand{\eqb}{\begin{equation}}
\newcommand{\eqe}{\end{equation}}
\newcommand{\eqbn}{\begin{equation*}}
\newcommand{\eqen}{\end{equation*}}

\newcommand{\rta}{\rightarrow}

\newcommand{\mcl}{\mathcal}

\newcommand{\eps}{\varepsilon}
\newcommand{\lis}{\mathrm{LIS}}
\newcommand{\lcl}{\mathrm{LCL}}
\newcommand{\lin}{\mathrm{LIN}}
\newcommand{\dist}{\mathrm{d}}

\newcommand{\Law}{\mathrm{Law}}

\newcommand{\Pp}[1]{\mathbb{P} \left(#1\right)}

\DeclareMathOperator{\Perm}{Perm}
\DeclareMathOperator{\Graph}{Graph}

\newcommand{\bbP}{\mathbb{P}}
\newcommand{\bbE}{\mathbb{E}}

\newcommand{\N}{\mathbb{N}}
\newcommand{\bbR}{\mathbb{R}}

\newcommand{\efrak}{\mathfrak{e}}
\newcommand{\sfrak}{\mathfrak{s}}

\newcommand{\Ec}[1]{\mathbb{E} \left[#1\right]}

\newcommand{\Pc}[1]{\mathbb{P} \left[#1\right]}
\newcommand{\Ppp}[2]{\mathbb{P}_{#1} \left(#2\right)}
\newcommand{\Pppt}[2]{\widetilde{\mathbb{P}}_{#1} \left(#2\right)}
\newcommand{\Pppsq}[3]{\mathbb{P}_{#1} \left(#2\mathrel{}\middle|\mathrel{}#3\right)}
\newcommand{\Ecsq}[2]{\mathbb{E} \left[#1\mathrel{}\middle|\mathrel{}#2\right]}
\newcommand{\Ecpsq}[3]{\mathbb{E}_{#1} \left[#2\mathrel{}\middle|\mathrel{}#3\right]}

\newcommand{\Ppsq}[2]{\mathbb{P} \left(#1\mathrel{}\middle|\mathrel{}#2\right)}
\newcommand{\Ecp}[2]{\mathbb{E}_{#1} \left[#2\right]}
\newcommand{\Ecpt}[2]{\widetilde{\mathbb{E}}_{#1} \left[#2\right]}
\newcommand{\Var}[1]{\mathrm{Var}\left(#1\right)}

\newcommand{\Ff}{\mathcal{F}}
\newcommand{\Gf}{\mathcal{G}}
\newcommand{\revF}{\overleftarrow{\mathcal{F}}}

\newcommand{\indicator}[1]{\mathds{1}_{#1}}

\newcommand{\intervalle}[4]{\mathopen{#1}#2
                            \mathclose{}\mathpunct{},#3
                            \mathclose{#4}}
\newcommand{\intervalleff}[2]{\intervalle{[}{#1}{#2}{]}}

\newcommand{\intervallefo}[2]{\intervalle{[}{#1}{#2}{)}}
\newcommand{\intervalleoo}[2]{\intervalle{(}{#1}{#2}{)}}
\newcommand{\intervalleentier}[2]{\intervalle\llbracket{#1}{#2}
                               \rrbracket}
\newcommand{\enstq}[2]{\left\lbrace#1\mathrel{}\middle|\mathrel{}#2\right\rbrace}
 
\let\originalleft\left
\let\originalright\right
\renewcommand{\left}{\mathopen{}\mathclose\bgroup\originalleft}
\renewcommand{\right}{\aftergroup\egroup\originalright}

\DeclareMathOperator{\LIS}{LIS}

\DeclareMathOperator{\LIN}{LIN}
\DeclareMathOperator{\LCL}{LCL}
\DeclareMathOperator{\Leb}{Leb}

\DeclareMathOperator{\divv}{div}
\DeclareMathOperator{\mer}{mer}
\DeclareMathOperator{\lmax}{lmax}
\newcommand{\TLis}[1]{T^{#1}}
\newcommand{\DMC}{M}
\newcommand{\lfa}{\raisebox{1.5pt}{\smash{\scaleto{\leftarrow}{2pt}}}}
\newcommand{\rga}{\raisebox{1.5pt}{\smash{\scaleto{\rightarrow}{2pt}}}}
\newcommand{\bunderline}[2][4]{\underline{#2\mkern-#1mu}\mkern#1mu}
\newcommand{\tpr}{\pi}
\newcommand{\stt}{H}
\newcommand{\sss}{s}
\newcommand{\sttm}{\tau}
\newcommand{\unc}{\mathrm{u}}
\newcommand{\Dir}{\mathrm{Dir}}
\newcommand{\DirM}{\mathrm{DirMult}}
\newcommand{\Mst}[1]{M^*_{#1}}
\newcommand{\stmer}{\mathfrak{t}_{\mer}}
\newcommand{\stmerpr}{\mathfrak{t}'_{\mer}}
\newcommand{\stdivv}{\mathfrak{h}_{\divv}}
\newcommand{\xmer}{V_{\mer}}
\newcommand{\lrv}{X}
\newcommand{\uleaf}[1]{L_{#1}}
\newcommand{\uedge}[1]{E_{#1}}
\newcommand{\usign}[1]{S_{#1}}
\newcommand{\svfq}{\varphi}
\newcommand{\svfbQ}{\Phi}
\newcommand{\resproc}{\mathcal{X}}
\newcommand{\auxproc}{\mathscr{X}}
\newcommand{\teps}{\mathcal{T}_\eps}
\newcommand{\stpmg}{\uptau}
\newcommand{\evnt}{\mathcal{E}}
\newcommand{\svx}{x}
\newcommand{\symcst}{0.901}
\newcommand{\leftevent}{\mathcal{H}'_{\ell}}

\newcommand{\sth}{\eta}
\newcommand{\stht}[1]{%
  \mathchoice
    {
      \eta_{\lower 0.3ex \hbox{$\scriptstyle #1$}}
    }{
      \eta_{\lower 0.3ex \hbox{$\scriptstyle #1$}}
    }{
      \eta_{\lower 0.3ex \hbox{$\scriptscriptstyle #1$}}
    }{
      \eta_{\lower 0.3ex \hbox{$\scriptscriptstyle #1$}}
    }%
}

\newcommand{\sthtpr}[1]{%
  \mathchoice
    {
      \eta'_{\lower 0.3ex \hbox{$\scriptstyle #1$}}
    }{
      \eta'_{\lower 0.3ex \hbox{$\scriptstyle #1$}}
    }{
      \eta'_{\lower 0.3ex \hbox{$\scriptscriptstyle #1$}}
    }{
      \eta'_{\lower 0.3ex \hbox{$\scriptscriptstyle #1$}}
    }%
}

\title{The longest increasing subsequence of Brownian separable permutons}
 \date{ }
 \author{
\begin{tabular}{c} Arka Adhikari\\[-2pt]\small University of Maryland \end{tabular}	
\begin{tabular}{c} Jacopo Borga\\[-2pt]\small MIT \end{tabular}
\begin{tabular}{c} Thomas Budzinski\\[-2pt]\small  ENS de Lyon \end{tabular}\\[15pt]
\begin{tabular}{c} William Da Silva\\[-2pt]\small University of Vienna \end{tabular} 
\begin{tabular}{c} Delphin Sénizergues\\[-2pt]\small Universit\'e Paris Nanterre \end{tabular} 
}

\begin{document}

\newpage

\maketitle
\thispagestyle{empty}
\vspace{-0.5cm}

\begin{abstract}
We establish a scaling limit result for the length $\lis(\sigma_n)$ of the longest increasing subsequence  of a permutation $\sigma_n$ of size $n$ sampled from the Brownian separable permuton $\bm{\mu}_p$ of parameter $p\in(0,1)$, which is the universal limit of pattern-avoiding permutations. Specifically, we prove that
\[\frac{\lis(\sigma_n)}{n^{\alpha}} \xlongrightarrow[n\rightarrow \infty]{\mathrm{a.s.}} X,\] 
where $\alpha=\alpha(p)$ is the unique solution in the interval $(1/2,1)$ to the equation	
    \begin{equation*}
        \frac{1}{4^{\frac{1}{2\alpha}}\sqrt{\pi}}\,\frac{\Gamma\big(\tfrac{1}{2}-\tfrac{1}{2\alpha}\big)}{\Gamma\big(1-\tfrac{1}{2\alpha}\big)}=\frac{p}{p-1},
    \end{equation*}
    and $X=X(p)$ is a non-deterministic and a.s.\ positive and finite random variable, which is a measurable function of the Brownian separable permuton. 
    Notably, the exponent $\alpha(p)$ is an increasing continuous function of $p$ with $\alpha(0^+)=1/2$, $\alpha(1^-)=1$ and $\alpha(1/2)\approx0.815226$, which corresponds to the permuton limit of uniform separable permutations. We prove analogous results for the size of the largest clique of a graph sampled from the Brownian cographon of parameter $p\in(0,1)$.
\end{abstract}

 

\tableofcontents

\bigskip

%
%

\section{Introduction}

The Brownian separable permuton is the universal random limiting permuton describing the scaling limit of many families of random pattern-avoiding permutations. The Brownian cographon is the universal random limiting graphon describing the scaling limit of various dense random graphs sampled from graph classes. In this paper, we solve the following open problems:
\begin{itemize}
\item What is the length $\lis(\sigma_n)$ of the longest increasing subsequence of a random permutation $\sigma_n$ of size $n$ sampled from the Brownian separable permuton?

\item What is the size $\lcl(G_n)$ of the largest clique of a random graph $G_n$ on $n$ vertices sampled from the Brownian cographon? 
\end{itemize}
Specifically, we prove a scaling limit result for both quantities, that is,\footnote{The natural couplings underlying the almost sure convergence will be described in detail in Sections~\ref{sect:perm-main}~and~\ref{sect:graph-main}.}
\[ \frac{\lis(\sigma_n)}{n^{\alpha}} \xlongrightarrow[n\rightarrow \infty]{\mathrm{a.s.}} X \qquad\text{and}\qquad \frac{\lcl(G_n)}{n^{\alpha}} \xlongrightarrow[n\rightarrow \infty]{\mathrm{a.s.}} X,  \]
where $\alpha$ is explicitly determined through a simple equation and the limiting random variable $X$ is non-deterministic, almost surely positive and finite, and is a deterministic measurable function of the Brownian separable permuton and the Brownian cographon, respectively. In particular, our main results -- Theorems \ref{thm:main_permutations} and \ref{thm:main_graphs} -- solve \cite[Conjecture 1.5]{bdsg-lis-perm}
and go significantly further, as they establish precise scaling limits.

\smallskip

In the remainder of this section, we introduce both the Brownian separable permuton and the Brownian cographon, present precise statements of our results, and discuss the motivations behind these questions, including relevant references.

\medskip

\noindent\textbf{Acknowledgments.} 
We thank Nathanaël Berestycki, Valentin Féray, Ewain Gwynne, Anders Karlsson and Allan Sly for some helpful discussions and for sharing their interest in this problem. 
J.B.\ was partially supported by NSF grant DMS-2441646. W.D.S.\ acknowledges the support of the Austrian Science Fund (FWF) grant on “Emergent
branching structures in random geometry” (DOI: 10.55776/ESP534). 

\subsection{Scaling limit for the length of the longest increasing subsequence}\label{sect:perm-main}

A \textbf{permuton} is a Borel probability measure $\mu$ on the unit square $ [ 0,1 ] ^ 2$ with two uniform marginals, that is, $\mu([a,b]\times[0,1]) = \mu( [0,1] \times [a,b] )=b-a$ for all $0\le a < b\le 1$. Permutons are natural objects to describe the scaling limits of random permutations and have been studied quite intensively in the past decade; see~\cite[Section 2.1]{borga-thesis} for an introduction to the theory of permutons and, in particular, to the notion of permuton limit. The Brownian separable permutons $\bm{\mu}_p$ are a one-parameter family of \emph{random} permutons, indexed by $p\in(0,1)$, which are known to describe the permuton limit of many families of pattern-avoiding permutations~\cite{bassino-separable-permuton, maazoun-separable-permuton, bassino2022scaling}. Moreover, the Brownian separable permutons represent the critical regime for the skew Brownian permutons introduced in~\cite{borga-skew-permuton}, a larger universality class for generalized pattern-avoiding permutations which enjoy several connections with classical objects studied in random geometry, such as Liouville quantum gravity and Schramm--Loewner evolutions~\cite{borga-skew-permuton, bhsy-baxter-permuton, bgs-meander}. 

Given a (possibly random) permuton $\mu$, there is a natural way to sample permutations of size $n$ from it:  conditional on $\mu$,  sample a sequence $(Z_i)_{i\geq 1}$ of independent points in the unit
square $[0, 1]^2$ with distribution $\mu$. The first $n$ points $(Z_i)_{i\leq n}$ induce a random permutation $\Perm[\mu,n]$: for any $i,j\in \{1,\dots,n\}$, let $\Perm[\mu,n](i) = j$ if the point with $i$-th lowest $x$-coordinate has $j$-th lowest $y$-coordinate. Note that $\Perm[\mu,n]$ are naturally coupled for different values of $n$.

We are interested in studying the longest increasing subsequence $\lis(\sigma_n)$ of the permutations
\begin{equation}\label{eq:finite_perm_brownian}
    \sigma_n=\sigma_n(p)\coloneqq \Perm[\bm{\mu}_p,n]
\end{equation}
sampled from the Brownian separable permuton  $\bm{\mu}_p$ of parameter $p\in(0,1)$. The precise definition of the Brownian separable permuton is given in~\cref{sect:brow-perm-cograp}. We stress that the definition is given for the sake of completeness, but for the purposes of this paper, it will be enough to know the distribution of $\sigma_n$, which we now describe.

Let $T_n=T_n(p)$ be a uniform planar rooted binary tree with $n$ leaves, decorated (given the tree) with i.i.d.\ Bernoulli($p$) $\oplus/\ominus$ signs on its nodes, \emph{i.e.}\ $\Pp{\oplus}=1-\Pp{\ominus}=p$. We will say that $T_n$ is a $p$-\textbf{signed uniform binary tree}.
See, for instance, the left-hand side of~\cref{fig:patterns-trees}.
Given $T_n$, we construct $\pi_n$ via the following procedure:
\begin{itemize}
\item Number the leaves of $T_n$ from $1$ to $n$ from left to right.
\item For each node $v$ of $T_n$ decorated by $\ominus$, swap the order of the two children of $v$ (together with their subtrees), but keep the subtrees oriented in the same way.
\item Finally, define $\pi_n=\pi_n(1)\dots\pi_n(n)$ as the sequence of numbers for the leaves of the reordered tree from left to right.
\end{itemize}

\begin{figure}[b!]
	\begin{center}
		\includegraphics[width=\textwidth]{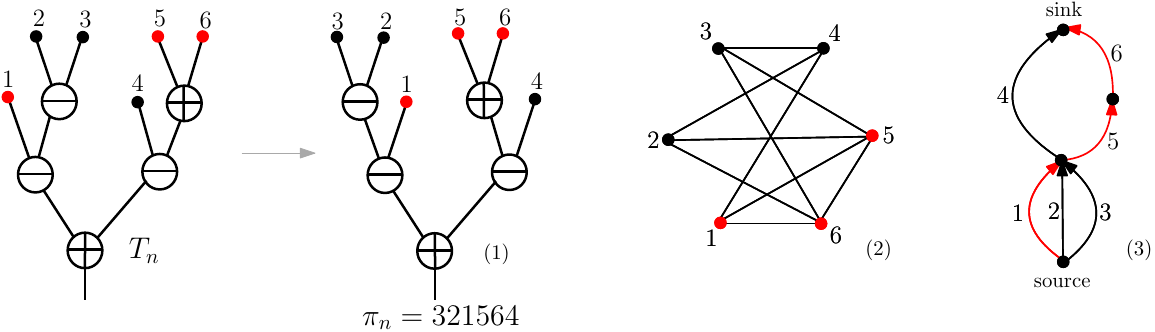}  
		\caption{\textbf{Left:}  A planar rooted binary tree with $6$ leaves, decorated with $\oplus/\ominus$ signs on its nodes. \textbf{Right:} (1) The tree  is obtained from the tree on the left by swapping the children of each $\ominus$ decorated node as explained below \eqref{eq:finite_perm_brownian}. The corresponding permutation $\pi_n$ is obtained by reading the labels of the leaves from left to right. (2) The graph  is obtained from the tree on the left by following the strategy detailed below \eqref{eq:finite_graph_brownian}. (3) The rooted planar map is obtained from the tree on the left by operating parallel compositions at $\ominus$ nodes and series compositions at $\oplus$ nodes. Here the leaves of the tree on the left correspond to the edges of the map. The map has a natural orientation from the \emph{source} to the \emph{sink}, inherited from the tree.
        Note that the tree leaves highlighted in red on the tree on the left induce a largest positive subtree (see \cref{sect:positive-subtrees} for terminology). These three leaves induce a longest increasing subsequence of the corresponding permutation $\pi_n$, a largest clique of the corresponding graph, and a longest directed path of the corresponding planar map. \label{fig:patterns-trees}}
	\end{center}
	\vspace{-3ex}
\end{figure}

Then, as explained for instance in \cite[Section 5.1]{bbfgp-universal}, the permutation $\pi_n$ has the same law as $\sigma_n$ in~\eqref{eq:finite_perm_brownian}. An example of this construction is given in~\cref{fig:patterns-trees}. Moreover, as we will carefully explain in~\cref{remk:coupled-trees}, if the trees $T_n$ are coupled for different values of $n$ via the R\'emy algorithm introduced in \cref{sect:remy-algo}, then the corresponding permutations $\sigma_n$ are coupled exactly as in~\eqref{eq:finite_perm_brownian}.

Our interest in studying the longest increasing subsequence $\lis(\sigma_n)$ is motivated by various factors, including the general study of longest increasing subsequences in non-uniform permutation models and the connections between the longest increasing subsequence of permutations sampled from skew Brownian permutons and certain models of directed metrics on random planar maps. We refer the reader to~\cite[Sections 1.1.2 and 1.3.1]{bdsg-lis-perm} for further details. 
Our main result on the Brownian separable permutons is the following. We denote by $\Gamma$ the standard Gamma function.

\begin{thm}[Scaling limit for the length of the longest increasing subsequence of the Brownian separable permutons]\label{thm:main_permutations}
    For all $p\in(0,1)$, let $\sigma_n=\Perm[\bm{\mu}_p,n]$ be a random permutation of size $n$ sampled from the Brownian separable permuton $\bm{\mu}_p$ of parameter $p$, as introduced in \eqref{eq:finite_perm_brownian}. Recall that the permutations $\sigma_n$ are naturally coupled for different values of $n$. We have the almost sure convergence
    \begin{equation}\label{eq:as-conv-perm}
        \frac{\lis(\sigma_n)}{n^{\alpha}} \xlongrightarrow[n\rightarrow \infty]{a.s.} X,
    \end{equation}  
    where $\alpha=\alpha(p)$ is the unique solution in the interval $(1/2,1)$ to the equation	
        \begin{equation}\label{eq:eq-to-defn-alpha}
            \frac{1}{4^{\frac{1}{2\alpha}}\sqrt{\pi}}\,\frac{\Gamma\big(\tfrac{1}{2}-\tfrac{1}{2\alpha}\big)}{\Gamma\big(1-\tfrac{1}{2\alpha}\big)}=\frac{p}{p-1},
        \end{equation}
    and the limiting random variable $X=X(p)$ is non-deterministic and almost surely positive and finite.
    Moreover, $X$ is a deterministic measurable function of the Brownian separable permuton $\bm{\mu}_p$.
\end{thm}

Setting $\zeta_{\mathbb{Z}}(s)\coloneqq \frac{1}{4^s\sqrt{\pi}}\,\frac{\Gamma(1/2-s)}{\Gamma(1-s)}$, we can write the equation in~\eqref{eq:eq-to-defn-alpha} as $\zeta_{\mathbb{Z}}(\tfrac{1}{2\alpha})=\frac{p}{p-1}$. The function $\zeta_{\mathbb{Z}}(s)$ is the spectral zeta function of $\mathbb{Z}$ introduced in~\cite{friedli2017spectral}. The beautiful formula in \cite[Theorem 2]{karlsson2022volumes} allows us to rewrite our equation in~\eqref{eq:eq-to-defn-alpha} in the following elegant product form:
\[\prod_{i=1}^\infty \frac{(i-\tfrac{1}{2\alpha})^2}{i(i-\tfrac{1}{\alpha})}=\frac{p}{p-1}.\]
We refer the reader to~\cite{karlsson2020spectral,karlsson2022volumes} for more background on spectral zeta functions.
The graph of $\alpha(p)$ is shown in~\cref{fig:alpha_p_graph}. A heuristic explanation for why the exponent $\alpha$ satisfies the equation~\eqref{eq:eq-to-defn-alpha} is given in~\cref{sect:heuristic-calculation}. However, at present, we have no explanation for the appearance of the number-theoretic function $\zeta_{\mathbb{Z}}$. Further details on the fact that $X$ is a deterministic measurable function of the Brownian separable permuton are given in~\cref{sect:measurability}. 

\begin{figure}[!htb]
	\begin{center}
		\includegraphics[width=0.55\textwidth]{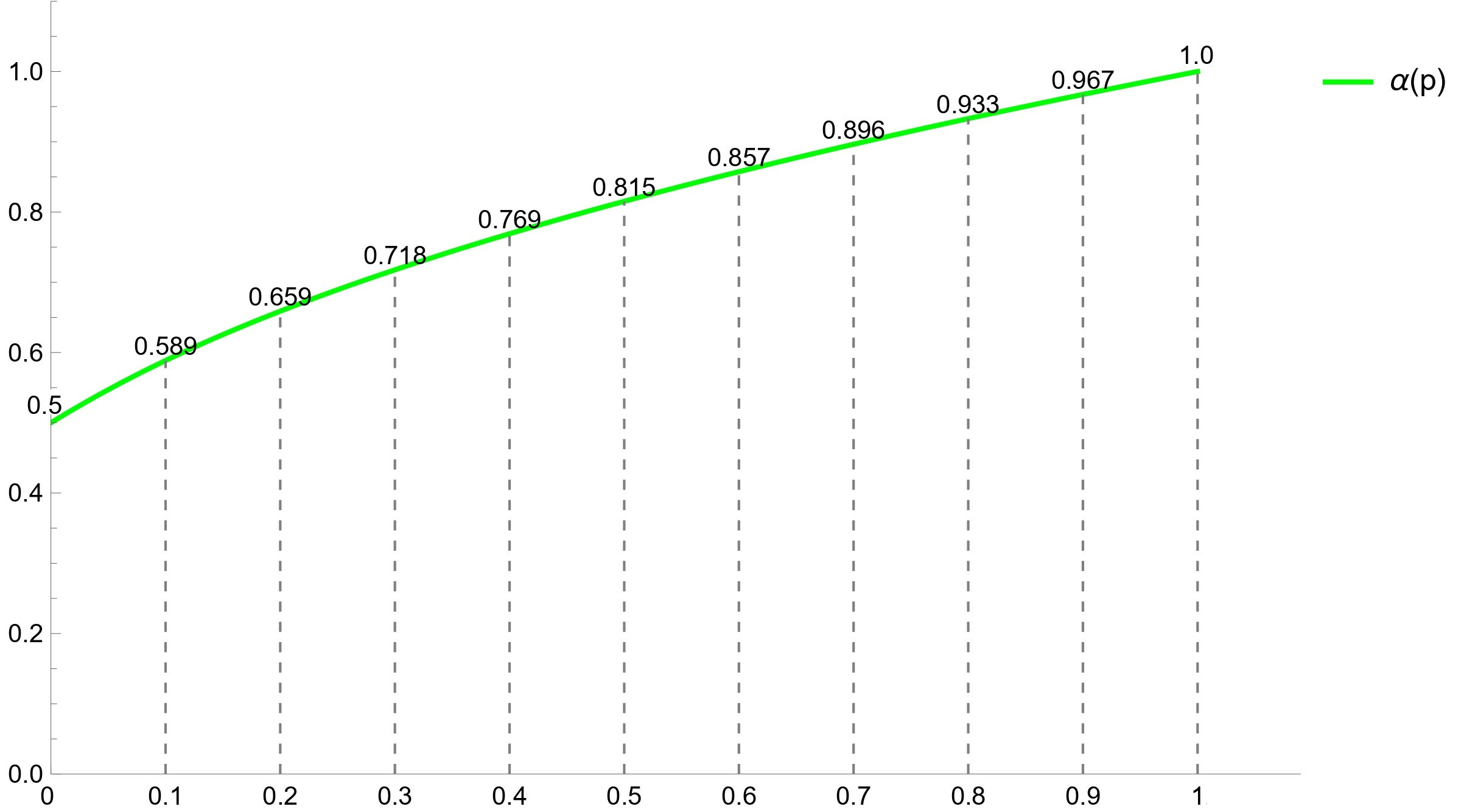}  
		\caption{The graph of the function $\alpha(p)$ from \cref{thm:main_permutations} for $p\in(0,1)$. The values of $\alpha(i/10)$ for $i=0,1,\dots,10$ are shown on top of the function, where $\alpha(0)=\alpha(0^+)$ and $\alpha(1)=\alpha(1^-)$.\label{fig:alpha_p_graph}}
	\end{center}
	\vspace{-3ex}
\end{figure}

We believe that the law of $X(p)$ is absolutely continuous with respect to the Lebesgue measure  and that for all $p \neq p'$, the laws of $X(p)$ and $X(p')$ are distinct -- although we do not prove this here. This belief is also supported by the numerical simulations in~\cref{sect:num-sim}.

We recall that the study of the longest increasing subsequence of Brownian separable permutons was initiated in \cite{bassino-lis} and the best previously known result on the behavior of $\lis(\sigma_n)$ was the following.

\begin{thm}[{\cite[Theorem 1.1]{bdsg-lis-perm}}]\label{thm:prev-bound}
	There exist two explicit functions $\alpha_* :(0,1)\to(1/2,1)$ and $\beta^* :(0,1)\to(1/2,1)$ such that for all $p\in(0,1)$:
	\begin{itemize}
		\item $1/2<\alpha_*(p)\leq\beta^*(p)<1$.
        \item For all $\eps>0$, it holds with probability tending to one as $n\to\infty$ that 
		\begin{equation*}
			n^{\alpha_*(p)-\eps} \leq \LIS(\sigma_n) \leq n^{\beta^*(p)+\eps}.
		\end{equation*}
	\end{itemize}	
\end{thm}

We emphasize that, prior to the present work, even the existence of the exponent $\alpha(p)$ was an open problem. The graphs of $\alpha_*(p),$ $\alpha(p)$ and $\beta^*(p)$ are shown on the left-hand side of~\cref{fig:alpha_beta_p_graph} for comparison.
We refer to \cite[Remark 1.2]{bdsg-lis-perm} for the exact equations characterizing $\alpha_*(p)$ and $\beta^*(p)$.
It is interesting to note that the previously known lower bound $\alpha_*(p)$ for the exponent $\alpha(p)$ is very close to the exact value for all $p\in(0,1)$, even though $\alpha_*(p)<\alpha(p)$ for all $p\in(0,1)$ as shown  on the right-hand side of~\cref{fig:alpha_beta_p_graph}. We also note that the values of $\alpha(p)$ given in~\cref{fig:alpha_p_graph} are consistent with the numerical simulations in \cite[Appendix A]{bdsg-lis-perm}.

\begin{figure}
	\begin{center}
		\includegraphics[width=0.5\textwidth]{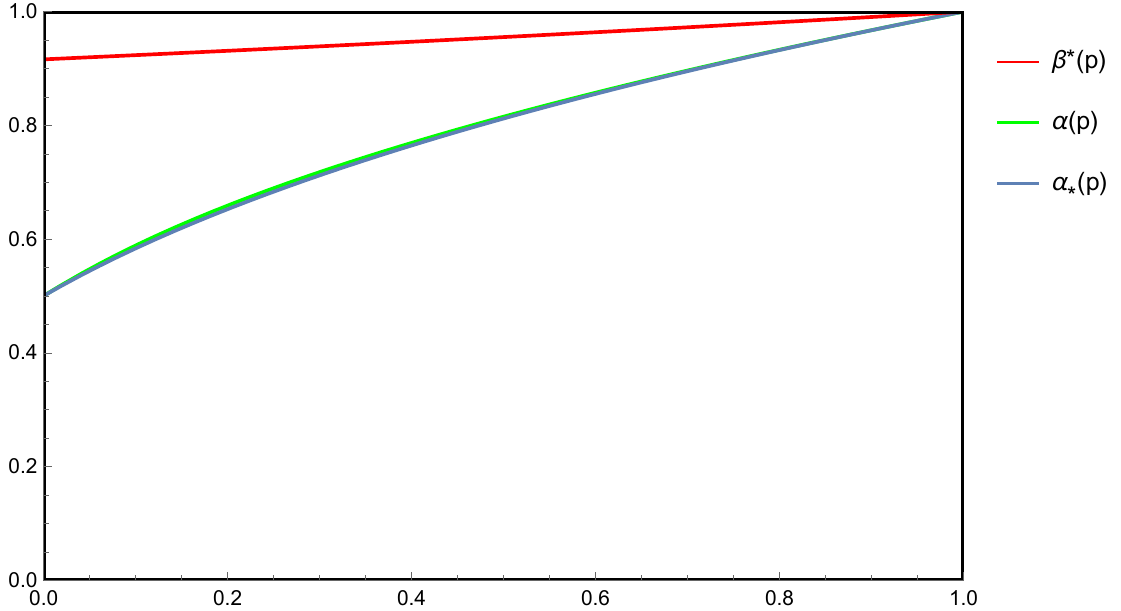} 
        \includegraphics[width=0.46\textwidth]{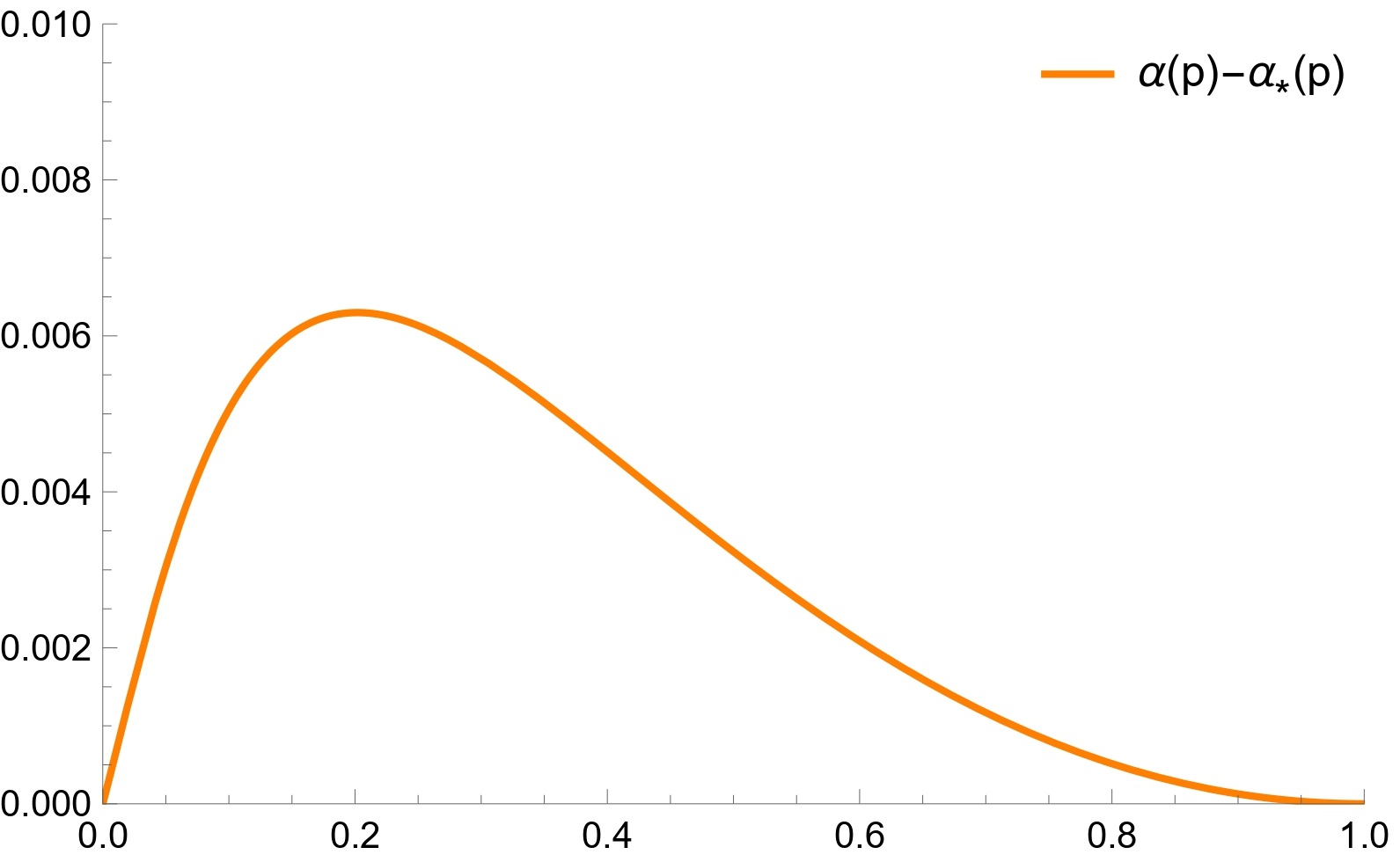} 
		\caption{\textbf{Left:} The graph of the functions $\alpha_*(p),$ $\alpha(p)$ and $\beta^*(p)$ from \cref{thm:main_permutations} and \cref{thm:prev-bound}. \textbf{Right:} The graphs of $\alpha(p)-\alpha_*(p)$. 
        The maximum is reached at $p\approx0.202$ where it takes the value $\alpha(p)-\alpha_*(p)\approx0.0063$. 
        Moreover, $\alpha(p)-\alpha_*(p)>0$ for all $p\in(0,1)$. \label{fig:alpha_beta_p_graph}}
	\end{center}
	\vspace{-3ex}
\end{figure}

\paragraph{Uniform separable permutations vs.\ permutations sampled from the permuton.}\label{rmk:non-uniform}
For the reader familiar with the notion of separable permutation (a beautiful introduction, including connections to algebraic geometry and Kontsevich's polynomial interchange problem, can be found in Ghys's book~\cite{ghys2016singular}), we stress that $\sigma_n(p)$ is a.s.\ a separable permutation~\cite[Definition 5.1]{bbfgp-universal} for all $p\in(0,1)$, but is \emph{not} a uniform separable permutation, even when $p=1/2$. Indeed, there are multiple rooted binary trees with $n$ leaves decorated with $\oplus/\ominus$ signs on their nodes that are mapped to the same separable permutation by the construction described below~\eqref{eq:finite_perm_brownian} and in \cref{fig:patterns-trees}. Nevertheless, when $p=1/2$, the random permutations $\sigma_n$ have \emph{the same} permuton limit as uniform separable permutations~\cite{bassino-separable-permuton}, and the two models are conjectured to display exactly the same LIS behavior; see \cref{conj:sep-and-cograph} below and \cite[Conjecture 1.6]{bdsg-lis-perm}.

Moreover, for all $p\in(0,1)$, the permutation $\sigma_n(p)$ defined as in \eqref{eq:finite_perm_brownian} has the following consistency property: for $k<n$, the permutation induced 
by $k$ uniform indices in $\sigma_n$ has the same law as $\sigma_k$; see \cite[Proposition 2.9]{bbfgp-universal}. In addition, the model $\sigma_n(1/2)$ can be viewed as the unique model in the universality class of uniform separable permutations that satisfies this consistency property, which makes it perhaps more natural than the uniform model.

\paragraph{Largest positive subtrees: local convergence and scaling limit.}
Due to the connection presented in \cref{fig:patterns-trees} between trees and permutations, \cref{thm:main_permutations} has a natural translation in terms of the $p$-signed uniform binary tree $T_n$. With a slight abuse of notation, let $\lis(T_n)$ denote the maximal number of leaves of a subtree\footnote{See \cref{sect:positive-subtrees} for the precise notion of subtrees.} of $T_n$ whose pairwise highest common ancestors all carry $\oplus$ signs. We will refer to this as the size of the ``largest positive subtree'' of $T_n$.
Then, by construction of the permutation $\pi_n$ from $T_n$, we have $\lis(T_n)=\lis(\pi_n)\stackrel{\text{d}}{=}\lis(\sigma_n)$, so that \cref{thm:main_permutations} may be interpreted as a scaling limit result for $\lis(T_n)$. This is precisely the point of view we adopt to prove \cref{thm:main_permutations}, as explained in the proof strategy outlined in Section~\ref{sect:proof-strat}. As a key intermediate step, we will prove in \cref{thm:localconv} the local convergence of a largest positive subtree in a variant of the model $T_n$. As a by-product of our arguments, we will also obtain in Theorem~\ref{thm:fragmentation_tree} a description of the limiting \emph{shape} of a largest positive subtree in the same variant of $T_n$. See \cref{sect:dynamical-picture} for a more detailed discussion.

\subsection{Scaling limit for the size of the largest clique and independent set}\label{sect:graph-main}

A \textbf{graphon} is an equivalence class of symmetric measurable functions $W : [0, 1]^2\to\{0, 1\} $ (\emph{i.e.}\ $ W ( x , y ) = W ( y , x ) $ for all $ x , y \in [0,1] $) under the equivalence relation $\sim$, where $W\sim U$ if there exists a measurable, invertible, Lebesgue measure preserving function $\phi:[0, 1]\to[0, 1]$ such that $W( \phi(x) ,  \phi(y) ) = U( x, y )$ for almost every $x,y\in[0, 1].$ A graphon can be thought of as a continuous analog of the adjacency matrix of a graph, viewed up to relabeling its continuous vertex set. Graphons have been used for many different purposes, with the original goal of having a notion of scaling limit for dense graphs~\cite{lovasz2012large}.

The \textbf{Brownian cographon} $\bm{W}_p$ of parameter $p\in(0,1)$ is a \emph{random} universal graphon describing, for instance, the graphon limit of uniform random cographs when the number of vertices tends to infinity
\cite{bassino2022random,stufler2021graphon} -- in this case $p=1/2$. In \cite{lenoir2023graph, lenoir2023subgraph}, the author proves that $\bm{W}_p$ is also the graphon limit of many other natural families of graphs, showing the universality of $\bm{W}_p$.

Let $(U_i)_{i\geq 1}$ be i.i.d.\ uniform random variables on $[0, 1]$. Given a graphon $W$, the random graph $\Graph[W,n]$ induced by $W$ of size $n$ is defined as follows: consider the first $n$ random variables $(U_i)_{i\leq n}$ and consider $n$ vertices
$\{ v_1 ,  v_2 , \dots , v_n \}$.
We connect the vertices $v_i$ and $v_j$ with an edge if and only if $W(U_i,U_j)=1$.
This definition can be naturally extended to the case of random graphons (see \cite[Section 3.2]{bassino2022random} for further details) and gives a coupling for the graphs $\Graph[W,n]$ for different values of $n$.

An \textbf{independent set} of a graph $G$ is a subset of its vertices such that any two distinct vertices in the subset are not adjacent, and a \textbf{clique} of a graph $G$ is a subset of its vertices such that every two distinct vertices in the subset are adjacent.

We are interested in the size of the largest independent set $\LIN(G_n)$ and the size of the largest clique\footnote{These two quantities are usually denoted by $\alpha(\cdot)$ and $\omega(\cdot)$ in the literature, but we preferred to adopt a different notation since it is more consistent with the one used for permutations.} $\LCL(G_n)$ of the graphs
\begin{equation}\label{eq:finite_graph_brownian}
    G_n=G_n(p)\coloneqq \Graph[\bm{W}_p,n]
\end{equation}
sampled from the \textbf{Brownian cographon} $\bm{W}_p$ of parameter $p\in[0,1]$.  
Our main motivation comes from connections with the \emph{Erd\H{o}s--Hajnal conjecture}, as explained in \cite[Section 1.2.2]{bdsg-lis-perm}.

As in the case of permutons, the exact law of $\bm{W}_p$ is not important for our results and is provided in \cref{sect:brow-perm-cograp} for completeness, but the law of $G_n$ can be described in terms of a uniform rooted binary tree $T_n$ with $n$ leaves, decorated with i.i.d.\ Bernoulli($p$) $\oplus/\ominus$ signs on its nodes. Specifically, numbering the leaves of $T_n$ from $1$ to $n$ from left to right, the law of $G_n$ is described as follows: there is an edge between the vertices $v_i$ and $v_j$ if and only if the highest common ancestor of the $i$-th and $j$-th leaves in  $T_n$ is decorated by a $\oplus$.
Note that by flipping the $\oplus/\ominus$ signs, $\LIN(G_n(p))$ has the same law as $\LCL(G_n(1-p))$ for all $p\in (0,1)$.

\begin{thm}[Scaling limit for the size of the largest clique and independent set of the Brownian cographon]\label{thm:main_graphs}

For all $p\in(0,1)$, let $G_n=\Graph[\bm{W}_p,n]$ be a random graph of size $n$ sampled from the Brownian cographon $\bm{W}_p$ of parameter $p$, as explained in \eqref{eq:finite_graph_brownian}. Recall that the graphs $G_n$ are naturally coupled for different values of $n$. We have the almost sure convergence
	\[ \frac{\LCL(G_n)}{n^{\alpha(p)}} \xlongrightarrow[n\rightarrow \infty]{a.s.} X(p) \qquad\text{and}\qquad \frac{\LIN(G_n)}{n^{\alpha(1-p)}} \xlongrightarrow[n\rightarrow \infty]{a.s.} X(1-p),  \]
where $\alpha(p)\in(1/2,1)$ and the law of $X(p)$ are as in~\cref{thm:main_permutations}. Moreover, $X(p)$ and $X(1-p)$ are deterministic measurable functions of the Brownian cographon $\bm{W}_p$.
\end{thm}

The study of the largest clique and independent sets of the Brownian cographons was initiated in~\cite{bassino-lis} and the best previously known result on the behavior of $\LCL(G_n)$ and $\LIN(G_n)$ was~\cite[Theorem 1.3]{bdsg-lis-perm}, which is the counterpart of \cref{thm:prev-bound} above. We note right away that just like for permutations, if $G_n$ is built from $T_n$ as in Figure~\ref{fig:patterns-trees}, then $\LCL(G_n)=\LIS(T_n)$ by construction. In particular, the convergence statements of Theorems~\ref{thm:main_permutations} and~\ref{thm:main_graphs} are equivalent.

\subsection{Some conjectures and additional comments}

In this final section, we outline a few conjectures and related comments that extend and contextualize our main results.

\paragraph{Uniform separable permutations and cographs.}
We believe that our results can be extended to the length of the longest increasing subsequence in uniform separable permutations and the size of the largest clique and independent set in uniform cographs (and, more generally, to most natural models known to converge to the Brownian separable permuton or the Brownian cographon).

\begin{conj}\label{conj:sep-and-cograph}
    Let $\alpha(1/2)$ and $X(1/2)$ be as in \cref{thm:main_permutations}. Let $\overline{\sigma}_n$ and $\overline{G}_n$ be a uniform separable permutation and a uniform separable cograph of size $n$, respectively. Then there is a deterministic constant $c>0$ such that we have the convergences in distribution
	\[ \frac{\LIS(\overline{\sigma}_n)}{n^{\alpha(1/2)}} \xlongrightarrow[n\rightarrow \infty]{\mathrm{d}} c \cdot X(1/2) \qquad\text{and}\qquad \frac{\LIN(\overline{G}_n)}{n^{\alpha(1/2)}} \xlongrightarrow[n\rightarrow \infty]{\mathrm{d}} c\cdot X(1/2). \]
\end{conj}

The numerical simulations in~\cref{sect:num-sim} suggest that $c\approx \symcst$. Note that \cref{conj:sep-and-cograph} is more precise than \cite[Conjecture 1.6]{bdsg-lis-perm}.

\smallskip

We also expect that some of the ideas in our paper could be adapted to study the length of the longest increasing subsequence in random recursive separable permutations~\cite{feray2023permuton}. We emphasize that these permutations have a random permuton limit that differs from the Brownian separable one.

\paragraph{The limiting random variable $X(p)$.}
In this work, we do not determine the exact law of the limiting random variable $X(p)$ appearing in \cref{thm:main_permutations}. It would be interesting to identify the distribution of $X(p)$ or at least establish further properties. For instance, three natural questions are whether the laws of $X(p)$ and $X(p')$ are distinct for all $p \neq p'$, what the expectation of $X(p)$ is, or what the tail behavior of $X(p)$ is. Some numerical simulations of these random variables are presented in~\cref{sect:num-sim}.

\paragraph{Hausdorff dimension of the largest increasing subset of the Browian separable permuton.}

The (closed) support of the Brownian separable permuton $\bm{\mu}_p$ is the (random) set $\text{supp} (\bm{\mu}_p)$ defined as the intersection of all closed sets $K\subset [0,1]^2$ with $\bm{\mu}_p(K) = 1$.
We say that a set $A\subset [0,1]^2$ is increasing if  
\[A\subset \left( [0,t] \times [0,s]\right) \cup \left( [t,1] \times [s,1] \right)\qquad \text{for all } (t,s) \in A .\] 
An example of an increasing set is the graph of a non-decreasing function $[0,1] \rta [0,1]$, or a subset of such a graph.
Let $\mcl{I}_p$ denote the collection of increasing subsets of $\text{supp}(\bm{\mu}_p)$. Then, we conjecture that for all $p \in (0,1)$, almost surely,
\[
\underset{A\in \mcl{I}_{p}}{\sup} \; \{\dim A \} = \alpha(p),
\]
where $\alpha(p)$ is defined as in \cref{thm:main_permutations}, and $\dim(\cdot)$ denotes the Hausdorff dimension. 
We also believe that the limiting random variable $X(p)$ of~\cref{thm:main_permutations} is connected to the maximal $\alpha(p)$-dimensional Hausdorff measure (with a potential gauge correction) of such a largest increasing subset.

\paragraph{A directed version of the Liouville quantum gravity metric.} Our results also describe the length of the longest directed path from the source to the sink in a natural model\footnote{As in the case of separable permutations, this natural model of series–parallel maps is the unique one within the universality class of uniform series–parallel maps that satisfies a certain consistency property; recall the discussion on uniform separable permutations on page~\pageref{rmk:non-uniform}. Nevertheless, we expect the asymptotic behavior of the longest directed path to be the same in both the uniform and non-uniform models.} of rooted series–parallel maps~\cite[Proposition 6]{bbf-bipolar-bijection}; see also~\cref{fig:patterns-trees}. 
Rooted series-parallel maps are graphs with two distinguished vertices called source and sink, formed recursively by two simple composition operations (\emph{parallel composition} and \emph{series composition}). 
They have been used to model series and parallel electric circuits, and are a natural model of last-passage percolation in random geometry. 
The scaling limit of different models of series-parallel maps (in the Gromov--Hausdorff sense) has been studied in~\cite{stufler2016limits, amankwah2025scaling}.
See the discussion in \cite[Section  1.3.1]{bdsg-lis-perm} for more connections between longest increasing subsequences in the skew Brownian permuton and models of last-passage percolation in random geometry, including Baxter permutations and bipolar orientations~\cite{bm-baxter-permutation}.

In the continuum,
skew Brownian permutons were constructed in \cite{borga-skew-permuton} using Liouville quantum gravity and a pair of Schramm--Loewner evolutions. This connection was
recently further strengthened in~\cite{borga2024reconstructing}. 
In this correspondence, it is conjectured that Brownian separable permutons can be constructed using $\text{SLE}_4/\text{CLE}_4$ and critical Liouville quantum gravity ($\gamma=2$). In connection to the aforementioned discrete picture, we expect the exponent $\alpha(p)^{-1}$ to coincide with the Hausdorff dimension of a directed version of the critical Liouville quantum gravity metric space; a space whose construction still represents a major challenge but which, at least in the critical regime, appears significantly more approachable in light of the results presented in this paper.

\section{Proof strategy}\label{sect:proof-strat}

We now explain how we prove \cref{thm:main_permutations} and \cref{thm:main_graphs}. Readers only interested in a rough idea of the proof strategy can simply browse through \cref{sect:positive-subtrees} and take a look at the diagram in \cref{fig:proof-diagram}. Reading through this entire section will provide a good overview of all main ideas that go into the proof of our main results.

\subsection{Main results for the size of the largest positive subtree of a $p$-signed uniform binary tree}\label{sect:positive-subtrees}

From the description of the law of $\sigma_n$ and $G_n$ in terms of a uniform rooted binary tree $T_n$ with i.i.d.\ Bernoulli($p$) $\oplus/\ominus$ signs on its internal vertices, the convergences of \cref{thm:main_permutations} and \cref{thm:main_graphs} immediately follow if we prove an analogous scaling limit result for the size of the ``largest positive subtree'' of $T_n$, when the trees $(T_n)_{n\geq 1}$ are coupled via the R\'emy algorithm. We are now going to precisely define the ``largest positive subtree'' and introduce the R\'emy algorithm. Then, we will state our main results on the ``largest positive subtree'' of $T_n$.

\begin{figure}[b!]
	\begin{center}
		\includegraphics[width=\textwidth]{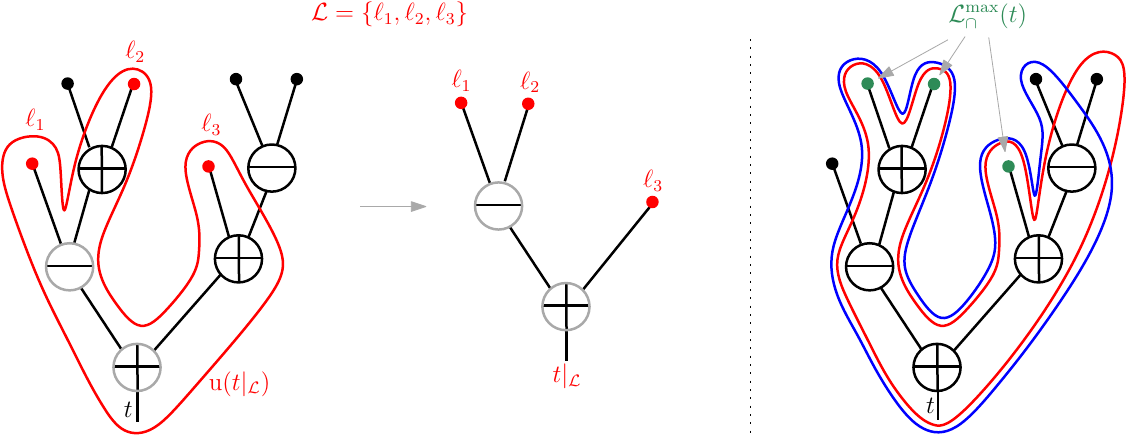}  
		\caption{\textbf{Left:} A sign-decorated binary tree $t$ with a subtree $t|_\mathcal{L}$ induced by the three leaves $\ell_1, \ell_2, \ell_3$. The uncontracted version $\unc(t|_\mathcal{L})$ of $t|_\mathcal{L}$ is highlighted by the red curve on $t$. \textbf{Right:} The same sign-decorated binary tree $t$ with all (\emph{i.e.}\ two) maximal positive (uncontracted) subtrees of $t$ highlighted in red and blue. We colored in green the leaves in  $\mathcal{L}^{\max}_{\cap}(t)$. 
    \label{fig:notation-subtree}}
	\end{center}
	\vspace{-5ex}
\end{figure}

\paragraph{Notation for trees.} We always consider planar rooted (planted) trees, drawn with the root at the bottom. By planted, we mean that the root has an additional edge extending below it, which is not connected to any other vertex; see \cref{fig:notation-subtree}. All trees in this paper are planted (sometimes without explicit mention). The \textbf{degree} $\deg(v)$ of a vertex $v$ is the number of edges incident to $v$ (in particular, provided the tree is not a single leaf, the root of a rooted (planted) binary tree has degree $3$). 
The vertices of a tree are of two types: the \textbf{leaves}, which are the vertices of degree one, and the \textbf{nodes}, which are all the other vertices (\emph{i.e.}\ the internal vertices). 
The \textbf{size} $|t|$ of a tree $t$ is equal to the number of leaves in $t$. 
Moreover, all the \textbf{subtrees} of any tree $t$ considered in this paper are obtained by first selecting a subset $\mathcal{L}$ of the leaves of $t$, then removing from $t$ all the vertices that do not have a descendant in $\mathcal{L}$, and finally contracting all the  vertices of degree 2 that may appear. 
We denote such a subtree by $t|_\mathcal{L}$. 
Moreover, we denote by $\unc(t|_\mathcal{L})$ the \textbf{uncontracted version} of $t|_\mathcal{L}$, \emph{i.e.}\ the tree obtained by only removing from $t$ all the vertices that do not have a descendant in $\mathcal{L}$. Both notions of subtrees are seen as planar rooted (planted) trees, where the root is given by the vertex with the lowest height in the original tree $t$.

A \textbf{sign-decorated tree} $t$ is a  planar rooted (planted) tree such that all nodes are decorated with a $\oplus$ or $\ominus$ sign.\footnote{The signs $\oplus$ or $\ominus$ will be the only signs that we  use to decorate the vertices of all trees considered in this paper, so we will often simply refer to the signs without explicitly saying the $\oplus$ or $\ominus$ signs.} Given a sign-decorated tree $t$ and one of its nodes $v$, we denote by $\sss(v)$ the sign decorating $v$. 
Nodes carrying $\oplus$ (resp.\ $\ominus$) signs are sometimes called positive (resp.\ negative).
A subtree of a sign-decorated tree $t$ is still a sign-decorated tree where all the signs are inherited from $t$.
A sign-decorated tree $t$ is \textbf{positive} if all the nodes are decorated with $\oplus$ signs.
For any finite sign-decorated tree $t$, we write $\lis(t)$ for the \textbf{maximal size} (\emph{i.e.}\ the maximal number of leaves) of a positive subtree of $t$.  A \textbf{maximal positive subtree} of $t$ is a subtree of $t$ of size $\lis(t)$. The notation $\lis(t)$ is motivated by the following fact:  $\lis(t)$ is equal to the size of the longest increasing subsequence of the  permutation obtained from $t$ using the procedure described below \eqref{eq:finite_perm_brownian} and in \cref{fig:patterns-trees}. Similarly, $\lis(t)$ is equal to the size of the largest clique of the  graph obtained from $t$ using the procedure described below \eqref{eq:finite_graph_brownian} and the length of the longest directed path from the source to the sink of the corresponding series-parallel planar map.
Finally, we denote by $\mathcal{L}^{\max}_{\cap}(t)$ the set of leaves in the intersection of \emph{all} maximal positive subtrees of $t$.

\paragraph{The R\'emy algorithm.}\label{sect:remy-algo}
We describe a coupling between the $p$-signed uniform binary trees $T_n=T_n(p)$ for all $n \geq 1$, which will be of fundamental importance in this paper. This coupling is a natural adaptation of the R\'emy algorithm~\cite{remy1985procede} to sign-decorated trees. To build this coupling, we define $T_n$ by induction on $n$ as follows (\emph{cf.}\ \cref{fig:remy}):
\begin{itemize}
    \item The tree $T_1$ is just a single leaf with an additional edge extending below it (recall that our trees are planted). 
    \item Conditionally on $T_1,\dots,T_n$, let $\uedge{n}$ be chosen uniformly at random among the $2n-1$ edges of $T_n$, and let $\usign{n}$ be an independent Bernoulli$(p)$ $\oplus/\ominus$ sign. We build $T_{n+1}$ by adding in the middle of $\uedge{n}$ a new vertex with sign $\usign{n}$, and grafting a new leaf on the left of this vertex with probability $1/2$ or on the right of this vertex with probability $1/2$. We see $T_{n+1}$ as a rooted (planted) tree: when $E_n$ is the special edge extending below the root of $T_n$, the root of $T_{n+1}$ is the new vertex; otherwise it is the same as that of $T_n$.
\end{itemize}
We also note that the transitions of the R\'emy algorithm are even simpler when it is run in reverse: in the above coupling, for any $1<m\leq n$, the law of $T_{m-1}$ conditionally on $\left( T_m, T_{m+1}, \dots, T_n \right)$ is obtained by removing a uniform random leaf $\uleaf{m}$  of $T_m$ and contracting the corresponding node.

\begin{figure}[b!]
	\begin{center}
		\includegraphics[width=\textwidth]{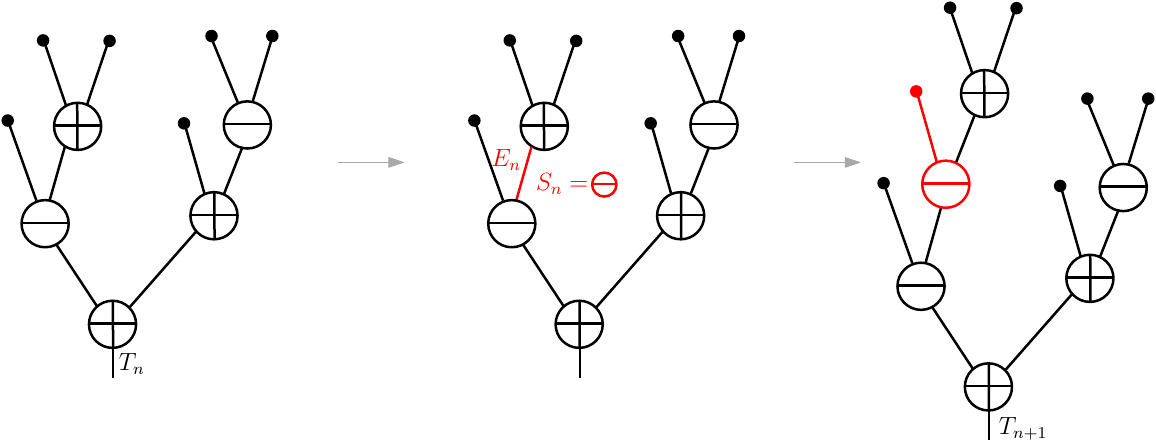}  
		\caption{Our version of the R\'emy algorithm for $p$-signed uniform binary trees. 
        In this realization, we start with a $p$-signed uniform binary  tree $T_n$ with $n$ leaves, then we select a uniform edge $\uedge{n}$, sample a $\ominus$ sign (this happens with probability $1-p$) and finally, we graft a new leaf to the \emph{left} (this happens with probability $1/2$) of the new vertex on $\uedge{n}$ decorated by $\ominus$. \label{fig:remy}}
	\end{center}
	\vspace{-3ex}
\end{figure}

\medskip

We are now ready to state our three main results on $\lis(T_n)$.

\begin{thm}[Existence of the exponent]\label{thm:main}
	For all $p\in(0,1)$, there exists $\alpha=\alpha(p)\in(1/2,1)$ such that we have the convergence in probability
	\begin{equation}\label{eqn:existence-exponent}
    \frac{\log \lis(T_n)}{\log n} \xlongrightarrow[n \to \infty]{\mathbb{P}} \alpha. 
    \end{equation}
    Moreover, 
    \begin{equation}\label{eqn:alpha_as_limsup}
    \alpha = \limsup_{n \to \infty} \frac{\log \Ec{\lis(T_n)}}{\log n}.
    \end{equation}
\end{thm}

The fact that $\alpha(p)\in(1/2,1)$, provided the exponent exists, is a simple consequence of~\cref{thm:prev-bound}. The latter actually yields 
the better bounds $\alpha(p) \in \left(\alpha_*(p),\beta^*(p) \right)$, but the bounds $\alpha(p) \in \left(1/2,1 \right)$ are sufficient for our arguments. Another important note is that the description of $\alpha$ in~\eqref{eqn:alpha_as_limsup} and Markov's inequality provide a quantitative upper bound on the probability that $\LIS(T_n)$ exceeds $n^{\alpha}$ by a large amount. This fact will be helpful later to obtain certain important estimates (\cref{prop:rough_regularity}) used to prove the next main result (more explanations are given in \cref{sect:dynamical-picture}).

\begin{thm}[Determination of the exponent]\label{thm:main2}
	For all $p\in(0,1)$, let $\alpha=\alpha(p)\in(1/2,1)$ be as in \cref{thm:main}. Then  $\alpha$ is the unique solution in the interval $(1/2,1)$ to the equation in~\eqref{eq:eq-to-defn-alpha}, that is,
    \begin{equation*}
        \frac{1}{4^{\frac{1}{2\alpha}}\sqrt{\pi}}\,\frac{\Gamma\big(\tfrac{1}{2}-\tfrac{1}{2\alpha}\big)}{\Gamma\big(1-\tfrac{1}{2\alpha}\big)}=\frac{p}{p-1}.
    \end{equation*}
\end{thm}

Our third main result gives the scaling limit of $\lis(T_n)$. 
Recall that the trees $(T_n)_{n\geq 1}$ are coupled via the R\'emy algorithm.

\begin{thm}[Scaling limit of $\LIS(T_n)$]\label{thm:main3}
For all $p\in(0,1)$, we have the almost sure convergence
	\[ \frac{\lis(T_n)}{n^{\alpha}} \xlongrightarrow[n\rightarrow \infty]{a.s.} X,  \]
where the value $\alpha=\alpha(p)$  is as in Theorems~\ref{thm:main} and~\ref{thm:main2} and the limiting random variable $X=X(p)$ is non-deterministic and almost surely positive and finite. 
\end{thm}

Note that Theorems~\ref{thm:main2}~and~\ref{thm:main3} immediately imply Theorems~\ref{thm:main_permutations}~and~\ref{thm:main_graphs} due to the description of the law of $\sigma_n$ and $G_n$. The only fact that does not immediately follow from the above results is that the random variable $X$ in \cref{thm:main_permutations} (resp.\ \cref{thm:main_graphs}) is a deterministic measurable function of the Brownian separable permuton (resp.\ the Brownian cographon). This fact will be justified in \cref{sect:measurability}.

\medskip

The rest of this section is organized as follows. We first explain our strategy to prove the first main result, that is,  the existence of the exponent $\alpha$ in~\cref{thm:main}. This is done via a supermultiplicativity argument explained in~\cref{sect:supermultiplicativity-expl}.
Then, in a second phase, we determine the exact value of $\alpha(p)$ by showing that $\alpha(p)$ is the unique solution to the equation in~\eqref{eq:eq-to-defn-alpha} in the interval $(1/2,1)$. The latter step is the most involved part of the paper and will be explained in Sections~\ref{sect:heuristic-calculation}~and~\ref{sect:dynamical-picture}.
At the end of~\cref{sect:dynamical-picture}, we will also explain how we prove our last main result, that is, the scaling limit in~\cref{thm:main3}.

We invite the reader to compare the explanations in the next sections with the diagram in~\cref{fig:proof-diagram} which should be helpful to visualize the structure of the full proof.

\medskip

\emph{Important note}. From now on, we consider $p\in (0,1)$ as a fixed parameter and omit its dependency in our notation. We note that our constants may still depend on $p$.

\begin{figure}[p]
	\begin{center}
		\includegraphics[width=\textwidth]{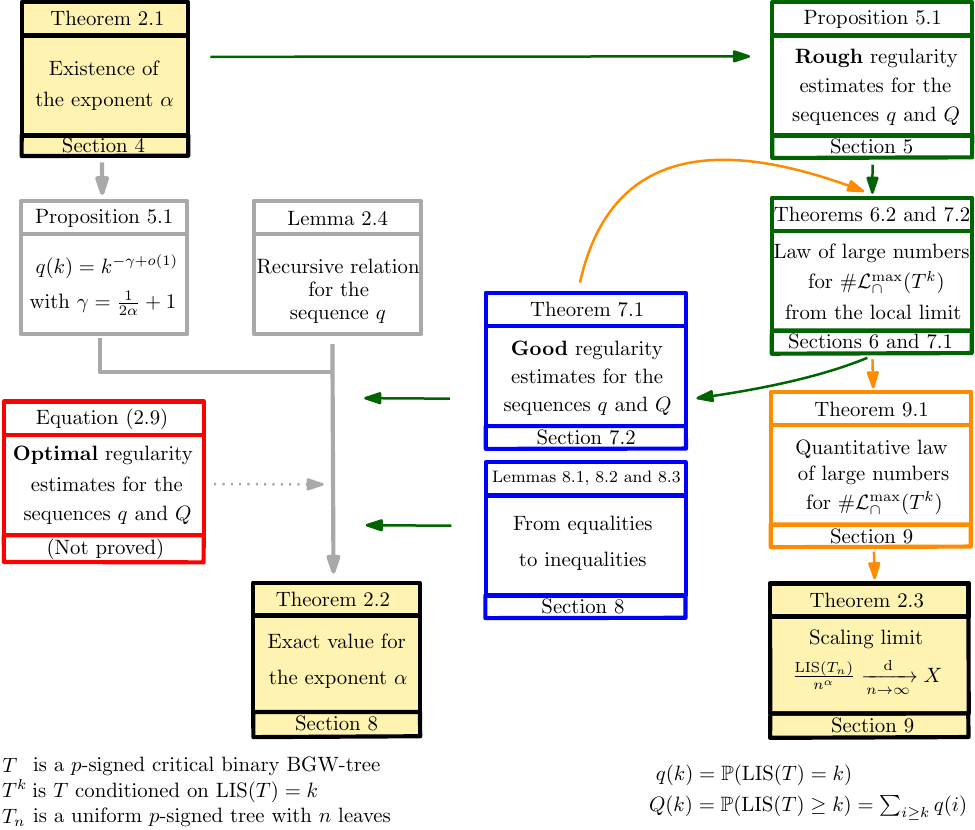} 
		\caption{
        A diagram for the strategy of the proof of the three main results of this paper, \emph{i.e.}\ Theorems~\ref{thm:main},~\ref{thm:main2}~and~\ref{thm:main3}. The existence of the exponent $\alpha$ (\cref{thm:main}) is proved in \cref{sect:supermultiplicativity} via a supermultiplicativity argument. 
        \textbf{The gray path:} It is relatively simple to deduce from the existence of $\alpha$ that $q(k)=\Pp{\LIS(T)=k}=k^{-\frac{1}{2\alpha}-1+o(1)}$ (\cref{prop:rough_regularity}). To determine the exact value of $\alpha$ (\cref{thm:main2}), we first find an appropriate recursive relation (\cref{lem:q_rel}) for the sequence $q(k)$. 
        Assuming certain optimal regularity estimates~\eqref{eq:dreamed asymptotics}, one can determine the value of $\alpha$ from the recursive relation for $q(k)$; but unfortunately such optimal regularity estimates are out of reach with our methods. \textbf{The green path:} To circumvent this issue, in \cref{subsec:rough regularity}, we first deduce from the existence of the exponent $\alpha$ certain rough regularity estimates (\cref{prop:rough_regularity}). Then, from these estimates, we obtain in \cref{sect:loacal-lim} a local convergence result for a maximal positive subtree seen from a uniformly chosen leaf in the related tree model $T^k$ (\cref{thm:localconv}) and in \cref{sect-lln-int} a corresponding law of large numbers for $\# \mathcal{L}^{\max}_{\cap}(T^k)$ (\cref{thm:convergence_size_intersection}). Finally, in \cref{sect-bett-reg-q-Q}, we obtain from this law of large numbers certain good regularity estimates (\cref{thm:good_regularity}). 
        The latter estimates, combined with a new technique which converts the recursive relation for $q(k)$ into certain auxiliary inequalities (Lemmas~\ref{lem:comparison q hat q},~\ref{lem:asymptotics q hat q}~and~\ref{lem:monotonicity}), allow us to determine $\alpha$ in \cref{sec:value exponent}. \textbf{The orange path:} Starting from the good regularity estimates and running again part of the proof of the local convergence theorem, we  obtain in \cref{sec:sc_limit_Xn} a quantitative version of the law of large numbers for $\# \mathcal{L}^{\max}_{\cap}(T^k)$ (\cref{thm:quantitative_moments}) and finally deduce  our scaling limit result for $X_n=\LIS(T_n)$ (\cref{thm:main3}). 
        \label{fig:proof-diagram}}
	\end{center}
	\vspace{-3ex}
\end{figure}

\subsection{Existence of the exponent $\alpha$ via supermultiplicativity}\label{sect:supermultiplicativity-expl}

\medskip

We introduce the compact notation 
\begin{equation} \label{eq:def_X_X_m}
    X_n\coloneqq\lis(T_n).
\end{equation}
Our first objective will be to show the existence of the exponent $\alpha$, proving \cref{thm:main}. 
This is done in \cref{sect:supermultiplicativity} using a supermultiplicativity argument. 
The basic idea here would be to decompose the trees $T_n$ in some subregions to obtain a supermultiplicativity equation for the random variables $X_n$ of the form 
\begin{equation}\label{eq:rec-rell-dom}
        X_{m \cdot n} \succeq \sum_{i=1}^{X_m}X_n^{(i)},
\end{equation}
where the $X_n^{(i)}$ are i.i.d.\ copies of $X_n$ and are independent of $X_m$. 
Here $\succeq$ denotes stochastic domination (see \cref{sect:not-proba} for further details if necessary). This domination could be interpreted as a ``super-branching property'' for the sequence $(X_{2^k})_{k\geq 0}$. Comparison with branching processes could then be applied to get that $(X_{2^k})^{\frac{1}{k}}$ converges in probability to some constant, and hence by monotonicity we would have
\[ 
\frac{\log X_n}{\log n }
\xrightarrow[n \to +\infty]{\mathbb{P}} \alpha. \]
The issue is that we will not be able to write an equation as simple as~\eqref{eq:rec-rell-dom} because most sensible ways to decompose uniform binary trees into several regions result in regions with \emph{random} and \emph{different} sizes (\cref{lem:decomp-tree2}). This will have the consequence that our supermultiplicativity equation (see \eqref{eqn:superbranching_v3}) will be more complicated than~\eqref{eq:rec-rell-dom}. However, we can exactly describe the joint law of such random regions (\cref{lem:decomp-tree2}). This will enable us to compare our supermultiplicativity equation with a more standard one (see~\eqref{eqn:stochastic_domination2}) and to prove \cref{thm:main}.

\subsection{The sequences $q$ and $Q$ and a heuristic derivation of the exponent $\alpha$}\label{sect:heuristic-calculation}

The next step is to determine the exact value of $\alpha$ by showing that $\alpha$ is the unique solution to the equation~\eqref{eq:eq-to-defn-alpha} in the interval $(1/2,1)$. 
For the remainder of~\cref{sect:proof-strat}, we omit most technical details and make frequent use of the symbol $\approx$ to indicate that two quantities behave similarly, in a deliberately imprecise sense.

\subsubsection{The sequences $q$ and $Q$}

Here, our first idea is to randomize the size of $T_n$ by introducing a critical binary Bienaymé--Galton--Watson tree $T$ with i.i.d.\ Bernoulli($p$) signs on its nodes (sampled independently of the tree), so that $T$ conditioned to have size $n$ is distributed as $T_n$.
We introduce the two fundamental sequences $q=(q(k))_{k\geq 1}$ and $Q=(Q(k))_{k\geq 1}$, defined for all $k \geq 1$ by
\begin{equation} \label{eq:def_q_Q}
    q(k)\coloneqq\Pp{\lis(T)=k} \qquad \mbox{and} \qquad Q(k)\coloneqq\Pp{\lis(T) \geq k}=\sum_{i \geq k} q(i).
\end{equation} 
We first explain intuitively why the sequences $q$ and $Q$ are related to the exponent $\alpha$ that we aim to compute. 
Write $Q(k)=\sum_{n\geq 1} \Pp{|T|=n} \cdot \Pp{\LIS(T_n)\geq k}$. 
Using Theorem~\ref{thm:main}, we get that $\Pp{\LIS(T_n)\geq k}$ is very small if $k\gg n^\alpha$ and very close to $1$ if $k\ll n^\alpha$. 
This suggests that 
\[Q(k)\approx \Pp{|T|\geq k^{\frac{1}{\alpha}}} \approx k^{-\frac{1}{2\alpha}},\] 
using the standard fact that $\Pp{|T|\geq n}\approx n^{-\frac{1}{2}}$ (\cref{lem:prelim_GW}). 
A rigorous version of this argument (\cref{prop:rough_regularity}) yields
\begin{equation}\label{eq:bigQrel}
        q(k)=k^{-\frac{1}{2\alpha}-1+o(1)}\qquad \text{ and }\qquad Q(k)= k^{-\frac{1}{2\alpha} +o(1)} \qquad\text{as }k \to \infty.
\end{equation}
To simplify the notation, it will be convenient to set
\begin{equation}\label{eq:gamma-alpha-rel}
\gamma \coloneqq  \frac{1}{2\alpha}+1\in \Big(\frac{3}{2},2\Big),
\end{equation}
so that $q(k)=k^{-\gamma+o(1)}$ and $Q(k)=k^{-\gamma+1+o(1)}$ by \eqref{eq:bigQrel}. To determine $\gamma$ (or equivalently $\alpha$), we write a recursion for the sequence $q$ by decomposing the critical binary Bienaymé--Galton--Watson tree $T$ at the root. We obtain the following equation.

\begin{lem}[Recursive relation for the sequence $q$]\label{lem:q_rel}
Let $q(k)$ be as in \eqref{eq:def_q_Q}. Then for $k\geq 2$,
\begin{align}\label{eq:q_rel}
	q(k)
	&= (1-p) q(k) \sum_{i=1}^{k-1} q(i)  
	+p  \sum_{i=1}^{(k-1)/2} q(i) q(k-i) + \frac{1-p}{2} q(k)^2 +\mathds{1}_{\{k \,\mathrm{ even}\}}\frac p 2 q(k/2)^2.
\end{align}
\end{lem}

\begin{proof}
    Decompose the critical binary Bienaymé--Galton--Watson tree $T$ at the root. 
    The four summands correspond to the following four cases: the root (which is necessarily a node because $k\geq 2$)
    is decorated by a $\ominus/\oplus$ sign and the left and right subtrees have maximal positive subtrees of different/equal sizes.
\end{proof}

Note that one can explicitly compute that $q(1)=\Pp{\LIS(T)=1}=(1+\sqrt{p})^{-1}$ and then use \eqref{eq:q_rel} to numerically compute the first few values of the sequence $q$. 
Unfortunately, although the recursive relation in~\eqref{eq:q_rel} entirely determines the sequence $q$, it is very sensitive to initial conditions, see for instance~\cref{fig:sequences}. 
As a result, it seems quite hard to deduce anything accurate on the sequence $q$ directly from~\eqref{eq:q_rel}, without any additional input.

 \begin{figure}[b!]
	\begin{center}
		\includegraphics[width=0.6\textwidth]{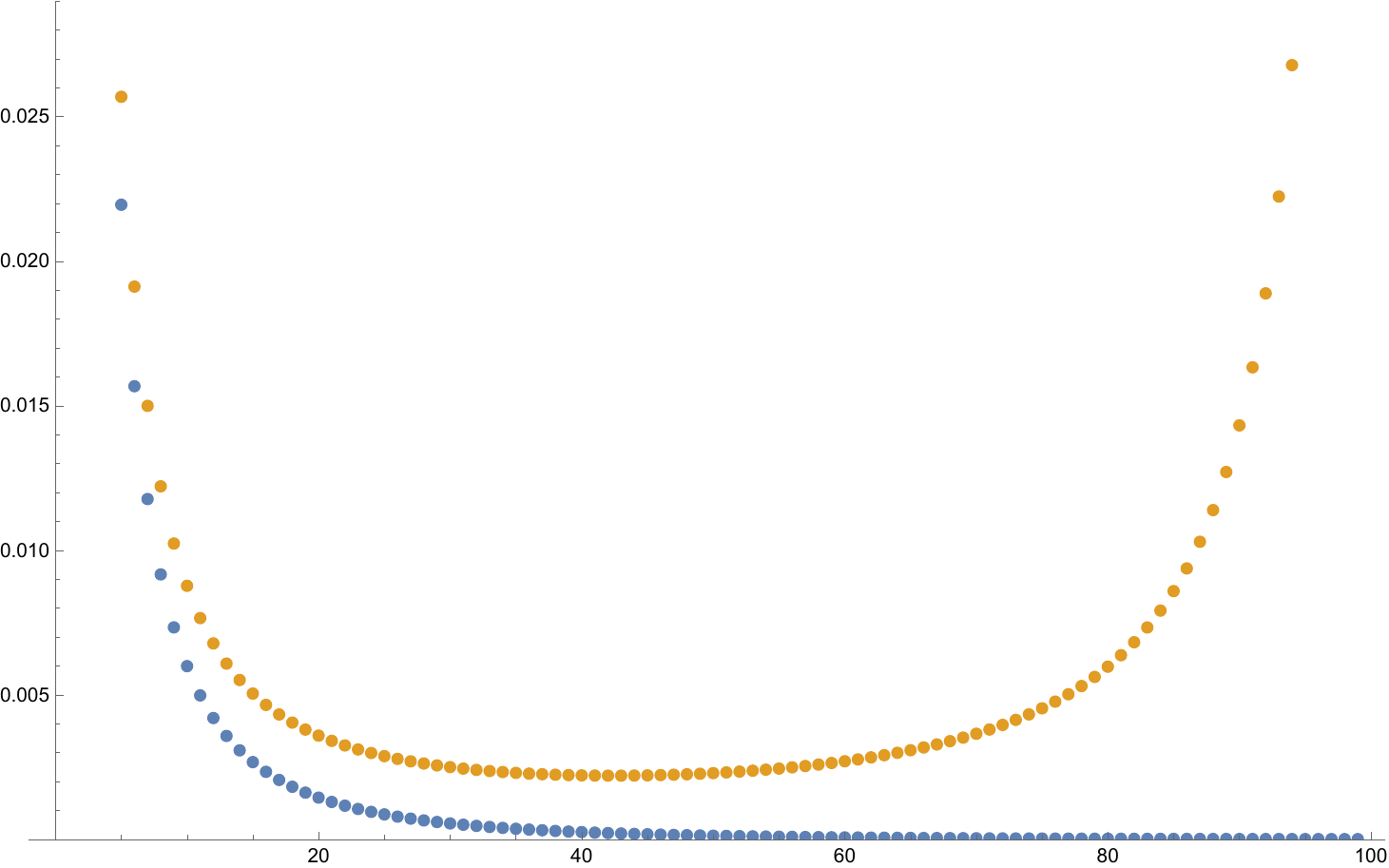}  
		\caption{We fix $p=1/2$. In blue, the first 100 values of the sequence $q(k)$ computed using the recursive relation in~\eqref{eq:q_rel} with initial condition $q(1)=2-\sqrt{2}$. In orange, the first 100 values of the sequence $\widetilde{q}(k)$ computed using the same recursive relation in~\eqref{eq:q_rel} with initial condition $\widetilde{q}(1)=2-\sqrt{2}+0.005$. The first values of each sequence are not shown because they are too large.} \label{fig:sequences}
	\end{center}
	\vspace{-3ex}
\end{figure}

\subsubsection{A heuristic calculation for the exponent $\alpha$ assuming optimal regularity estimates}

We now explain how one could determine the value of $\gamma$ assuming the following \textbf{optimal regularity estimates} about the sequences $q$ and $Q$ (which we will never be able to establish): For some constant $A>0$, assume
\begin{align}\label{eq:dreamed asymptotics}
	Q(k) \sim \frac{A}{\gamma -1}\cdot k^{-\gamma +1}, \qquad q(k)\sim A k ^{-\gamma}, \qquad q(k-1)- q(k) \sim \gamma A k^{-\gamma -1} \qquad \text{as } k\rightarrow \infty.
\end{align}
Rearranging the terms of~\eqref{eq:q_rel}, we get
\begin{align*}
	q(k) \left(1 + \frac{1-p}{p}Q(k)\right) = \left(\sum_{i=1}^{(k-1)/2} q(i) q(k-i) + \frac{1-p}{2 p} q(k)^2 +\mathds{1}_{\{k\text{ even}\}}\frac 1 2 q(k/2)^2\right),
\end{align*}
so that 
\begin{align}\label{eq:LHS=RHS q(k)}
	\frac{1-p}{p}Q(k) q(k)   =  \left(\sum_{i=1}^{(k-1)/2} q(i) q(k-i) \right) - q(k) + \left(\frac{1-p}{2 p} q(k)^2 +\mathds{1}_{\{k\text{ even}\}}\frac 1 2 q(k/2)^2\right).
\end{align}
Then we could determine the value of $\gamma$ by inserting the asymptotics~\eqref{eq:dreamed asymptotics} into the equation \eqref{eq:LHS=RHS q(k)} and equating their leading term as $k\rightarrow \infty$. 
First, the left-hand side would be
\begin{align*}
	\frac{1-p}{p}Q(k) q(k) =\frac{1-p}{p}\frac{A^2}{\gamma-1} k^{-2\gamma +1}  + o(k^{-2\gamma +1}).
\end{align*}
Note already that some terms on the right-hand side will not contribute to the leading term of order $k^{-2\gamma +1}$, because
\begin{align*}
	\frac{1-p}{2 p} q(k)^2 +\mathds{1}_{\{k\text{ even}\}}\frac 1 2 q(k/2)^2 = O(k^{-2\gamma}) =  o(k^{-2\gamma +1}).
\end{align*}
By Abel summation, the sum on the right-hand side of~\eqref{eq:LHS=RHS q(k)} can be rewritten as
\begin{align*}
	\sum_{i=1}^{(k-1)/2} q(i) q(k-i) 
	&=
	\sum_{i=1}^{(k-1)/2} Q(i) q(k-i) - \sum_{i=2}^{(k+1)/2} Q(i) q(k-i+1)\\
	&= q(k-1)
	-Q\Big(\Big\lfloor \frac{k+1}{2}\Big\rfloor\Big) q\Big(\Big\lceil \frac{k+1}{2}\Big\rceil\Big)  + \sum_{i=2}^{(k-1)/2} Q(i) (q(k-i)-q(k-i+1)).
\end{align*}
We use this to express the right-hand side of~\eqref{eq:LHS=RHS q(k)} and 
allow ourselves to replace any of the terms $q(k)$, $Q(k)$ by their asymptotics given in \eqref{eq:dreamed asymptotics} and neglect any term that would be $o(k^{-2\gamma +1})$.  
We get (the entire argument could be turned into an actual proof with some additional details that we spare to the reader)
\begin{align*}
	& \left( \sum_{i=1}^{(k-1)/2} q(i) q(k-i) \right) - q(k) \\
	&\approx q(k-1)- q(k) - \frac{A}{\gamma -1} \left(\frac{k}{2}\right)^{-\gamma +1} \cdot A \left(\frac{k}{2}\right)^{-\gamma}
	+ \sum_{i=2}^{(k-1)/2}  \frac{A}{\gamma -1} i^{-\gamma+1} \cdot \gamma A (k-i)^{-\gamma -1}\\ 
	&\approx - \frac{A^2 \cdot 2^{2\gamma -1}}{\gamma -1}\cdot k^{-2\gamma +1}+ \frac{\gamma A^2}{\gamma-1}\cdot k^{-2\gamma +1}\cdot \sum_{i=2}^{(k-1)/2}   \frac{1}{k}\cdot \left(\frac{i}{k}\right)^{-\gamma+1}  \left(1-\frac{i}{k}\right)^{-\gamma -1} \\
	&\approx \frac{A^2}{\gamma -1} \cdot k^{-2\gamma +1} \cdot \left(-2^{2\gamma-1} + \gamma \int_0^{\frac{1}{2}}x^{-\gamma +1} (1-x)^{-\gamma -1} \dd x \right),
\end{align*}
where the integral comes as the limit of a Riemann sum. Equating the leading term on the left-hand side and right-hand side of \eqref{eq:LHS=RHS q(k)} would ensure that 
\begin{align}\label{eq:relation gamma p}
	\frac{1-p}{p} = -2^{2\gamma -1} + \gamma\int_0^{\frac{1}{2}}x^{-\gamma +1} (1-x)^{-\gamma -1} \dd x,
\end{align}
which after the substitution $\alpha=\tfrac{1}{2(\gamma-1)}$ from \eqref{eq:gamma-alpha-rel} and some manipulations\footnote{We will formally do these computations in the second part of the proof of \cref{lem:monotonicity}~on~page~\pageref{eq:neg-deriv}.} would lead to the equation in~\eqref{eq:eq-to-defn-alpha}.

\begin{remark} 
	It is possible to run a different version of the above argument and identify the exponent $\gamma$ (and $\alpha$) from a set of assumptions that are weaker than \eqref{eq:dreamed asymptotics}, namely that there exists a slowly varying function $\svfq$ such that
	\begin{align*}
		q(k)=k^{-\gamma}\svfq(k) \qquad \text{and that} \qquad \frac{q(k+1)}{q(k)}=1 + O\left(\frac{1}{k}\right).
	\end{align*}
	We will be able to prove that the first asymptotic holds thanks to \cref{thm:good_regularity} below, but the second one will remain out of reach using our methods. Hence, we will need a different strategy explained in the next subsection.
\end{remark}

\subsection{Our proof strategy: good but not optimal regularity estimates}\label{sect:dynamical-picture}

We now describe the strategy used in this paper.
Most of the paper will be devoted to proving some \textbf{good regularity estimates} for the sequences $q$ and $Q$. Specifically,  we show that there exist slowly varying functions $\svfq$ and $\svfbQ$ (see \cref{sect:not-proba} for definitions) such that 
\begin{align}\label{eq:good regularity estimates}
        q(k)=k^{-\frac{1}{2\alpha}-1} \svfq(k)
        \qquad \text{and}\qquad 
        Q(k)=k^{-\frac{1}{2\alpha}} \svfbQ(k),
\end{align}
where $\alpha$ is the exponent of \cref{thm:main}. 
Although these estimates are not as precise  as~\eqref{eq:dreamed asymptotics}, they will still be sufficient to conclude with some additional effort.

\medskip

The rest of this section is dedicated to outlining the key ideas behind the structure of this paper. 

\paragraph{The dynamical picture.}
The control \eqref{eq:good regularity estimates} on the sequences $q$ and $Q$ will come from understanding a dynamical version of the problem. 
We will work on a probability space where the whole sequence $(T_n)_{n\geq 1}$ is coupled under Rémy's algorithm described in Section~\ref{sect:positive-subtrees}: 
under this construction, conditionally on $(T_i)_{i\geq n}$,
the tree $T_{n-1}$ can be obtained by removing a uniform leaf $\uleaf{n}$ from $T_{n}$ and contracting the corresponding node. 
This ensures that the corresponding sequence $(X_n)_{n\geq 1}$ defined in \eqref{eq:def_X_X_m} is non-decreasing and changes at each step by either $0$ or $1$.
We denote by $\sttm_k$ the first time that $X_n$ attains any number $k\geq 1$.  
Note that studying the asymptotic behavior of the process $(\sttm_k)_{k\geq 1}$ is equivalent to studying that of $(X_n)_{n\geq 1}$
since
 \begin{equation}\label{eq:sttm-X-rel}
     X_{n}=k\qquad \text{for all } n\in \intervalleentier{\sttm_{k}}{\sttm_{k+1}-1}.
 \end{equation}
It just happens that the sequences $q$ and $Q$ are more naturally controlled in terms of the stopping times $(\sttm_k)_{k\geq 1}$, as we can see below.

\paragraph{How the sequence $(\sttm_k)_{k\geq 1}$ is related to $q$ and $Q$.}
Let the $p$-signed critical binary Bienaymé--Galton--Watson tree $T$ be defined on the same probability space as, and independent of, the sequence $(T_n)_{n\geq 1}$.
Since $T$ conditioned to have size $n$ is distributed as $T_n$, we have that $\LIS(T) \overset{\dist}{=} X_{|T|}$. 
Combining this with the asymptotics $\Pp{|T| \geq n}\approx c n^{-\frac{1}{2}}$ that holds for some constant $c>0$ (\cref{lem:prelim_GW}), we get that
\begin{equation}\label{eq:Q-to-tau}
	Q(k) =\Pp{\LIS(T)\geq k}= \Pp{X_{|T|} \geq k} \stackrel{\eqref{eq:sttm-X-rel}}{=} \Pp{|T| \geq \sttm_k}
	 \approx c \cdot \Ec{\sttm_k^{-1/2}}.
\end{equation}
Similarly, we have $q(k)=Q(k)-Q(k+1) \approx c \cdot \Ec{\sttm_k^{-1/2} -\sttm_{k+1}^{-1/2}}$.

\paragraph{Regularity of $q$, $Q$ and $(\sttm_k)_{k\geq 1}$.}
By Karamata's characterization theorem~\cite[Theorem~1.4.1]{bingham1989regular} for slowly varying functions, in order to prove the estimates \eqref{eq:good regularity estimates}, it is enough to show that, for any $\svx \geq 1$,  we have $\frac{Q(\svx k)}{Q(k)}\to \svx^{-\frac{1}{2\alpha}}$ as $k\to \infty$ and the analogous estimate for $q$. 
Our objective is hence to control the ratios $\frac{q(\svx k)}{q(k)}$ and $\frac{Q(\svx k)}{Q(k)}$. 
From the previous paragraph, this could be obtained by showing that in some sense, when $k$ is large, we have
\begin{equation}\label{eq:goal-to-reach}
    \sttm_{\lceil \svx k \rceil } \approx \svx^{\frac{1}{\alpha}}\sttm_k.
\end{equation}
We start by controlling what happens between the two consecutive increase times $\tau_k$ and $\tau_{k+1}$.

\paragraph{Controlling the time $(\sttm_{k+1}-\sttm_k)$
between successive increases of $X$.}
First note that for a given $n\geq 2$, 
we have $X_{n-1}=X_{n}-1$ if, and only if, the leaf $L_n$ removed from $T_{n}$ to obtain $T_{n-1}$ belongs to the set $\mathcal{L}_{\cap}^{\max}(T_{n})$ of leaves that are in \emph{all} positive subtrees of $T_{n}$ of maximal size. 
Hence, we can write
\begin{align}\label{eq:increase of X in remy algo intro}
X_{n-1}	= X_{n} -  \indicator{L_{n}\in \mathcal{L}_{\cap}^{\max}(T_{n})},
\end{align}
where conditionally on $(T_{i})_{i\geq n}$, the probability of the event $\{L_{n}\in \mathcal{L}_{\cap}^{\max}(T_{n})\}$ is $\frac{1}{n} \# \mathcal{L}_{\cap}^{\max}(T_{n})$.

Now let us apply this argument in a cascading manner: we start at time $n = \tau_{k+1} - 1$, so that $X_n = k$, and remove one leaf after another from the tree until its $\LIS$ drops below $k$, that is, until we reach $\tau_k - 1$. 
Heuristically, we get for all $j \geq 0$,
\begin{align*}
\Ppsq{\tau_{k+1} - \tau_{k}> j}{T_{\tau_{k+1}}}
&= \Ppsq{L_{\tau_{k+1}-1} \notin \mathcal{L}_{\cap}^{\max}(T_{\tau_{k+1}-1}) ,\dots, L_{\tau_{k+1}-j} \notin \mathcal{L}_{\cap}^{\max}(T_{\tau_{k+1}-j})}{T_{\tau_{k+1}}}\\
&= \Ecsq{\left(1 - \frac{\# \mathcal{L}_{\cap}^{\max}(T_{\tau_{k+1}-1})}{\tau_{k+1}-1}\right) \cdot \ldots \cdot \left(1 - \frac{\# \mathcal{L}_{\cap}^{\max}(T_{\tau_{k+1}-j})}{\tau_{k+1}-j}\right) }{T_{\tau_{k+1}}}\\
&\approx\,\,\, \left(1 - \frac{\# \mathcal{L}_{\cap}^{\max}(T_{\tau_{k+1}-1})}{\tau_{k+1}}\right)^j,
\end{align*}
which suggests the approximation in distribution
\begin{equation} \label{eq:law_tau_k+1_k_geom}
\Law(\tau_{k+1} - \tau_{k} | T_{\tau_{k+1}}) \overset{\mathrm{d}}{\approx} \mathrm{Geom}\left(\frac{\# \mathcal{L}_{\cap}^{\max}(T_{\tau_{k+1}-1})}{\tau_{k+1}}\right).
\end{equation}

\paragraph{A key insight.}
The key insight (justified in the next paragraphs) underlying the strategy that we follow for the rest of the paper is that, in a suitable probabilistic sense, we have
\begin{equation}\label{eq:key insight}
	\#\mathcal{L}_{\cap}^{\max}(T_{n}) \approx \lambda X_{n} \qquad \text{and hence } \qquad	\# \mathcal{L}_{\cap}^{\max}(T_{\tau_{k+1}-1}) \approx \lambda k,
\end{equation}
for some \textit{a priori} unknown $\lambda \in (0,1)$ depending only on $p$.
Hence, from \eqref{eq:law_tau_k+1_k_geom} we can compare the law of $\sttm_{k+1}-\sttm_k$ conditional on $\sttm_{k+1}$ with a geometric random variable of parameter $\frac{\lambda k}{\sttm_{k+1}}$, which has expectation $\frac{\sttm_{k+1}}{\lambda k}$. 
From this, we can deduce that $\Ecsq{\tau_j}{\tau_{j+1}} \approx \tau_{j+1}\left(1-\frac{1}{\lambda j}\right)$. By the Markov property (and assuming some concentration of $\tau_k$ conditionally on $\tau_{\lceil \svx k \rceil}$), for $x \geq 1$, we get 
\begin{align}\label{eq:regularity of tau from key insight}
		\tau_k \approx \Ecsq{\tau_k}{\tau_{\lceil \svx k \rceil}} \approx \tau_{\lceil \svx k \rceil} \cdot \prod_{j=k}^{\lceil \svx k \rceil -1}\left(1-\frac{1}{\lambda j}\right) \approx  \tau_{\lceil \svx k \rceil} \cdot \svx^{-\frac{1}{\lambda}}.
\end{align}
Since from \eqref{eq:bigQrel} we expect that $\frac{Q(xk)}{Q(k)}\approx x^{-\frac{1}{2\alpha}}$ and from \eqref{eq:Q-to-tau} that $Q(k) \approx c \cdot \Ec{\sttm_k^{-1/2}}$, the last display implies in particular that the constant $\lambda$ is actually equal to $\alpha$ (see \cref{rmk:lambda=alpha}).
Moreover, recalling \eqref{eq:goal-to-reach}, this is exactly what we need to derive the regularity estimates for $q$ and $Q$. 
The justification for the key insight in \eqref{eq:key insight} will be developed in the next paragraphs.

\paragraph{A first rough estimate.}
It is possible to use the above line of reasoning using the inequality $\#\mathcal{L}_{\cap}^{\max}(T_{n}) \leq X_{n}$, valid deterministically, instead of the more precise \eqref{eq:key insight}.
This (together with a symmetric argument with time going forwards) 
leads to the following estimates that hold for some constant $C>0$, for all $k\geq 1$ and all $\svx\geq 1$:
\begin{align}\label{eq:rough regularity intro}
C \svx^{-C} \leq \frac{q(\svx k)}{q(k)} \leq C \qquad \text{and} \qquad  (1+o_{k\rightarrow\infty}(1)) \svx^{-C} \leq \frac{Q(\svx k)}{Q(k)} \leq 1. 
\end{align}
These, together with the asymptotics \eqref{eq:bigQrel}, are referred to as the \textbf{rough regularity estimates} and constitute the content of \cref{prop:rough_regularity}, proved in Section~\ref{subsec:rough regularity}.
These estimates are not precise enough to obtain \eqref{eq:good regularity estimates}, nor to extract the exponent $\gamma$ from \eqref{eq:q_rel}. 
However, they will still prove useful in the next part of the argument, which is itself used to prove \eqref{eq:key insight}.

\paragraph{A local convergence argument.} In order to get \eqref{eq:key insight}, we rely on a local convergence argument developed in \cref{sect:loacal-lim}.
There, we work with a slightly different tree model, denoted by $T^k$, which has the distribution of the $p$-signed critical binary Bienaymé--Galton--Watson tree $T$ conditioned on the event $\{\LIS(T)=k\}$.\footnote{Beware that $\TLis{k}$ and $T_n$ stand for two very different conditionings; to avoid any potential risk of confusion we shall always stick to $k$ for the conditioning on the LIS and $n$ for that on the number of leaves.}
More precisely, for any finite sign-decorated binary tree $t$, we write $t^{\lmax}$ for the \emph{leftmost} maximal positive subtree of $t$. This is the maximal positive subtree of $t$ informally defined by going left at any negative node $v$ where the left and right subtrees above $v$ have equal $\LIS$. We will study the local convergence of $T^k$ around a leaf $L^k$ chosen uniformly at random among the $k$ leaves of $(T^k)^{\lmax}$; see \cref{thm:localconv} and the discussion below it.

The main tool that we use is some \emph{exploration process} that records (among other things) the LIS of the subtree above node $v$, as $v$ travels along the spine of the tree from the root to the marked leaf $L^k$. 
It turns out that this process is a Markov chain and this is precisely our reason for studying the conditioned model $T^k$ with fixed LIS instead of $T_n$.
 
The proof of local convergence relies on the construction of a coupling between the exploration processes associated with $T^k$ and $T^{k'}$, for different large values of $k$ and $k'$, ensuring that the two processes coalesce very fast as $k,k'\to\infty$.
Since the transitions of the Markov chain are given in terms of the sequence $q(k)$, the rough regularity estimates \eqref{eq:rough regularity intro} play an important role in this proof.

\paragraph{Motivation for the local convergence.}
The point of the former result is to prove that the probability that the uniform leaf $L^k$ belongs to $\mathcal{L}_{\cap}^{\max}(T^k)$ tends to some deterministic limit $\lambda\in \intervalleoo{0}{1}$, which can be interpreted as the probability of some event happening in the local limit.
This in turn ensures that, as $k\rightarrow \infty$, 
\begin{align}\label{eq:convergence size intersection intro}
\frac{\#\mathcal{L}_{\cap}^{\max}(T^k)}{\LIS(T^k)} = \frac{\#\mathcal{L}_{\cap}^{\max}(T^k)}{k}  \overset{\bbP}{\longrightarrow} \lambda.
\end{align}
This is the content of Theorem~\ref{thm:convergence_size_intersection} proved in~\cref{sect-lln-int}.
With some additional work, it is possible to transfer this result to our original model $T_n$, which has the same distribution as $T$ conditionally on $\{|T|=n\}$. 
This ensures that for any large $n\geq 1$, we have 
\begin{align}\label{eq:convergence proportion of }
	\#\mathcal{L}_{\cap}^{\max}(T_{n})\approx \lambda \LIS(T_n) = \lambda X_n,
\end{align}
obtaining the key insight from \eqref{eq:key insight}. 
While the existence of the local limit is motivated by the concentration result in~\eqref{eq:convergence size intersection intro}, it also provides valuable insights into the landscape around a uniformly chosen leaf in a maximal positive subtree.

\paragraph{Obtaining the good regularity estimates and the exponent $\alpha$.}
To summarize, we used the rough regularity estimate \eqref{eq:rough regularity intro}, obtained earlier, to derive \eqref{eq:convergence proportion of }.
This, in turn, allows us to rigorously justify the arguments leading to
\eqref{eq:regularity of tau from key insight}, and thereby obtain the \textbf{good regularity estimates}~in~\eqref{eq:good regularity estimates} for the sequences $q$ and $Q$.
This last step actually turns out to be quite challenging and will
take up the whole of~\cref{sect-bett-reg-q-Q}. 

Once \eqref{eq:good regularity estimates} is established, we will need a new technique in \cref{sec:value exponent} to extract the value of the exponent~$\alpha$. 
Here the main idea is to 
compare the recursive relation for $q$ from \cref{lem:q_rel} with two related \emph{inequalities}, each satisfied by a more regular test sequence.
These test sequences, combined with \eqref{eq:good regularity estimates}, will then yield sharp enough upper and lower bounds for the sequence~$q$ to extract the exact exponent $\alpha$.

\paragraph{As a by-product: the scaling limit result.}
Using as an input the newly obtained good regularity estimates \eqref{eq:good regularity estimates}, we can revisit the local convergence argument and obtain in \cref{thm:quantitative_moments} a more quantitative version of \eqref{eq:convergence size intersection intro}. 
We find that there exists $\varepsilon>0$ such that for $k$ large we have 
\begin{align}\label{eq:convergence size intersection quantitative intro}
	\Ec{|\#\mathcal{L}_{\cap}^{\max}(T^k) - \alpha k|} \leq  k^{-\varepsilon}.
\end{align}
Going back to \eqref{eq:increase of X in remy algo intro} and assuming that \eqref{eq:convergence size intersection quantitative intro} can be leveraged to get some quantitative control over $|\#\mathcal{L}_{\cap}^{\max}(T_n) - \alpha X_n|$, we get that roughly 
\begin{align*}
	\Ecsq{X_{n-1}}{(T_{i})_{i\geq n}}= X_{n} - \frac{\#\mathcal{L}_{\cap}^{\max}(T_{n})}{n}
	 \approx X_{n} - \frac{\alpha X_{n}}{n} \approx \left(1-\frac{\alpha}{n}\right) X_{n}.
\end{align*}
If the last display held with an equal sign, it would be easy by some (backward) martingale argument to show that $n^{-\alpha} X_n$ converges almost surely to some random variable.
The proof of \cref{thm:main3} in \cref{sec:sc_limit_Xn} follows this idea, by writing $n^{-\alpha} X_n$ as some (backward) martingale part plus some error term, and showing that the estimates that we have are sufficient to control the error term. 
Some additional work will then be needed to ensure that all the properties of the limit $\lrv$ claimed in \cref{thm:main3} hold.

\paragraph{Another by-product: scaling limit of largest positive subtrees of $T^k$.}
Let us finally mention that another by-product of our proof is a scaling limit result for the shape of  
the leftmost maximal positive subtree of $T^k$, which complements the local convergence mentioned before (\cref{thm:localconv}).
If $t$ is a finite sign-decorated binary tree, we recall that $t^{\lmax}$ stands for its leftmost maximal positive subtree. We can equip it with the metric coming from the restriction to $t^{\lmax}$ of the graph distance on $t$ (equivalently, this is the graph distance on the uncontracted version of $t^{\lmax}$). If $a>0$ and $S$ is a metric space, we also denote by $a \cdot S$ the metric space obtained by multiplying the distances in $S$ by a factor $a$.
\begin{thm}[Scaling limit of the leftmost positive subtree of $T^k$]\label{thm:fragmentation_tree}
    Let $\alpha$ be the unique solution in the interval $(1/2,1)$ to the equation in~\eqref{eq:eq-to-defn-alpha} and recall that $\gamma-1\stackrel{\eqref{eq:gamma-alpha-rel}}{=} \frac{1}{2\alpha}\in (1/2,1)$. 
    Recall also the slowly varying function $\varphi$ of \eqref{eq:good regularity estimates}.
    Then the following convergence in distribution holds for the Gromov--Hausdorff topology: 
    \[\frac{\varphi(k)}{k^{\gamma -1}} \cdot (T^k)^{\lmax} \xrightarrow[k \to \infty]{\mathrm{d}} \mathcal{T}_{\gamma},\]
    where $\mathcal{T}_{\gamma}$ is the self-similar fragmentation tree with self-similarity index $1-\gamma$ and with binary dislocation measure $\nu_{\gamma}$ given by
    \begin{equation}\label{eqn:dislocation_measure}
        \int f(s_1, s_2) \nu_{\gamma} (\mathrm{d}\mathbf{s}) = \int_0^{1/2} f(1-x,x) \, x^{-\gamma} (1-x)^{-\gamma} \, \mathrm{d}x,
    \end{equation}
    for any measurable function $f:[0,1]^2 \to \bbR_{>0}$.
\end{thm}
Since this theorem is not needed anywhere in the proof of the main results, we omit precise definitions of continuous fragmentation trees and refer to~\cite{haasmiermont12fragmentation} instead. We will give a short proof of Theorem~\ref{thm:fragmentation_tree} in Section~\ref{subsec:fragmentation_tree} as an easy consequence of the good regularity estimates on $q$, together with a general convergence criterion established by Haas and Miermont in~\cite{haasmiermont12fragmentation}. Since Theorem~\ref{thm:fragmentation_tree} is the result that comes naturally as a by-product of our arguments, we do not try to show the analogous convergence for $(T_n)^{\lmax}$, which appears to be more challenging.

\section{Notation and preliminary results}\label{sec:def}

\subsection{Basic notation and definitions}\label{sect:not-proba}

We denote by $\N=\{0,1,2,\dots\}$ the set of non-negative integers, and for all real numbers $a < b$, we write the integer interval $[a,b]\cap\mathbb{Z}$ as $\intervalleentier{a}{b}$. 
The set of positive real numbers is denoted by $\bbR_{>0}$, while the set of non-negative real numbers is denoted by $\bbR_{\ge 0}$.
Moreover, we often write $a\wedge b$ for $\min(a,b)$ and $a \vee  b$ for $\max(a,b)$. 

For sums, we use the convention $\sum_{k=x}^y\coloneqq\sum_{k=\lceil x \rceil}^{\lfloor y \rfloor}$, for all $x<y\in\mathbb R_{\geq 0}$. The same convention is naturally extended to the sequences $q=(q(k))$ and $Q=(Q(k))$  defined in~\eqref{eq:def_q_Q} for which we write $Q(x)$ instead of $Q(\lceil x \rceil)$ when $x\ge 1$.

For two real-valued functions $f$ and $g$ with domain $D$, we write
$
f(x)=o(g(x))
$
if
$
\frac{f(x)}{g(x)}\xrightarrow[x\to\infty]{}0.
$
We write
$
f(x)=O(g(x))
$
if there exists a constant $C>0$ such that
$
f(x)\leq C\,g(x)
$
for all $x\in D$.
Moreover, given a parameter $K>0$, we write
$
f(x)=o_K(g(x))
$
if
$
\frac{f(x)}{g(x)}\xrightarrow[x\to\infty]{}0,
$
where the rate of convergence depends on $K$ (but not on $x$).
Similarly, given a parameter $K>0$, we write
$
f(x)=O_K(g(x))
$
if there exists a constant $C_K>0$ depending on $K$ (but not on $x$) such that
$
f(x)\leq C_K\,g(x)
$
for all $x\in D$.

We recall that a \textbf{slowly varying function} is a measurable function $\svfq : \bbR_{>0} \to \bbR_{>0}$ such that for all $\svx>0$, we have $\frac{\svfq(\svx z)}{\svfq(z)} \to 1$ as $z \to \infty$.
A measurable function $\svfq : \bbR_{>0} \to \bbR_{>0}$ is called \textbf{regularly varying} of index $\beta$ if for all $\svx>0$, we have $\frac{\svfq(\svx z)}{\svfq(z)} \to x^{\beta}$ as $z \to \infty$.
A comprehensive reference for the theory of regular variation is~\cite{bingham1989regular}. In particular, we will need the following standard results on slowly varying functions. The first item is the Uniform Convergence Theorem~\cite[Theorem 1.2.1]{bingham1989regular} and the second is a particular case of the Potter bounds~\cite[Theorem 1.5.6, item (ii)]{bingham1989regular}.

\begin{lem}\label{lem:slow_variation}
    Let $\varphi : \bbR_{>0}\to \bbR_{>0}$ be a slowly varying function.
    \begin{enumerate}
        \item\label{item:slow_variation1} For any constant $K>0$, the ratio $\frac{\varphi(a)}{\varphi(b)}$ converges to $1$ as $a,b \to \infty$ uniformly in $a,b$ as long as $K^{-1} \leq \frac{a}{b} \leq K$.
        \item\label{item:slow_variation2} For any $\delta>0$, there is an absolute constant $K=K(\delta)$ such that for all integers $1 \leq a<b$, we have $\frac{\varphi(a)}{\varphi(b)} \leq K \left( \frac{a}{b} \right)^{-\delta}$.
        \end{enumerate}
\end{lem}

For a sequence of real-valued random variables, we write $Y_n\xlongrightarrow[n\to \infty]{\mathbb P}Y$ to denote convergence in probability to the random variable $Y$, we write $Y_n\xlongrightarrow[n\to \infty]{\mathrm{a.s.}}Y$ to denote almost sure convergence, and $Y_n\xlongrightarrow[n\to \infty]{\mathrm{d}}Y$ to denote convergence in distribution. When two random variables $Y$ and $Z$ have the same law, we write $Y\stackrel{\mathrm{d}}{=}Z$.

We also recall that a real-valued random variable $X$ \textbf{stochastically dominates} another real-valued random variable $Y$ if there exists a coupling of $X$ and $Y$ such that $X\geq Y$ a.s., and we write $X \succeq Y$.

\medskip

A summary of the main notation used throughout this paper can be found in the index on page~\pageref{sect:Index-of-notation}.

\subsection{The Brownian separable permuton and the Brownian cographon}\label{sect:brow-perm-cograp}

For completeness, we describe the constructions of the Brownian separable permutons and the Brownian cographons. We refer the reader to \cite{maazoun-separable-permuton,bbfgp-universal,bassino2022random} for more details.

Under $\bbP$, we call \textbf{signed Brownian excursion} a triplet $(\efrak, \sfrak, p)$ consisting of a (normalized) Brownian excursion $\efrak$, together with an independent sequence $\sfrak$ of i.i.d.\ Bernoulli$(p)$ $\oplus/\ominus$ signs, \emph{i.e.}\ $\bbP(\oplus)=1-\bbP(\ominus)=p$. 
One should think of the sequence $\sfrak$ as being indexed by the local minima of $\efrak$.

We define the following random relation $\vartriangleleft_{\efrak, \sfrak, p}$ on $[0,1]$:
conditional on $\efrak$, if $x,y\in[0,1]$ with $x<y$, and if $\min_{[x,y]}\efrak$ is reached at a unique point which is a strict local minimum $m_{x,y}\in(x,y)$, then

\begin{equation}\label{eq:exc_to_perm}
	\begin{cases}
		x\vartriangleleft_{\efrak, \sfrak, p} y &\quad\text{if}\quad \sfrak(m_{x,y})=\oplus,\\
		y\vartriangleleft_{\efrak, \sfrak, p} x &\quad\text{if}\quad \sfrak(m_{x,y})=\ominus.\\
	\end{cases}
\end{equation}
Standard properties of the Brownian excursion (see, for instance, \cite[Lemma 2.1]{bdsg-lis-perm}) ensure the existence of a random subset $ \mcl R_\efrak \subset [0,1]$ such that: 
\begin{itemize}
    \item Almost surely, the complement of $\mcl R_\efrak$ has Hausdorff dimension $1/2$.
    \item For every $x,y\in  \mcl R_\efrak$ with $x<y$, the minimum $\min_{[x,y]}\efrak$ is reached at a unique point which is a strict local minimum.
\end{itemize} 
In particular, the restriction of the relation $\vartriangleleft_{\efrak, \sfrak, p}$ to $\mcl R_\efrak$ is a total order.
Introducing the function
\begin{equation}\label{eq:function_for_sep}
    \forall t\in[0,1], \quad
	\psi_{\efrak, \sfrak, p}(t)\coloneqq\Leb\left( \big\{x\in[0,1]\mid x \vartriangleleft_{\efrak, \sfrak, p} t\big\}\right),
\end{equation}
then the (biased) \textbf{Brownian separable permuton} of parameter $p\in(0,1)$ is the push-forward of the Lebesgue measure on $[0,1]$ via the mapping $(\mathbb{I},\psi_{\efrak, \sfrak, p})$, where $\mathbb{I}$ denotes the identity. That is,
\begin{equation}\label{eq:sep_perm}
	\bm{\mu}_p(\cdot)\coloneqq (\mathbb{I},\psi_{\efrak, \sfrak, p})_{*}\Leb(\cdot)=\Leb\left(\{t\in[0,1] \mid (t,\psi_{\efrak, \sfrak, p}(t))\in \cdot \,\}\right).
\end{equation}
Heuristically speaking, $\psi_{\efrak, \sfrak, p}$ is the \emph{continuum permutation} of the elements in the interval $[0,1]$ induced by the order $\vartriangleleft_{\efrak, \sfrak, p}$ and $ \bm{\mu}_p$ is the diagram of $\psi_{\efrak, \sfrak, p}$. We highlight that  $\bm{\mu}_p$ is a \emph{random} permuton.

\begin{remark}\label{remk:coupled-trees}
Recall from the discussion around \eqref{eq:finite_perm_brownian} that, given $\bm{\mu}_p$, one can sample a corresponding permutation $\sigma_n=\sigma_n(p)=\Perm[\bm{\mu}_p,n]$ of size $n$. We also described the law of $\sigma_n$ in terms of a $p$-signed uniform binary  tree $T_n=T_n(p)$. We now explain how to obtain $T_n$ directly from $\bm{\mu}_p$, or more precisely, how one can couple $T_n$ with the signed excursion $(\efrak, \sfrak, p)$ used in the construction of $\bm{\mu}_p$. As a corollary, we will see that the trees $T_n$ are coupled for different values of $n$ exactly as in the R\'emy algorithm from \cref{sect:remy-algo}.

For this, let $(U_i)_{i\geq 1}$ be a sequence of i.i.d.\ uniform random variables in $[0,1]$. For each $n\geq 1$, consider the first $n$ points $(U_i)_{i\leq n}$ in $[0,1]$ and define the tree $t(\efrak,\sfrak,p;(U_i)_{i\leq n})$ recursively as follows. Let $(\overline{U}_i)_{i\leq n}$ be the order statistic of $(U_i)_{i\leq n}$ and let $m_i$ be the minimum value of $\efrak$ on $[\overline{U}_i, \overline{U}_{i+1}]$ for all $i<n$. If $n = 1$, then $t(\efrak,\sfrak,p;U_1)$ is a leaf.
 Otherwise, let $m = \min_i\{m_i\}$ and denote by $j$ the index such that $m_j 
= m$. Then $t(\efrak,\sfrak,p;(U_i)_{i\leq n})$ is the tree whose root is decorated by $\sfrak(m)$ and whose children are $t(\efrak,\sfrak,p;(\overline{U}_i)_{i\leq j})$ and $t(\efrak,\sfrak,p;(\overline{U}_i)_{j+1\leq i\leq n})$.
Then $t(\efrak,\sfrak,p;(U_i)_{i\leq n})\stackrel{\mathrm{d}}{=}T_n(p)$ and the trees $t(\efrak,\sfrak,p;(U_i)_{i\leq n})$ are coupled for different values of $n$ exactly as in the R\'emy algorithm -- see for instance~\cite[Section 2.6]{legall-tree-survey}~or~\cite[Exercise 7.4.11]{pitman2006combinatorial}.
\end{remark}

\bigskip

We now turn to the Brownian cographon. Given the signed excursion $(\efrak, \sfrak,p)$, the \textbf{Brownian cographon} $\bm{W}_p$ of parameter $p\in[0,1]$ is defined (following \cite{bassino2022random}) as the graphon equivalence class of the random function
\begin{align}\label{eq:Brownian-cog}
	\bm{W}_p: &[0, 1]^2 \to \{0, 1\},\\
	&(x,y)\mapsto\,\,\,\mathds{1}_{\sfrak(m_{x,y})=\oplus}, \notag
\end{align}
where, as in \eqref{eq:exc_to_perm}, if $(x,y)\in\mcl R_\efrak^2$ and $x<y$, we denote by $m_{x,y}\in[x,y]$ the unique strict local minimum $\min_{[x,y]}\efrak$. If $x$ or $y$ is not in $\mcl R_\efrak$ then we arbitrarily set $\sfrak(m_{x,y})=\oplus$. This choice does not change the law of $\bm{W}_p$ because a.s.\ $\Leb([0,1]^2\setminus \mcl R_\efrak^2)=0$. We highlight that $\bm{W}_p$ is a \emph{random} graphon.

\subsection{A simple estimate for critical binary Bienaymé--Galton--Watson trees}

Recall that $T$ denotes a critical binary Bienaymé--Galton--Watson tree with i.i.d.\ Bernoulli($p$) signs on its nodes, so that $T$ conditioned to have size $n$ is distributed as $T_n$. We will often use the following simple estimate on the size of $T$.

\begin{lem}\label{lem:prelim_GW}
	For all $n \geq 1$, set $r(n)\coloneqq \Pp{|T|=n}$. Then
	\begin{equation} \label{eq:asympt r(m)}
		r(n)=\frac{1}{2^{2n-1}} \cdot \frac{1}{n} \binom{2n-2}{n-1} = \frac{1}{2\sqrt{\pi}} n^{-3/2}+O \left( n^{-5/2} \right). 
	\end{equation}
	In particular, $r(n)$ is decreasing in $n$.
\end{lem}

\begin{proof}
	There are $\frac{1}{n} \binom{2n-2}{n-1}$ binary trees with $n$ leaves, \emph{i.e.}\ the $(n-1)$-th Catalan number. Moreover, each binary tree with $n$ leaves has $n-1$ nodes.
	Since $T$ is a critical binary Bienaymé--Galton--Watson tree, we immediately get the first equality in~\eqref{eq:asympt r(m)}. The second one is then deduced by a standard application of Stirling's formula.
	The fact that $r(n)$ is decreasing in $n$ follows by a simple direct calculation.
\end{proof}

\subsection{Decomposing $p$-signed uniform binary trees}\label{sect:decomp-binary}

Recall from \cref{sect:positive-subtrees} that given a subset $\mathcal{L}$ of the set of leaves of a sign-decorated binary tree $t$, we denote by $t|_\mathcal{L}$ the sign-decorated binary subtree with $\# \mathcal{L}$ leaves obtained from $t$ by removing all vertices that do not have a descendant in $\mathcal{L}$ and contracting all vertices of degree $2$ that may appear.

\begin{lem}\label{lem:decomp-tree}
Let $n \geq m \geq 1$ and let $T_n$ be a $p$-signed uniform binary tree with $n$ leaves. Conditionally on $T_n$, let $\mathcal{L}_m=\{\uleaf{1}, \dots, \uleaf{m} \}$ be a uniform random subset of the set of leaves of $T_n$ with cardinality $m$. 
Then $T_n|_{\mathcal{L}_m}$ has the same law as $T_m$. 
\end{lem}

\begin{proof}
    This is an immediate consequence, for instance, of \cref{remk:coupled-trees}.
\end{proof}

We also note that $T_n|_{\mathcal{L}_m}$ has $m-1$ nodes
which can be seen as nodes of the larger tree $T_n$. There is a natural correspondence between, on the one hand, the connected components of the complement of those vertices in $T_n$ and, on the other hand, the edges of $T_n|_{\mathcal{L}_m}$. 
 For any edge $e$ of $T_n|_{\mathcal{L}_m}$, we denote by $T_n^e$ the corresponding connected component in $T_n$; \emph{cf.}\ \cref{fig:subtree-notation}.

 Let $e_1, \dots, e_{2m-1}$ be the edges of $T_n|_{\mathcal{L}_m}$, labelled so that $e_1, \dots, e_m$ are the edges incident to the leaves $\uleaf{1}, \dots, \uleaf{m}$. For all $i\in\{1,\dots, 2m-1\}$, we set
 \begin{equation*}
        N_i\coloneqq \text{the number of leaves of $T_n$ that belong to the region $T_n^{e_i}$},
 \end{equation*}
 and 
 \begin{alignat*}{3}
        &\widetilde{N}_i\coloneqq N_i-1 &&\quad\text{ if } i \in \{1,\dots, m\}, \\
        &\widetilde{N}_i\coloneqq N_i &&\quad\text{ if } i \in \{m+1,\dots,2m-1\}. 
 \end{alignat*}
 That is, $\widetilde{N}_i$ counts the number of leaves of $T_n$ in each region, excluding the leaves $\uleaf{1}, \dots, \uleaf{m}$.

 \begin{figure}[ht]
	\begin{center}
    \includegraphics[width=\textwidth]{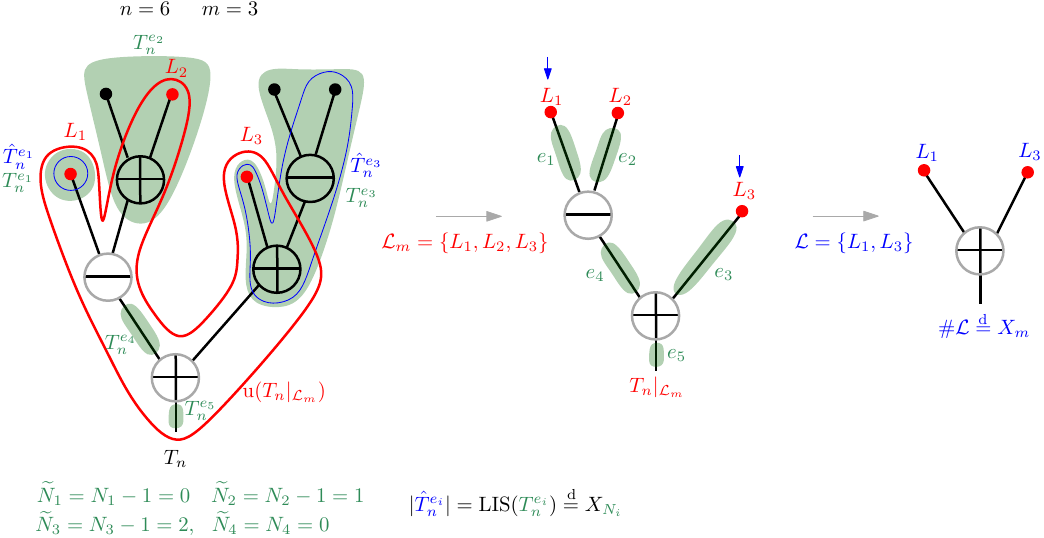}  
		\caption{A diagram for the notation introduced in \cref{sect:decomp-binary}. 
        The blue notation  
        refers to the proof of \cref{lem:superbranching_ineq}. 
        \textbf{Left:} The tree $T_n$, from which we select a subset of leaves $\mathcal{L}_m = \{\uleaf{1},\uleaf{2},\uleaf{3}\}$ and form the uncontracted tree $\text{u}(T_n|_{\mathcal{L}_m})$ highlighted in red.
        \textbf{Middle:} The tree $T_n|_{\mathcal{L}_m}$ with the two leaves in the set $\mathcal{L}$ corresponding to a maximal positive subtree marked by two blue arrows. To each edge $e$ of $T_n|_{\mathcal{L}_m}$ corresponds a region $T^e_n$ highlighted in green on the left part. We select a positive maximal subtree $\hat{T}^e_n$ in any such region adjacent to the subset $\mathcal{L}$.
        \textbf{Right:} The corresponding maximal positive subtree.
        \label{fig:subtree-notation}}
	\end{center}
	\vspace{-3ex}
\end{figure}

 \medskip
 
We recall the definition of the \emph{Dirichlet} distribution and of the \emph{Dirichlet-multinomial} distribution, which will describe the law of the vector $\big( \widetilde{N}_i \big)_{1 \leq i \leq 2m-1}$.

\begin{defn}\label{defn:Dirichlet-multinomial}
    Let $n \geq 0, k \geq 1$ be integers, let $\alpha_1, \dots, \alpha_k>0$ and $\alpha_0=\sum_{i=1}^k \alpha_i$. 
    \begin{itemize}
    \item 
    The Dirichlet distribution $\Dir \left( \alpha_1, \dots, \alpha_k \right)$ is the law on $k$-tuples of non-negative real numbers $(x_1, \dots, x_k)$ which sum up to $1$ given by the probability density function
     \[\frac{\Gamma(\alpha_0)}{\Gamma(\alpha_1)\cdots \Gamma(\alpha_k)} \prod_{i=1}^k x_i^{\alpha_i - 1}.\]

    \item The Dirichlet-multinomial distribution $\DirM \left( n; \alpha_1, \dots, \alpha_k \right)$ is the law on $k$-tuples of non-negative integers $(n_1, \dots, n_k)$ which sum up to $n$ given by the probability mass function
        \[ \frac{\Gamma(\alpha_0) \Gamma(n+1)}{\Gamma(n+\alpha_0)} \prod_{i=1}^k \frac{\Gamma(n_i+\alpha_i)}{\Gamma(\alpha_i) \Gamma(n_i+1)}. \]
    In other words, a $\DirM \left( n; \alpha_1, \dots, \alpha_k \right)$-distributed random variable $\mathbf{N}=\left( N_1, \dots, N_k \right)$ can be obtained as follows: let $\mathbf{D}=(D_1, \dots, D_k)$ follow the Dirichlet distribution $\Dir \left( \alpha_1, \dots, \alpha_k \right)$; then, conditionally on $\mathbf{D}$, sample the vector $\mathbf{N}$ as a multinomial vector with parameters $(n;\mathbf{D})$.
    \end{itemize}
\end{defn}

Note that, by the law of large numbers, we immediately obtain the following result, stated in a convenient form for later use.

\begin{lem}\label{lem:discrete-multinomial-distrd-rv}
    Fix an integer $m\geq1$. Let $\mathbf{N}=(N_i)_{i\leq 2m-1}$ be a $\DirM \left(n; \alpha_1, \dots, \alpha_{2m-1} \right)$-distributed random vector. Then
    \begin{equation*}
        \left(\frac{N_i}{n}\right)_{i\leq 2m-1}\xrightarrow[n\to \infty]{\mathrm{d}}(D^m_i)_{i\leq 2m-1},
    \end{equation*}
    where $(D^m_i)_{i\leq 2m-1}$ is a Dirichlet random vector with parameter $(\alpha_1, \dots, \alpha_{2m-1})$.
\end{lem}
We are making explicit the dependence on $m$ of the random variable $D^m_i$ because later on we will use this result in the case when $(D^m_i)_{i\leq 2m-1}$ is a Dirichlet random vector with parameter $(\frac 1 2, \dots, \frac 1 2)$ and the dependence on $m$ will play a key role.

We also record a simple result about Dirichlet random variables.

\begin{lem}\label{lem:diric-distrd-rv}
    Let $(D_1,\dots D_{n}) \sim \Dir(1/2,\dots , 1/2)$. Then

    \begin{equation*}
    \left(\frac{n}{2} D_1, \frac{n}{2}D_2\right)\xrightarrow[n\to\infty]{\mathrm{d}}(G_1,G_2),
    \end{equation*}
    where
    $(G_1,G_2)$ is a pair of independent of $\mathrm{Gamma}\left(\frac{1}{2}\right)$-distributed 
    random variables.
\end{lem}

\begin{proof}
    Let $(G_i)_{i \geq 1}$ be independent $\mathrm{Gamma}\left(\frac{1}{2}\right)$-distributed random variables. For any $n \geq 1$, we have the equality in distribution \cite[Chapter XI.4]{devroye2006nonuniform}
    \begin{align*}
        \frac{1}{\sum_{i=1}^n G_i } \cdot (G_1,G_2,\dots,G_{n}) \overset{\mathrm{d}}{=} (D_1,\dots D_{n}).
    \end{align*}
    Since $\frac{n/2}{\sum_{i=1}^n G_i}$ converges to $1$ in probability by the law of large numbers, the claimed result follows from Slutzky's lemma.
\end{proof}

We finally state our last preliminary result on the decomposition of $p$-signed uniform binary  trees.

\begin{lem}\label{lem:decomp-tree2}
     Conditionally on $T_n|_{\mathcal{L}_m}$, the vector $\big( \widetilde{N}_i \big)_{1 \leq i \leq 2m-1}$ follows the Dirichlet-multinomial distribution $\DirM \left( n-m; \frac{1}{2}, \dots, \frac{1}{2} \right)$. 
     Moreover, conditionally on $T_n|_{\mathcal{L}_m}$ and $(\widetilde{N}_i)_{1 \leq i \leq 2m-1}$,
     for any $i$ with $1\leq i \leq m$, the region $T_n^{e_i}$ corresponding to edge $e_i$ (adjacent to the leaf $L_i$) has the same distribution as $T_{N_i}$, and the regions $(T_n^{e_i})_{1\leq i\leq m}$ are jointly independent.

     In particular, since the laws of  $\big( \widetilde{N}_i \big)_{1\leq i \leq 2m-1}$ and $(T_n^{e_i})_{1\leq i \leq m}$ conditionally on $T_n|_{\mathcal{L}_m}$ do not depend on $T_n|_{\mathcal{L}_m}$, then $\big( \widetilde{N}_i \big)_{1\leq i \leq 2m-1}$ and $(T_n^{e_i})_{1\leq i \leq m}$ are independent of $T_n|_{\mathcal{L}_m}$.
\end{lem}

\begin{proof}
    The result follows, for instance, by combining the two facts stated in \cite[Exercises 2.2.2 and 7.4.13]{pitman2006combinatorial}.
\end{proof}

\section{Existence of the exponent $\alpha$ via supermultiplicativity}\label{sect:supermultiplicativity}

Recall that $X_n=\lis(T_n)$ denotes the maximal size (number of leaves) of a positive subtree of the $p$-signed uniform binary  tree $T_n$ with $n$ leaves. The main goal of this section is to prove~\cref{thm:main}, \emph{i.e.}\ the existence of the exponent $\alpha$.

The proof of \cref{thm:main} consists of a supermultiplicativity argument. We first state a result of Dekking and Grimmett~\cite{DG88} which, although not directly applicable to our setting, has a similar flavor.

\begin{prop}[{\cite[Theorem 1]{DG88}}]\label{prop:dg88}
	Let $(Y_n)_{n \geq 1}$ be a sequence of random variables with values in the set of positive integers. We assume that for any $m,n \geq 1$, we have the stochastic domination
	\begin{equation}\label{eqn:superbranching_v1}
        Y_{m+n} \succeq \sum_{i=1}^{Y_m}Y_n^{(i)},
    \end{equation}
    where the $Y_n^{(i)}$ are i.i.d.\ copies of $Y_n$ and are independent of $Y_m$. Let $\lambda=\sup_{n \geq 1} \bbE[Y_n]^{1/n}$. If $\lambda<+\infty$, then we have the convergence in probability
    \[ Y_n^{1/n} \xrightarrow[n \to +\infty]{\mathbb{P}} \lambda. \]
\end{prop}

In~\cite{DG88}, processes satisfying~\eqref{eqn:superbranching_v1} are called \textbf{superbranching processes}, by analogy with the branching property enjoyed by e.g.\ Bienaymé--Galton--Watson processes. Since we are interested in a quantity $X_n$ with polynomial growth, it would be natural to try and obtain a superbranching equation of the form
\begin{equation}\label{eqn:superbranching_v2}
    X_{m\cdot n} \succeq \sum_{i=1}^{X_m} X_n^{(i)}.
\end{equation}
This would mean breaking a tree of size $m\cdot n$ into $m$ subtrees of size $n$. However, the most natural ways to decompose uniform binary trees into several regions result in regions with \emph{random} and \emph{different} sizes (recall \cref{lem:decomp-tree2} in \cref{sect:decomp-binary}), which is why our supermultiplicativity equation \eqref{eqn:superbranching_v3} will be more complicated than~\eqref{eqn:superbranching_v2}.

Recall the definition of the \emph{Dirichlet-multinomial} distribution from \cref{defn:Dirichlet-multinomial} in \cref{sect:decomp-binary}.
We now state the superbranching-like equation that will be useful to us.

\begin{lem}[Supermultiplicativity equation for $X_n$]\label{lem:superbranching_ineq}
For all $n \geq m \geq 1$, we have the stochastic domination
    \begin{equation}\label{eqn:superbranching_v3}
        X_n \succeq \sum_{i=1}^{X_m} X_{N_i}^{(i)},
    \end{equation}
where:
\begin{itemize}
	\item The vector $(N_i-\mathds{1}_{i\leq m})_{1 \leq i \leq 2m-1}$ follows the $\DirM \left( n-m; \frac{1}{2}, \dots, \frac{1}{2} \right)$ distribution  and is independent of $X_m$.
    \item Conditionally on $(N_i)_{1 \leq i \leq 2m-1}$, the random variable $X_{N_i}^{(i)}$ has the law of $X_{N_i}$ for all $i\in \intervalleentier{1}{m}$ and the  
    $(X_{N_i}^{(i)})_{1 \leq i \leq m}$ are jointly independent and are independent of $X_m$. 
\end{itemize}
\end{lem}

\begin{remark} 
    Note that since $X_m\leq m$, the random variables $N_i$ involved in  \eqref{eqn:superbranching_v3} are all among the first $m$ coordinates of the vector $(N_i-\mathds{1}_{i\leq m})_{1 \leq i \leq 2m-1}$. That is, when the indicator function is equal to 1.
\end{remark}

\begin{proof}
    Let $n \geq m \geq 1$. 
    We will build a subtree of $T_n$ with only $\oplus$ nodes, whose total number of leaves will have the same distribution as the right-hand side of~\eqref{eqn:superbranching_v3}.

    Conditionally on $T_n$, let $\mathcal{L}_m=\{\uleaf{1}, \dots, \uleaf{m} \}$ be a uniform random subset of the set of leaves of $T_n$ with size $m$. We consider the subtree $T_n|_{\mathcal{L}_m}$ as in \cref{lem:decomp-tree}. Then $T_n|_{\mathcal{L}_m}$ has the same law as $T_m$.
 
    Recall the notation introduced above \cref{lem:decomp-tree2} and in \cref{fig:subtree-notation}.  
    Then, thanks to \cref{lem:decomp-tree2}, the vector $( \widetilde{N}_i)_{1 \leq i \leq 2m-1}=(N_i-\mathds{1}_{i\leq m})_{1 \leq i \leq 2m-1}$ follows the $\DirM \left( n-m; \frac{1}{2}, \dots, \frac{1}{2} \right)$ distribution and is independent of $T_n|_{\mathcal{L}_m}$. Moreover, conditionally on $(N_i)_{1 \leq i \leq 2m-1}$, the regions $T_n^{e_i}$ for $1\leq i\leq m$, corresponding to the edges $e_i$, are independent copies of $T_{N_i}$ and are independent of $T_n|_{\mathcal{L}_m}$. 
    
    Now consider the largest positive subtree of $T_n|_{\mathcal{L}_m}$ (if there is more than one such maximal subtree, pick one arbitrarily), and denote by $\mathcal{L} \subset \mathcal{L}_m=\{ \uleaf{1}, \dots, \uleaf{m}\}$ the set of its leaves as on Figure~\ref{fig:subtree-notation}. Since $T_n|_{\mathcal{L}_m}$ has the same law as $T_m$, the size of $\mathcal{L}$ has the same law as $X_m=\lis(T_m)$. Moreover, for each $\uleaf{i} \in \mathcal{L}$, consider a maximal positive subtree $\widehat{T}^{e_i}_n$ of $T_n^{e_i}$.
    Since $T_n^{e_i}$ has the same law as $T_{N_i}$ by \cref{lem:decomp-tree2}, the size $\lis(T_n^{e_i})$ of $\widehat{T}^{e_i}_n$ has the same law as $X_{N_i}$.
    
    The subtree of $T_n$ induced by the leaves of $\bigcup_{\uleaf{i} \in \mathcal{L}} \widehat{T}_n^{e_i}$ has only $\oplus$ nodes and, thanks to the previous paragraph, its number of leaves has the same distribution as $\sum_{i=1}^{X_m} X_{N_i}^{(i)}$, where $N_i$ and the $X_{N_i}^{(i)}$ are as in the statement of the Lemma. This concludes the proof.
\end{proof}

With \cref{lem:superbranching_ineq} in our hands, we can turn to the proof of \cref{thm:main}.

\begin{proof}[Proof of \cref{thm:main}] 
    For $n\geq 1$, we set $\alpha_n \coloneqq  \frac{\log \bbE[X_n]}{\log n} $ so that 
    \begin{equation*}
        \bbE[X_n]=n^{\alpha_n}.
    \end{equation*}
    Let $\alpha=\limsup_{n \to \infty} \alpha_n$ as in~\eqref{eqn:alpha_as_limsup}.
    By Markov's inequality, it is immediate  that for any $\eps>0$, the probability that $X_n \geq n^{\alpha+\eps}$ goes to $0$ as $n \to \infty$. 
    Therefore, to prove \eqref{eqn:existence-exponent}, we only need to show that for any $\eps>0$,
    \begin{equation}\label{eq:expectation}
        \Pp{X_n \leq n^{\alpha-\eps}}\to 0 \quad \text{as }n \to \infty.
    \end{equation}
    Once we have this, the fact that $\alpha \in ( 1/2,1 )$ is an immediate consequence of~\cref{thm:prev-bound}. 
    We divide the proof of~\eqref{eq:expectation} into five main steps.

    \medskip
    
    \noindent\underline{\emph{Step 1:}} \emph{Turning~\eqref{eqn:superbranching_v3} into an equation closer to~\eqref{eqn:superbranching_v2}.} We first turn~\eqref{eqn:superbranching_v3} into an equation closer to~\eqref{eqn:superbranching_v2} in order to compare $X_n$ with a branching process. This goal is reached in \eqref{eqn:stochastic_domination2}. We fix $m\le n$. The precise choice of $m$ will be decided towards the end of the proof in Step 4.
    
    Roughly speaking, \cref{lem:discrete-multinomial-distrd-rv} ensures that the random variables $N_i$ for $1\le i\le 2m-1$ appearing in  \eqref{eqn:superbranching_v3} behave like $n \cdot D_i^{m}$, where $(D_1^{m}, D_2^{m}, \dots, D_{2m-1}^{m}) \sim \mathrm{Dir}\left(\frac{1}{2}, \dots , \frac{1}{2}\right)$, so that they are of order $\frac{n}{m}$ and are random at that scale. In order to obtain a version of~\eqref{eqn:superbranching_v3} with i.i.d.\ terms on the right-hand side, we will keep only the values of $i$ for which $N_i> \frac{n}{m}$. 

    With the notation of \cref{lem:superbranching_ineq}, we can write the stochastic domination 
    \begin{equation}\label{eqn:stochastic_domination}
    	X_n \succeq \sum_{i=1}^{X_m} X_{N_i}^{(i)} \succeq \sum_{i=1}^{X_m} X_{N_i}^{(i)} \mathbbm{1}_{\{N_i\geq  \frac{n}{m}\}}\succeq \sum_{i=1}^{\widetilde{Y}_m} X_{\lfloor\frac{n}{m}\rfloor}^{(i)},
    \end{equation}
    where
    \begin{equation*}
    \widetilde{Y}_m=\widetilde{Y}_m(n)\overset{\mathrm{\mathrm{d}}}{=} \sum_{i=1}^{X_m}\mathbbm{1}_{\{N_i\geq  \frac{n}{m}\}},
    \end{equation*}
    and the random variables $X_{\lfloor\frac{n}{m}\rfloor}^{(i)}$ are independent copies of $X_{\lfloor\frac{n}{m}\rfloor}$, also independent of $\widetilde{Y}_m$. Note that the last domination in~\eqref{eqn:stochastic_domination}  is possible because \cref{lem:superbranching_ineq} guarantees that, conditional on $(N_i)$, the $X_{N_i}^{(i)}$ are all independent and are also independent of $X_m$.

    The domination in~\eqref{eqn:stochastic_domination} is closer to~\eqref{eqn:superbranching_v2} but we also need to replace $\widetilde{Y}_m(n)$ with a random variable $Y_m$ that does not depend on $n$.
    By~\cref{lem:discrete-multinomial-distrd-rv}, we have the convergence in distribution 
    \begin{align*}
    \sum_{i=1}^{X_m}\mathbbm{1}_{\{N_i\geq  \frac{n}{m}\}}\xrightarrow[n\rightarrow \infty]{\mathrm{d}} \sum_{i=1}^{X_m}\mathbbm{1}_{\{D^m_i\geq \frac{1}{m}\}},
    \end{align*}
    where $(D^m_1,\dots D^m_{2m-1}) \sim \Dir(1/2,\dots , 1/2)$ and is independent of $X_m$. 
    We then introduce\footnote{Note the extra factor $2$ compared to the previous display.}
    \begin{align}\label{eq:defn-Ym}
    Y_m\coloneqq  \sum_{i=1}^{X_m}\mathbbm{1}_{\{D^m_i\geq \frac{2}{m}\}},
    \end{align}
    so that for any $k\in \intervalleentier{1}{m}$, we have 
    \begin{align*}
    \Pp{\widetilde{Y}_m(n)\geq k} = \Pp{\sum_{i=1}^{X_m}\mathbbm{1}_{\{N_i\geq  \frac{n}{m}\}} \geq k} \xrightarrow[n \rightarrow \infty]{} \Pp{\sum_{i=1}^{X_m}\mathbbm{1}_{\{D^m_i\geq \frac{1}{m}\}} \geq k} > \Pp{Y_m \geq k}.
    \end{align*}
    Now, because of the strict inequality in the last display, the probability $\Pp{\widetilde{Y}_m(n)\geq k}$ is eventually larger than $\Pp{Y_m \geq k}$ for all $k\in \intervalleentier{1}{m}$. Hence, for all $m\geq 1$, there exists $n_0=n_0(m)>0$ such that
    $\widetilde{Y}_m(n)\succeq Y_m$ for $n\geq n_0$. Summarizing, 
    \begin{equation}\label{eqn:stochastic_domination2}
    	X_n \succeq \sum_{i=1}^{Y_m} X_{\lfloor\frac{n}{m}\rfloor}^{(i)} \qquad\text{for all $m\geq 1$ and $n\geq n_0(m)$},
    \end{equation}
    where $Y_m$ is as in~\eqref{eq:defn-Ym}. This domination is now closer to~\eqref{eqn:superbranching_v2} as we want.
    
    \medskip

    \noindent\underline{\emph{Step 2:}} \emph{Lower bounding the expectation of $Y_m$.}
   Note that $Y_m$ has expectation 
    \begin{equation}\label{eq:exp_computations}
        \bbE[Y_m] =\bbE[X_m] \cdot \Pp{D^m_1 \geq \frac{2}{m}}=m^{\alpha_m}\cdot \Pp{D^m_1 \geq  \frac{2}{m}}, 
    \end{equation}
    since the vector $(D^m_i)_{1 \leq i \leq 2m-1}$ is independent of $X_m$.     
    Using~\cref{lem:diric-distrd-rv}, we can set
    \begin{equation}\label{eq:defn-b}
    b\coloneqq  \inf_{m\geq 1} \left\{\Pp{D^m_1 \geq \frac{2}{m}}\right\} >0,
    \end{equation}
    which gives
    \begin{equation}\label{eq:low_bound_offspring}
    \bbE[Y_m]\geq m^{\alpha_m}b
    \qquad\text{ for all }m\geq 1.
    \end{equation}

    \medskip

    \noindent\underline{\emph{Step 3:}} \emph{Lower bounding $X_n$ by comparing it with an appropriate branching process.}
    Let $n_0=n_0(m)$ be as in~\eqref{eqn:stochastic_domination2}. By iterating~\eqref{eqn:stochastic_domination2}, an induction on $k$ shows that the process $(X_{n_0 m^k})_{k \geq 0}$ stochastically dominates a Bienaymé--Galton--Watson process $(Z_k)_{k \geq 0}$ started with one particle and with offspring distribution given by the law of $Y_m$. Therefore, setting 
    \begin{equation}\label{eq:defn_k}
    k=k(m,n)\coloneqq\Big\lfloor\frac{\log (n/n_0)}{\log m}\Big\rfloor,
    \end{equation}
    so that $n_0 m^{k+1} > n \geq  n_0 m^{k}$, we get $X_n \succeq X_{n_0 \cdot m^{k}} \succeq Z_{k}$.    

    \smallskip

    We now claim that there exists a threshold $m_0>0$, such that for every $m\geq m_0$, there exist $c=c(m)>0$ and $n'_0=n'_0(m)>0$ such that for all $n \geq n'_0(m)$: 
    \begin{equation}\label{eq:pos_proba_event}
        \Pp{X_{n} \geq c\cdot (m^{\alpha_m}b)^k} \geq c,
    \end{equation}
    where  $k=k(m,n)$ is as in \eqref{eq:defn_k} (and so its value is fixed by $m, n$).
    Indeed, since $X_n \succeq Z_{k}$, for any choice of the constant $c>0$, we can write
    \begin{equation*}
        \Pp{X_{n} \geq c\cdot (m^{\alpha_m}b)^k}\geq  \Pp{Z_{k} \geq c\cdot (m^{\alpha_m}b)^k} \geq  \Pp{(\Ec{Y_m})^{-k} Z_{k} \geq c},
    \end{equation*}
    where in the last equality we used that the offspring distribution $Y_m$ of the branching process $(Z_i)_{i\geq 1}$ has expectation at least $m^{\alpha_m} b$ thanks to~\eqref{eq:low_bound_offspring}. 
    By standard results (see for instance~\cite{kesten1966limit}), $(\Ec{Y_m})^{-k} Z_k$ converges almost surely as $k\to \infty$ to some $Z_\infty(m)$, which is positive on the event of non-extinction. Now, since $\Ec{Y_m}>1$ for $m$ large enough, there exists $m_0>0$ such that the event of non-extinction has positive probability (which depends on $m$) for $m>m_0$.
    Therefore, we can fix $c=c(m)>0$ such that $\Pp{Z_\infty(m)\ge c}\ge 2c$. 
    This ensures that $\Pp{(\Ec{Y_m})^{-k} Z_k \geq c} \ge c$ for large enough $k$ (\emph{i.e.}\ for $n\geq n'_0(m)$), establishing \eqref{eq:pos_proba_event}.

    \smallskip

    In addition, when the event in~\eqref{eq:pos_proba_event} occurs, we have that
    \begin{align*}
    	X_n
    	&\,\,\,\geq\,\,\, c \cdot \exp \left( k (\alpha_m \log m+\log b)\right)\\ 
    	&\stackrel{\eqref{eq:defn_k}}{\geq} c  \cdot 
        \min\left\{\exp \left( \left( \frac{\log(n/n_0)}{\log m}-1 \right) ( \alpha_m \log m+\log b)\right), 
        \exp \left( \frac{\log(n/n_0)}{\log m} ( \alpha_m \log m+\log b)\right)\right\},
    \end{align*}
    where the minimum accounts for the fact that $ \alpha_m \log m+\log b$ could \textit{a priori} be negative. In any case,
    this implies that when the event in~\eqref{eq:pos_proba_event} occurs, we have
    \begin{equation}\label{eqn:comparison_with_branching}
        \frac{\log X_n}{\log n} \geq \alpha_m + \frac{\log b}{\log m} +O_m \left( \frac{1}{\log n} \right).
    \end{equation}
    
    \medskip

    \noindent\underline{\emph{Step 4:}} \emph{Getting the right lower bound for $X_n$ with positive probability.} We will now specify the value of $m$. Let $\eps>0$. Recalling that $b\in(0,1)$ and $\alpha=\limsup_{m \to \infty} \alpha_m$, we choose $m$ large enough so that $m\ge m_0$ and
    \begin{equation}
        \frac{\log b}{\log m} \ge -\eps/3
        \qquad\text{and}\qquad
        \alpha_m \geq \alpha-\eps/3.
    \end{equation}
    Now that $m$ is fixed, \eqref{eq:pos_proba_event}~and~\eqref{eqn:comparison_with_branching} imply that there is a constant $c>0$ such that for $n$ large enough, we have
    \begin{equation}\label{eqn:existence_exponent_with_positive_P}
        \Pp{\frac{\log X_n}{\log n} \geq \alpha-\eps}=\Pp{X_n \geq n^{\alpha-\eps}} \geq c.
    \end{equation}
    In order to prove~\eqref{eq:expectation}, it remains to turn the probability in~\eqref{eqn:existence_exponent_with_positive_P} from $c$ to $1-\eps$. This is the goal of the next (and last) step.

    \medskip

    \noindent\underline{\emph{Step 5:}} \emph{Getting the right lower bound for $X_n$ with high probability.}
    The idea will be to use~\eqref{eqn:stochastic_domination2} again and argue that if each term on the right-hand side of~\eqref{eqn:stochastic_domination2} has a positive probability to be large (thanks to~\eqref{eqn:existence_exponent_with_positive_P}), then the sum is very likely to be large, provided the number $Y_m$ of terms is high enough.
    For this, we first claim that we can fix $m'>0$ large enough so that
    \begin{equation}\label{eq:Y-bound}
        \Pp{Y_{m'}\leq \frac{1}{\eps c}} \leq \eps,
    \end{equation}
    where $c$ is as in~\eqref{eqn:existence_exponent_with_positive_P}.
    Indeed, recalling that $Y_m= \sum_{i=1}^{X_m}\mathbbm{1}_{\{D^m_i\geq \frac{2}{m}\}}$, we have
    \begin{equation*}
        \Pp{Y_{m}\leq \frac{1}{\eps c}} \leq \Pp{X_{m}\leq \sqrt{m}}+\Pp{\sum_{i=1}^{\sqrt{m}}\mathbbm{1}_{\{D^m_i\geq \frac{2}{m}\}}\leq \frac{1}{\eps c}}.
    \end{equation*}
    The first term on the right-hand side tends to zero as $m$ tends to infinity thanks to~\cref{thm:prev-bound}, so we focus on the second term. Setting $S_m\coloneqq \sum_{i=1}^{\sqrt{m}}\mathbbm{1}_{\{D^m_i\geq \frac{2}{m}\}}$ and noting that $\Ec{S_m}\geq \lfloor\sqrt{m}\rfloor b$ by definition of $b$ in \eqref{eq:defn-b}, we get that for $m$ large enough, 
     \begin{equation*}
        \Pp{S_m > \frac{1}{\eps c}}\geq \Pp{S_m > \frac{1}{\eps c \lfloor\sqrt{m}\rfloor b}\Ec{S_m}}\geq \Bigg(1-\frac{1}{\eps c \lfloor\sqrt{m}\rfloor b}\Bigg)^2 \,\,\frac{\Ec{S_m}^2}{\Ec{S_m^2}},
    \end{equation*}
    where the last inequality follows from the Paley–Zygmund inequality. The first factor on the right-hand side tends to one as $m$ tends to infinity, so we are left to show that $\frac{\Ec{S_m}^2}{\Ec{S_m^2}}\xrightarrow[m\to\infty]{}1$. Note that, computing the two expectations, we have that 
    \begin{equation*}
        \frac{\Ec{S_m}^2}{\Ec{S_m^2}}\geq \frac{\lfloor\sqrt{m}\rfloor^2 \Pp{D^m_1\geq \frac{2}{m}}^2}{\sqrt{m} \Pp{D^m_1\geq \frac{2}{m}}+m \Pp{D^m_1\geq \frac{2}{m},D^m_2\geq \frac{2}{m}}}.
    \end{equation*}
    Thanks to~\cref{lem:diric-distrd-rv}, we know that $(mD^m_1,mD^m_2)\xrightarrow[m\to\infty]{\mathrm{d}}(G_1,G_2)$,
    where $(G_1,G_2)$ is a pair of independent Gamma random variables of parameter 1/2. This implies that the right-hand side of the above equation tends to one as we wanted, completing the proof of~\eqref{eq:Y-bound}.
    
    With~\eqref{eq:Y-bound} in our hands we can now return to~\eqref{eqn:stochastic_domination2}. For $n$ large enough, we have
    \begin{align*}
        \Pp{X_n < \left( \frac{n}{m'}  \right)^{\alpha-\eps}} & \leq \Pp{Y_{m'} \leq \frac{1}{\eps c}} + \Pp{\forall 1 \leq i \leq \frac{1}{\eps c}, \, X^{(i)}_{\lfloor \frac{n}{m'} \rfloor} < \left( \frac{n}{m'} \right)^{\alpha-\eps}}\\
        &\leq \eps + \left( 1-c \right)^{1/\eps c}\\
        &\leq \eps+\exp(-\eps^{-1})\\
        &\leq 2\eps,
    \end{align*}
    where in the second inequality we used~\eqref{eq:Y-bound}~and~\eqref{eqn:existence_exponent_with_positive_P} with $\frac{n}{m'}$ in place of $n$. This proves~\eqref{eq:expectation}.
\end{proof}

\section{Rough regularity estimates for the sequences $q$ and $Q$}\label{subsec:rough regularity}

Recall the definition of the sequences $q$ and $Q$ from~\eqref{eq:def_q_Q}, and note 
from~\cref{sect:dynamical-picture} that one of the main objectives of this paper is to establish the good regularity estimates~\eqref{eq:good regularity estimates} for these sequences. This goal will be accomplished in~\cref{thm:good_regularity}. As explained in~\cref{sect:dynamical-picture} (see also the diagram in~\cref{fig:proof-diagram}), the proof of~\cref{thm:good_regularity} relies crucially on the local convergence result established in~\cref{sect:loacal-lim}. To obtain this result, we now prove the following \textit{a priori} rough regularity estimates for the sequences $q$ and $Q$.

\begin{prop}[Rough regularity estimates]\label{prop:rough_regularity}
	We have 
    \begin{equation*}
        q(k)=k^{-\frac{1}{2\alpha}-1+o(1)}\qquad \text{ and }\qquad Q(k)= k^{-\frac{1}{2\alpha} +o(1)} \qquad\text{as }k \to \infty,
    \end{equation*}
    where $\alpha \in ( 1/2,1 )$ is as in \cref{thm:main}. Moreover, there is a constant $C>0$ such that for any real $\svx \geq 1$, we have
	\begin{equation}\label{eqn:regularity_Q}
	(1+o(1))\, \svx^{-C} \leq \frac{Q(\svx k)}{Q(k)} \leq 1 \qquad\text{as }k \to \infty,
	\end{equation}
	and, for any $k \geq 1$,
	\begin{equation}\label{eqn:regularity_q}
    C \, \svx^{-C} \leq \frac{q(\svx k)}{q(k)} \leq C, 
    \end{equation}
    where we are using the conventions $Q(z)=Q(\lceil z \rceil)$ and $q(z)=q(\lceil z \rceil)$ for all $z\in\mathbb{R}_{> 0}$ as in \cref{sect:not-proba}.
\end{prop}

The proof relies on the coupling of the decorated trees $T_n$ between various values of $n$ obtained through R\'emy's algorithm introduced in  \cref{sect:remy-algo}. Recall that one can run R\'emy's algorithm both forwards and backwards. 
For $n \geq 0$, we denote by $\Ff_n$ the $\sigma$-algebra generated by $(T_i)_{1 \leq i \leq n}$ and by $\revF_n$ the $\sigma$-algebra generated by $(T_i)_{i \geq n}$. 

We recall that we set $X_n=\lis(T_n)$ for all $n \geq 1$. Note that under this coupling, for all $n \geq 1$, we have $X_{n+1}-1 \leq X_n \leq X_{n+1}$, and \cref{thm:prev-bound} guarantees that $X_n \to +\infty$ in probability. Therefore, for all $k \geq 1$, the stopping time
\begin{equation}\label{eq:defn-stp-time}
	\sttm_k \coloneqq  \min \{ n \geq 1 \mid X_n=k\}
\end{equation} 
is almost surely finite (note that we also have the trivial a.s.\ bound $\sttm_k\geq k$). 
In particular, let $\alpha$ be given by \cref{thm:main}. Since $X_n$ is non-decreasing in $n$, for all $\eps>0$ and $k \geq 1$, we have
\begin{equation}\label{eq:sttm-bounds}
    \Pp{\sttm_k> k^{1/\alpha+\eps}}
=
\Pp{X_{\lfloor k^{1/\alpha+\eps} \rfloor} < k}
\quad \text{and} \quad 
\Pp{\sttm_k\le k^{1/\alpha-\eps}}
=
\Pp{X_{\lfloor k^{1/\alpha-\eps} \rfloor}\ge k}.
\end{equation}
Hence, \cref{thm:main} implies $\sttm_k=k^{1/\alpha+o(1)}$ in probability.
We start by making this statement a bit more precise, which will be useful later. We will rely on the description of the exponent $\alpha$ as a limsup, as we established in~\eqref{eqn:alpha_as_limsup} of \cref{thm:main}.

\begin{lem}[Control on the negative moments of $\sttm_k$]\label{cor:lower_tail_sigma}
    Let $0<\xi < \alpha$. Then we have
    \begin{equation}\label{eq:tail_sigma-1}
        \Ec{\sttm_k^{-\xi}} = k^{-\frac{\xi}{\alpha}+o(1)} \qquad \text{as } k \to \infty.
    \end{equation} 
    In particular, we have
    \begin{equation}\label{eq:tail_sigma-2}
        \Ec{\sttm_k^{-1/2}} = k^{-\frac{1}{2\alpha}+o(1)} \qquad \mbox{ and } \qquad \Ec{\sttm_k^{-3/2}} = o \left( \frac{1}{k} \Ec{\sttm_k^{-1/2}} \right) \qquad \text{as } k \to \infty.
    \end{equation}
\end{lem}

\begin{proof}
    The lower bound for~\eqref{eq:tail_sigma-1} is immediate using $\sttm_k=k^{1/\alpha+o(1)}$ in probability: for all $\eps>0$,
    \[
    \Ec{\sttm_k^{-\xi}}
    \geq
    \Ec{\sttm_k^{-\xi} \indicator{\sttm_k \le k^{1/\alpha+\eps}}}
    \geq 
    k^{-\frac{\xi}{\alpha}-\eps\xi}\, \Pc{\sttm_k \le k^{1/\alpha+\eps}} 
    \sim 
    k^{-\frac{\xi}{\alpha}-\eps\xi} \quad \text{as } k\to\infty.
    \]
    
    \noindent For the upper bound, let $\eps>0$ and let $k \geq 1$. We note that $\sttm_k \geq k$, so $\sttm_k^{-\xi} \leq k^{-\xi}$. For any $0<x\leq k^{-\xi}$, using Markov's inequality and the description~\eqref{eqn:alpha_as_limsup} of $\alpha$ as a limsup, we have 
    \[ \Pp{\sttm_k^{-\xi} \geq x} 
    = \Pp{\sttm_k \leq x^{-1/\xi}}
    =\Pp{X_{\lfloor x^{-1/\xi} \rfloor} \geq k} 
    \leq \frac{1}{k} \Ec{X_{\lfloor x^{-1/\xi} \rfloor}} 
    \leq \frac{1}{k} \lfloor x^{-1/\xi} \rfloor^{\alpha+\eps} \]
    if $k$ is large enough (which implies that $x^{-1/\xi}$ is large enough). Note that this last inequality still holds if $x>k^{-\xi}$, because then the left-hand side is $0$.
    By integrating over $x$, since $\frac{\alpha+\eps}{\xi}>1$, we find
    \begin{align*}
    \Ec{\sttm_k^{-\xi}}
    = \int_0^\infty \Pp{\sttm_k^{-\xi}\geq x}\dd x 
    \leq  \int_0^{\infty} 1 \wedge \left(\frac{1}{k} x^{-\frac{\alpha+\eps}{\xi}}\right) \mathrm{d}x 
    &\leq \int_0^{k^{-\frac{\xi}{\alpha+\epsilon}}} 1 \ \dd x + \frac{1}{k} \int_{k^{-\frac{\xi}{\alpha+\epsilon}}}^{\infty}  x^{-\frac{\alpha+\eps}{\xi}} \dd x \\
    &= O \left( k^{-\frac{\xi}{\alpha+\eps}} \right).
    \end{align*}
    Since this is true for any $\eps>0$, this proves~\eqref{eq:tail_sigma-1} and the first claim in~\eqref{eq:tail_sigma-2}. 
    
    For the second claim in~\eqref{eq:tail_sigma-2}, we fix $\eps>0$ small enough to have $\frac{1}{2}+\eps<\alpha$, which is possible because $\alpha>1/2$ by \cref{thm:main}. 
    We write, using the fact that $\sttm_k\geq k$ deterministically, 
    \[ \Ec{\sttm_k^{-3/2}} 
    = \frac{1}{k}\Ec{\frac{k}{\sttm_k} \cdot \sttm_k^{-1/2} }
    \leq \frac{1}{k} \Ec{\left( \frac{k}{\sttm_k} \right)^{\eps}\cdot \sttm_k^{-1/2} } 
    = k^{\eps-1} \Ec{\sttm_k^{-1/2-\eps}}.  \]
    Using \eqref{eq:tail_sigma-1} with $\xi=\frac{1}{2}+\eps$, we get $\Ec{\sttm_k^{-3/2}} 
    \leq k^{-\frac{1}{2\alpha}-1-\left(\frac{1}{\alpha}-1 \right)\eps+o(1)}$, whereas $\frac{1}{k}\Ec{\sttm_k^{-1/2}}=k^{-\frac{1}{2\alpha}-1+o(1)}$ again by~\eqref{eq:tail_sigma-1}. This proves~\eqref{eq:tail_sigma-2}.
\end{proof}

\cref{prop:rough_regularity} will be an easy consequence of the estimates in the last \cref{cor:lower_tail_sigma} and the following result, which relates the sequences $q$ and $Q$ to the stopping times $\sttm_k$ in order to compare them.

\begin{lem}[Relation between $q$, $Q$ and the stopping times $\sttm_k$]\label{lem:q_and_sigma}
	We have
	\begin{equation}\label{eqn:large_Q_and_sigma}
	Q(k) = \left( 1+o\left( \frac{1}{k} \right) \right) \, \frac{1}{\sqrt{\pi}} \, \Ec{\sttm_{k}^{-1/2}} \qquad\text{as } k \to \infty,
	\end{equation}
    and
    \begin{equation}\label{eqn:q_sigma_lower_and_upper}
	(1+o(1))\, \frac{1}{2\sqrt{\pi}} \, \frac{1}{k} \Ec{\sttm_{k+1}^{-1/2}} \leq q(k) \leq (1+o(1)) \,\frac{1}{2p \sqrt{\pi}} \, \frac{1}{k} \Ec{\sttm_k^{-1/2}} \qquad\text{as } k \to \infty.
	\end{equation}
    Moreover, there exists a constant $C>0$ such that for all $k\geq 1$, 
    \begin{equation}\label{eq:bound-kq/Q}
        \frac{1}{C} \, \frac{Q(k)}{k} \leq q(k) \leq C \, \frac{Q(k)}{k}.
    \end{equation}
    \end{lem}

\begin{proof}
    We divide the proof into five main steps.

    \medskip

    \noindent\underline{\emph{Step 1:}} \emph{Proof of~\eqref{eqn:large_Q_and_sigma}.} For this, recall that we set  $r(n)=\Pp{|T|=n}$ for all $n \geq 1$ and, from \cref{lem:prelim_GW}, we know that 
    \begin{equation} \label{eq:rn-est}
    r(n) 
    =\frac{1}{2\sqrt{\pi}}  n^{-3/2} \left(1+O \left( \frac{1}{n} \right)\right)
    \quad \text{as } n\to\infty,
    \end{equation}
    and $r(n)$ is decreasing in $n$.
    Therefore, summing over all possible values of $|T|$ and noting that $ \Pp{X_n \geq k}= \Pp{\sttm_k \leq n}$, we can write as $k\to\infty$:
	\begin{multline*}
		Q(k) = \Pp{\lis(T) \geq k} 
        = \sum_{n \geq 1} r(n) \Pp{X_n \geq k}
		= \Ec{\sum_{n \geq \sttm_k} r(n)}
		=  \frac{1}{2\sqrt{\pi}}\Ec{\sum_{n \geq \sttm_k} \left( 1+O \left( \frac{1}{n} \right) \right) n^{-3/2}}\\ 
        =\frac{1}{\sqrt{\pi}} \Ec{\sttm_k^{-1/2}} +O \left( \Ec{\sttm_k^{-3/2}} \right)
        =\frac{1}{\sqrt{\pi}} \Ec{\sttm_k^{-1/2}} +o \left(\frac{1}{k} \Ec{\sttm_k^{-1/2}} \right),
	\end{multline*}
    where the last equality follows from the second item of \eqref{eq:tail_sigma-2} in \cref{cor:lower_tail_sigma}. Thus \eqref{eqn:large_Q_and_sigma} holds.

    \medskip

    \noindent\underline{\emph{Step 2:}} \emph{Preparation for the proof of~\eqref{eqn:q_sigma_lower_and_upper}.}
    The key idea will be to use Rémy's algorithm forwards for the upper bound and backwards for the lower bound.
    As above, we first write
    \begin{align}
		q(k) = \sum_{n \geq 1} r(n) \Pp{X_n=k}
        = \Ec{\sum_{n=\sttm_k}^{\sttm_{k+1}-1} r(n)}
        &= \frac{1}{2\sqrt{\pi}} \left( 1+ O \left( \frac{1}{k} \right) \right)\Ec{\sum_{n=\sttm_k}^{\sttm_{k+1}-1} n^{-3/2} },
        \label{eq:q(k)_sum_sigma(k)}
        \end{align}
        where for the last equality we used \eqref{eq:rn-est} and that $\sttm_k \geq k$. Bounding every term of the sum by either its largest or smallest term yields 
        \begin{align}
        \frac{1}{2\sqrt{\pi}} \left( 1+ O \left( \frac{1}{k} \right) \right) \Ec{\frac{\sttm_{k+1}-\sttm_k}{\sttm_{k+1}^{3/2}}} \leq q(k)\leq \frac{1}{2\sqrt{\pi}} \left( 1+ O \left( \frac{1}{k} \right) \right) \Ec{\frac{\sttm_{k+1}-\sttm_k}{\sttm_k^{3/2}}}.
        \label{eqn:q_upper_bound_with_sigma}
        \end{align}

    \medskip

    \noindent\underline{\emph{Step 3:}} \emph{Proof of the upper bound of~\eqref{eqn:q_sigma_lower_and_upper}.}
    Let $n \geq k \geq 1$, and let us condition on $X_{n}=k$. Recalling the terminology introduced at the beginning of \cref{sect:positive-subtrees}, let $T_n^{\max}$ be a positive uncontracted subtree of $T_n$ with $k$ leaves (if it is not unique, pick one arbitrarily). Let also $\uedge{n}$ be the edge of $T_n$ that is picked by the R\'emy algorithm run forwards at step $n$. Note that $T_n^{\max}$ is a tree with $k$ leaves and vertex degrees $1$, $2$ or $3$, so it has at least $2k-1$ edges (this would be an equality if we considered the contracted version of $T_n^{\max}$ since all the vertices would have degree $1$ or $3$). If the edge $\uedge{n}$ picked by the R\'emy algorithm belongs to $T_n^{\max}$ and $\usign{n}=\oplus$, then the union of $T^{\max}_n$ and the new leaf $\uleaf{n+1}$ forms a positive subtree of size $n+1$, so $\sttm_{k+1}=n+1$. It follows that on the event $\{X_n=k\}$,
    \begin{equation} \label{eq:lb_sigma_geom}
    \Ppsq{\sttm_{k+1}=n+1}{\Ff_n} \geq p \frac{2k-1}{2n-1} \geq p \frac{k-1}{n}, 
    \end{equation}
    since $\mathbb{P}(\usign{n}=\oplus)=p$.
    
    From here, we will obtain the upper bound of~\eqref{eqn:q_sigma_lower_and_upper} by comparing the increment $\sttm_{k+1}-\sttm_k$ with a geometric random variable of parameter $p \frac{k}{\sttm_k}$ conditionally on $\sttm_k$. More precisely, by an immediate induction, for all $j \geq 0$, we obtain from \eqref{eq:lb_sigma_geom} that
    \[
    \Ppsq{\sttm_{k+1}>\sttm_k+j}{\sttm_k} \leq \prod_{i=0}^{j-1} \left( 1-p\frac{k-1}{\sttm_k+i} \right)
    \leq \exp \bigg( -p(k-1) \sum_{i=0}^{j-1} \frac{1}{\sttm_k+i} \bigg),
    \]
    where, when $j=0$, we interpret the empty product as $1$.
    Moreover, since
    \[
    \sum_{i=0}^{j-1} \frac{1}{\sttm_k+i} \ge \int_{\sttm_k}^{\sttm_k+j} \frac{\mathrm{d}x}{x}
    =
    \log\left(1+\frac{j}{\sttm_k}\right),
    \] 
    we deduce that
    \begin{equation}\label{eqn:bound_tail_increment_sigma}
    \Ppsq{\sttm_{k+1}>\sttm_k+j}{\sttm_k} 
    \leq \exp \bigg(-p(k-1) \log\left(1+\frac{j}{\sttm_k}\right) \bigg) 
    = \left( 1+\frac{j}{\sttm_k} \right)^{-p(k-1)}. 
    \end{equation}
    Therefore, we have for $k$ large enough so that $p(k-1)>1$,
    \begin{align*}
        \Ecsq{\sttm_{k+1}-\sttm_k}{\sttm_k} = \sum_{j \geq 0} \Ppsq{\sttm_{k+1}>\sttm_k+j}{\sttm_k}
        \leq \sum_{j \geq 0} \left( 1+\frac{j}{\sttm_k} \right)^{-p(k-1)}
        &\leq 1+ \int_0^{\infty} \left( 1+\frac{t}{\sttm_k} \right)^{-p(k-1)} \mathrm{d}t\\
        &= \left( 1+O \left( \frac{1}{k}\right)\right) \frac{\sttm_k}{pk} +1,
    \end{align*}
    where the constant in the $O$ notation is deterministic.
    This entails that 
    \begin{align}\label{eq:for-fut-ref}
    \Ec{\frac{\sttm_{k+1}-\sttm_k}{\sttm_k^{3/2}}} 
    &= \left( 1+O \left( \frac{1}{k} \right)\right) \frac{1}{pk}\Ec{\sttm_k^{-1/2}} + \Ec{\sttm_k^{-3/2}}.
    \end{align}
    Plugging this back into
    \eqref{eqn:q_upper_bound_with_sigma}, we obtain
    \[ q(k) \leq \frac{1}{2p\sqrt{\pi}} \left( 1+O \left( \frac{1}{k} \right)\right) \frac{1}{k} \Ec{\sttm_k^{-1/2}} + O \left( \Ec{\sttm_k^{-3/2}}\right).  \]
Finally, the last error term is negligible by the right-hand side of \eqref{eq:tail_sigma-2} in \cref{cor:lower_tail_sigma}, which proves the upper bound in~\eqref{eqn:q_sigma_lower_and_upper}.

\medskip

    \noindent\underline{\emph{Step 4:}} \emph{Proof of the lower bound of~\eqref{eqn:q_sigma_lower_and_upper}.} 
    The argument will be similar to the upper bound, but this time we will rely on the lower bound of~\eqref{eqn:q_upper_bound_with_sigma} and use the R\'emy algorithm run backwards. 
    More precisely, let $n \geq k \geq 2$, and let us condition on $X_{n}=k$, so that $\sttm_k \leq n$. 
    We recall from \cref{sect:positive-subtrees} that, conditionally on $\revF_{n}$, the decorated tree $T_{n-1}$ is obtained by removing a uniform leaf $\uleaf{n}$ of $T_{n}$. 
    As before, let $T_{n}^{\max}$ be a positive uncontracted subtree of $T_{n}$ with maximal size, \emph{i.e.}\ with $k$ leaves. 
    If $\uleaf{n} \notin T_{n}^{\max}$ (which occurs with probability $1-\frac{k}{n}$), then $T_{n}^{\max} \subset T_{n-1}$, which implies $X_{n-1}=X_{n}=k$, and so $\sttm_k < n$. Therefore, on the event $\{X_n=k\}$, we have that
    \begin{equation}\label{eqn:cond_proba_hitting_sigmak}
        \Ppsq{\sttm_{k}<n}{\revF_{n}} \geq \Ppsq{\uleaf{n} \notin T_{n}^{\max}}{\revF_{n}} = 1-\frac{k}{n}. 
    \end{equation}
    We refer to the end of the present section for a discussion on why this last inequality is not an equality, and why this matters a lot for the rest of the paper.
    \noindent For all $0 \leq j \leq \sttm_{k+1} - k$, we get by induction from \eqref{eqn:cond_proba_hitting_sigmak} that
    \begin{equation}\label{eqn:reverse_Remy_geometric_comparison}
        \Ppsq{\sttm_{k}<\sttm_{k+1}-j}{\sttm_{k+1}} \geq \prod_{i=1}^{j} \left( 1-\frac{k}{\sttm_{k+1}-i} \right) \geq \left( 1-\frac{k}{\sttm_{k+1}-j} \right)^{j},
    \end{equation}
    where, when $j=0$, we interpret the empty product as $1$.
    Our upper bound argument (see~\eqref{eqn:bound_tail_increment_sigma}) suggests that the right order of magnitude for $j$ is $\frac{\sttm_{k+1}}{k}$. Therefore, we now fix a (large) constant $K>0$, and assume $k>2K$.
    By summing over $j$, we find that on the event $\{K\sttm_{k+1}/k \leq \sttm_{k+1}-k \}$, we have
    \begin{align*}
        \Ecsq{\sttm_{k+1}-\sttm_k}{\sttm_{k+1}} &= \sum_{j \geq 0} \Ppsq{\sttm_{k}<\sttm_{k+1}-j}{\sttm_{k+1}}\\
        &\geq \sum_{j=0}^{ \lfloor K\sttm_{k+1}/k \rfloor} \left( 1-\frac{k}{\sttm_{k+1}-\frac{K}{k} \sttm_{k+1}} \right)^j\\
        &= \left( 1-\frac{K}{k} \right) \frac{\sttm_{k+1}}{k} \left( 1- \left( 1-\frac{k}{(1-K/k)\sttm_{k+1}} \right)^{\lfloor K\sttm_{k+1}/k \rfloor +1} \right)\\
        &\geq \left( 1-\frac{K}{k} \right) \frac{\sttm_{k+1}}{k} \left( 1- \mathrm{e}^{-K} \right),
    \end{align*}
    where the last line just used that $1-K/k \leq 1$ and $\lfloor K\sttm_{k+1}/k \rfloor +1 \geq K\sttm_{k+1}/k$.
    It follows that
    \begin{align}
    \Ec{\frac{\sttm_{k+1}-\sttm_k}{\sttm_{k+1}^{3/2}}} \nonumber
    &\geq \Ec{\Ecsq{\sttm_{k+1}-\sttm_k}{\sttm_{k+1}} \, \sttm_{k+1}^{-3/2} \, \indicator{\{K\sttm_{k+1}/k \leq \sttm_{k+1}-k}\}} \nonumber \\
    &\geq \left( 1-\frac{K}{k} \right) \left( 1-\mathrm{e}^{-K} \right) \frac{1}{k}  \Ec{\sttm_{k+1}^{-1/2} \indicator{\{K\sttm_{k+1}/k \leq \sttm_{k+1}-k}\}}. \label{eqn:lower_bd_increment_sigma_with_indicator}
    \end{align}
    To conclude, we argue that removing the indicator only costs a factor $1+o(1)$. If the event in the indicator does not occur, then $\sttm_{k+1} < \frac{k}{1-K/k} \leq 2k$. Hence, using this with the trivial bound $\sttm_{k+1}\geq k+1$, and then Markov's inequality, we have
    \begin{align*}
    \Ec{\sttm_{k+1}^{-1/2}\indicator{\{K\sttm_{k+1}/k > \sttm_{k+1}-k}\}} &\leq (k+1)^{-\frac{1}{2}} \cdot \Pp{\sttm_{k+1} < 2k} \\
    &= (k+1)^{-\frac{1}{2}} \cdot \Pp{ X_{2k-1} \geq k+1}\\
    &\leq  (k+1)^{-\frac{3}{2}} \Ec{X_{2k-1}} = k^{-\frac{3}{2}+\alpha+o(1)},
    \end{align*}
    by using~\eqref{eqn:alpha_as_limsup} for the last inequality. On the other hand, by \cref{cor:lower_tail_sigma}, we have $\Ec{\sttm_{k+1}^{-1/2}}=k^{-\frac{1}{2\alpha}+o(1)}$. Moreover, by \cref{thm:main}, we have $\frac{1}{2}<\alpha<1$, which implies $-\frac{3}{2}+\alpha<-\frac{1}{2\alpha}$. Therefore, we have
    \[ \Ec{\sttm_{k+1}^{-1/2}\indicator{\{K\sttm_{k+1}/k \leq \sttm_{k+1}-k}\}} = \left( 1+o(1) \right) \Ec{\sttm_{k+1}^{-1/2}},\]
    so~\eqref{eqn:lower_bd_increment_sigma_with_indicator} becomes
    \[ \Ec{\frac{\sttm_{k+1}-\sttm_k}{\sttm_{k+1}^{3/2}}} \geq \left( 1-\frac{K}{k} \right) \left( 1-\mathrm{e}^{-K} \right) (1+o(1)) \frac{1}{k} \Ec{\sttm_{k+1}^{-1/2}} \]
    as $k \to \infty$. Since the above inequality is true for any constant $K>0$, combined with the lower bound of~\eqref{eqn:q_upper_bound_with_sigma}, this proves the lower bound of~\eqref{eqn:q_sigma_lower_and_upper}.

    \medskip

    \noindent\underline{\emph{Step 5:}} \emph{Proof of~\eqref{eq:bound-kq/Q}.}
    The upper bound on $q$ from \eqref{eqn:q_sigma_lower_and_upper} implies $q(k)=o(Q(k))$. Therefore $Q(k+1)=Q(k)-q(k)=(1+o(1))Q(k)$ and again by \eqref{eqn:q_sigma_lower_and_upper}, 
    \begin{equation}\label{eq:k-k1-equival}
        \Ec{\sttm_{k+1}^{-1/2}} = (1+o(1)) \Ec{\sttm_k^{-1/2}}\qquad \text{as $k \to \infty$.}
    \end{equation}
    Hence, we may replace $\Ec{\sttm_{k+1}^{-1/2}}$ by $\Ec{\sttm_{k}^{-1/2}}$ in the lower bound on $q$ of \cref{lem:q_and_sigma}. Therefore, there exists a constant $C>0$ such that $\frac{1}{C} \frac{Q(k)}{k} \leq q(k) \leq C \frac{Q(k)}{k}$ for all $k$, as we wanted.
    \end{proof}
    
    We can now complete the proof of \cref{prop:rough_regularity}.
    
    \begin{proof}[Proof of \cref{prop:rough_regularity}]
        The estimates $q(k)=k^{-\frac{1}{2\alpha}-1+o(1)}$ and $Q(k)= k^{-\frac{1}{2\alpha} +o(1)}$ are immediate by combining \cref{lem:q_and_sigma} with \cref{cor:lower_tail_sigma}. Moreover, by \cref{lem:q_and_sigma}, there exists a constant $C>0$ such that for any $k$, we have
        \begin{equation}\label{eq:proof-rough}
            \frac{Q(k+1)}{Q(k)}=1-\frac{q(k)}{Q(k)} \geq 1-\frac{C}{k}.
        \end{equation}
        The estimate on $Q$ then follows by writing a telescopic product: for all $\svx \ge 1$,
        \begin{multline*} 
        \frac{Q(\svx k)}{Q(k)} \geq \prod_{i=k}^{\svx k-1} \left( 1-\frac{C}{i} \right) = \exp \left( -\sum_{i=k}^{\svx k-1} \left(\frac{C}{i} + O \left( \frac{1}{i^2}\right) \right) \right) \\
        \geq \exp \left(-C \log \frac{\svx k}{k} +O \left( \frac{1}{k} \right) \right) = (1+o(1))\, \svx^{-C}.
        \end{multline*}
        Finally, it remains to prove the estimate~\eqref{eqn:regularity_q} on $q$. By \eqref{eq:bound-kq/Q} in \cref{lem:q_and_sigma}, there exists a constant $C>0$ such that $\frac{1}{C} \frac{Q(k)}{k} \leq q(k) \leq C \frac{Q(k)}{k}$ for all $k$,
        so we can deduce the estimate on $q$ from that on $Q$.
\end{proof}

\noindent\emph{What is still missing?} As can be seen in this last proof, understanding the polynomial decay of $Q$ and $q$ is closely related to understanding the limit of the ratio $\frac{k q(k)}{Q(k)}$, which will be an important part of the rest of the paper. Note that \cref{lem:q_and_sigma} provides both an upper and a lower bound for this ratio, but neither is tight. Our strategy will be to turn the proof of the \emph{lower} bound of \cref{lem:q_and_sigma} (Step 4 of the proof) into a more precise argument. We first notice that the ``greedy'' part of the argument is the discussion leading to~\eqref{eqn:cond_proba_hitting_sigmak}, where we pick \emph{one} specific maximal positive subtree $T_{n}^{\max}$. To avoid such a waste, we can notice that $\sttm_k=n$ if and only if the leaf $\uleaf{n}$ selected from $T_{n}$ by the backward Rémy's algorithm belongs to \emph{all} the maximal positive subtrees of $T_{n}$. Therefore, on the event $\{X_n=k\}$, we may write
\begin{equation}\label{eq:discussion-for-good-est}
    \Ppsq{\sttm_{k}<n}{\revF_{n}} = 1-\frac{\# \mathcal{L}^{\max}_{\cap} (T_n)}{n}, 
\end{equation}
or equivalently,
\begin{equation}\label{eq:discussion-for-good-est2}
\Ppsq{\sttm_k=n}{\revF_{n}}   = \frac{\# \mathcal{L}_{\cap}^{\max}(T_{n})}{n}\mathds{1}_{X_{n}=k},
\end{equation}
where we recall that $ \mathcal{L}^{\max}_{\cap}(t)$ stands for the set of leaves in the intersection of \emph{all} positive subtrees of $t$ with maximal size. In particular, \cref{sect-lln-int} will be devoted to proving that the ratio $\frac{\# \mathcal{L}^{\max}_{\cap}(T_n)}{X_n}$ is concentrated around a certain -- non-explicit for now -- deterministic constant $\lambda$ (this is the law of large numbers in~\cref{thm:convergence_size_intersection}). The constant $\lambda$ is exactly the multiplicative constant by which the lower bound in \eqref{eqn:q_sigma_lower_and_upper} of \cref{lem:q_and_sigma} is off. The proof of the law of large numbers will rely on the local convergence of a maximal positive subtree seen from a uniformly chosen leaf, resulting from \cref{sect:loacal-lim}. Finally, in \cref{sect-bett-reg-q-Q}, equipped with this concentration result, we will run again the proof of \cref{prop:rough_regularity} in a more precise way to prove \cref{thm:good_regularity}. We stress that the proof of the law of large numbers will not provide by itself an explicit value for the constant $\lambda$. However, once \cref{thm:good_regularity} is proved, we will be able to argue \emph{a posteriori} that $\lambda=\alpha$; see \cref{rmk:lambda=alpha}.

\medskip

\noindent\emph{Note.} The proofs of the local convergence and the law of large numbers presented in~\cref{sect:loacal-lim} and~\cref{sect-lln-int} are both intricate and rich in details. A reader encountering these sections for the first time may prefer to initially skip ahead, assuming the validity of the law of large numbers stated in~\cref{thm:convergence_size_intersection}, and proceed directly to the proof of~\cref{thm:good_regularity} in~\cref{sect-bett-reg-q-Q}. Following this approach also has the additional benefit of immediately seeing how the proof of~\cref{prop:rough_regularity} lays the ground for the more refined proof of~\cref{thm:good_regularity}. 
The reader is then encouraged to visit~\cref{sect:loacal-lim} and~\cref{sect-lln-int}, as these sections offer deep insights into the landscape around a uniformly chosen leaf in a maximal positive subtree.

 \section{The local convergence of the leftmost maximal positive subtree of $T^k$ around a uniformly chosen leaf}\label{sect:loacal-lim}

In the previous section, we proved the rough regularity estimates for the sequences $q$ and $Q$ stated in \cref{prop:rough_regularity}. 
Along the way, we realized that the problem of finding sharper estimates was closely related to the existence of the local limit of a maximal positive subtree. We have all freedom to decide which maximal positive subtree to consider: to fix ideas, we will look at the \emph{leftmost} maximal subtree, which is defined by going left at each negative node carrying two subtrees with equal $\lis$ (see \cref{sect:mark-desc} for a precise definition).
The main goal of this section is to establish this local convergence (\cref{thm:localconv}) for such leftmost maximal positive subtree around a uniformly chosen leaf. It turns out that the most natural model to state this local convergence result is the $p$-signed critical binary Bienaymé--Galton--Watson model $T$ conditioned on $\lis(T)=k$, denoted by $T^k$, in the regime where $k \to \infty$.
This result will be later used in \cref{sect:better-bounds} to improve the rough regularity estimates for the sequences $q$ and $Q$; completing the proof of \cref{thm:good_regularity}.
The proof of the local convergence hinges upon a key coupling argument which forms the bulk of this section, and essentially enables us to compare the situations when $\lis(T)=k$ and when $\lis(T)=k'$.

This section is structured as follows. In \cref{sect:prem-and-loc} we build the ground to state our local convergence result (\cref{thm:localconv}) by introducing a Markov chain that records the behavior of $T^k$ and its leftmost maximal positive subtree along the spine from the root to the chosen leaf. \cref{sect:trans-prob} is devoted to certain preliminary computations on this Markov chain that will be later used in \cref{sect:coupling-arg} to run the main coupling argument. 

\medskip

\noindent\emph{Note.} All the subtrees considered in \cref{sect:loacal-lim} are uncontracted (see \cref{sect:positive-subtrees}).

\subsection{Preliminary results and statement of the local convergence}\label{sect:prem-and-loc}

The main goal of this section is to state our local convergence result (\cref{thm:localconv}) in \cref{sect:local-lim-statement}. We first introduce some preliminary notation and results in \cref{sect:mark-desc}.

\subsubsection{Markovian description of the tree $T^k$ around a uniformly chosen leaf of the leftmost maximal positive subtree} 
\label{sect:mark-desc}

For any finite sign-decorated binary tree $t$, we denote by $t^{\lmax}$ the \textbf{leftmost maximal positive subtree} of $t$. More precisely, $t^{\lmax}$ can be constructed recursively in the following way (\emph{cf.}\ \cref{fig:intersection-max-trees}):
\begin{itemize}
	\item If $t$ consists of a single leaf, then $t^{\lmax}=t$.
	\item If $t$ consists of a root vertex connected to two trees $t_{\lfa}$ on the left and $t_{\rga}$ on the right, then:
        \begin{itemize}
            \item If the root of $t$ is decorated with a $\oplus$ sign, then $t^{\lmax}$ consists of a $\oplus$ node connected to the two trees $t_{\lfa}^{\lmax}$ and $t_{\rga}^{\lmax}$.
	       \item If the root of $t$ is decorated with a $\ominus$ sign and $\lis(t_{\lfa}) \geq  \lis(t_{\rga})$, then $t^{\lmax}=t_{\lfa}^{\lmax}$.
	       \item If the root of $t$ is decorated with a $\ominus$ sign and $\lis(t_{\lfa}) < \lis(t_{\rga})$, then $t^{\lmax}=t_{\rga}^{\lmax}$.
        \end{itemize}
\end{itemize}

Recall that $\TLis{k}$ denotes the $p$-signed critical binary Bienaymé--Galton--Watson tree $T$ conditioned on the event $\{\lis(T)=k\}$.
Conditionally on $\TLis{k}$, let $L^k$ be a leaf chosen uniformly at random from the $k$ leaves of the leftmost maximal positive subtree 
\begin{equation}\label{eq:defn-tklmax}
    T^{k,\lmax}\coloneqq(\TLis{k})^{\lmax}.
\end{equation}
We first prove that the exploration of $\TLis{k}$ along the spine directed towards $L^k$ satisfies a nice Markov property. 
We denote by $\rho_k$ the root of $\TLis{k}$ and by $\TLis{k}_{\lfa}$ and $\TLis{k}_{\rga}$ the left and right subtrees of $\rho_k$ in $\TLis{k}$, respectively. 
We also denote by $D \in \{\lfa,\rga\}$ the direction such that $L^k \in \TLis{k}_{D}$. For the unconditioned $p$-signed critical binary Bienaymé--Galton--Watson tree $T$, we analogously define the root $\rho$ and the left and right subtrees $T_{\lfa}$ and $T_{\rga}$ when $\deg(\rho)=3$. 

\begin{lem}[Decomposition of $\TLis{k}$ at the root]
\label{lem:proba-comp}
    Let $k \geq 1$. Then the law of $\left( \lis(\TLis{k}_{\lfa}), \lis(\TLis{k}_{\rga}), D \right)$ is described as follows:
    \begin{enumerate}
		\item \label{item:proba-comp1} With probability $\frac{1}{2 q(1)} \mathds{1}_{k=1}$, the tree $\TLis{k}$ consists of a single vertex, which is a leaf.
		\item \label{item:proba-comp2} For all $1\leq i < k$, with probability $\frac 1 2 (1-p) q(i)$, the root $\rho_k$ is negative, $\lis(\TLis{k}_{\lfa})=i$ and $\lis(\TLis{k}_{\rga})=k$. 
        Moreover, in this case $D=\rga$. 
		\item \label{item:proba-comp3} For all $1\leq i \leq k$, with probability $ \frac 1 2 (1-p) q(i)$, the root $\rho_k$ is negative, $\lis(\TLis{k}_{\lfa})=k$ and $\lis(\TLis{k}_{\rga})=i$. Moreover, in this case $D=\lfa$.
		\item \label{item:proba-comp4} For all $1 \leq i \leq k-1$, with probability $ \frac p 2  \frac{q(i) q(k-i)}{q(k)}$, the root $\rho_k$ is positive, $\lis(\TLis{k}_{\lfa})=i$ and $\lis(\TLis{k}_{\rga})=k-i$. Moreover, conditionally on this, the direction $D$ is $\lfa$ with probability $\frac{i}{k}$ and $\rga$ with probability $\frac{k-i}{k}$.
    \end{enumerate}
    Moreover, conditionally on $\{( \lis(\TLis{k}_{\lfa}), \lis(\TLis{k}_{\rga}), D )=(k_1, k_2, d)\}$, the trees $\TLis{k}_{\lfa}$ and $\TLis{k}_{\rga}$ are independent trees with respective laws $\TLis{k_1}$ and $\TLis{k_2}$, and $L^k$ is a uniform leaf of $(\TLis{k}_{d})^{\lmax}$.
\end{lem}

\begin{proof}
    The first claim simply follows from  
    \[ \Pp{\TLis{k}= \{ \rho_k \} }
    =\frac{\Pp{T=\{\rho\},\lis(T)=k}}{\Pp{\lis(T)=k}}=
    \frac{\Pp{T=\{\rho\}}}{\Pp{\lis(T)=k}} \indicator{k=1}=
    \frac{1}{2 q(1)} \mathds{1}_{k=1}.
    \]
    
    For the second claim, fix $1\leq i < k$. We can rewrite $\Pp{\sss(\rho_k)=\ominus, \lis(\TLis{k}_{\lfa})=i, \lis(\TLis{k}_{\rga})=k}$ as
    \begin{multline*}
        \frac{\Pp{\deg(\rho)=2, \sss(\rho)=\ominus, \lis(T_{\lfa})=i, \lis(T_{\rga})=k, \lis(T)=k}}{\Pp{\lis(T)=k}} \\
        = \frac{\Pp{\deg(\rho)=2, \sss(\rho)=\ominus, \lis(T_{\lfa})=i, \lis(T_{\rga})=k}}{q(k)}.
    \end{multline*}
    By the branching property at the root vertex $\rho$ of $T$, we conclude that
    \[
    \Pp{\sss(\rho_k)=\ominus, \lis(\TLis{k}_{\lfa})=i, \lis(\TLis{k}_{\rga})=k}
    =
     \frac 1 2 (1-p)q(i) .
    \]
    Moreover, on the above event, we have $D=\rga$ since $\lis(\TLis{k}_{\rga})>\lis(\TLis{k}_{\lfa})$, so $L^k$ must be in $T_{\rga}$.

    The proof of the third claim is the same as the second one. Note that in this case $D=\lfa$ since $\lis(\TLis{k}_{\lfa})\ge \lis(\TLis{k}_{\rga})$ and we are considering the leftmost maximal subtree as defined at the beginning of \cref{sect:mark-desc}.

    For the fourth claim, fix $1\leq i \leq k-1$ so that in particular $k-1 \geq 1$. We rewrite the probability $\Pp{\sss(\rho_k)=\oplus, \lis(\TLis{k}_{\lfa})=i, \lis(\TLis{k}_{\rga})=k-i}$ as
    \begin{equation*}
        \frac{\Pp{\deg(\rho)=2, \sss(\rho)=\oplus, \lis(T_{\lfa})=i, \lis(T_{\rga})=k-i, \lis(T)=k}}{\Pp{\lis(T)=k}}= \frac{p}{2} \frac{q(i)q(k-i)}{q(k)},
    \end{equation*}
    where for the equality we used that $\lis(T)=\lis(T_{\lfa})+\lis(T_{\rga})=k$ is implied by all the other conditions. The claim about the conditional distribution of $D$ immediately follows.

    The very last claim about the conditional independence  and the conditional distribution of $L^k$ is immediate.
\end{proof}

By iterating this Markov property, we can introduce a Markov chain that describes the law of the exploration of $\TLis{k}$ and its leftmost maximal positive subtree along the spine directed towards $L^k$.

\begin{figure}[b!]
	\begin{center}
		\includegraphics[width=0.4\textwidth]{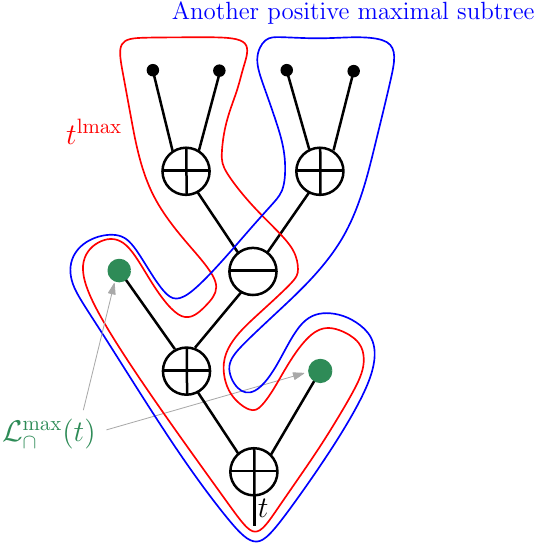}  
		\caption{A sign-decorated binary tree $t$ with two maximal positive subtrees highlighted in red and blue. The leftmost maximal positive subtree $t^{\lmax}$ is the red tree. The leaves in $\mathcal{L}_{\cap}^{\max}(t)$ are highlighted in green. \label{fig:intersection-max-trees}}
	\end{center}
	\vspace{-3ex}
\end{figure}

\subsubsection{Description of the Markov chain and statement of the local convergence}\label{sect:local-lim-statement}

For every $h\geq 0$, consider the vertex $v_h$ in the tree $\TLis{k}$ which is the ancestor of the marked leaf $L^k$ at height $h\geq 0$; \emph{cf.}\ \cref{fig:Markov-chain-notation}. 
In particular, $v_0= \rho_k$ is the root of the tree $\TLis{k}$ and, denoting by $\sth-1$ the height of $L^k$ in $\TLis{k}$, we have that $v_{\sth-1}=L^k$
and we can interpret $v_\sth$ as a fictitious vertex lying above $L^k$ playing the role of a \emph{cemetery}. Importantly, we consider $v_{\sth-1}$ as a leaf, but not $v_\sth$.

Given the sign-decorated binary tree $\TLis{k}$, we introduce\footnote{There are two reasons why the third component $\Mst{h}$ is marked with a star. First, it is the most important component of the process $Z^k$ since the transition probabilities of $Z^k$ will only depend on the value of $\Mst{h}$, as we will see in~\cref{sect:trans-prob2}. Second, starting from~\eqref{eq:decreasing-eq}, we will mainly focus on an accelerated version of $\Mst{}$ which we will simply denote by $\DMC$.} the process $Z^k$,
\begin{equation}\label{eq:def_MC_Z}
(Z^{k}_h)_{h\ge 0} \coloneqq (S_h,D_h,\Mst{h},W_h)_{h\ge 0},
\end{equation}
with values in $\{\ominus,\oplus\} \times \{{\lfa},{\rga}\} \times \N \times \N$, defined as follows (and illustrated on  \cref{fig:Markov-chain-notation}):
\begin{enumerate}
	\item Let $(S_0,D_0,\Mst{0},W_0) = (\ominus,{\lfa},k,0)$, by convention.
	\item For any $h\in\intervalleentier{1}{\sth -1}$:
	\begin{itemize}
		\item $S_h$ is the \textbf{sign} in $\{\ominus, \oplus\}$ of $v_{h-1}$.
		\item $D_h$ is the \textbf{direction} in $\{{\lfa}, {\rga}\}$ indicating whether $v_h$ is the left or the right child of $v_{h-1}$.
		\item $\Mst{h}$ is the number of leaves of $\TLis{k,\lmax}$ \emph{weakly}\footnote{Note that the \emph{weakly} condition is relevant only when $h=\eta-1$ so that $v_{\eta-1}$ is a leaf, \emph{i.e.}\ at the very end of the process.} above $v_h$ (the size of the \textbf{main} tree, which  contains the marked leaf).
		\item $W_h$ is the number of leaves in a maximal positive subtree of the subtree rooted at the sibling of $v_h$, \emph{i.e.}\ the child of $v_{h-1}$ which is not $v_h$ (we use the letter $W$ since it is $M$ upside down). Note that if  $S_h=\oplus$, then $W_h$ is the number of leaves of $\TLis{k,\lmax}$ \emph{weakly} above the sibling of $v_h$.
	\end{itemize}
        \item We use the convention that the chain is sent at time  $\sth$ to some \textbf{cemetery state} 
        \begin{equation}\label{eq:cemetery state}
            \dagger \coloneqq (\oplus,{\lfa},0,0)
        \end{equation}
        and stays there afterwards. 
        Note that defining the chain in this way allows us to write
        \[\sth= \inf\enstq{h\geq 0}{\Mst{h}=0}.\]
\end{enumerate}

\begin{figure}[ht]
	\begin{center}
		\includegraphics[width=0.85\textwidth]{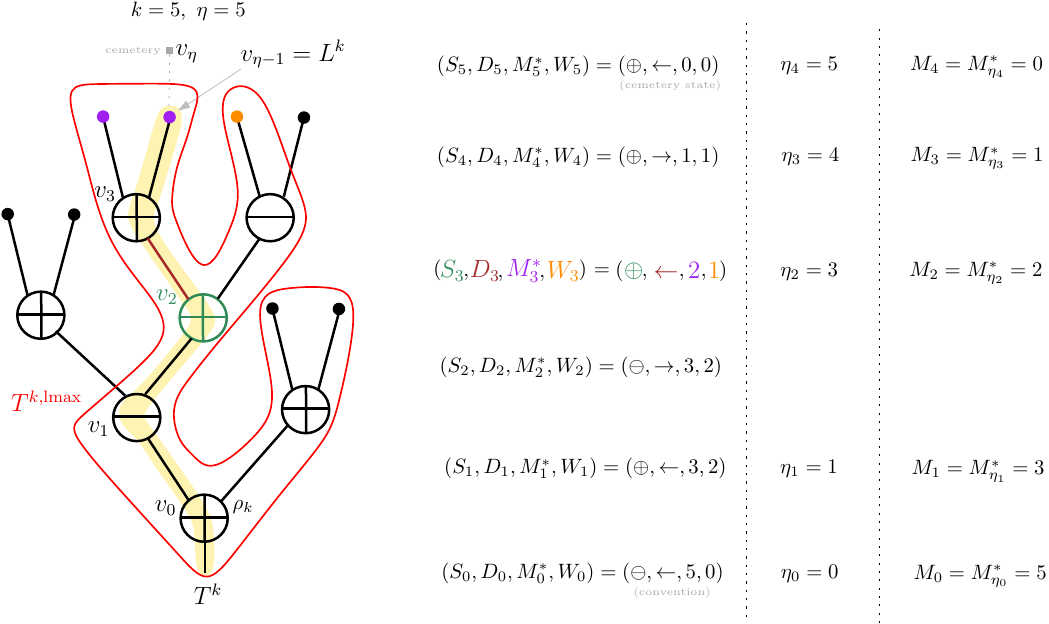}  
		\caption{A schematic representation of the notation introduced for the process $(Z^{k}_h)_{h\ge 0} = (S_h,D_h,\Mst{h},W_h)_{h\ge 0} $. Here $k=5$ and $\sth=5$. The spine from the root to the marked leaf $L^k$ is highlighted in yellow and $\TLis{k,\lmax}$ is marked in red. We used the other colors to highlight the elements of the tree $\TLis{k}$ which determine the values of $(S_3,D_3,M^*_3,W_3)$. The stopping times $\stht{t}$ and the decreasing Markov chain $\DMC_t=\Mst{\stht{t}}$ on the right will be introduced in \eqref{eq:defn-stopping-times} and \eqref{eq:decreasing-eq}. \label{fig:Markov-chain-notation}}
	\end{center}
	\vspace{-3ex}
\end{figure}
We will denote by $S$ the process $(S_h)_{h\ge 0}$ in the first component of $Z^k$. 
Similarly, we denote by $D$, $\Mst{}$, and $W$ the processes in the other components of $Z^k$. Since $Z^k$ is sent to the absorbing cemetery state $\dagger$ from time $\sth$ onward, we shall restrict the study of $Z^k$ to that of $(Z^k_h)_{0\le h \le \sth}$.

It follows from \cref{lem:proba-comp} that $Z^k$ is a Markov chain, whose transition probabilities will be given in \cref{sect:trans-prob}.
Recall that $\mathcal{L}_{\cap}^{\max}(t)$ denotes the set of leaves of $t$ contained in all the maximal positive subtrees of $t$.
Note that the only way for $T^k$ to have maximal positive subtrees that do not contain $L^k$ is to find some $h \geq 1$ such that the node $v_{h-1}$ is negative and $\Mst{h} = W_h$. Hence, we can express the event that $L^k$ belongs to $\mathcal{L}^{\max}_\cap(T^k)$ as a function of the trajectory of the chain, that is,
\begin{align}\label{eq:the target leaf belong to the intersection tree}
\left\{L^k\in \mathcal{L}^{\max}_\cap(T^k)\right\} =\Bigg\{ \sum_{h\geq 1} \mathds{1}\{S_h=\ominus \text{ and }\Mst{h}=W_h\} = 0\Bigg\}.
\end{align} 
This chain will be used later in \cref{thm:convergence_size_intersection} to establish, through \eqref{eq:the target leaf belong to the intersection tree} and the local convergence in~\cref{thm:localconv}, that 
\[ \frac{\# \mathcal{L}^{\max}_{\cap}(\TLis{k})}{k}  \to \lambda >0\] 
in probability as $k$ tends to infinity. This is why we consider this four-component Markov chain.

\medskip

We want to prove the convergence of the process $Z^k$ defined in~\eqref{eq:def_MC_Z} ``seen from the end'', \emph{i.e.}\ of the time-reversal of $Z^k$, because our interest lies in the behavior of $Z^k$ near the chosen leaf $L^k$.

\begin{thm}[Local convergence of the leftmost maximal positive subtree around a uniformly chosen leaf]\label{thm:localconv}
    The process $(Z^k_{\sth-h})_{0 \leq h \leq \sth}$ converges in distribution as $k\to\infty$ to a limiting process $(\overset{\leftarrow}{Z}_h)_{h \geq 0}$, in the sense of finite-dimensional distributions.
\end{thm}

The proof of the theorem above will mainly rely on a coupling argument; see \cref{prop:coupling_merging} and \cref{cor:full_coupling}.

\medskip

We conclude this section by explaining why \cref{thm:localconv} can be interpreted as a local convergence result. Indeed, one can deduce\footnote{This proof is standard and omitted, as it is not strictly necessary for our purposes.} from \cref{thm:localconv} that the triplet $(\TLis{k}, \TLis{k,\lmax}, L^k)$, viewed as a tree $\TLis{k}$ marked at the leaf $L^k$ and equipped with the distinguished subtree $\TLis{k,\lmax}$, converges in the Benjamini--Schramm sense to a triplet
\[\left(\TLis{\infty},\TLis{\infty,\lmax},L^\infty\right).\]
This limit consists of an infinite tree $\TLis{\infty}$ with a marked leaf $L^\infty$ and a distinguished infinite “leftmost maximal” positive subtree $\TLis{\infty,\lmax}$. In particular, the triplet $\left(\TLis{\infty}, \TLis{\infty,\lmax}, L^\infty\right)$ is constructed as follows:

\begin{itemize}
    \item Start by letting $L^\infty$ be the marked leaf, with an infinite spine below it. Set $u_0=L^\infty$ and for all $h\geq 1$, let $u_h$ be the $h$-th vertex below $L^\infty$ along this infinite spine. The infinite tree $\TLis{\infty}$ should be regarded as rooted at the bottom of this spine, in line with our usual convention of rooting trees at their bottom vertex.
    \item Then, for all $h\geq 1$, conditioning on the value $\overset{\leftarrow}{Z}_h=(\overset{\leftarrow}{S}_h,\overset{\leftarrow}{D}_h,\overset{\leftarrow}{\Mst{h}},\overset{\leftarrow}{W}_h)$:
    \begin{itemize}
    \item Assign the sign $\overset{\leftarrow}{S}_h$ to $u_h$.
    \item Attach an independent tree $\TLis{\overset{\leftarrow}{W}_h}$ to $u_h$ on the left\footnote{Here left and right refer to the usual notions for the plane where the infinite spine is drawn.} (resp.\ right) side of the infinite spine
    if $\overset{\leftarrow}{D}_h={\rga}$ (resp.\ if $\overset{\leftarrow}{D}_h={\lfa}$).
    \end{itemize}
    \item Let $\TLis{\infty}$ be the tree obtained through the procedure above regarded as rooted at the bottom of this spine and with the marked leaf $L^\infty$. 
    \item Finally, $\TLis{\infty,\lmax}$ is the union of the spine $\{ u_h, \, h \geq 0\}$ and of the trees $\TLis{\overset{\leftarrow}{W}_h,\lmax}$ for $h \geq 0$ such that $\overset{\leftarrow}{S}_h = \oplus$. This is also an infinite tree regarded as rooted at the bottom of the spine.
\end{itemize}

\subsection{Transition probabilities and the associated decreasing Markov chain}\label{sect:trans-prob}

The first step in the proof of \cref{thm:localconv} is to compute the transition probabilities of the  Markov chain $(Z^{k}_{h})_{0 \leq h \leq \sth}$ introduced in \eqref{eq:def_MC_Z}. This is done in \cref{sect:trans-prob2}. Then, in \cref{sect:decreasing-chain} we will introduce a simpler (single-component) decreasing Markov chain $\DMC$ that is a reduced and time-changed version of $Z^k$.
The heart of \cref{thm:localconv} will then be to produce a coupling of $\DMC$ and $\DMC'$ started from two values $k$ and $k'$, in which the two chains merge fast enough as $k,k'\to \infty$. This  goal is achieved in \cref{sect:coupling-arg}.

\subsubsection{Transition probabilities of the Markov chain}
\label{sect:trans-prob2}

Recall that the process $(Z^k_{h})_{0 \leq h \leq \sth}$ defined in \eqref{eq:def_MC_Z} is a Markov chain. 
Let
\[
 \tpr\left((s,d,m,w);(s',d',m',w')\right)
\]
denote its transition probabilities. We start by noticing that by the very end of \cref{lem:proba-comp}, these transition probabilities do not depend on $s, d, w$.

We first determine these transition probabilities when $m\geq 2$. From \cref{item:proba-comp2,item:proba-comp3} of \cref{lem:proba-comp}, we have
\begin{equation} \label{eq:trans1}
    \tpr((*,*,m,*)  ; (\ominus, d, m ,w')) = \frac{1-p}{2} q(w')
\end{equation}
whenever either $d={\lfa},\,\, m\geq 2, \,\, 1\leq w'\leq m$, or $d={\rga},\,\, m\geq 2, \,\, 1\leq w'< m$.
Additionally, from \cref{item:proba-comp4} of \cref{lem:proba-comp}, we have
\begin{align}\label{eq:trans2}
	\tpr((*,*,m,*) ; (\oplus, d, m' ,m-m')) =\frac p 2 \frac{m'}{m} \frac{q(m') q(m-m')}{q(m)} \,\quad \text{if } d\in \{{\lfa},{\rga}\},\,\, m\geq 2,\,\, 1\leq m' \leq m-1.
\end{align}
Note that this indeed defines a Markov kernel, since
\begin{align} \label{eq:Markov_kernel}
	&\sum_{(s',d',m',w')}	\tpr((*,*,m,*); (s',d',m',w')) \notag\\
&=	\sum_{m'=1}^{m-1}\frac{1-p}{2}q(m') + 	\sum_{m'=1}^{m}\frac{1-p}{2}q(m') + \sum_{m'=1}^{m-1}\frac{p}{2} \frac{m'}{m} \frac{ q(m') q(m-m')}{q(m)} +\sum_{m'=1}^{m-1}\frac{p}{2}  \frac{m'}{m} \frac{ q(m') q(m-m')}{q(m)} \notag\\
	&= \frac{1}{q(m)}\cdot \left((1-p) q(m)\sum_{m'=1}^{m-1}q(m') + \frac{1}{2}(1-p) q(m)^2 + \frac{p}{2}\sum_{m'=1}^{m-1} \frac{m' + m-m'}{m}q(m') q(m-m') \right) \notag\\
	&= \frac{1}{q(m)}\cdot q(m)=1, 
\end{align}
where for the second-to-last equality we used the recursive relation for $q(m)$ introduced in \cref{lem:q_rel}.

For $m=1$, note that $\tpr((*,*,1,*);(\oplus, *, * ,*))=0$. Moreover, from \cref{item:proba-comp1} of \cref{lem:proba-comp}, we have probability 
\[\tpr((*,*,1,*);(\ominus, {\lfa}, 0 ,0))= \frac{1}{2q(1)}\] 
to jump to the cemetery state $(\oplus,{\lfa},0,0)$, and complementary probability 
\[\tpr((*,*,1,*);(\ominus, {\lfa}, 1 ,1))=1-\frac{1}{2q(1)}\]
to jump to $(\ominus, {\lfa}, 1 ,1)$.
Since the transition probabilities only depend on the third coordinate $m$, most of the rest of this section will be devoted to the study of the Markov chain $\Mst{}=(\Mst{h})_{0 \leq h \leq \sth}$.

\subsubsection{The associated decreasing Markov chain}\label{sect:decreasing-chain}

The Markov chain $\Mst{}$ is non-increasing, but it will be convenient for us to make it decreasing. For this,
we consider the stopping times $(\stht{t})_{t\geq 0}$, where $\stht{0}=0$ and for $t\geq 1$,
\begin{align}\label{eq:defn-stopping-times}
	\stht{t}\coloneqq\inf\enstq{h > \stht{t-1}}{\Mst{h} < \Mst{h-1}}
\end{align}
is the $t$-th decreasing ladder epoch; 
see the second column on the right-hand side of \cref{fig:Markov-chain-notation} for an example. 
This means that either $S_{\stht{t}}=\oplus$, so that $v_{\stht{t}-1}$ is the $t$-th vertex along the spine from the root to $L^k$ that carries a $\oplus$ sign, or that we reached the fictitious vertex $v_{\sth}$ 
and that $\Mst{\stht{t}}=0$. 
We let
\[\mathtt{T} \coloneqq \inf\enstq{t\geq 0}{\Mst{\stht{t}}=0}\]
be the total number of times that the chain $\Mst{}$ decreases during the process. Note that $\stht{\mathtt{T}}=\sth$. 

From now on, we focus our attention on the decreasing Markov chain $M$ defined by
\begin{equation}\label{eq:decreasing-eq}
    (\DMC_t)_{0\leq t \leq \mathtt{T}}\coloneqq\left(\Mst{\stht{t}}\right)_{0\leq t \leq \mathtt{T}}\,,
\end{equation}
see the third column on the right-hand side of \cref{fig:Markov-chain-notation} for an example.
Indeed, from the transitions of the chain, we can easily check that, conditionally on the trajectory $(\DMC_t)_{0\leq t \leq \mathtt{T}}$, the distribution of the trajectory of the whole process $Z^k$ is given as follows:

\begin{itemize}
\item The  waiting times $(\stht{t}-\stht{t-1})_{1\leq t\leq \mathtt{T}-1}$ are independent geometric random variables with respective parameters $\theta(\DMC_{t-1})$, where for $m \geq 2$ (note that $\DMC_{t-1}\geq 2$ when $1\leq t\leq \mathtt{T}-1$),
\begin{align} \label{eq:trans_Mtau}
\theta(m)= \Ppsq{S_{1} = \oplus}{\Mst{0}=m} &= \sum_{m'=1}^{m-1} \sum_{d'\in\{{\lfa},{\rga}\}}\tpr((*,*,m,*) ; (\oplus, d' , m' ,m-m')) \notag \\
&= 1-(1-p)\bigg(\sum_{m'=1}^{m-1}q(m')+\frac{q(m)}{2}\bigg)\notag\\
&= 1-(1-p)\bigg(1-Q(m)+\frac{Q(m)-Q(m+1)}{2}\bigg)\notag\\
&= p+(1-p)\frac{Q(m)+Q(m+1)}{2},
\end{align}
where in the third equality we used \eqref{eq:trans2} and the fact that we know that the second line of \eqref{eq:Markov_kernel} equals one. 
Furthermore, the final waiting time $\stht{\mathtt{T}}-\stht{\mathtt{T}-1}$ is another independent geometric random variable, but with parameter $\theta(1)= \frac{1}{2q(1)}$.

\item For any $t\in \intervalleentier{0}{\mathtt{T}-1}$ and  $h\in \intervalleentier{1}{\stht{t+1}-\stht{t}-1}$, 
we have 
\begin{equation} \label{eq:between_wait_times}
\Mst{\stht{t}+h}=\DMC_t\qquad\text{ and }\qquad S_{\stht{t}+h}=\ominus.
\end{equation}
Moreover, conditionally on $(\DMC_t)_{0\leq t\leq \mathtt{T}}$ and $(\stht{t})_{0\leq t\leq {\mathtt{T}}}$, the random variables $(D_{\stht{t}+{h}},W_{\stht{t}+{h}})$ are independent for all $t\in \intervalleentier{0}{\mathtt{T}-1}$ and  ${h}\in \intervalleentier{1}{\stht{t+1}-\stht{t}-1}$. Their conditional distribution is given as follows:
for $d={\lfa},1\leq w\leq m$, or for $d={\rga},1\leq w< m$,
\begin{align}
    \Ppsq{D_{\stht{t}+{h}}=d, W_{\stht{t}+{h}}=w }{ {\DMC_t}=m} 
    &=\Ppsq{D_1=d, W_1=w}{S_1=\ominus, {\Mst{0}}=m} \notag\\
    &= \frac{\tpr\left((*,*,m,*);(\ominus,d,m,w)\right)}{\Ppsq{S_1=\ominus}{{\Mst{0}}=m}} \notag\\
    &=\frac{1}{2}\frac{ q(w)}{\left(1-\frac{Q(m)+Q(m+1)}{2}\right)}, \label{eq:cond_law_D_P_between}
\end{align} 
where for the last equality we used \eqref{eq:trans1} and \eqref{eq:trans_Mtau}.

\item Finally, for all $t\in \intervalleentier{0}{\mathtt{T}-1}$, we have $S_{\stht{t}}=\oplus$,  
\begin{equation*}
    W_{\stht{t}}={\Mst{\stht{t}-1}-\Mst{\stht{t}}\stackrel{\eqref{eq:between_wait_times}}{=}\DMC_{t-1}-\DMC_{t}},
\end{equation*}
and 
\begin{equation} \label{eq:cond_law_D_P_DMC}
\Ppsq{D_{\stht{t}}={\lfa}}{S,{\Mst{}},W}=\Ppsq{D_{\stht{t}}={\rga}}{S,{\Mst{}},W}=\frac 1 2.
\end{equation}
\end{itemize}

\noindent {We now compute the transition probabilities of the decreasing Markov chain $M$.} By the branching property, we have
\[
\tpr(m ; m') \coloneqq \Ppsq{\DMC_{t+1}=m'}{\DMC_t=m}
	=\Ppsq{\Mst{\stht{1}}=m'}{\Mst{0}=m},
\]
so $\DMC$ is a time-homogeneous Markov chain.
Moreover, the transition probabilities $\tpr(\cdot \,;\, \cdot)$ can be expressed, for all $m\geq 2$ and $1 \leq m'<m$, as
\begin{align}\label{eq:transitions decreasing chain}
	\tpr(m ; m')
    &=\sum_{s=0}^\infty \Ppsq{\stht{1}\ge s, \Mst{s+1}=m'}{\Mst{0}=m}\notag \\
    &= \sum_{s=0}^\infty\bigg((1-p)\bigg(\sum_{j=1}^{m-1}q(j)+\frac{q(m)}{2}\bigg)\bigg)^s \cdot \Ppsq{\Mst{s+1} =m'}{\Mst{0} = m,\stht{1}\ge s} \qquad \text{ by \eqref{eq:trans_Mtau}}\notag \\
    &= \sum_{s=0}^\infty\bigg((1-p)\bigg(\sum_{j=1}^{m-1}q(j)+\frac{q(m)}{2}\bigg)\bigg)^s \cdot \Ppsq{\Mst{1} =m'}{\Mst{0} = m} \, \quad \text{ by the Markov property}\notag \\
    &= \sum_{s=0}^\infty\bigg((1-p)\bigg(\sum_{j=1}^{m-1}q(j)+\frac{q(m)}{2}\bigg)\bigg)^s \cdot p \cdot\frac{m'}{m}\cdot  \frac{q(m') q(m-m')}{q(m)} \,\,\qquad \text{ by \eqref{eq:trans2}} \notag \\
	&= \frac{1}{1+\frac{(1-p)}{p}\left(Q(m)-\frac{q(m)}{2}\right)} \cdot \frac{m'}{m}\cdot  \frac{q(m') q(m-m')}{q(m)}.
\end{align} 
Moreover $\tpr(1;0)=1$ and the chain is killed at $0$.

\medskip

\noindent\emph{Important discussion.} It can also be useful to write the transition probabilities in the following way:
\begin{align}\label{eq:transitions decreasing chain2}
	\tpr(m \,;\, m-w)&= \Ppsq{\DMC_t - \DMC_{t+1}=w}{\DMC_t=m} \notag \\
    &= \frac{1}{1+\frac{(1-p)}{p}\left(Q(m)-\frac{q(m)}{2}\right)} \cdot \frac{m-w}{m} \cdot \frac{q(m-w)}{q(m)} \cdot q(w),
\end{align}
where it is reasonable to expect that all the factors except $q(w)$ are close to $1$ when $w$ is fixed and $m$ is large.\footnote{Equation \eqref{eqn:regularity_q} in \cref{prop:rough_regularity} proves that $\frac{q(m-w)}{q(m)}$ is bounded from above and below by positive constants, but not that it is close to $1$.} 
This suggests  comparing the jumps of $\DMC$ with those of a decreasing random walk with step distribution $\mathcal Q$ given by $\left(q(i)\right)_{i\geq 1}$. Note from the rough regularity estimates in~\cref{prop:rough_regularity} that $q(i) = i^{-\gamma+o(1)}$ as $i\to\infty$, with $\gamma \in ( 3/2,2 )$ as in \eqref{eq:gamma-alpha-rel}. In particular, such a random walk has jumps with infinite mean.

\begin{remark}\label{rmk:start-notation}
    The decreasing process $(\DMC_t)_{0\le t\le \mathtt{T}}$ was defined in~\eqref{eq:decreasing-eq} as a time-change of $(\Mst{h})_{0 \leq h \leq \sth}$, which in turn is the third coordinate of the process $(Z^k_{h})_{0 \leq h \leq \sth}$ defined in \eqref{eq:def_MC_Z}. 
    We will use the notation $\mathbb P_k$ and $\mathbb E_k$ to denote the corresponding probability measures and expectations of the process $Z^k$. 
    In particular, by construction, under $\mathbb{P}_{k}$, the processes $(\DMC_t)_{0\le t\le \mathtt{T}}$ and $(\Mst{h})_{0 \leq h \leq \sth}$ are both started at $k$ because $Z^k_0=(S_0,D_0,\Mst{0},W_0) = (\ominus,{\lfa},k,0)$, by convention.
\end{remark}

\begin{remark}\label{rem:point-process}
It can be useful to think of the set of values taken by the Markov chain $\DMC$ under the measure $\bbP_k$ as a point process $\mathscr{P}^k=\{\DMC_t, \ t\ge 0\}$ on $\N$. Indeed, \cref{thm:localconv} restricted to $M$ may be recast as the (vague) convergence of $\mathscr{P}^k$, as $k\to\infty$, to a limiting point process $\mathscr{P}$. 
A nice feature of $\mathscr{P}^k$ is that conditioning on $\mathscr{P}^k\cap [m,\infty)$ allows for more flexibility than conditioning on the first steps of $\DMC$, as it allows to condition on partial information about one of the jumps. In particular, this ``spatial'' Markov property explains why we can expect independence across scales, even though one single step of $\DMC$ can in theory jump over several scales (see the ``spatial'' filtration introduced in \eqref{eq:def_Ff}). However, we will not explicitly use this point process representation.
\end{remark}

\subsection{The coupling argument}\label{sect:coupling-arg}

Let $\DMC$ and $\DMC'$ be two versions of the decreasing Markov chain introduced in \eqref{eq:decreasing-eq}, started from two large values $k$ and $k'$.
The goal of this section is to construct a coupled version of those chains (see \cref{defn:black-golden coupling}) such that, with high probability as $k$ and $k'$ tend to infinity, the last finite number of steps of the  two chains before extinction are the same (\cref{prop:coupling_merging}). We will then quickly explain how to upgrade this coupling to the full Markov chain $Z^k$ (\cref{cor:full_coupling}), and how to deduce Theorem~\ref{thm:localconv}.
The bulk of this section is devoted to the proof of \cref{prop:coupling_merging}, which will be completed in \cref{sect:proof-coupling}, building on the coupling estimates of \cref{sect:4tech-lem}.

\subsubsection{Introducing the coupling: the black and golden instructions} \label{sec:coupling_blue_red}

We assume for the moment that there exists an infinite deterministic set of good scales, denoted by $\mathrm{GoodScales}\subset \mathbb{N}$, that satisfy certain regularity conditions; see \cref{defn:good-scales} for a precise definition and \cref{lem:existence good scales} for the existence of such good scales. In particular, these good scales are \emph{well-separated}, in the sense that if $\ell,\ell'\in\mathrm{GoodScales}$ with $\ell'>\ell$ then $\ell'> 7\ell$.

Then in \cref{defn:black-golden coupling}, we construct a coupling $\mathbb{P}_{k,k'}$ of the two Markov chains $\DMC$ and $\DMC'$ in \eqref{eq:decreasing-eq}, with marginals $\mathbb{P}_k$ and $\mathbb{P}_{k'}$ respectively,
by using a set of rules for determining what step the chains $\DMC$ and $\DMC'$ take at every site of $\N$. The sort of coupling that we have in mind is what Aldous and Thorisson coined as \emph{shift-couplings} \cite{aldous1993shift}, since the two chains will \textit{a priori} merge at different times in their own evolutions.

First,  we introduce the \textbf{black instructions} $(B_m)_{m\geq 1}$ as independent random variables defined so that 
\begin{align}\label{eq:blue-distrib}
	\forall m'<m, \qquad \Pp{B_m = m'} = \tpr(m;m'),
\end{align}
where $\tpr(m;m')$ is as in \eqref{eq:transitions decreasing chain}.
Then, for every good scale $\ell\in \mathrm{GoodScales}$, we introduce the \textbf{golden instructions} $(G_m)_{5\ell \leq m \leq 7\ell}$ that are coupled together for a given good scale $\ell\in \mathrm{GoodScales}$, but independent between different good scales and independent of the black instructions $(B_m)_{m\geq 1}$, with marginal law
\begin{align}\label{eq:red-law}
    \forall m'  \in \intervalleentier{\ell}{4\ell}, \quad
	\Pp{G_m= m' } = \Ppsq{B_m = m'}{B_m  \in \intervalleentier{\ell}{4\ell}}.
\end{align}
The coupling of $(G_m)_{5\ell \leq m \leq 7\ell}$ will be chosen so that, with uniformly positive probability over all $\ell\in\mathrm{GoodScales}$, we have
\[ \forall a,a' \in \intervalleentier{5\ell}{7\ell}, \qquad G_a = G_{a'}. \]
The existence of such coupling of the golden instructions will be shown later in \cref{lem:prop-good-scales-item3} of \cref{lem:prop-good-scales}.

\begin{defn}[The black-golden coupling]\label{defn:black-golden coupling}
    We say that the decreasing chains $\DMC$ and $\DMC'$ have the \textbf{black-golden coupling} when the dynamics of $\DMC$ (and $\DMC'$) are described for all $t\geq 0$ as follows:\footnote{We stress that here we are using a \emph{unique} set of black and golden instructions for \emph{both} chains $\DMC$ and $\DMC'$.}
    \begin{equation}\label{eq:coupling X and X'}
    	\DMC_{t+1} = \left\lbrace \begin{aligned}
    	& G_{\DMC_t} \qquad \text{if } \DMC_t \in \intervalleentier{5\ell}{7\ell} \text{ for some good scale $\ell$  and } B_{\DMC_t}  \in \intervalleentier{\ell}{4\ell}; \\
    	&B_{\DMC_t} \qquad \text{otherwise.}
    \end{aligned}
	\right.
    \end{equation}
    Moreover, when two chains $\DMC$ and $\DMC'$ are started from $k$ and $k'$ respectively, we denote their joint law under this coupling by $\mathbb{P}_{k,k'}$.
\end{defn}

In more intuitive terms, the chains $\DMC$ and $\DMC'$ typically follow the black instructions, except in a specific situation: when a chain is instructed to jump -- via a black instruction -- from a position in $\intervalleentier{5\ell}{7\ell}$ to somewhere inside $\intervalleentier{\ell}{4\ell}$, it then follows the corresponding golden instruction instead.
It is clear from this definition that the marginals $M$ and $M'$ under $\mathbb{P}_{k,k'}$ are indeed Markov chains with transition probabilities $\pi$ as in \eqref{eq:transitions decreasing chain}.

For all $m \in \mathbb{N}$, the black instruction $B_m$ should be thought of as the instruction to take a \emph{small step} from $m$ to $B_m$, while for all $m \in \intervalleentier{5\ell}{7\ell}$ the golden instruction $G_m$ should be thought of as the instruction to take a \emph{large step} from $m \geq 5\ell$ to $G_m \in \intervalleentier{\ell}{4\ell}$. The reason why we consider golden instructions pointing towards $\intervalleentier{\ell}{4\ell}$ rather than $\intervalleentier{0}{4\ell}$ is that we will run a multiscale argument that requires the chains not to jump across several good scales at once. Furthermore, the reason for the terminology for colors is that when taking a golden instruction, the chains will merge with positive probability and then stick together,\footnote{We emphasize that this is not the only mechanism by which the two chains can merge. In fact, it is totally possible for them to merge through two black instructions that happen to land at the same value, causing the chains to stick together from that point onward. However, since our aim is to lower bound the probability that the two chains will merge fast enough, it suffices for us to focus on the case involving the golden instructions.} ending up in a win, as precisely described in the next result.

We will write $\xmer$ for the merging value of $\DMC$ and $\DMC'$, \emph{i.e.}\ the largest common value taken by the chains $\DMC$ and $\DMC'$. Let also $\stmer$ (resp.\ $\stmerpr$) be the time when $\DMC$ (resp.\ $\DMC'$) visits $\xmer$. Note that by definition of the black-golden coupling, 
    \begin{equation*}
        (\DMC_{\stmer +s})_{s\geq 0} = (\DMC'_{\stmerpr+s})_{s \geq 0}.
    \end{equation*}
    Finally, for $k \geq A \geq 1$, we denote by $s(A,k)$ the number of good scales $\ell \in \intervalleentier{A}{k}$ such that $7 \ell<k$.

\begin{prop}[Fast merging of $\DMC$ and $\DMC'$] \label{prop:coupling_merging}
    Let $\DMC$ and $\DMC'$ be defined as the decreasing chains started from $k$ and $k'$, respectively, under the black-golden coupling $\mathbb{P}_{k,k'}$ of~\cref{defn:black-golden coupling}. 
    Then, the following assertions hold:
    \begin{enumerate}
        \item\label{item:fast-merg-coupling}  There exists a constant $c \in (0,1)$ such that for all $k'\geq k\geq 1$ and all $A\geq 1$,
    \begin{equation*} 
        \Ppp{k,k'}{\xmer < A} \leq (1-c)^{s(A,k)}.
    \end{equation*}
        \item\label{item:inf-lim-t-mer-coupling} Furthermore, the total times $\mathtt{T} - \stmer$ and $\mathtt{T}' - \stmerpr$, representing the durations between merging and dying, tend to infinity in probability as $k, k' \to \infty$.
    \end{enumerate}
\end{prop}

The proof of \cref{prop:coupling_merging} will occupy most of the remainder of \cref{sect:loacal-lim}. Once this is proved, we easily obtain a similar result for the whole Markov chain $Z^k$. We note that with the notation of~\eqref{eq:defn-stopping-times}, the time where $M^*$ (resp.\ $(M^*)'$) first hits $\xmer$ is given by $\stht{\stmer}$ (resp.\ $\sthtpr{\stmerpr}$). We also recall that $\eta=\stht{\mathtt{T}}$ (resp.\ $\eta'=\sthtpr{\mathtt{T}'}$) is the hitting time of $0$ by $M^*$ (resp.\ $(M^*)'$).

\begin{cor}\label{cor:full_coupling}
    For any $k' \geq k \geq 1$, there exists a coupling $\widetilde{\mathbb{P}}_{k,k'}$ of the full Markov chains $Z^{k}$ and $Z^{k'}$ started from $k$ and $k'$ respectively, under which:
    \begin{itemize}
    \item The pair $(M,M')$ follows the black-golden coupling $\mathbb{P}_{k,k'}$ of~\cref{defn:black-golden coupling}. 
    \item For any $h \geq 1$, we have $Z^{k}_{\stht{\stmer}+h} = {Z}^{k'}_{\sthtpr{\stmerpr}+h}$.
    \end{itemize}
    In particular:
    \begin{enumerate}
    \item\label{item:full_coupling_quantitative}
    For all $A \geq 1$, we have
        \[\Pppt{k,k'}{\xmer < A} \leq (1-c)^{s(A,k)}. \] 
    \item\label{item:full_coupling_convergence}
    For all $A \geq 1$, we have    
        \[ \Pppt{k,k'}{\xmer \geq A} \to 1 \quad \text{as } k,k'\to \infty. \]
    \item\label{item:full_coupling_time}
    The durations $\eta-\stht{\stmer}$ and $\eta'-\sthtpr{\stmerpr}$ go to infinity in probability as $k,k' \to \infty$.
    \end{enumerate}
    \end{cor}

\begin{proof}[Proof of Corollary~\ref{cor:full_coupling} given Proposition~\ref{prop:coupling_merging}]
    It is explained in Section~\ref{sect:decreasing-chain} how to reconstruct $Z^{k}$ from $M$ using independent decorations $\eta, S, D, W$, and similarly for $Z^{k'}$. Since $M$ and $M'$ merge at respective times $\stmer$ and $\stmerpr$, we can choose these decorations such that for all $t\geq 0$ and all $h\geq 1$,
    \[\stht{\stmer+t+1}-\stht{\stmer+t}=\sthtpr{\stmerpr+t+1}-\sthtpr{\stmerpr+t}\quad\text{and}\quad (S,D,W)_{\stht{\stmer}+h}=(S',D',W')_{\sthtpr{\stmerpr}+h}.\]
    This implies that $Z^{k}$ after time $\stht{\stmer}+1$ and $Z^{k'}$ after time $\sthtpr{\stmerpr}+1$ coincide, which shows the existence of $\widetilde{\mathbb{P}}_{k,k'}$. \cref{item:full_coupling_quantitative} then follows immediately from \cref{item:fast-merg-coupling} of Proposition~\ref{prop:coupling_merging}, and implies \cref{item:full_coupling_convergence} because our set of good scales is infinite. Finally, \cref{item:full_coupling_time} follows from \cref{item:inf-lim-t-mer-coupling} of Proposition~\ref{prop:coupling_merging} since $\eta-\stht{\stmer} \geq \mathtt{T}-\stmer$ and $\eta'-\sthtpr{\stmerpr} \geq \mathtt{T}'-\stmerpr$.
\end{proof}

\begin{proof}[Proof of Theorem~\ref{thm:localconv} given Corollary~\ref{cor:full_coupling}]
    We need to prove that there exists a process $(\overset{\leftarrow}{Z}_t)_{ t \ge 0}$ such that, for all $r\geq 1$,
    \begin{equation} \label{eq:TV_distance_DMC}
    d_{\mathrm{TV}}\left((Z^{k}_{\eta-t})_{t\leq r}, (\overset{\leftarrow}{Z}_{{t}})_{t\leq r}\right) \rightarrow 0 \quad \text{as } k\to \infty,
    \end{equation} 
    where we take the convention that $Z^{k}_{\eta-t} = (\ominus,{\lfa},k,0)$ when $t>\eta$.
    We prove \eqref{eq:TV_distance_DMC} by Cauchy characterization. Consider the two decreasing Markov chains $Z^{k}$ and $Z^{k'}$ started at $k$ and $k'$ respectively and having the coupling $\widetilde{\mathbb{P}}_{k,k'}$ of Corollary~\ref{cor:full_coupling}. 
    By standard coupling techniques (see e.g.\ \cite[Proposition 4.7]{levin2017markov}),
    we can write
    \[
    d_{\mathrm{TV}}\Big((Z^{k}_{\eta-{t}})_{t\leq r}, ({Z}^{k'}_{\eta'-{t}})_{t\leq r}\Big) 
    \le
    \Pppt{k,k'}{Z^{k}_{\eta-{t}} \neq {Z}^{k'}_{\eta'-{t}} \text{ for some $t\leq r$}} 
    \leq
    \Pppt{k,k'}{\eta-\stht{\stmer} \leq r}
    \]
    since, under the coupling, the two chains $Z^{k}$ and $Z^{k'}$ stick together after $\stht{\stmer}$ and $\sthtpr{\stmerpr}$.
   Then, \cref{item:full_coupling_time} of \cref{cor:full_coupling} guarantees that the probability on the right-hand side of the last display goes to $0$ as $k,k'\to\infty$, which proves the desired Cauchy characterization.
    
    We conclude that for any fixed $r$, the random vector $\left( Z^{k}_{\eta-t} \right)_{t \leq r}$ started from $k$ has a limit as $k\to\infty$ that we write $\big( \overset{\leftarrow}{Z}_t \big)_{t \leq r}$.
    Finally, an application of Kolmogorov's extension theorem shows that the limit defines a valid stochastic process $\overset{\leftarrow}{Z}$.
\end{proof}

All we have left to do to prove \cref{thm:localconv} is therefore to prove \cref{prop:coupling_merging}, which will be completed in \cref{sect:proof-coupling}. As it forms the most technical part of the paper, the reader may on a first read directly jump to \cref{sect:better-bounds}.

\medskip

The rest of this section is organized as follows: in \cref{sect:descript} we will give an informal description of the strategy for the proof of \cref{prop:coupling_merging}. Then, \cref{sect:good-scales} is devoted to the description of the good scales, \cref{sect:4tech-lem} to the proof of the main technical estimates, and finally \cref{sect:proof-coupling} to the conclusion of the proof of \cref{prop:coupling_merging}.

\subsubsection{Description of the strategy for the coupling argument}\label{sect:descript}

We now describe the strategy for the proof of~Item~\ref{item:fast-merg-coupling}~of~\cref{prop:coupling_merging}; the second item will follow as a simple consequence. 
We will show that there exists a constant $c\in(0,1)$ such that:
\begin{enumerate}
    \item From any starting point $k$, for each good scale $\ell$ such that $7\ell \leq k$, the chain $\DMC$ visits $\intervalleentier{6\ell}{7\ell}$ with probability at least $c$ (this will be a consequence of \cref{lem:prop-good-scales-item1} in \cref{lem:prop-good-scales}).

    \item If $\DMC$ is ever in $\intervalleentier{6\ell}{7\ell}$, it will try to jump below $4\ell$ from somewhere in $\intervalleentier{5\ell}{7\ell}$ with probability at least $c$, meaning it will use a golden instruction with probability at least $c$ (this is a consequence of \cref{lem:prop-good-scales-item2} in \cref{lem:prop-good-scales}). Moreover, conditional on this, the same facts hold for $\DMC'$ with probability at least $c/2$ (this is a consequence of \cref{lem:cond_M_item_goodscale}).

    \item For all good scales $\ell$, we can couple the $(G_m)_{5\ell \leq m \leq 7\ell}$ so that with probability at least $c$ they are all equal (this is \cref{lem:prop-good-scales-item3} in \cref{lem:prop-good-scales}).

    \item\label{it:iterate the argument at every scale} Since, for any good scale, \cref{lem:prop-good-scales-item1} in \cref{lem:prop-good-scales} holds conditionally on all the information above $7\ell$, we can iterate the argument at each scale.
\end{enumerate}

We recall from the discussion around \eqref{eq:transitions decreasing chain2} that $\DMC$ can be compared to a decreasing random walk with jump distribution $(q(i))_{i\ge 1}$, with $q(i) = i^{-\gamma+o(1)}$ as $i\to\infty$, and $\gamma \in ( 3/2,2 )$ {as in \eqref{eq:gamma-alpha-rel}. Informally speaking, the \emph{good scales} that we need in Steps $1$, $2$ and $3$ are scales where the behavior of this walk is reasonably close to that of a stable subordinator.\footnote{The convergence of such a walk to a $(\gamma-1)$-stable subordinator would be immediate given \cref{thm:good_regularity}, but unfortunately proving \cref{thm:good_regularity} is precisely our long-term goal.} 

If all the above is true (and using the fact that $M$ and $M'$ are essentially independent up until merging),
there is a uniformly positive probability at every good scale that both $\DMC$ and $\DMC'$ follow two golden instructions landing at the same height. 
Since we can iterate this attempt to land at the same height by successive conditionings above each good scale (thanks to Step~\ref{it:iterate the argument at every scale} above), with probability tending to one as the starting points approach infinity, the two chains will eventually follow two golden instructions landing at the same height. From this, we can conclude by construction of the black-golden coupling that the future trajectories of $\DMC$ and $\DMC'$ will stick together beyond this point, completing the proof of \cref{item:fast-merg-coupling}~of~\cref{prop:coupling_merging}.

\subsubsection{Definition and existence of good scales and additional preliminary estimates}\label{sect:good-scales}

In this whole section, the value of the constants $c$ and $C$ may change from line to line. 
Recall that, in \cref{subsec:rough regularity}, we proved the rough regularity estimates for the sequences $q$ and $Q$ stated in \cref{prop:rough_regularity}, which have the following immediate consequences.
\begin{cor}[Bounds on $q$ and $Q$]
\label{cor:bound ratio q and Q}
Fix  $r,R>0$. Then there exist constants $c,C>0$ such that:
\begin{enumerate}
\item For all $k\geq 1$, 
\begin{align}\label{eqn:regularity_q-part1}
            c\leq \frac{kq(k)}{Q(k)} \leq C.
\end{align}
 \item For all $1 \leq k' \leq k$,
\begin{equation}\label{eqn:regularity_q-part2}
\frac{q(k')}{q(k)}\geq c.
\end{equation}
\item For all $k,k'\geq 1$ with $r \leq \frac{k'}{k} \leq R$,
\begin{align}\label{eqn:regularity_q-part3}
c\leq \frac{q(k')}{q(k)} \leq C \qquad \text{and} \qquad c\leq \frac{Q(k')}{Q(k)} \leq C.
\end{align}
\end{enumerate}
\end{cor}

\begin{proof}
The bounds in \eqref{eqn:regularity_q-part1} are exactly the ones in \cref {eq:bound-kq/Q} in \cref{lem:q_and_sigma}. 
The upper bound in \eqref{eqn:regularity_q} of  \cref{prop:rough_regularity} gives that $\frac{q(k)}{q(k')}\leq C$ for all $1 \leq k' \leq k$, which immediately implies \eqref{eqn:regularity_q-part2}.
Finally, \cref{prop:rough_regularity} ensures that as long as the ratio $\frac{k'}{k}$ is bounded above and below by some fixed constants,
 so are the ratios $\frac{q(k')}{q(k)}$ and $\frac{Q(k')}{Q(k)}$. This gives the
 bounds in \eqref{eqn:regularity_q-part3}. 
\end{proof}

In order to show the existence of large jumps, we will need an additional regularity result holding (at least) for an infinite number of integers  $\ell$.

\begin{lem}[Good scale condition]
\label{lem:existence good scales}
	There exists a constant $C>0$ such that there are infinitely many integers $\ell \geq 1$ satisfying the inequality
     \begin{align}\label{eqn:good_scale2}
        \sum_{i=1}^{\ell-1} i q(i) \leq C \ell Q(\ell).
    \end{align}
\end{lem}

\begin{proof}
	We prove that there exists a constant $c>0$ such that for infinitely many values of $\ell$ we have
    \begin{equation}\label{eqn:good_scale}
		\ell Q(\ell) \geq c \sum_{i=1}^{\ell-1} Q(i).
	\end{equation}
    By the bounds in~\eqref{eqn:regularity_q-part1}, this would imply our claim.
    For all $\ell \geq 1$, we write $\mathbf{Q}(\ell)=\sum_{i=1}^{\ell-1} Q(i)$. 
    By \cref{prop:rough_regularity}, we know that $Q(i)= i^{-\frac{1}{2\alpha} +o(1)}$ where $\alpha\in(1/2,1)$. 
    This entails that 
    \begin{equation}\label{eq:for_contra}
        \mathbf{Q}(\ell)=\ell^{1-\frac{1}{2\alpha} +o(1)}.
    \end{equation}
    Now, assume for a contradiction that for any choice of the constant $c>0$, \eqref{eqn:good_scale} fails for all large enough $\ell$. 
    Then, for all large enough $\ell$, we have $\mathbf{Q}(\ell+1)-\mathbf{Q}(\ell)=Q(\ell) < \frac{c}{\ell} \mathbf{Q}(\ell)$, so that $\mathbf{Q}(\ell+1) < \left( 1+\frac{c}{\ell} \right) \mathbf{Q}(\ell)$. Hence, for all $\ell$ large enough, 
    \[\mathbf{Q}(\ell) \leq \ell^{c+o(1)}.\] 
    We get a contradiction with \eqref{eq:for_contra} if 
    $c<1-\frac{1}{2\alpha}$. Since $1-\frac{1}{2\alpha}>0$ as $\alpha\in(1/2,1)$, we conclude that there exists a constant $c>0$ such that \eqref{eqn:good_scale} holds for infinitely many values of $\ell$.
\end{proof}

We can now introduce the good scales that we mentioned at the beginning of~\cref{sec:coupling_blue_red}.

\begin{defn}[Good scales]\label{defn:good-scales}
Fix $C>0$ as in \cref{lem:existence good scales}. We say that an infinite set $\mathrm{GoodScales}\subset \N$ is a collection of \textbf{good scales} if the following two conditions hold: 
\begin{enumerate}
    \item\label{item:good-scale-condition} Every $\ell\in\mathrm{GoodScales}$ satisfies the condition $\sum_{i=1}^{\ell-1} i q(i) \leq C \ell Q(\ell)$.
    \item For all $\ell,\ell'\in\mathrm{GoodScales}$ with $\ell'>\ell$, we have that $\ell'>7\ell$. 
\end{enumerate}
If $\ell\in\mathrm{GoodScales}$, we say that the interval $\intervalleentier{\ell}{7\ell}$ is a \textbf{good interval}.
\end{defn}

\noindent \cref{lem:existence good scales} guarantees the existence of these collections of good scales and from now on we fix $\mathrm{GoodScales}$ to be one arbitrary, but fixed, such collection. Note that good intervals are disjoint by the second condition in the definition.

\begin{remark} \label{rk:all_scales_good}
Note that, assuming the good regularity estimates in~\cref{thm:good_regularity}, there exists $C>0$ such that any integer $\ell\ge 1$ satisfies the good scale condition in~\cref{item:good-scale-condition} of \cref{defn:good-scales}. Indeed, if $q(i)=i^{-\frac{1}{2\alpha}-1} \svfq(i)$ as in~\cref{thm:good_regularity}, then by Karamata's Theorem (see \cite[Theorem 1.5.11]{bingham1989regular}), we have $\sum_{i=1}^{\ell-1} i q(i)\sim\frac{\ell^2 q(\ell)}{1-\frac{1}{2\alpha}}$ as $\ell\to\infty$, and $\ell^2 q(\ell) \le C \ell Q(\ell)$ by the bounds in \eqref{eqn:regularity_q-part1}.
Therefore, provided we knew \cref{thm:good_regularity}, we could take $\mathrm{GoodScales}$ to be any collection of well-separated integers. This will be important later for the scaling limit result in \cref{sec:sc_limit_Xn}.
\end{remark}

The bounds in \cref{cor:bound ratio q and Q} and \cref{lem:existence good scales} help us control the transition probabilities $\tpr(m ; m')$ from \eqref{eq:transitions decreasing chain} of the decreasing Markov chain $\DMC$, as shown in the next lemma.

\begin{lem}[Bounds for the transition probabilities of $M$]
\label{lem:trans-bound}
    Fix two constants $c_1,c_2>0$ with $c_1<c_2$. Then there exist $c>0$ and $C>0$ such that the following is true:
    \begin{enumerate}
        \item\label{item:trans-bound0} If $c_1<1$, then for all $j\ge 1$,
        \begin{equation*}
            \forall i < j, \quad
            \tpr(j;i)\geq c  \frac{i}{j}  q(i)  
            \qquad\text{ and }\qquad 
            \forall i\leq c_1 j,
            \quad 
            \tpr(j;i)\leq C  \frac{i}{j}      q(i).
        \end{equation*}
        \item\label{item:trans-bound2} For  all $\ell\geq 1$, and all $j\geq 1$ and $x,y\geq 1$ such that $c_1 \ell \le x \le y \le c_2\ell < j$,
        \begin{equation*}
    		\Ppsq{\DMC_{t+1}\in \intervalleentier{x}{y}}{\DMC_t=j} \geq c \frac{y-x}{j}  Q(\ell).
        \end{equation*}
        \item\label{item:trans-bound1} If $c_1<1$, then for all $\ell\in \mathrm{GoodScales}$, and all $j\geq 1$ and $y\geq 1$ such that $y\le (c_1 j) \wedge (c_2 \ell)$,
        \begin{equation*}
            \Ppsq{\DMC_{t+1}\in \intervalleentier{1}{y}}{\DMC_t=j}
        \le
        C \frac{\ell}{j}  Q(\ell).
        \end{equation*}
        \item\label{item:trans-bound3} If $c_1>1$, then for all $\ell\in \mathrm{GoodScales}$ and $j\in \intervalleentier{c_1\ell}{c_2\ell}$, writing $\Delta \DMC_1 = \DMC_0 - \DMC_1$,
        \begin{equation*}
        \mathbb{E}_{j}\big[\Delta \DMC_1  \mathds{1}_{\Delta \DMC_1 \le \ell}\big] 
        \leq C \ell Q(\ell).
        \end{equation*}
    \end{enumerate}
\end{lem}

\begin{proof}
The bounds in \cref{item:trans-bound0} follow from \cref{cor:bound ratio q and Q}. Indeed, by \eqref{eq:transitions decreasing chain}, we have
\begin{equation} \label{eq:pi(j,i)}
\forall i<j, \quad \tpr(j;i)=\frac{1}{1+\frac{(1-p)}{p}(Q(j) -\frac{q(j)}{2})} \cdot \frac{i}{j}\cdot  \frac{q(i) q(j-i)}{q(j)}.
\end{equation}
Now for all $j\geq 1$, the ratio $\frac{1}{1+\frac{(1-p)}{p}(Q(j)-\frac{q(j)}{2})}$ is bounded above and below by constants. 
Besides, on the one hand, by \eqref{eqn:regularity_q-part2}, 
\[
\forall i\le j, \qquad \frac{q(j-i)}{q(j)}\geq c.
\]
On the other hand, by the first item of \eqref{eqn:regularity_q-part3}, if $c_1<1$,
\[ 
    \forall i \le c_1 j, \qquad \frac{q(j-i)}{q(j)} \leq C.
\]
since
$1-c_1\leq1-\frac{i}{j}=\frac{j-i}{j}\leq 1$.
The bounds in  \cref{item:trans-bound0} immediately follow.

\medskip

We now move on to \cref{item:trans-bound2}. We recall from \cref{sect:not-proba} the notation for sums between real numbers, which will be extensively used in the rest of the proof. Let $\ell\geq 1$, and $j\geq 1$ and $x,y\geq 1$, such that $c_1 \ell \le x \le y \le c_2\ell < j$.
Allowing the value of the constant $c$ to change from line to line, we get by \cref{item:trans-bound0} that
	\begin{align}\label{eq:idwveivfuo}
		\Ppsq{\DMC_{t+1}\in \intervalleentier{x}{y}}{\DMC_t=j} = \sum_{i=x}^{y} \tpr(j;i) &\geq c \sum_{i=x}^{y} \frac{i}{j} \cdot q(i) \notag \\
		&\geq c \cdot \frac{\ell}{j} \sum_{i=x}^{y}q(i) \geq c \frac{\ell}{j} (y-x) q(\ell),
	\end{align}
where in the last inequality we used the first bound in \eqref{eqn:regularity_q-part3}, which holds since we are summing over $i \in \intervalleentier{x}{y} \subseteq \intervalleentier{c_1 \ell}{c_2 \ell}$. By the bounds in \eqref{eqn:regularity_q-part1}, this establishes \cref{item:trans-bound2}.

\medskip

We now assume that $c_1<1$ and  prove \cref{item:trans-bound1}. Fix $\ell\in \mathrm{GoodScales}$ and consider $j\geq 1$ and $y\geq 1$ such that $y \le (c_1 j) \wedge (c_2 \ell)$. Using the fact that we sum over $i\leq y\leq c_1 j$, together with the second bound in \cref{item:trans-bound0}, we get 
	\begin{align*}
	\Ppsq{\DMC_{t+1}\in \intervalleentier{1}{y}}{\DMC_t=j} = \sum_{i=1}^{y} \tpr(j;i) \leq C \sum_{i=1}^{y} \frac{i}{j}  q(i) 
     \leq \frac{C}{j} \bigg(\sum_{i=1}^{\ell-1} i q(i) + \sum_{i=\ell}^{y} i q(i) \bigg).
\end{align*}
Since $y\le c_2 \ell$, the latter sum can be handled again using first the bounds in \eqref{eqn:regularity_q-part3} and then in \eqref{eqn:regularity_q-part1}, yielding
\[
\sum_{i=\ell}^{y} i q(i) 
\le C \ell^2 q(\ell)
\le C \ell Q(\ell).
\]
The first one is also bounded by $C \ell Q(\ell)$ by the good scale condition of \cref{defn:good-scales}, which proves \cref{item:trans-bound1}.

It only remains to prove \cref{item:trans-bound3} for which we assume that $c_1>1$.  Taking $\ell\in \mathrm{GoodScales}$ and $j\in \intervalleentier{c_1\ell}{c_2\ell}$, we have
\[
\mathbb{E}_{j}\big[\Delta \DMC_1  \mathds{1}_{\Delta \DMC_1 \le \ell}\big]
=
\sum_{i=j-\ell}^{j-1} (j-i)\tpr(j;i).
\]
Any $i$ in this sum satisfies $(c_1-1)\ell \le i < j \leq c_2 \ell$ (with $c_1-1>0$).  Therefore, we have
\[
\tpr(j;i) \stackrel{\eqref{eq:pi(j,i)}}{\le} C\frac{i}{j}\cdot  \frac{q(i) q(j-i)}{q(j)} \stackrel{\eqref{eqn:regularity_q-part3}}{\le} C q(j-i).
\]
Recalling that $\ell$ is a good scale, we finally get by the good scale condition (\cref{defn:good-scales}):
\[
\mathbb{E}_{j}\big[\Delta \DMC_1  \mathds{1}_{\Delta \DMC_1 \le \ell}\big]
\le
C \sum_{i=j-\ell}^{j-1} (j-i)q(j-i)
=
C \sum_{i=1}^{\ell} i q(i)
\le
C \ell Q(\ell).
\] 
This concludes the proof of \cref{item:trans-bound3} and therefore establishes \cref{lem:trans-bound}.
\end{proof}

\subsubsection{Statements and proofs of the three main ingredients for the coupling}\label{sect:4tech-lem}

In this section, we consider two chains $\DMC$ and $\DMC'$ under the black-golden coupling (see \cref{defn:black-golden coupling}). Recall the notation $\bbP_{k,k'}$ for the probability measure of this coupling under which $\DMC$ starts at $k$ and $M'$ at $k'$. 
Whenever we work with events that only depend on $M$, we use $\bbP_{k}$ for shorter expressions.

We write $\stt_i$ for the first hitting time of $\intervalleentier{0}{i}$ by the chain $\DMC$, that is
\begin{equation}\label{eq:stt-defn}
    \stt_i\coloneqq\inf\left\{t\geq 1:\DMC_t\in\intervalleentier{0}{i}\right\}.
\end{equation}
We define $\stt_i'$ similarly for $\DMC'$.
We will often write $\DMC(\stt_i)$ instead of $\DMC_{\stt_i}$ to avoid notational clutter.
Note that for any $h \in \intervalleentier{0}{k}$ we have  
\begin{align*}
\DMC(H_h) &= \sup \left( \{\DMC_t, \ t\geq 0\} \cap \intervalleff{0}{h} \right), \\
\DMC(H_h-1) &= \inf \left( \{\DMC_t, \ t\geq 0\} \cap \intervallefo{h+1}{\infty} \right).
\end{align*}

Define for all $\ell\in\mathrm{GoodScales}$ the $\sigma$-algebra (\emph{cf.}\ \cref{rem:point-process})
\begin{equation} \label{eq:def_Ff}
\Gf_{\ell} \coloneqq  \sigma\left(G_m : m\in \bigcup_{\ell'>\ell}\intervalleentier{5\ell'}{7\ell'}\right)
\vee 
\sigma\Bigg(\intervallefo{7\ell+1}{+\infty} \cap \{\DMC_t, t \geq 0\}
,
\intervallefo{7\ell+1}{+\infty} \cap \{\DMC_t', t \geq 0\}
\Bigg).
\end{equation}
This $\sigma$-algebra records the values of the golden instructions for all golden intervals strictly above $7\ell$, as well as the trajectories of $\DMC$ and $\DMC'$ from their starting point up to the last time that they take a value strictly larger than $7\ell$, \emph{i.e.}\ $(\DMC_t)_{0\leq t \leq H_{7\ell}-1}$ and $(\DMC'_t)_{0\leq t \leq H'_{7\ell}-1}$.

\begin{lem}[Spatial Markov property at level $7\ell$]
\label{lem:spatial Markov property}
Fix $\ell\in\mathrm{GoodScales}$ and $k,k'\geq 7\ell+1$. Let $\DMC$ and $\DMC'$ be defined under the black-golden coupling $\mathbb{P}_{k,k'}$ of~\cref{defn:black-golden coupling}. Then the following assertions are true:
\begin{enumerate}
\item\label{lem:spatial markov property item1} 
For $m\in\intervalleentier{7\ell +1}{k}$ and  $m'\in\intervalleentier{7\ell +1}{k'}$, the $\Gf_\ell$-conditional distribution of $(\DMC(H_{7\ell}), \DMC'(H'_{7\ell}))$ on the event $\{\DMC(H_{7\ell}-1)=m, \DMC'(H'_{7\ell}-1)=m'\}$ is given by
\begin{align}
\Law\left((\DMC(H_{7\ell}), \DMC'(H'_{7\ell})) \ | \  \Gf_\ell \right) = \Law((B_m,B_{m'}) \ | \ B_m \leq 7\ell, B'_m \leq 7\ell).
\end{align}
\item\label{lem:spatial markov property item2} Conditionally on $\sigma(\Gf_{\ell}, (\DMC(H_{7\ell}),\DMC'(H'_{7\ell}))$, the chains $(\DMC(H_{7\ell}+t))_{t\geq 0}$ and $(\DMC'(H'_{7\ell}+t))_{t\geq 0}$ have the distribution of $\DMC$ and $\DMC'$ under the black-golden coupling $\bbP_{\DMC(H_{7\ell}),\DMC'(H'_{7\ell})}$.
\end{enumerate}
\end{lem}

\begin{proof}
We first prove \cref{lem:spatial markov property item1}.
Fix any $m,m'\geq 7\ell +1$ and let us reason on the event $\{\DMC(H_{7\ell}-1)=m, \DMC'(H'_{7\ell}-1)=m'\}$.
We know by definition of $H_{7\ell}$ and $H'_{7\ell}$ that 
$\DMC(H_{7\ell})\leq 7\ell$ and $\DMC'(H'_{7\ell})\leq 7\ell$.
From the definition of the black-golden coupling, it is clear that we have 
$\DMC(H_{7\ell})=B_m$, because a golden instruction $G_m$ 
(if taken at all)
would land in some interval $\intervalleentier{\ell'}{4\ell'}$ with $\ell'>7\ell$, which would contradict the fact that $\DMC(H_{7\ell})\leq 7\ell$. 
We similarly have $\DMC'(H'_{7\ell})=B'_m$.
Now we have 
\begin{align*}
\Law((\DMC(H_{7\ell}),\DMC'(H'_{7\ell})) \ | \ \Gf_\ell) = \Law((B_m,B_{m'}) \ | \ \Gf_\ell) = \Law((B_m,B_{m'}) \ | \  B_m \leq 7\ell, B'_m \leq 7\ell),
\end{align*}
which proves \cref{lem:spatial markov property item1}.

We now turn to \cref{lem:spatial markov property item2}. 
Just note that by definition of the black-golden coupling, the future trajectories $(\DMC(H_{7\ell}+t))_{t\geq 0}$ and $(\DMC'(H'_{7\ell}+t))_{t\geq 0}$ are entirely determined by
the positions $(\DMC(H_{7\ell}),\DMC'(H'_{7\ell}))$
and
the values of the black and golden instructions below $7\ell$, namely $(B_m)_{1 \leq m \leq 7\ell}$ and $(G_m)_{5\ell' \leq m\leq 7\ell', \ell' \leq  \ell }$. 
Note also that these random variables are independent of 
$\Gf_\ell, \DMC(H_{7\ell}),\DMC'(H'_{7\ell})$, because those are determined from the value of the black and golden instructions \emph{above} $7\ell +1$.
This ensures that \cref{lem:spatial markov property item2} holds.
\end{proof}
}

The next lemma presents the three main ingredients for the proof of~\cref{prop:coupling_merging}, following the strategy outlined in \cref{sect:descript}. Then, in \cref{lem:cond_M_item_goodscale}, we will extend the second item to incorporate the coupling with $\DMC'$.

\begin{lem}\label{lem:prop-good-scales}
There exists a constant $c\in (0,1)$ such that the following three properties hold:
\begin{enumerate}
\item \label{lem:prop-good-scales-item1} 
For any $\ell\in\mathrm{GoodScales}$ and $m>7\ell$,
\[
\Ppsq{B_m\in \intervalleentier{6\ell}{7\ell} 
}{B_m\leq 7\ell} \geq c.
\]

\item \label{lem:prop-good-scales-item2}  For any  $\ell\in\mathrm{GoodScales}$ and $k\in \intervalleentier{6\ell}{7\ell}$,
\begin{align*}
	\Ppp{k}{\DMC(\stt_{4\ell}-1) \in \intervalleentier{5\ell}{7\ell}\,,\, \DMC(\stt_{4\ell}) \in \intervalleentier{\ell}{4\ell} }\geq c.
\end{align*}

\item \label{lem:prop-good-scales-item3} For any $\ell\in \mathrm{GoodScales}$, we can couple the golden instructions $(G_m)_{5\ell \leq m \leq 7\ell}$, with law given by \eqref{eq:red-law}, in such a way that 
	\begin{align*}
		\Pp{\forall a,a' \in \intervalleentier{5\ell}{7\ell}, \ G_a = G_{a'} }\geq c.
	\end{align*}
\end{enumerate}
\end{lem}

\noindent See \cref{fig-good-scales2} for a schematic representation of the quantities involved in \cref{lem:prop-good-scales} and some further explanations on the role of the three items. 

\begin{figure}[ht]
	\begin{center}
		\includegraphics[width=1\textwidth]{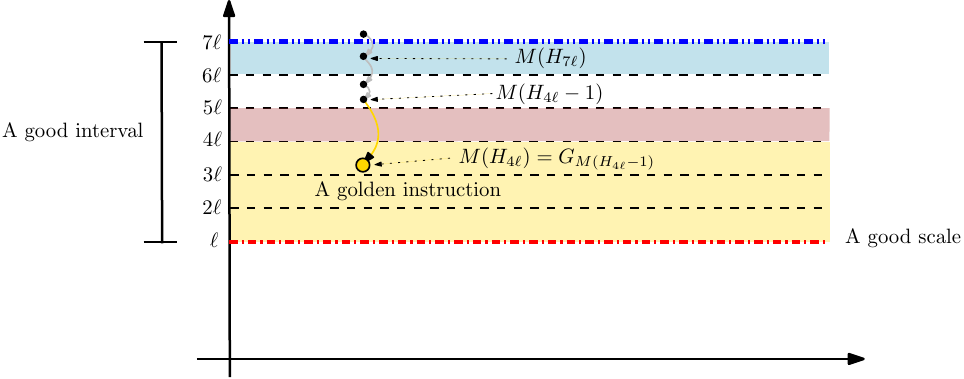}  
		\caption{A schematic representation for the statement of \cref{lem:prop-good-scales}. A good scale $\ell$ is shown as a dashed-dotted red line and the corresponding value of $7\ell$ is shown as a dashed-dotted blue line.  The good interval  $\intervalleentier{\ell}{7\ell}$ is divided into six sub-intervals by the dashed black lines, with the sub-intervals $\intervalleentier{6\ell}{7\ell}$, $\intervalleentier{4\ell}{5\ell}$ and $\intervalleentier{\ell}{4\ell}$ highlighted in blue, brown and gold, respectively. \cref{lem:prop-good-scales-item1,lem:prop-good-scales-item2} ensure that there is a uniformly positive probability among all good scales $\ell$ that $\DMC$ enters the blue sub-interval and then (possibly after a few steps) jumps across the brown sub-interval; this trajectory is schematically represented by the black dots. The jump across the brown good sub-interval corresponds to a golden instruction landing at the golden dot at height $G_{\DMC(\stt_{4\ell}-1)}$ inside the golden interval $\intervalleentier{\ell}{4\ell}$. Since \cref{lem:prop-good-scales-item3} allows us to couple all the golden instructions in such a way that they land at \emph{the same} place with positive probability, this (with the additional help of \cref{lem:cond_M_item_goodscale}) will allow us to argue that there is a uniformly positive probability over all good scales $\ell$ that $\DMC$ and $\DMC'$ meet. Since \cref{lem:prop-good-scales-item1} guarantees that this actually holds conditionally on the scales above, we will be able to conclude that $\DMC$ and $\DMC'$ meet with probability tending to one as their starting points tend to infinity.}
        \label{fig-good-scales2}
	\end{center}
	\vspace{-3ex}
\end{figure}

\medskip

We prove the three items of \cref{lem:prop-good-scales} sequentially.

\begin{proof}[Proof of \cref{lem:prop-good-scales-item1} of \cref{lem:prop-good-scales}] 
Let  $\ell$ be a good scale and $m>7\ell$.
We write
\begin{equation} \label{eq:cond_M<5ell_ratio}
\Ppsq{B_m \in \intervalleentier{6\ell}{7\ell}}{B_m\le 7\ell}
=
\frac{\Pp{B_m \in \intervalleentier{6\ell}{7\ell}}}{\Pp{B_m\le 7\ell}}
=
\frac{\Ppsq{M_{t+1} \in \intervalleentier{6\ell}{7\ell}}{M_t=m}}{\Ppsq{M_{t+1} \in \intervalleentier{1}{7\ell}}{M_t=m}},
\end{equation}
for any $t\geq 0$. 
Now by Item~\ref{item:trans-bound2}~in~\cref{lem:trans-bound} (with $c_1=6$, $c_2=7$, $x=6\ell$ and $y=7\ell$) there is a constant $c>0$ (independent of $\ell$) such that for all $m>7\ell$,
\[
\Ppsq{\DMC_{t+1}\in \intervalleentier{6\ell}{7\ell}}{\DMC_t=m}  \geq c \frac{\ell}{m} Q(\ell).
\]
Likewise, by Item~\ref{item:trans-bound1}~in~\cref{lem:trans-bound} (with $c_1=6/7$, $c_2=6$, and $y=6\ell-1$), there is a constant $C>0$ (independent of $\ell$) such that for all $m>7\ell$,
\[
\Ppsq{\DMC_{t+1}\in \intervalleentier{1}{6\ell-1}}{\DMC_t=m}
\le
C \frac{\ell}{m} Q(\ell).
\]
Combining the previous two displays, we find that for all $m>7\ell$,
\[
\Ppsq{B_m \in \intervalleentier{6\ell}{7\ell}}{B_m\le 7\ell} 
\ge \frac{c}{C+c}.
\]
This proves \cref{lem:prop-good-scales-item1} of \cref{lem:prop-good-scales}.
\end{proof}

\begin{proof}[Proof of \cref{lem:prop-good-scales-item2} of \cref{lem:prop-good-scales}]

We divide the proof into five steps.

 \bigskip

\noindent\emph{\underline{Step 1}: Introducing a good event $E_\ell$.}
For $a>0$ and $\ell\ge 1$, let $b=b(\ell) \coloneqq \left\lceil \frac{a}{Q(\ell)} \right\rceil,$ where $Q(\cdot)$ is as in \eqref{eq:def_q_Q}. Define the following event $E_\ell$: 
\begin{itemize}
\item[(i)]  $\DMC$ takes at most $b$ jumps smaller than $\ell$, and then a big jump of size between $3\ell$ and $4\ell$.
\item[(ii)]  The sum of the above small jumps is actually smaller than $\ell$.
\end{itemize}
Note that the event $E_{\ell}$ depends implicitly on $b$, and therefore on $a$. If the chain $\DMC$ starts from $k\in \intervalleentier{6\ell}{7\ell}$ and satisfies (i) and (ii) above, it first travels less than distance $\ell$ in less than $b$ units of time (so it stays in $\intervalleentier{5\ell}{7\ell}$), and then lands in $\intervalleentier{\ell}{4\ell}$ by a big jump, jumping over $\intervalleentier{4\ell}{5\ell}$. Therefore, 
\begin{equation} \label{eq:E_ell_DMC}
\forall k\in \intervalleentier{6\ell}{7\ell}, \quad
\Ppp{k}{\DMC(H_{4\ell}-1) \in \intervalleentier{5\ell}{7\ell}, \DMC(H_{4\ell}) \in \intervalleentier{\ell}{4\ell}}
\ge 
\Ppp{k}{E_\ell}.
\end{equation}

 We will prove in the next steps that we can find $a>0$ such that for some constant $c>0$ and all $\ell\in\mathrm{GoodScales}$,
\begin{equation} \label{eq:E_ell}
\forall k\in\intervalleentier{6\ell}{7\ell}, \quad
\Ppp{k}{E_\ell} \ge c.
\end{equation}
Assuming \eqref{eq:E_ell}, we note that \eqref{eq:E_ell_DMC} yields the claim in~\cref{lem:prop-good-scales-item2} of \cref{lem:prop-good-scales}. It remains to prove \eqref{eq:E_ell}.

\bigskip

\noindent\emph{\underline{Step 2}: Splitting the event $E_{\ell}$.}
We treat $a>0$ as a constant that will be determined at the very end of the proof (the constants $c, C$ appearing in the estimates will not depend on $a$). For convenience, set for the decreasing Markov chain $\DMC$,
\[\Delta \DMC_t \coloneqq \DMC_{t-1}-\DMC_{t},\qquad t\ge 1.\] 
Note that $\Delta \DMC_t\geq 1$ for all $t\ge 1$.
Let $\ell\in\mathrm{GoodScales}$ and $k\in\intervalleentier{6\ell}{7\ell}$. We first partition $E_\ell$ as
\begin{align}
\Ppp{k}{E_\ell}
&=
\Ppp{k}{\exists b'\leq b: \max\{\Delta \DMC_1,\dots, \Delta \DMC_{b'}\}\leq \ell, \Delta \DMC_{b'+1}\in \intervalleentier{3\ell}{4\ell}, k-\DMC_{b'}\leq \ell} \notag \\
&=
\sum_{b'=0}^b \Ppp{k}{\max\{\Delta \DMC_1,\dots, \Delta \DMC_{b'}\}\leq \ell, \Delta \DMC_{b'+1}\in \intervalleentier{3\ell}{4\ell}, k-\DMC_{b'}\leq\ell}.  \label{eq:E_ell_partition}
\end{align}
Note that even if the event $\{k-\DMC_{b'}\leq\ell\}$ is contained in $\{\max\{\Delta \DMC_1,\dots, \Delta \DMC_{b'}\}\leq \ell\}$, it will be convenient for us to keep both events explicit.
Fix $b'\in\intervalleentier{0}{b}$. Using the Markov property, 
\begin{align*}
&\Ppp{k}{\max\{\Delta \DMC_1,\dots, \Delta \DMC_{b'}\}\leq \ell, \Delta \DMC_{b'+1}\in \intervalleentier{3\ell}{4\ell}, k-\DMC_{b'}\leq\ell} \\
&=
\Ecp{k}{\mathds{1}\{\max\{\Delta \DMC_1,\dots, \Delta \DMC_{b'}\}\leq \ell\} \mathds{1}\{k-\DMC_{b'}\leq\ell\}\Ppp{\DMC_{b'}}{\Delta \DMC_{1}\in \intervalleentier{3\ell}{4\ell}}}.
\end{align*}
Recalling that $k\ge 6\ell$, on the event $\{k-\DMC_{b'}\leq\ell\}$, we have $\DMC_{b'}\ge 5\ell$. Moreover, $\DMC_{b'}\leq k \leq 7\ell$. Therefore, we can use \cref{item:trans-bound2} in \cref{lem:trans-bound} (with $c_1=1$, $c_2=4$, $j\in\intervalleentier{5\ell}{7\ell}$, $x=j-4\ell$ and $y=j-3\ell$)
to  bound, on the event $\{k-\DMC_{b'}\leq\ell\}$, the probability $\Ppp{\DMC_{b'}}{\Delta \DMC_{1}\in \intervalleentier{3\ell}{4\ell}}$ from below:
\[
\Ppp{\DMC_{b'}}{\Delta \DMC_{1} \in \intervalleentier{3\ell}{4\ell}}
=
\Ppp{\DMC_{b'}}{\DMC_1  \in \intervalleentier{\DMC_{b'}-4\ell}{\DMC_{b'}-3\ell}} 
\geq
c \frac{\ell}{\DMC_{b'}} Q(\ell)
\geq
c \, Q(\ell).
\]
This yields
\begin{multline} \label{eq:E_ell_b'}
\Ppp{k}{\max\{\Delta \DMC_1,\dots, \Delta \DMC_{b'}\}\leq \ell, \Delta \DMC_{b'+1}\in \intervalleentier{3\ell}{4\ell}, k-\DMC_{b'}\leq\ell} \\
\geq
c \, Q(\ell) \cdot  \Ppp{k}{\max\{\Delta \DMC_1,\dots, \Delta \DMC_{b'}\}\leq \ell, k-\DMC_{b'}\leq\ell}.
\end{multline}
For future purposes, we note that, assuming  $M_0=k$, we can rewrite the event $\{k-\DMC_{b'}\leq\ell\}$ as
\[
\{k-\DMC_{b'}\leq\ell\}  
= \bigg\{\sum_{i=1}^{b'}\Delta \DMC_i \mathds{1}_{\DMC_{i-1}\geq k-\ell}\leq\ell \bigg\}.
\]
Likewise, we will need to introduce a cutoff in order to later apply the estimates of \cref{lem:trans-bound}.
Therefore, we rewrite the event 
\begin{align*}
&\{\max\{\Delta \DMC_1,\dots, \Delta \DMC_{b'}\}\leq \ell, k-\DMC_{b'}\leq\ell\} \\
&= \{\max\{\Delta \DMC_1 \mathds{1}_{\DMC_{0}\geq k-\ell} ,\dots, \Delta \DMC_{b'} \mathds{1}_{\DMC_{b'-1}\geq k-\ell}\}\leq \ell, \  \sum_{i=1}^{b'}\Delta \DMC_i \mathds{1}_{\DMC_{i-1}\geq k-\ell} \leq\ell\} \\
&= \{\max\{\Delta \DMC_1 \mathds{1}_{\DMC_{0}\geq k-\ell} ,\dots, \Delta \DMC_{b'} \mathds{1}_{\DMC_{b'-1}\geq k-\ell}\}\leq \ell, \  \sum_{i=1}^{b'}\Delta \DMC_i \mathds{1}_{\Delta \DMC_i \le \ell}  \mathds{1}_{\DMC_{i-1}\geq k-\ell} \leq \ell\}.
\end{align*}
Hence, the probability on the right-hand side of \eqref{eq:E_ell_b'} is
\begin{align}\label{eq:E_ell_split_sum}
&\Ppp{k}{\max\{\Delta \DMC_1,\dots, \Delta \DMC_{b'}\}\leq \ell, k-\DMC_{b'}\leq\ell} \notag \\
&=\Ppp{k}{\max\{\Delta \DMC_1 \mathds{1}_{\DMC_{0}\geq k-\ell} ,\dots, \Delta \DMC_{b'} \mathds{1}_{\DMC_{b'-1}\geq k-\ell}\}\leq \ell, \  \sum_{i=1}^{b'}\Delta \DMC_i \mathds{1}_{\Delta \DMC_i \le \ell} \mathds{1}_{\DMC_{i-1}\geq k-\ell}  \leq\ell} \notag  \\
&\geq \Ppp{k}{\max\{\Delta \DMC_1 \mathds{1}_{\DMC_{0}\geq k-\ell} ,\dots, \Delta \DMC_{b'} \mathds{1}_{\DMC_{b'-1}\geq k-\ell}\}\leq \ell}  
 - \Ppp{k}{\sum_{i=1}^{b'}\Delta \DMC_i \mathds{1}_{\Delta \DMC_i \le \ell} \mathds{1}_{\DMC_{i-1}\geq k-\ell}  > \ell}.
\end{align}

\medskip

\noindent\emph{\underline{Step 3}: Lower bound for the first term in  \eqref{eq:E_ell_split_sum}.} We deal with the first term in  \eqref{eq:E_ell_split_sum}, bringing it down by successive applications of the Markov property to
\begin{equation} \label{eq:ma_DMC_1st_term_rec}
\Ppp{k}{\max\{\Delta \DMC_1 \mathds{1}_{\DMC_{0}\geq k-\ell} ,\dots, \Delta \DMC_{b'} \mathds{1}_{\DMC_{b'-1}\geq k-\ell}\}\leq \ell} 
\ge 
(1-C Q(\ell))^{b'}
\end{equation}
for some constant $C>0$.
Indeed, by the Markov property, 
\begin{multline}
\Ppp{k}{\max\{\Delta \DMC_1 \mathds{1}_{\DMC_{0}\geq k-\ell} ,\dots, \Delta \DMC_{b'} \mathds{1}_{\DMC_{b'-1}\geq k-\ell}\}\leq \ell} \\
=
\Ecp{k}{\mathds{1}\left\{\max\{ \Delta \DMC_1 \mathds{1}_{\DMC_{0}\geq k-\ell} ,\dots, \Delta \DMC_{b'-1} \mathds{1}_{\DMC_{b'-2}\geq k-\ell}\} \le \ell\right\}  \Ppsq{\Delta \DMC_{b'}\mathds{1}_{\DMC_{b'-1}\ge k-\ell}\le \ell}{\DMC_{b'-1}}}. \label{eq:rec_ma_DMC_step1}
\end{multline}
Now fix $j \leq 7\ell$. One can write 
\[
\Ppsq{\Delta \DMC_{b'}\mathds{1}_{\DMC_{b'-1}\ge k-\ell}\le \ell}{\DMC_{b'-1}=j} 
=
1-\Ppsq{\Delta \DMC_{b'}\mathds{1}_{\DMC_{b'-1} \ge k-\ell}> \ell}{\DMC_{b'-1}=j}.
\]
Since $k\ge 6\ell$, the last term is smaller than $\Ppsq{\Delta \DMC_{b'}\mathds{1}_{\DMC_{b'-1}\ge 5\ell}> \ell}{\DMC_{b'-1}=j}$. Besides, since $\ell\in\mathrm{GoodScales}$, we can use \cref{item:trans-bound1} in \cref{lem:trans-bound} (with $c_1=6/7$, $c_2=6$ and $y=j-\ell-1 $) to see that, for all $j\in \intervalleentier{5\ell}{7\ell}$, 
\[
\Ppsq{\Delta \DMC_{b'}\mathds{1}_{\DMC_{b'-1}\geq 5\ell}> \ell}{\DMC_{b'-1}=j}
=\Ppsq{\DMC_{1}\in\intervalleentier{1}{j-\ell-1}}{\DMC_{0}=j} 
\leq C Q(\ell). 
\]
The last bound actually holds for all $j\leq 7\ell$, since the probability on the left is zero if $j<5\ell$, so we proved that 
\[
\Ppsq{\Delta \DMC_{b'}\mathds{1}_{\DMC_{b'-1}\ge k-\ell}\le \ell}{\DMC_{b'-1}} 
\geq
1-CQ(\ell).
\]
Plugging this back into \eqref{eq:rec_ma_DMC_step1}, we obtain~\eqref{eq:ma_DMC_1st_term_rec} by induction on $b'$.

\medskip

\noindent\emph{\underline{Step 4}: Upper bound for the second term in  \eqref{eq:E_ell_split_sum}.} We now deal with the second term of \eqref{eq:E_ell_split_sum}. By Markov's inequality,
\[
\Ppp{k}{\sum_{i=1}^{b'}\Delta \DMC_i \mathds{1}_{\Delta \DMC_i \le \ell} \mathds{1}_{\DMC_{i-1}\geq k-\ell} > \ell}
\le 
\frac{1}{\ell} \sum_{i=1}^{b'} \Ecp{k}{\Delta \DMC_i \mathds{1}_{\Delta \DMC_i \le \ell} \mathds{1}_{\DMC_{i-1}\geq k-\ell}}.
\]
By the Markov property, each expectation is 
\[
    \Ecp{k}{\Delta \DMC_i \mathds{1}_{\Delta \DMC_i \le \ell} \mathds{1}_{\DMC_{i-1}\geq k-\ell}} 
=  
\Ecp{k}{\mathds{1}_{\DMC_{i-1}\geq k-\ell}\mathbb{E}_{\DMC_{i-1}}\big[\Delta \DMC_1  \mathds{1}_{\Delta \DMC_1 \le \ell}\big]}. 
\]
But now, on the event $\{\DMC_{i-1}\geq k-\ell\}$, we have $7\ell\ge k \geq \DMC_{i-1} \ge k-\ell \ge 5\ell$, \emph{i.e.}\ $\DMC_{i-1}\in\intervalleentier{5\ell}{7\ell}$. Hence, the inner expectation can be bounded using \cref{item:trans-bound3} in \cref{lem:trans-bound} as
\[
\mathbb{E}_{\DMC_{i-1}}\big[\Delta \DMC_1  \mathds{1}_{\Delta \DMC_1 \le \ell}\big]
\le  
C \ell Q(\ell).
\]

\noindent We deduce 
\[
\Ppp{k}{\sum_{i=1}^{b'}\Delta \DMC_i \mathds{1}_{\Delta \DMC_i \le \ell} \mathds{1}_{\DMC_{i-1}\geq k-\ell} > \ell}
\le
C b' Q(\ell).
\]

\medskip

\noindent\emph{\underline{Step 5}: Completing the proof of \eqref{eq:E_ell}.}
Going back to \eqref{eq:E_ell_split_sum}, the bounds from Steps 3 and 4 give
\[
\Ppp{k}{\max\{\Delta \DMC_1,\dots, \Delta \DMC_{b'}\}\leq \ell, k-\DMC_{b'}\leq\ell}
\geq
(1-CQ(\ell))^{b'} - Cb'Q(\ell).
\]
Then \eqref{eq:E_ell_partition}--\eqref{eq:E_ell_b'} provide, for some constants $c, C>0$,
\[
\Ppp{k}{E_\ell}
\geq
c \cdot Q(\ell) \cdot \bigg(\frac{1-(1-CQ(\ell))^{b+1}}{C Q(\ell)} - C b^2 Q(\ell)\bigg)\geq \frac{c}{C} \cdot\left(1-(1-CQ(\ell))^{b}-C^2(b Q(\ell))^2\right). 
\]
We recall that $b=b(\ell) \coloneqq \left\lceil \frac{a}{Q(\ell)} \right\rceil$.
To conclude, it remains to see that one can choose $a>0$ (small enough) so that the right-hand side is bounded from below by a positive constant, uniformly over $\ell$.  
For the first term, we note that
\[
1-(1-CQ(\ell))^{b}
\geq
1- \mathrm{e}^{-CbQ(\ell)}
\geq
1- \mathrm{e}^{-C a}.
\]
On the other hand, we can bound the second term by $C^2(a+Q(\ell))^2 \leq 2 C^2 a^2$ for $\ell$ large enough (where the threshold only depends on $a$). Overall we have proven the existence of two constants $c>0$ and $C>0$ such that for all $\ell \in \mathrm{GoodScales}$ large enough and $k\in\intervalleentier{6\ell}{7\ell}$,
\[
\Ppp{k}{E_\ell}
\geq
c  (1- \mathrm{e}^{-Ca} - 2C^2 a^2). 
\]
We can now take $a$ small enough so that this is positive, and remove the condition that $\ell$ is large enough by using $\mathbb{P}_k(E_{\ell})>0$ for all $\ell \in \mathrm{GoodScales}$ and $k\in\intervalleentier{6\ell}{7\ell}$. This concludes the proof of \eqref{eq:E_ell}.
\end{proof}

\begin{proof}[Proof of \cref{lem:prop-good-scales-item3} of \cref{lem:prop-good-scales}]
Let $\ell\in \mathrm{GoodScales}$. 
Recalling that $\Pp{B_{m}=m'}=\tpr(m;m')$ and using \cref{item:trans-bound0} in \cref{lem:trans-bound}, we get that for all $m\in \intervalleentier{5\ell}{7\ell}$,
\[
\forall m'\leq 4\ell, \quad 
\frac{c}{m} \cdot m' \cdot  q(m') \leq \Pp{B_{m}=m'}\leq \frac{C}{m} \cdot  m' \cdot  q(m'),
\]
and so
  \[\frac{c}{m} \sum_{m'=\ell}^{4\ell} m' \cdot  q(m') \leq \Pp{B_{m} \in \intervalleentier{\ell}{4\ell}}\leq \frac{C}{m} \sum_{m'=\ell}^{4\ell} m' \cdot  q(m').\]
As a consequence, there exists a constant $C>0$ such that
\begin{align}\label{eq:we-fuiebfo2}
\forall m\in \intervalleentier{5\ell}{7\ell}, \quad \forall m'  \in \intervalleentier{\ell}{4\ell}, \quad 
\frac{1}{C} \cdot \Pp{B_{5\ell}=m'}\leq \Pp{B_{m}=m'} \leq C \cdot \Pp{B_{5\ell}=m'},
\end{align}
and
\begin{align}\label{eq:we-fuiebfo}
\forall m\in \intervalleentier{5\ell}{7\ell}, \quad
\frac{1}{C} \cdot \Pp{B_{5\ell} \in \intervalleentier{\ell}{4\ell}}\leq \Pp{B_{m} \in \intervalleentier{\ell}{4\ell}} \leq C \cdot \Pp{B_{5\ell} \in \intervalleentier{\ell}{4\ell}}.
\end{align}
Combining \eqref{eq:we-fuiebfo2} and  \eqref{eq:we-fuiebfo}, we get that for all $m\in \intervalleentier{5\ell}{7\ell}$,
\begin{align*}
\forall m' \in \intervalleentier{\ell}{4\ell}, \quad
\Pp{G_m= m'} \stackrel{\eqref{eq:red-law}}{=} \frac{\Pp{B_m=m'}}{\Pp{B_m  \in \intervalleentier{\ell}{4\ell}}} \geq  \frac{1}{C^2} \cdot \frac{\Pp{B_{5\ell} = m'}}{ \Pp{B_{5\ell} \in \intervalleentier{\ell}{4\ell}}} = \frac{1}{C^2}\cdot \Pp{G_{5\ell} = m'}.
\end{align*}
This ensures that 
\begin{equation} \label{eq:coupling_sum_inf}
\sum_{m'=\ell}^{4\ell} \left(\inf_{m\in \intervalleentier{5\ell}{7\ell}} \Pp{G_m = m'}\right) \geq  c \sum_{m'=\ell}^{4\ell}  \Pp{G_{5\ell} = m'} = c,
\end{equation}
where in the last equality we used that the sum is one by definition \eqref{eq:red-law} of $G_{5\ell}$. By standard arguments (see for example \cite{prelov2022coupling}), this allows us to construct a coupling $G_m$, $m\in \intervalleentier{5\ell}{7\ell}$, of the golden instructions such that\footnote{Even if not strictly needed, we note that here we can take the same constant $c$ as in \eqref{eq:coupling_sum_inf}.}
    \[
    \Pp{\forall a, a'\in \intervalleentier{5\ell}{7\ell}, \ G_a = G_{a'}} \ge c.
    \]
This proves our claim.
\end{proof}

As explained in \cref{sect:descript}, we also need to show that the facts stated in \cref{lem:prop-good-scales-item2} of \cref{lem:prop-good-scales} for $\DMC$ also hold for $\DMC'$, conditional on them holding for $\DMC$, with uniformly positive probability across all good scales.

\begin{lem}\label{lem:cond_M_item_goodscale}
     Let $\DMC$ and $\DMC'$ be defined as the decreasing chains started from respectively $k$ and $k'$ under the black-golden coupling $\mathbb{P}_{k,k'}$ of~\cref{defn:black-golden coupling}.
     Let $c\in(0,1)$ be the constant in \cref{lem:prop-good-scales}. Then, for any  $\ell\in\mathrm{GoodScales}$ and for all $ k,k'\in \intervalleentier{6\ell}{7\ell}$, we have that
    \begin{equation*}
		\Pppsq{k,k'}{M'(H'_{4\ell} -1) \in \intervalleentier{5\ell}{7\ell}, M'(H'_{4\ell}) \in \intervalleentier{\ell}{4\ell} }{ M(H_{4\ell} -1) \in \intervalleentier{5\ell}{7\ell}, M(H_{4\ell}) \in \intervalleentier{\ell}{4\ell} } \ge \frac{c}{2}.
    \end{equation*}
\end{lem}

\begin{proof}
    Fix $\ell\in\mathrm{GoodScales}$ and $ k,k'\in \intervalleentier{6\ell}{7\ell}$, and define the finite set of paths
	\[
	\mathcal{M}_{k,\ell}\coloneqq\Big\{\vec{m}=(m_i)_{i\in\intervalleentier{0}{h-1}}:h>0,\,\,k=m_0 >m_1 > \ldots >m_{h-1},\,\, m_{h-1} \in \intervalleentier{5\ell}{7\ell}\Big\}.
	\]
	We then have
	\begin{multline*}
	\left\{M_0=k, M(H_{4\ell} -1) \in \intervalleentier{5\ell}{7\ell}, M(H_{4\ell}) \in \intervalleentier{\ell}{4\ell} )\right\} \\
	=
	\bigsqcup_{\vec{m}\in\mathcal{M}_{k,\ell}}\left\{(M_i)_{i\in\intervalleentier{0}{H_{4\ell}-1 }}=\vec{m}, B_{m_{h-1}}\in \intervalleentier{\ell}{4\ell}\right\},
	\end{multline*}
    where we recall that $B_{m_{h-1}}$ is the black instruction at height $m_{h-1}$.
    We now fix a path $\vec{m}=(m_i)_{i\in\intervalleentier{0}{h-1}}\in\mathcal{M}_{k,\ell}$ and set  $\leftevent\coloneqq \{M'(H'_{4\ell} -1) \in \intervalleentier{5\ell}{7\ell}, M'(H'_{4\ell}) \in \intervalleentier{\ell}{4\ell}\}$.  We will show that
	\begin{equation} \label{eq:conditionP}
		\Pppsq{k,k'}{\leftevent }{ (M_i)_{i\in\intervalleentier{0}{H_{4\ell}-1 }}=\vec{m}, B_{m_{h-1}}\in \intervalleentier{\ell}{4\ell}} \ge \frac{c}{2}.
	\end{equation}
    Note that this immediately implies the lemma statement.
    We need to introduce some new notation. Define, for all $t'>0$, the finite set of paths
	\begin{multline*}
     \mathcal{M}'_{k'}(\vec{m},t')\coloneqq\Big\{\vec{m}'=(m'_i)_{i\in\intervalleentier{0}{t'}}:  k'=m'_0 >m'_1 > \ldots >m'_{t'},\\
	 m'_{t'}=m_{t}\text{ for some $t\leq h-1$},\,\, m'_i\neq m_j \text{ for all $i< t'$ and $j< t$}\Big\},
	\end{multline*}
	that is, the set of all possible paths for $M'$ up to $t'$ that are merging into the path $\vec{m}$ at time $t'$. We also define the set
	\begin{multline*}
		{\mathcal{M}'^{\,\mathrm{c}}_{k'}}(\vec{m})\coloneqq\Big\{\vec{m}'=(m'_i)_{i\in\intervalleentier{0}{h'-1}}:h'>0, \,\, k'=m'_0 >m'_1 > \ldots >m'_{h'-1},\\ 
		m'_{h'-1} \in \intervalleentier{5\ell}{7\ell},\,\, m'_i\neq m_j \text{ for all $i \leq h'-1$ and $j \leq h-1$}\Big\},
	\end{multline*}
	that is, the set of all possible paths for $M'$ up to $\stt'_{4\ell}-1$ on the event $\leftevent$ that are not merging into the path $\vec{m}$ up to time $h-1$.
	We also introduce the events 
	\begin{align*}
        \mathcal{E}_{k',\vec{m}}&\coloneqq \bigsqcup_{t'>0} \left\{(M'_i)_{i\in\intervalleentier{0}{t'}}\in \mathcal{M}'_{k'}(\vec{m},t')\right\}, \\
        \mathcal{E}^{\mathrm{c}}_{k',\vec{m}}&\coloneqq\left\{(M'_i)_{i\in\intervalleentier{0}{H'_{4\ell}-1}}\in \mathcal{M}'^{\,\mathrm{c}}_{k'}(\vec{m})\, , \,B_{M'(H'_{4\ell}-1)}\in \intervalleentier{\ell}{4\ell}\right\}.
	\end{align*}
        Note that $\mathcal{E}^{\mathrm{c}}_{k',\vec{m}} \subset \leftevent$, and on the event $\left\{(M_i)_{i\in\intervalleentier{0}{H_{4\ell}-1 }}=\vec{m}, B_{m_{h-1}}\in \intervalleentier{\ell}{4\ell}\right\}$,
        \begin{equation}\label{eq:[artitioning}
            \leftevent=\mathcal{E}_{k',\vec{m}}\sqcup\mathcal{E}^{\mathrm{c}}_{k',\vec{m}}.
        \end{equation}
        
	We divide our analysis of \eqref{eq:conditionP} into two cases. 
	
	\medskip
	
	\noindent\emph{\underline{Case 1}:} We first assume that $\mathbb{P}_{k'}(\mathcal{E}_{k',\vec{m}}) \ge \frac{c}{2}$. Note that thanks to \eqref{eq:[artitioning},
	\begin{multline}\label{eq:conditionP10}
		\Pppsq{k,k'}{\leftevent }{(M_i)_{i\in\intervalleentier{0}{H_{4\ell}-1 }}=\vec{m}, B_{m_{h-1}}\in \intervalleentier{\ell}{4\ell}} 
		\ge 
		\Pppsq{k,k'}{\mathcal{E}_{k',\vec{m}}}{(M_i)_{i\in\intervalleentier{0}{H_{4\ell}-1 }}=\vec{m}, B_{m_{h-1}\in \intervalleentier{\ell}{4\ell}}}\\
        =\sum_{t'>0}\Pppsq{k,k'}{(M'_i)_{i\in\intervalleentier{0}{t'}}\in \mathcal{M}'_{k'}(\vec{m},t')}{(M_i)_{i\in\intervalleentier{0}{H_{4\ell}-1 }}=\vec{m}, B_{m_{h-1}\in \intervalleentier{\ell}{4\ell}}}.
	\end{multline}

    Now, fix $t'>0$. For any path $\vec{m}'=(m'_i)_{i\in\intervalleentier{0}{t'}}\in \mathcal{M}'_{k'}(\vec{m},t')$, the event $\left\{(M'_i)_{i\in\intervalleentier{0}{t'}}=\vec{m}'\right\}$ is measurable w.r.t.\ the black instructions  $(B_{m'_i})_{i\in\intervalleentier{0}{t'-1}}$, while $\left\{(M_i)_{i\in\intervalleentier{0}{H_{4\ell}-1}}=\vec{m},B_{m_{h-1}}\in \intervalleentier{\ell}{4\ell}\right\}$ is measurable w.r.t.\ the black instructions  $(B_{m_i})_{i\in\intervalleentier{0}{h-1}}$. 
    But, by the definition of $\mathcal{M}'_{k'}(\vec{m},t')$, we have that $m'_i\neq m_j$  for all $i< t'$ and $j< t$ (and so also for all $j \leq h-1$ since both paths are decreasing and $m_{t'}'=m_t$).
	Therefore, these two events are independent by the independence of the black instructions. As a result, we have
	\begin{equation}\label{eq:conditionP11}
	\Pppsq{k,k'}{\mathcal{E}_{k',\vec{m}}}{ (M_i)_{i\in\intervalleentier{0}{H_{4\ell}-1 }}=\vec{m}, B_{m_{h-1}}\in \intervalleentier{\ell}{4\ell}} = \Ppp{k,k'}{\mathcal{E}_{k',\vec{m}}} \ge \frac{c}{2}
	\end{equation}
    by our assumption that $\Ppp{k'}{\mathcal{E}_{k',\vec{m}}} \ge \frac{c}{2}$.
    Combining \eqref{eq:conditionP10} and \eqref{eq:conditionP11}, we get \eqref{eq:conditionP}.
	
	\medskip
	
	\noindent\emph{\underline{Case 2}:} We now assume that $\Ppp{k'}{\mathcal{E}_{k',\vec{m}}} < \frac{c}{2}$.  
    Because of \cref{eq:[artitioning}, we have the bound
	\begin{multline}\label{eq:eifbubefobew}
		\Pppsq{k,k'}{\leftevent }{(M_i)_{i\in\intervalleentier{0}{H_{4\ell}-1 }}=\vec{m}, B_{m_{h-1}}\in \intervalleentier{\ell}{4\ell}}\\
		\ge 
        \Pppsq{k,k'}{\mathcal{E}^{\mathrm{c}}_{k',\vec{m}}}{(M_i)_{i\in\intervalleentier{0}{H_{4\ell}-1 }}=\vec{m}, B_{m_{h-1}}\in \intervalleentier{\ell}{4\ell}}.
	\end{multline}
	Now, for any $\vec{m}'=(m'_i)_{i\in\intervalleentier{0}{h'-1}}\in \mathcal{M}'^{\,\mathrm{c}}_{k'}(\vec{m})$, the event $\left\{(M'_i)_{i\in\intervalleentier{0}{H'_{4\ell}-1}}=\vec{m}', B_{m'_{h'-1}}\in \intervalleentier{\ell}{4\ell}\right\}$ is measurable w.r.t.\ the black instructions  $(B_{m'_i})_{i\in\intervalleentier{0}{h'-1}}$, 
    while  $\left\{(M_i)_{i\in\intervalleentier{0}{H_{4\ell}-1 }}=\vec{m}, B_{m_{h-1}}\in \intervalleentier{\ell}{4\ell}\right\}$ is measurable w.r.t.\ the black instructions  $(B_{m_i})_{i\in\intervalleentier{0}{h-1}}$. 
    But, by definition of $\mathcal{M}'^{\,\mathrm{c}}_{k'}(\vec{m})$, we have that $m'_i\neq m_j$  for all $i\leq h'-1$ and $j\leq h$ and thus these two events are independent. 
    Combined with \eqref{eq:eifbubefobew}, this shows that
	\begin{equation}\label{eq:wefhvweifb}
		\Pppsq{k,k'}{\leftevent }{(M_i)_{i\in\intervalleentier{0}{H_{4\ell}-1 }}=\vec{m}, B_{m_{h-1}}\in \intervalleentier{\ell}{4\ell}}
		\ge  
        \Ppp{k,k'}{\mathcal{E}^{\mathrm{c}}_{k',\vec{m}}}.
	\end{equation}
	Now, since $\mathcal{E}^{\mathrm{c}}_{k',\vec{m}} \subset \leftevent \subset \mathcal{E}^{\mathrm{c}}_{k',\vec{m}} \cup \mathcal{E}_{k',\vec{m}}$, we can write
	\begin{equation*}
		\Ppp{k,k'}{\leftevent}
		= 
        \Ppp{k,k'}{\mathcal{E}_{k',\vec{m}}^{\mathrm{c}}}
		+ 
        \Pppsq{k,k'}{\leftevent}{ \mathcal{E}_{k',\vec{m}}}\cdot \Ppp{k,k'}{\mathcal{E}_{k',\vec{m}}}.
	\end{equation*}
	Thanks to \cref{lem:prop-good-scales-item2} in \cref{lem:prop-good-scales}, we have that $\Ppp{k,k'}{\leftevent} \geq c$. Since $\mathbb{P}_{k,k'}(\mathcal{E}_{k',\vec{m}}) < \frac{c}{2}$ by assumption, we immediately get that
	\begin{equation}\label{eq:unconditionedcons} 
        \Ppp{k,k'}{\mathcal{E}_{k',\vec{m}}^{\mathrm{c}}} > \frac{c}{2}. 
	\end{equation}
	Thus, we get \eqref{eq:conditionP} as a consequence of \eqref{eq:wefhvweifb} and \eqref{eq:unconditionedcons}.
	
	\medskip

    This concludes the proof of the lemma.
\end{proof}

\subsubsection{Completing the coupling argument: fast merging of the two walks}\label{sect:proof-coupling}

We are now ready to complete the proof of \cref{prop:coupling_merging}.

\begin{proof}[Proof of \cref{prop:coupling_merging}]
Most of the proof will be devoted to establishing~\cref{item:fast-merg-coupling}.
The gist of the proof is the following: by Items \ref{lem:prop-good-scales-item1} and \ref{lem:prop-good-scales-item2} of \cref{lem:prop-good-scales} (combined with \cref{lem:cond_M_item_goodscale}), at each good scale, there is positive probability that both $\DMC$ and $\DMC'$ follow a golden instruction, and therefore positive probability to merge thanks to \cref{lem:prop-good-scales-item3} of \cref{lem:prop-good-scales}. Moreover, by the specific conditioning in \cref{lem:prop-good-scales-item1} of \cref{lem:prop-good-scales}, we can iterate this argument at each good scale, so this will happen at a fair amount of scales. Since there are infinitely many good scales, it will happen with probability going to $1$ as the starting points $k$ and $k'$ of $\DMC$ and $\DMC'$ go to infinity. 

\smallskip

We now make this argument rigorous.
Without loss of generality, we assume that $k'\geq k$. 
Fix $A\geq 1$. For $k\ge A$, we recall that $s=s(A, k)$ is the number of good scales $\ell$ in $\intervalleentier{A}{k}$ such that $7\ell<k$.
We denote, for all $k\ge 1$, 
\begin{equation}\label{eq:enum-good-scales}
\ell^1=\ell^1(k)\coloneqq\max\{\ell'\in\mathrm{GoodScales}  \, :  \, 7\ell'<k\},
\end{equation}
with the convention $\ell^1(k)=0$ if no such scale exists, and then write $\ell^{i}=\ell^{i}(k)$ for the $i$-th good scale down from $\ell^1(k)$.

We start by noting that $\DMC$ and $\DMC'$ will meet above height $A$ if all the following events occur simultaneously at some scale $\ell^i$ with $1\le i\le s$ (recall \cref{fig-good-scales2}):
\begin{itemize}
    \item The chains $\DMC$ and $\DMC'$, at times $\stt_{7\ell^{i}}$ and $H'_{7\ell^{i}}$ respectively, land in the sub-interval $\intervalleentier{6\ell^{i}}{7\ell^{i}}$ (highlighted in blue in \cref{fig-good-scales2}).
    \item The chains $\DMC$ and $\DMC'$ both exit the interval $\intervalleentier{5\ell^{i}}{7\ell^{i}}$ by following a golden instruction, \emph{i.e.}\ landing in the golden sub-interval in \cref{fig-good-scales2}. 
    This means that $\DMC\big(\stt_{4\ell^i}-1\big) \in \intervalleentier{5\ell^i}{7\ell^i}$ and the black instruction $B_{\DMC(\stt_{4\ell^i}-1)}$ is inside the interval $\intervalleentier{\ell^i}{4\ell^i}$, and that the same holds for $\DMC'$.
    \item The golden instructions $G_{\DMC(\stt_{4\ell^{i}}-1)}$ and $G_{\DMC'(H'_{4\ell^{i}}-1)}$ are equal.
\end{itemize}

\noindent Therefore, for any $i \in \intervalleentier{1}{s}$ we introduce the event
\begin{equation} \label{eq:def_Ai_cap}
\evnt_i \coloneqq  \evnt_i^1 \cap (\evnt_i^1)' \cap \evnt_i^2 \cap (\evnt_i^2)' \cap \evnt_i^3,
\end{equation}
where 
\begin{equation*}
\left\lbrace
\begin{aligned}
\evnt_i^1
&=
\left\{\DMC(\stt_{7\ell^i})\in  \intervalleentier{6\ell^i}{7\ell^i} \right\}, \\
    \evnt_i^2
&=
\Big\{\DMC\big(\stt_{4\ell^i}-1\big) \in \intervalleentier{5\ell^i}{7\ell^i}, B_{\DMC(\stt_{4\ell^{i}}-1)}\in \intervalleentier{\ell^i}{4\ell^i}\Big\},\\
\evnt_i^3
&=
\left\{\forall a,a'\in \intervalleentier{5\ell^i}{7\ell^i}, \ G_{a}= G_{a'}\right\},
\end{aligned}
\right.
\end{equation*}
and $(\evnt_i^1)'$ and $(\evnt_i^2)'$ are constructed similarly to $\evnt_i^1$ and $\evnt_i^2$ but for the chain $\DMC'$. 
Recall the $\sigma$-field $\Gf_{\ell^i}$ of \eqref{eq:def_Ff}.
Note that for any $j<i$, the event $\evnt_{j}$ is $\Gf_{\ell^{i}}$-measurable.

The preceding discussion leads to
\begin{equation} \label{eq:bound_xmer_Ai}
\Ppp{k,k'}{\xmer < A}
\le
\Ppp{k,k'}{\bigcap_{i=1}^{s} \evnt_i^{\mathrm{c}}}.
\end{equation}

We start by proving that for any $i \in \intervalleentier{1}{s}$ we almost surely have 
\begin{align} \label{eq:lb_P(Ai)}
\Pppsq{k,k'}{\evnt_i}{\Gf_{\ell^i}}
\geq c^5/2,
\end{align}
where $c\in(0,1)$ is as in \cref{lem:prop-good-scales}.
We consider the events sequentially.

\medskip

\noindent\emph{\underline{Bound on $\Pppsq{k,k'}{\evnt_i^1 \cap (\evnt_i^1)'}{\Gf_{\ell^i}}$}}: 
By \cref{lem:spatial markov property item1} of \cref{lem:spatial Markov property} together with \cref{lem:prop-good-scales-item1} of \cref{lem:prop-good-scales}, 
we get that on the event $\{\DMC(H_{7\ell^i}-1)=m, \DMC'(H'_{7\ell^i}-1)=m'\}$ with  
$m\in\intervalleentier{7\ell^i +1}{k}$ and  $m'\in\intervalleentier{7\ell^i +1}{k'}$ such that $m \ne m'$,
we have 
\begin{align*}
\Pppsq{k,k'}{\evnt_i^1 \cap (\evnt_i^1)'}{\Gf_{\ell^i}} 
&= \Ppsq{B_m,B_{m'}\in \intervalleentier{6\ell}{7\ell} 
}{B_m\leq 7\ell, B_{m'}\leq 7\ell} \notag \\
&= \Ppsq{B_m\in \intervalleentier{6\ell}{7\ell} 
}{B_m\leq 7\ell} \cdot \Ppsq{B_{m'}\in \intervalleentier{6\ell}{7\ell} 
}{B_{m'}\leq 7\ell} \notag \\
&\geq c^2. 
\end{align*}
Moreover, if $m=m'$, the same computation gives a lower bound $c \geq c^2$, so the last inequality is true for every possible value $(m,m')$ that can be taken by the pair $(\DMC(H_{7\ell^i}-1),\DMC'(H'_{7\ell^i}-1))$. Therefore, we almost surely have
\begin{equation}\label{eq:bnd_evnt1_i}
    \Pppsq{k,k'}{\evnt_i^1\cap (\evnt_i^1)'}{\Gf_{\ell^i}} \geq c^2.
\end{equation}

\medskip

\noindent\emph{\underline{Bound on $\Pppsq{k,k'}{\evnt_i^1 \cap (\evnt_i^1)' \cap  \evnt_i^2 \cap (\evnt_i^2)' }{\Gf_{\ell^i}}$ }}: 
Let $j,j'\in \intervalleentier{6\ell^i}{7\ell^i}$.
Conditioning on $\evnt_i^2$, we have
\[
\Ppp{j,j'}{\evnt_i^2 \cap (\evnt_i^2)'}
=
\Ecp{j}{\mathds{1}_{\evnt_i^2} \cdot 
\Pppsq{j,j'}{(\evnt_i^2)'}{\evnt_i^2}}. 
\]
From there, we use \cref{lem:cond_M_item_goodscale} to get $\Pppsq{j,j'}{(\evnt_i^2)'}{\evnt_i^2} \geq c/2$ and then Item~\ref{lem:prop-good-scales-item2} in \cref{lem:prop-good-scales} to obtain
\begin{equation} \label{eq:bnd_evnt2_j}
\Ppp{j,j'}{\evnt_i^2 \cap (\evnt_i^2)'} \geq c^2/2.
\end{equation}
Note that the events $\evnt_i^2$ and $(\evnt_i^2)'$ are realized for the chains $\DMC$ and $\DMC'$ if and only if they are realized for the chains $(\DMC(H_{7\ell^i}+t))_{t\geq 0}$ and $(\DMC'(H'_{7\ell^i}+t))_{t\geq 0}$.
By \cref{lem:spatial markov property item2} of \cref{lem:spatial Markov property}, this ensures that
\begin{align*}
\Pppsq{k,k'}{\evnt_i^1 \cap (\evnt_i^1)' \cap  \evnt_i^2 \cap (\evnt_i^2)' }{\Gf_{\ell^i}} 
&= \Ecpsq{k,k'}{\mathds{1}_{\evnt_i^1 \cap (\evnt_i^1)'} \mathds{1}_{\evnt_i^2 \cap (\evnt_i^2)'} }{\Gf_{\ell^i}} \\
&=\Ecpsq{k,k'}{\mathds{1}_{\evnt_i^1\cap (\evnt_i^1)'} \Ppp{\DMC(H_{7\ell^i}),\DMC'(H'_{7\ell^i})}{\evnt_i^2\cap (\evnt_i^2)' }}{\Gf_{\ell^i}}\\
&\overset{\mathclap{\eqref{eq:bnd_evnt2_j}}}{\geq} \, \,  \frac{c^2}{2}\Ecpsq{k,k'}{\mathds{1}_{\evnt_i^1 \cap (\evnt_i^1)'} }{\Gf_{\ell^i}} 
\overset{\eqref{eq:bnd_evnt1_i}}{\geq} \frac{c^4}{2}
\end{align*}
almost surely.

\medskip

\noindent\emph{\underline{Bound on $\Pppsq{k,k'}{\evnt_i}{\Gf_{\ell^i}}$}}: 
Finally, by independence of black and golden instructions, and independence of golden instructions between different scales, the event $\evnt_i^3$ is independent of the chains up to the times $H_{4\ell^i}-1$ and $H'_{4\ell^i}-1$ and of $B_{\DMC(H_{4\ell^i}-1)}$ and $B_{\DMC'(H'_{4\ell^i}-1)}$, so it is independent of $\evnt_i^1 \cap (\evnt_i^1)' \cap  \evnt_i^2 \cap (\evnt_i^2)'$. 
Since it is also independent of $\Gf_{\ell^i}$, we have, using Item~\ref{lem:prop-good-scales-item3} of \cref{lem:prop-good-scales},
\begin{align*}
\Pppsq{k,k'}{\evnt_i^1 \cap (\evnt_i^1)' \cap  \evnt_i^2 \cap (\evnt_i^2)'\cap \evnt_i^3 }{\Gf_{\ell^i}}
= \Pp{\evnt_i^3} \Pppsq{k,k'}{\evnt_i^1 \cap (\evnt_i^1)' \cap  \evnt_i^2 \cap (\evnt_i^2)'}{\Gf_{\ell^i}} \geq c \cdot \frac{c^4}{2} \geq \frac{c^5}{2}
\end{align*}
almost surely, which proves \eqref{eq:lb_P(Ai)}.

\medskip

Now, using in cascade that $\evnt_1,\dots ,\evnt_i$ are $\Gf_{\ell^{i+1}}$-measurable together with the fact \eqref{eq:lb_P(Ai)} we get
\begin{align*}
\Ppp{k,k'}{\bigcap_{i=1}^{s} \evnt_i^{\mathrm{c}}} 
&=\Ecp{k,k'}{\mathds{1}_{\evnt_1^{\mathrm{c}}} \dots  \mathds{1}_{\evnt_{s-1}^{\mathrm{c}}}\Ppsq{\evnt_s^{\mathrm{c}}}{\Gf_{\ell^s}}}\\
&\leq \bigg(1-\frac{c^5}{2}\bigg) \cdot \Ecp{k,k'}{\mathds{1}_{\evnt_1^{\mathrm{c}}} \dots \mathds{1}_{\evnt_{s-1}^{\mathrm{c}}}} \\
&\leq \bigg(1-\frac{c^5}{2}\bigg)^s=\bigg(1-\frac{c^5}{2}\bigg)^{s(A, k)}.
\end{align*}
This proves~\cref{item:fast-merg-coupling}~of~\cref{prop:coupling_merging}.
\medskip

We now prove \cref{item:inf-lim-t-mer-coupling}, \emph{i.e.}\ the one about $\mathtt{T} - \stmer$ and $\mathtt{T}' - \stmerpr$, which will be a consequence of the first one. Recall from the statement of~\cref{prop:coupling_merging} that $\stmer$ and $\stmerpr$ are the respective first times when the two chains $\DMC$ and $\DMC'$ visit $\xmer$. By symmetry, we may focus on the chain $M$, say, and prove that $\mathtt{T}-\stmer$ goes to infinity in probability. For all $t> 0$, since the chain $\DMC$ is decreasing, we have
\[
\Ppp{k,k'}{\mathtt{T}-\stmer < t} 
=
\Ppp{k,k'}{\stmer > \mathtt{T} -t} 
=
\Ppp{k,k'}{\xmer < \DMC_{\mathtt{T}-t}}. 
\]
We now introduce another parameter $A>0$, that we will tune later, and split the above display as follows: 
    \begin{equation} \label{eq:bound_T-tmer_A}
    \Ppp{k,k'}{\mathtt{T}-\stmer < t}
    \le
    \Ppp{k,k'}{\xmer \le A} + \Ppp{k}{\DMC_{\mathtt{T}-t} > A}
    \le
    \Ppp{k,k'}{\xmer \le A} + \sup_{\textsf{k}}\Ppp{\textsf{k}}{\DMC_{\mathtt{T}-t} > A}.
    \end{equation}
    By~\cref{item:fast-merg-coupling}~of~\cref{prop:coupling_merging}, for any fixed $A>0$, the first term on the right-hand side goes to $0$ as $k,k'\to \infty$. It remains to prove that $\sup_{\mathsf{k}} \Ppp{\mathsf{k}}{\DMC_{\mathtt{T}-t} > A} \to 0$ as $A\to\infty$.
        To see this, we use a decomposition into blocks and \cref{lem:prop-good-scales-item1}~of~\cref{lem:prop-good-scales}.
        Let $s\ge 1$, and take $A$ large enough so that there are at least $s \cdot (t+1)$ good scales $\ell$ such that $\ell<A$. 
        We see $t+1$ as the number of blocks, each containing $s$ consecutive good scales. Then we note that, on the event $\{\DMC_{\mathtt{T}-t} > A\}$, the chain $\DMC$ only has at most $t$ units of time to visit $(t+1)$ blocks of good scales. As a consequence, on that event, at least one block of $s$ good scales is not visited by $\DMC$. By a union bound and iterating the bound of \cref{lem:prop-good-scales-item1} in \cref{lem:prop-good-scales}, this gives that for all $\mathsf{k}> A$,
        \[
        \Ppp{\mathsf{k}}{\DMC_{\mathtt{T}-t} > A}
        \le
        (t+1)(1-c)^s.
        \]
        The above estimate is actually uniform over $\mathsf{k}\ge 1$ since the probability is $0$ if $\mathsf{k}\le A$. This entails that $\sup_{\mathsf{k}} \Ppp{\mathsf{k}}{\DMC_{\mathtt{T}-t} > A} \to 0$ as $A\to\infty$, providing the missing estimate in \eqref{eq:bound_T-tmer_A}, and hence~\cref{item:inf-lim-t-mer-coupling}.
\end{proof}

\section{Good regularity estimates for the sequences $q$ and $Q$}\label{sect:better-bounds}

The main goal of this section is to prove the desired good regularity estimates.

\begin{thm}[Good regularity estimates: $q$ and $Q$ are regularly varying]
\label{thm:good_regularity}
	There exist two slowly varying functions $\svfq, \svfbQ$ such that for all $k \geq 1$,
    \begin{equation*}
        q(k)=k^{-\frac{1}{2\alpha}-1} \svfq(k)
        \qquad \text{and}\qquad 
        Q(k)=k^{-\frac{1}{2\alpha}} \svfbQ(k),
    \end{equation*}
    where $\alpha \in \left( 1/2,1 \right)$ is as in \cref{thm:main}.
\end{thm}

Recall that $\TLis{k}$ denotes the $p$-signed critical binary Bienaymé--Galton--Watson tree $T$ conditioned on the event that $\lis(T)=k$. A key step is to first prove in \cref{sect-lln-int} the following law of large numbers for the number $\#\mathcal{L}^{\max}_{\cap}(\TLis{k})$ of leaves  contained in the intersection of all maximal positive subtrees of $\TLis{k}$. 
This will rely on the local convergence in \cref{thm:localconv} (or, rather, on the good merging properties of the coupling in \cref{cor:full_coupling}).

\begin{thm}[Law of large numbers for $\#\mathcal{L}^{\max}_{\cap}(\TLis{k})$]\label{thm:convergence_size_intersection}
	We have the convergence in probability
	\[ \frac{\# \mathcal{L}^{\max}_{\cap}(\TLis{k})}{k}  \xlongrightarrow[]{\mathbb{P}} \lambda  \quad \text{as } k\to \infty,  \]
	where $\lambda$ is a constant in $(0,1)$.
\end{thm}

Building on this law of large numbers, we will then obtain in~\cref{sect-bett-reg-q-Q} the good regularity estimates for the sequences $q$ and $Q$, proving \cref{thm:good_regularity}.

\subsection{The size of the intersection of all maximal positive subtrees}\label{sect-lln-int}

The goal of this section is to prove the law of large numbers in~\cref{thm:convergence_size_intersection} using first and second moment arguments. Controlling the latter moments is the purpose of the next two subsections. We will aim for rather precise estimates, as such refined control will prove useful later in~\cref{sec:sc_limit_Xn} for establishing the full scaling limit result of~\cref{thm:main3}.

\subsubsection{The first moment}

We begin with the first moment. Recall the Markov chain $(Z^{k}_h)_{h \geq 0} = (S_h,D_h,\Mst{h},W_h)_{h \geq 0}$ introduced in \eqref{eq:def_MC_Z}. From \eqref{eq:the target leaf belong to the intersection tree}, for a leaf $L^k$ chosen uniformly at random among the $k$ leaves of the leftmost maximal positive subtree $T^{k,\lmax}$, we have
\begin{align}\label{eq:the target leaf belong to the intersection tree2}
\left\{L^k\in \mathcal{L}^{\max}_\cap(T^k)\right\} =\Bigg\{ \sum_{h\geq 1} \mathds{1}\{S_h=\ominus \text{ and }\Mst{h}=W_h\} = 0\Bigg\}.
\end{align}
We also note that
\begin{equation}\label{eq:proba-exp}
\mathbb{E}\left[\frac{\#\mathcal{L}^{\max}_{\cap}(\TLis{k})}{k} \right]=\Pp{L^k\in \mathcal{L}^{\max}_\cap(T^k)}.
\end{equation}
Therefore, proving the convergence of the first moment of $\mathcal{L}^{\max}_\cap(T^k)$ will crucially rely on controlling the number of \textbf{forks}
\begin{align} \label{eq:fork_Fk}
F \coloneqq \sum_{m\geq 1}\sum_{h\geq 1} \mathds{1}\{S_h=\ominus \text{ and }\Mst{h}=W_h=m\}
\end{align}
along the spine towards the marked leaf $L^k$. We also introduce the version of $F$ truncated below $A$:
\begin{equation} \label{eq:fork_Fk_trunc}
F_A \coloneqq \sum_{m\geq A}\sum_{h\geq 1} \mathds{1}\{S_h=\ominus \text{ and }\Mst{h}=W_h=m\}.
\end{equation}
The following claim says that $F$ is \emph{local}, in the sense that $F_A$ is small when $A$ is large, and so $F$ is well-approximated by a random variable that only depends on the finite number of values taken by the non-increasing chain $\Mst{}$ before extinction. Recall from \cref{rmk:start-notation} that under $\mathbb{P}_k$, the chain $\Mst{}$ starts from $k$.  
\begin{lem}[The number of forks is local] \label{lem:forks_mmt}
There exists a constant $C>0$ such that, for all $A\ge 1$,
\begin{equation} \label{eq:bound_local_F}
\sup_{k\ge 1} \Ecp{k}{F_A}
\leq
C \cdot  Q(A),
\end{equation}
where $Q(\cdot)$ is as in \eqref{eq:def_q_Q}. 
Moreover, for all $k\geq 1$,
\begin{equation} \label{eq:c0_neq_1}
\Ppp{k}{F =0}
\le 
\frac{1}{2q(1)} < 1.
\end{equation}
\end{lem}

\begin{proof}
We have for all $A\ge 2$ and $k\ge 1$,
\begin{align*}
\Ecp{k}{F_A}
&=\sum_{m=A}^\infty \sum_{h\geq 1} \Ppp{k}{S_h=\ominus \text{ and }\Mst{h}=W_h=m}\\ 
&= \sum_{m=A}^\infty \sum_{h\geq 1} \Ppp{k}{\Mst{h-1}=m}\cdot \Pppsq{k}{S_h=\ominus, \Mst{h}=W_h=m}{\Mst{h-1}=m}
\end{align*}
since, on the event $\{S_h=\ominus, \Mst{h}=W_h=m\}$, we must have that $\Mst{h-1} = m$. By the Markov property and the expression of the transition probabilities \eqref{eq:trans1}, the above display boils down to
\[
\Ecp{k}{F_A}
=\sum_{m=A}^\infty \Ecp{k}{\sum_{h\geq 0}\mathds{1}\{\Mst{h}=m\}} \cdot \frac{1}{2} (1-p) q(m).
\]
Now for fixed $m\ge 2$, conditionally on $M$ ever hitting the value $m$, the random variable $\sum_{h\geq 0}\mathds{1}\{\Mst{h}=m\}$ is geometric with parameter $\theta(m)=p+(1-p)\frac{Q(m)+Q(m+1)}{2}$ as shown in \eqref{eq:trans_Mtau}, and is $0$ if $M$ does not hit $m$. Hence, for all $m\ge 2$, we get the upper bound
\[
\Ecp{k}{\sum_{h\geq 0}\mathds{1}\{\Mst{h}=m\}} \le \bigg(p+(1-p)\frac{Q(m)+Q(m+1)}{2}\bigg)^{-1}
\le C,
\]
where $C$ is a constant that is independent of $k$ and $m$. The same argument carries over to the case $m=1$ using that $\theta(1)=\frac{1}{2q(1)}$ (see the discussion following \eqref{eq:trans_Mtau}).
As a result,
\[
\Ecp{k}{F_A}
\leq
\frac{C}{2} (1-p) \sum_{m=A}^\infty  q(m)=\frac{C}{2} (1-p) Q(A).
\]
This bound is uniform in $k$, and we deduce \eqref{eq:bound_local_F}.

To prove \eqref{eq:c0_neq_1}, we follow similar ideas. Fix $k\ge 1$. By a crude bound, we may only look at the term $m=1$ in \eqref{eq:fork_Fk}:
\[
\Ppp{k}{F =0}
\le
\Ppp{k}{\sum_{h\geq 1} \mathds{1}\{S_h=\ominus \text{ and }\Mst{h}=W_h=1\} = 0}.
\]
Note that if $\stht{\tt T - 1} < h < \stht{\tt T}$, then $\Mst{h} = W_h = 1$ and $S_h = \ominus$ (recall \eqref{eq:between_wait_times}). Thus we have the inclusion of events: 
\[
\bigg\{\sum_{h\geq 1} \mathds{1}\{S_h=\ominus \text{ and }\Mst{h}=W_h=1\} = 0\bigg\}
\subset
\{ \stht{\tt T} - \stht{\tt T -1} = 1\}.
\]
But by the discussion following \eqref{eq:trans_Mtau}, we know that $\stht{\tt T} - \stht{\tt T -1}$ is geometric with parameter $\theta(1)=\frac{1}{2q(1)}$. This provides \eqref{eq:c0_neq_1}.
\end{proof}

From \cref{lem:forks_mmt} and the fast merging property of~\cref{cor:full_coupling}, we deduce the convergence of the first moment. Recall the notation $\xmer$ introduced in~\cref{prop:coupling_merging}.

\begin{lem}[The first moment converges]
\label{lem:cvg_first_mmt}
The following assertions are true:
\begin{enumerate}
    \item\label{item:for-later-1} For $k'\geq k$ and for all $A>0$,
\begin{equation}\label{eq:key-first-mom-bound}
\left\vert\mathbb{E}\left[\frac{\#\mathcal{L}^{\max}_{\cap}(\TLis{k})}{k} \right] - \mathbb{E}\left[\frac{\# \mathcal{L}^{\max}_{\cap}(\TLis{k'})}{k'} \right]\right\vert 
\leq
2 \sup_{\mathsf{k}\ge 1}\Ecp{\mathsf{k}}{F_A} + \Pppt{k,k'}{\xmer \le A}. 
\end{equation}
\item\label{item:cvg_first_mmt-2} The first moment converges:
\[
\mathbb{E}\left[\frac{\#\mathcal{L}^{\max}_{\cap}(\TLis{k})}{k} \right] \to \lambda  \quad \text{as } k\to \infty.
\]
\item\label{item:cvg_first_mmt-3} The limit $\lambda$ is in $(0,1)$.
\end{enumerate}
\end{lem}

The quantitative bound in \cref{item:for-later-1} will also be used later in the proof of \cref{thm:quantitative_moments}. 

\begin{proof}
We first prove the bound in \cref{item:for-later-1} using \cref{cor:full_coupling}.

Let now $F$ and $F'$ be the number of forks along the trajectories of $(Z_h^{k})$ and $(Z_h^{k'})$ respectively, as in \eqref{eq:fork_Fk}, and $F_A$ and $F_A'$ their truncated version, as in \eqref{eq:fork_Fk_trunc}.
Under the coupling $\widetilde{\mathbb{P}}_{k,k'}$ of \cref{cor:full_coupling}, we can write $F-F'$ as 
\begin{align}\label{eq:F-diif}
F-F' &= \sum_{m\geq 1}\sum_{h\geq 1} \mathds{1}\{S_h=\ominus \text{ and }\Mst{h}=W_h=m\} - \sum_{m\geq 1}\sum_{h\geq 1} \mathds{1}\{S_h'=\ominus \text{ and }(\Mst{h})'=W_h'=m\}\notag\\
&=  \sum_{m \geq \xmer}\sum_{h\geq 1} \big(\mathds{1}\{S_h=\ominus \text{ and }\Mst{h}=W_h=m\}-  \mathds{1}\{S_h'=\ominus \text{ and }(\Mst{h})'=W_h'=m\}\big),
\end{align}
where in the last equality we used that $Z^{k}$ and $Z^{k'}$ are identical as soon as $\Mst{}$ and $(\Mst{})'$ are both strictly smaller than $\xmer$.
To conclude, note from \eqref{eq:the target leaf belong to the intersection tree2}--\eqref{eq:fork_Fk} that 
\[
\left\vert\mathbb{E}\left[\frac{\#\mathcal{L}^{\max}_{\cap}(\TLis{k})}{k} \right] - \mathbb{E}\left[\frac{\# \mathcal{L}^{\max}_{\cap}(\TLis{k'})}{k'} \right]\right\vert =\left\vert\Ppp{k}{F=0} - \Ppp{k'}{F'=0}\right\vert,
\] 
and
\begin{align*}
    \left\vert\Ppp{k}{F=0} - \Ppp{k'}{F'=0}\right\vert
    &\leq \Ecpt{k,k'}{\left\vert\mathds{1}_{F=0} - \mathds{1}_{F'=0}\right\vert}\notag\\
    &\leq \Ecpt{k,k'}{1\wedge \left\vert F - F'\right\vert}\notag\\
    &= \Ecpt{k,k'}{(1\wedge|F-F'|) \mathds{1}_{\xmer} \geq A} 
    + \Ecpt{k,k'}{(1\wedge|F-F'|) \mathds{1}_{\xmer} < A} \notag\\
    &\le
    2 \sup_{\mathsf{k}\ge 1}\Ecp{\mathsf{k}}{F_A} + \Pppt{k,k'}{\xmer < A},
\end{align*}
where for the last inequality we used~\eqref{eq:F-diif}. This completes the proof of~\cref{item:for-later-1}.

\medskip

We now prove the convergence in \cref{item:cvg_first_mmt-2} by Cauchy characterization. 
The first term on the right-hand side of~\eqref{eq:key-first-mom-bound} is small as $A\rightarrow \infty$ because of \eqref{eq:bound_local_F} in \cref{lem:forks_mmt}, while the second term is small for fixed $A$ and $k,k'\rightarrow \infty$ as a result of~\cref{item:full_coupling_convergence} in~\cref{cor:full_coupling}.
This proves the Cauchy characterization and so \cref{item:cvg_first_mmt-2}.

\medskip

It remains to check~\cref{item:cvg_first_mmt-3}, \emph{i.e.}\ that the limit $\lambda$ is in $(0,1)$. The fact that  $\lambda$ is not $1$ is a straightforward consequence of \eqref{eq:c0_neq_1}. We now check that it is non-zero. We fix $A \geq 1$ and write 
\begin{align*}
F&=\sum_{m\geq 1}\sum_{h\geq 1} \mathds{1}\{S_h=\ominus \text{ and }\Mst{h}=W_h=m\}\\ 
&=F_{A+1} + \sum_{m=1}^A\sum_{h\geq 1} \mathds{1}\{S_h=\ominus \text{ and }\Mst{h}=W_h=m \}.
\end{align*}
We recall from \eqref{eq:stt-defn} the notation  $\stt_A= \inf\enstq{t\geq 1}{M_t \leq A}$. Using the strong Markov property of $Z$ together with the fact that the transitions only depend on $M$ (\cref{sect:trans-prob2}), we have
\[
\Ppp{k}{F=0} = \Ecp{k}{\mathds{1}_{F_{A+1}=0} \cdot \Ppp{M({\stt_A})}{F=0}}
\geq \inf_{1\leq \mathsf{k} \leq A}\Ppp{\mathsf{k}}{F=0} \cdot  \Ppp{k}{F_{A+1}=0},
\]
where for the inequality we used that $1 \leq M({\stt_A})\leq A$.
But by Markov's inequality
\[
\Ppp{k}{F_{A+1}=0} 
= 
1-\Ppp{k}{F_{A+1}\ge 1} 
\ge 
1 - \sup_{\mathsf{k}\ge 1} \Ecp{\mathsf{k}}{F_{A+1}},
\]
which is non-zero if $A$ is large enough thanks to \eqref{eq:bound_local_F}. Combining the last two displayed equations, we get a uniform lower bound for $\Ppp{k}{F=0}$, as we wanted.
\end{proof}

\subsubsection{The second moment}

We now turn to the second moment. 
First, let us recall some notation and introduce some more.
Recall from the beginning of \cref{sect:mark-desc} and \eqref{eq:def_MC_Z} that conditionally on $\TLis{k}$, we let $L^k$ be a leaf chosen uniformly at random among the $k$ leaves of the leftmost maximal positive subtree $\TLis{k,\lmax}$ and denote, for $h \geq 0$, by $v_h$ the ancestor (if it exists) of $L^k$ at height $h$; see~\cref{fig:Markov-chain-notation}. 
The process $Z^k=(Z^k_h)_{h \geq 0}=(S_h,D_h,\Mst{h},W_h)_{h \geq 0}$ defined in \eqref{eq:def_MC_Z} records some useful information about the subtrees above each of the vertices $(v_h)_{h\geq 0}$.
Now, pick $\widetilde{L}^k$ a second leaf with the same distribution as $L^k$, independently of $L^k$ conditionally on $T^k$.
We introduce the process $\widetilde{Z}^k=(\widetilde{Z}^k_h)_{h \geq 0}$ similarly to the process $Z^k$, recording information along the ancestral line of $\widetilde{L}^k$; 
and with $\dagger$ as in \eqref{eq:cemetery state}, we let
\begin{align}\label{eq:definition stdivv}
\stdivv \coloneqq \inf \enstq{h\geq 0}{Z_h^{k} \neq \widetilde{Z}_h^{k} \ \text{or}  \ Z_h^{k} =\widetilde{Z}_h^{k}= \dagger},
\end{align}
so that $v_{\stdivv-1}$ is the highest common ancestor of $L^k$ and $\widetilde{L}^k$.
Note that as soon as $L^k \neq \widetilde{L}^k$, then $\stdivv<\min(\eta,\widetilde{\eta})$.

We record some useful facts in a preliminary lemma. 
The proof would follow straightforwardly from the arguments presented in \cref{sect:trans-prob} so we omit it.

\begin{lem}[Following the paths leading to the two marked leaves]
\label{lem:double markov chain}
The following properties hold:
\begin{enumerate}
\item The two processes $Z^k$ and $\widetilde{Z}^k$ have the same law.
\item The process $(Z^k_h,\widetilde{Z}^k_h)_{h\geq 0}$ is  a Markov chain.
\item The divergence height $\stdivv$ is a stopping time for the Markov chain $(Z^k_h,\widetilde{Z}^k_h)_{h\geq 0}$. 
\item Whenever $L^k\neq \widetilde{L}^k$, the two leaves $L^k$ and $\widetilde{L}^k$ belong respectively to one and the other subtree above $v_{\stdivv-1}$ so that 
\begin{align*}
(\Mst{\stdivv},W_{\stdivv}) = (\widetilde{W}_{\stdivv},\widetilde{M}^*_{\stdivv}).
\end{align*}
\item Conditionally on $(Z^k_{h},\widetilde{Z}^k_{h})_{0\leq h \leq \stdivv}$, the two processes $(Z^k_{\stdivv+h})_{h\geq 0}$ and $(\widetilde{Z}^k_{\stdivv+h})_{h\geq 0}$ are independent. 
\end{enumerate}
\end{lem}

We then prove a first estimate, which roughly says that the explorations towards $L_k$ and $\widetilde{L}_k$ do not stick together for too long.
\begin{lem}[Divergence occurs quickly]
\label{lem:divergence occurs quickly}
There exists a constant $C>0$ such that for any $k\geq A > 1$,  
\begin{align*}
\Ppp{k}{\Mst{\stdivv}< A} \leq \frac{A^2}{k} + \frac{C}{A}.
\end{align*}
\end{lem}

\begin{proof}
Let $k\geq A > 1$. 
We write 
\begin{align*}
\Ppp{k}{\Mst{\stdivv}< A} = \Ppp{k}{\Mst{\stdivv}< A, \Mst{\stdivv}+W_{\stdivv}\leq  A^2} +  \Ppp{k}{\Mst{\stdivv}< A, \Mst{\stdivv}+W_{\stdivv} > A^2}. 
\end{align*}
We show below that the first term is small because it would require picking the second leaf $\widetilde{L}^k$ quite close to $L^k$. 
For the second term, we show that such an event can only happen if the process $M$ has an unlikely trajectory.

Assume that $L^k \neq \widetilde{L}^k$, meaning $\stdivv<\eta$. Note that $S_{\stdivv}=\oplus$ because $L^k$ and $\widetilde{L}^k$ are leaves of the same maximal positive subtree.
Thus we have $\Mst{\stdivv}+W_{\stdivv}= \Mst{\stdivv-1}$.
Recall from \eqref{eq:decreasing-eq} the definition of the process $M$ as a time-changed version of $M^*$,
so that by the remark above $\Mst{\stdivv}+W_{\stdivv}$ and $\Mst{\stdivv}$ are two consecutive values along the trajectory $\enstq{M_t}{ t \ge 0}$.
Recall from \eqref{eq:stt-defn}, the random variable $H_m$ which is such that $M(H_m) = \sup (\intervalleentier{1}{m}\cap \enstq{M_t}{ t\geq 0})$. 
For any $m\ge 1$, 
we can associate a node $v^{m}$ in the original tree $T^k$, which corresponds to the lowest vertex on the spine from the root to the leaf $L^k$ such that $T^{k,\lmax}$ has at most $m$ leaves (weakly) above $v^{m}$. 
We denote by $\mathcal{L}^{k,\lmax}(v^{m})$ the set of leaves of $T^{k,\lmax}$ above vertex $v^{m}$, which has cardinality $\DMC(\stt_m)$, that is at most $m$. 

Consider the event $\{\Mst{\stdivv}< A, \, \Mst{\stdivv}+W_{\stdivv} \leq  A^2\}$.
By the above discussion, if this event is realized, then necessarily the leaf $\widetilde{L}^k$ 
belongs to $\mathcal{L}^{k,\lmax}(v^{A^2})$ (note that this is obvious when $L^k = \widetilde{L}^k$).
Since by definition this tree has $\DMC(\stt_{A^2}) \le A^2$ leaves, 
by reasoning conditionally on $Z^k$, we get 
\begin{align*}
\Ppp{k}{\Mst{\stdivv}< A, \, \Mst{\stdivv}+W_{\stdivv}\leq  A^2} 
\leq 
\Pp{\widetilde{L}^k \in \mathcal{L}^{k,\lmax}(v^{A^2})}
=
\Ec{\frac{\DMC(\stt_{A^2})}{k}}
\leq \frac{A^2}{k}.
\end{align*}
We now handle the second term. 
By a similar discussion, we have 
\begin{align*}
\Ppp{k}{\Mst{\stdivv}< A, \, \Mst{\stdivv}+W_{\stdivv} > A^2} \leq \Ppp{k}{M(H_{A^2}) <A }
&
=
\Ec{\sum_{i=1}^{A-1} \tpr(\DMC(\stt_{A^2}-1),i)}\\
&\leq \sum_{i=1}^{A-1} C \frac{A}{A^2}q(i) \leq \frac{C}{A}
\end{align*}
for some constant $C$, by the second bound in \cref{item:trans-bound0} of \cref{lem:trans-bound}.
\end{proof}

Recalling the first moment convergence provided by \cref{lem:cvg_first_mmt}, we now show that the second moment converges.
\begin{lem}[The second moment converges] \label{lem:cvg_2nd_mmt}
The following assertions are true:
\begin{enumerate}
\item\label{item:2nd_mmt_1} For any $k\geq A > 1$ we have 
\begin{align*}
\left|\mathbb{E}\left[\left(\frac{\#\mathcal{L}^{\max}_{\cap}(\TLis{k})}{k}\right)^2 \right] - \lambda^2  \right| \leq \mathbb{E}_k[F_A] + 3 \sup_{\mathsf{k}\geq A} \left|\Ec{\frac{\#\mathcal{L}^{\max}_{\cap}(\TLis{\mathsf{k}})}{\mathsf{k}}}-\lambda \right| + 7\left(\frac{A^2}{k} +\frac{C}{A}\right).
\end{align*}
\item\label{item:2nd_mmt_2} We have the convergence 
\end{enumerate}
\[
\mathbb{E}\left[\left(\frac{\#\mathcal{L}^{\max}_{\cap}(\TLis{k})}{k}\right)^2 \right] \to \lambda^2  \qquad \text{as } k\to \infty.
\]
\end{lem}
\noindent 
The precise form of the first item will be used to obtain quantitative bounds in \cref{thm:quantitative_moments}.

\begin{proof}
    Conditionally on $\TLis{k}$, let $L^k$ and $\widetilde{L}^k$ be two  (possibly equal) independent leaves  chosen uniformly at random among the $k$ leaves of the leftmost maximal positive subtree $\TLis{k,\lmax}$. 
    Recall from \eqref{eq:fork_Fk} the random variable 
    \[
    F = \sum_{m\geq 1}\sum_{h\geq 1} \mathds{1}\{S_h=\ominus \text{ and }\Mst{h}=W_h=m\}, 
    \]
    counting the number of forks along the spine towards the marked leaf $L^k$. 
    Denote by $\widetilde{F}$ the analogous quantity for the marked leaf $\widetilde{L}^k$. 
    We have 
    \begin{align*}
    \mathbb{E}\left[\left(\frac{\#\mathcal{L}^{\max}_{\cap}(\TLis{k})}{k}\right)^2 \right]=\Ppp{k}{F=0,\widetilde{F}=0}.
    \end{align*}
    Recall that $\stdivv$ 
    is one plus the height of the highest common ancestor of $L^k$ and $\widetilde{L}^k$
    and set
    \begin{align*}
    &F_{{\leq} \stdivv } \coloneqq \sum_{m\geq 1}\sum_{h=1}^{\stdivv} \mathds{1}\{S_h=\ominus \text{ and }\Mst{h}=W_h=m\},\\
    &F_{{>} \stdivv } \coloneqq F - F_{{\leq}\stdivv }
    \quad \text{ and }\quad
    \widetilde{F}_{{>} \stdivv } \coloneqq \widetilde{F} -  F_{{\leq}\stdivv },
    \end{align*}
    where we note that $F_{\leq \stdivv }$ is a function of $(Z^k_{h},\widetilde{Z}^k_{h})_{0\leq h \leq \stdivv}$ and $F_{> \stdivv }$ and $\widetilde{F}_{> \stdivv }$ are a function of $(Z^k_{h},\widetilde{Z}^k_{h})_{h \geq \stdivv+1}$. 
    With this notation, we get
    \begin{equation} \label{eq:Lmax^2_Fdiv}
    \mathbb{E}\left[\left(\frac{\#\mathcal{L}^{\max}_{\cap}(\TLis{k})}{k}\right)^2 \right]=\Ppp{k}{F_{\leq \stdivv }=0,F_{>\stdivv }=0,\widetilde{F}_{>\stdivv }=0}.
    \end{equation}
    By Lemma~\ref{lem:double markov chain},
    we get that conditionally on $(Z^k_{h},\widetilde{Z}^k_{h})_{0\leq h \leq \stdivv}$, the processes $(Z^k_{h})_{h \geq \stdivv+1}$ and $(\widetilde{Z}^k_{h})_{h \geq \stdivv+1}$ are distributed as two independent versions of $(Z_h)_{h\geq 1}$ 
    under the respective measures $\bbP_{\Mst{\stdivv}}$ and $\bbP_{W_{\stdivv}}$.\footnote{Note that here we are using that the Markov property of $(Z^k_{h})_h$ holds when conditioning either on the full state $Z^k_h=(S_h,D_h,\Mst{h},W_h)$ or just on $\Mst{h}$, since the transitions depend only on this third coordinate; see~\cref{sect:trans-prob2}.}
Introduce the notation 
\[\epsilon(k)\coloneqq  \Ppp{k}{F=0}-\lambda = \mathbb{E}\left[\frac{\#\mathcal{L}^{\max}_{\cap}(\TLis{k})}{k} \right]-\lambda.\] 
This ensures that we have 
\begin{align*}
\Pppsq{k}{F_{>\stdivv }=0,\widetilde{F}_{>\stdivv }=0}{(Z^k_{h},\widetilde{Z}^k_{h})_{0\leq h \leq \stdivv}} = (\lambda + \epsilon(\Mst{\stdivv})) \cdot(\lambda +\epsilon(W_{\stdivv})).
\end{align*}
Using the last display allows us to rewrite the right-hand side of \eqref{eq:Lmax^2_Fdiv} as
\begin{align}\label{eq:probability no fork}
    &\Ppp{k}{F_{\leq \stdivv }=0,F_{>\stdivv }=0,\widetilde{F}_{>\stdivv }=0} \notag\\
    &= \mathbb{E}_k\left[\mathds{1}_{F_{\leq \stdivv }=0}\cdot
     (\lambda + \epsilon(\Mst{\stdivv})) \cdot(\lambda +\epsilon(W_{\stdivv}))\right] \notag\\
    &= \lambda^2 - \lambda^2\Ppp{k}{F_{\leq \stdivv } \geq 1} 
    + \Ecp{k}{\mathds{1}_{F_{\leq \stdivv }=0} 
    \cdot (\lambda \epsilon(\Mst{\stdivv}) + \lambda \epsilon(W_{\stdivv}) + \epsilon(\Mst{\stdivv})\epsilon(W_{\stdivv}))}.
\end{align}

Let us first handle the term $\Ppp{k}{F_{\leq \stdivv }\geq 1}$ in \eqref{eq:probability no fork}.  
For all $A>1$, we can write
    \begin{equation*}
        \Ppp{k}{F_{\leq \stdivv }\geq 1}=\Ppp{k}{F_{\leq \stdivv }\geq 1,\Mst{\stdivv} \geq A}+\Ppp{k}{F_{\leq \stdivv}\geq 1 ,\Mst{\stdivv} < A }.
    \end{equation*}
Now on the event $\{\Mst{\stdivv} \geq A\}$, we have $F_{\leq \stdivv }\leq F_A$, where we recall that $F_A$ was introduced in \eqref{eq:fork_Fk_trunc}. 
Hence, by Markov's inequality,
    \begin{equation*}
        \Ppp{k}{F_{\leq \stdivv } \geq 1,\Mst{\stdivv} \geq A}
        \leq 
        \Ppp{k}{F_A\geq 1}
        \leq
        \mathbb{E}_k[F_A].
    \end{equation*}
We conclude that
\begin{equation*}
        \Ppp{k}{F_{\leq \stdivv } \geq 1} \leq \mathbb{E}_k[F_A] + \Ppp{k}{F_{\leq \stdivv} \geq 1,\Mst{\stdivv} < A } 
        \le \mathbb{E}_k[F_A] + \Ppp{k}{\Mst{\stdivv} < A}.
    \end{equation*}

Now let us control the last term in \eqref{eq:probability no fork}.
We have 
\begin{align*}
&\left|\Ecp{k}{\mathds{1}_{F_{\leq \stdivv }=0} 
    \cdot (\lambda \epsilon(\Mst{\stdivv}) + \lambda \epsilon(W_{\stdivv}) + \epsilon(\Mst{\stdivv})\epsilon(W_{\stdivv}))}
\right|\\ 
&\leq \Ecp{k}{
    \lambda |\epsilon(\Mst{\stdivv}) | + \lambda |\epsilon(W_{\stdivv})| 
    + |\epsilon(\Mst{\stdivv})\epsilon(W_{\stdivv})|}\\
& \leq \Ecp{k}{
\mathds{1}_{\{\Mst{\stdivv} \geq A,W_{\stdivv} \geq A\}} 
    \cdot (\lambda |\epsilon(\Mst{\stdivv}) | + \lambda |\epsilon(W_{\stdivv})| + |\epsilon(\Mst{\stdivv})\epsilon(W_{\stdivv}))|} \\
    &\qquad + 3  \Ppp{k}{\{\Mst{\stdivv} \geq A, W_{\stdivv} \geq A\}^{\mathrm{c}}} \\
    &\leq 3 \sup_{m\geq A} |\epsilon(m)| + 3  \Ppp{k}{\{\Mst{\stdivv} \geq A,W_{\stdivv} \geq A\}^{\mathrm{c}}},
\end{align*}
where we used the two bounds $\lambda\leq 1$ and $|\epsilon(\cdot )| \leq 1$. 
Finally, using a union bound and the fact that $\Mst{\stdivv}$ and $W_{\stdivv} = \widetilde{M}^*_{\stdivv}$ have the same distribution by symmetry of $L^k$ and $\widetilde{L}^k$, we can write 
\begin{align*}
\Ppp{k}{\{\Mst{\stdivv} \geq A,W_{\stdivv} \geq A\}^{\text{c}}} \leq 2 \Pp{\Mst{\stdivv} < A}. 
\end{align*}
Putting everything together and using Lemma~\ref{lem:divergence occurs quickly}, we get 
\begin{align*}
\left|\mathbb{E}\left[\left(\frac{\#\mathcal{L}^{\max}_{\cap}(\TLis{k})}{k}\right)^2 \right] - \lambda^2\right| 
&\leq \mathbb{E}_k[F_A] + 3 \sup_{m\geq A} |\epsilon(m)| + 7\Pp{\Mst{\stdivv} < A}\\
&\leq \mathbb{E}_k[F_A] + 3 \sup_{\mathsf{k}\geq A} \left|\Ec{\frac{\#\mathcal{L}^{\max}_{\cap}(\TLis{\mathsf{k}})}{\mathsf{k}}}-\lambda \right| + 7 \left(\frac{A^2}{k} +\frac{C}{A}\right).
\end{align*}
Hence, \cref{item:2nd_mmt_1} is proved.

The convergence in \cref{item:2nd_mmt_2} is a consequence of \cref{item:2nd_mmt_1}, together with \cref{lem:forks_mmt} and \cref{lem:cvg_first_mmt}, by taking $A$ going to infinity with $k$, with $A=o(\sqrt{k})$. 
\end{proof}

\cref{lem:cvg_first_mmt} and \cref{lem:cvg_2nd_mmt} entail our law of large numbers in \cref{thm:convergence_size_intersection}.
\begin{proof}[Proof of \cref{thm:convergence_size_intersection}] 
The result follows from Chebyshev's inequality since, for any $\delta>0$,
\[
\Pp{\bigg| \frac{\# \mathcal{L}^{\max}_{\cap}(\TLis{k})}{k} - \lambda \bigg| \ge \delta}
\le
\frac{1}{\delta^2} \Ec{\bigg( \frac{\# \mathcal{L}^{\max}_{\cap}(\TLis{k})}{k} - \lambda \bigg)^2}
\to 0 \qquad \text{as } k\to \infty,
\]
by \cref{lem:cvg_first_mmt} and \cref{lem:cvg_2nd_mmt}.
\end{proof}

\subsection{Good regularity estimates for the sequences $q$ and $Q$}\label{sect-bett-reg-q-Q}

With the law of large numbers of \cref{thm:convergence_size_intersection} in hand, we can now show the good regularity estimates for the sequences $q$ and $Q$ stated in \cref{thm:good_regularity}. 

Recall, from \eqref{eq:defn-stp-time}, the a.s.\ finite stopping times
\begin{equation*}
	\sttm_k \coloneqq  \min \{ n \geq 1 \mid X_n=k\}, \qquad k \geq 1.
\end{equation*}
We first state a preliminary lemma, which improves the estimates in \eqref{eqn:q_sigma_lower_and_upper} of \cref{lem:q_and_sigma} and constitutes the main input towards \cref{thm:good_regularity}. The constant $\lambda$ was introduced in \cref{thm:convergence_size_intersection}.

\begin{lem}[Asymptotics for the sequence $q(k)$] \label{lem:equiv_q(k)}
As $k\to\infty$, we have
\[
q(k) \sim \frac{1}{2\sqrt{\pi}\lambda k} \Ec{\sttm_k^{-1/2}}.
\]
\end{lem}

Before proving the lemma, we show how it implies \cref{thm:good_regularity}.

\begin{proof}[Proof of \cref{thm:good_regularity} given \cref{lem:equiv_q(k)}]
The proof is basically the same as the end of the proof of \cref{prop:rough_regularity} on page~\pageref{eq:proof-rough}. By \eqref{eqn:large_Q_and_sigma} in \cref{lem:q_and_sigma} and \cref{lem:equiv_q(k)},
\begin{equation} \label{eq:estim_Q_ratio}
\frac{Q(k+1)}{Q(k)} 
= 1 - \frac{q(k)}{Q(k)}
= 1 - \frac{1}{2\lambda k} + o\bigg(\frac{1}{k}\bigg)
= \exp \left( -\frac{1}{2\lambda k} + o \left( \frac{1}{k} \right) \right)
\quad \text{as } k\to\infty.
\end{equation}
We first show the claim on $Q$. Let $\svx \geq 1$. Multiplying the estimates \eqref{eq:estim_Q_ratio} over a telescopic product, we have as $k\to\infty$,
\begin{equation}\label{eq:cons-alpha-lambda}
    \frac{Q(\svx k)}{Q(k)}
=
\exp \left( \sum_{i=k}^{\svx k-1} \left( -\frac{1}{2\lambda i} + o\left(\frac{1}{i}\right)\right) \right)
=
\exp\left( -\frac{1}{2\lambda}\sum_{i=k}^{\svx k-1}\frac{1}{i} + o(1)\right) 
\to \svx^{-\frac{1}{2\lambda}}.
\end{equation} 
By Karamata's characterization theorem \cite[Theorem 1.4.1]{bingham1989regular}, we obtain the existence of a slowly varying function $\svfbQ$ such that $Q(k)=k^{-\frac{1}{2\lambda}} \svfbQ(k)$. In particular, \cref{prop:rough_regularity} implies $\lambda=\alpha$, which proves the part of \cref{thm:good_regularity} on $Q$.

Finally, by comparing \cref{lem:equiv_q(k)} with \eqref{eqn:large_Q_and_sigma} in \cref{lem:q_and_sigma}, we have
\[
q(k)
\sim \frac{1}{2\lambda k} Q(k)
= \frac{1}{2\alpha} \svfbQ(k) k^{-1-\frac{1}{2\alpha}}
\quad \text{as } k\to\infty,
\]
which concludes the proof.
\end{proof}

The rest of this section is devoted to the proof of \cref{lem:equiv_q(k)}.
\begin{proof}[Proof of \cref{lem:equiv_q(k)}]
We divide the proof into five steps. The core of the proof is Step 3, where we refine our proof of the lower bound on $q(k)$ in \cref{lem:q_and_sigma}. The improvement will come from the precise asymptotics for $\#\mathcal{L}^{\max}_{\cap}(\TLis{k})$ obtained in \cref{thm:convergence_size_intersection}.

\bigskip
\noindent\emph{\underline{Step 1}: Truncation.}
    We fix a constant $K \geq 1$ (to be thought of as a large constant to be tuned at the end of the proof in Step 5) and for the rest of the proof we assume that $k>K$. The proof starts in the same way as the proof of \cref{lem:q_and_sigma}: by the second equality in \eqref{eq:q(k)_sum_sigma(k)}, for all $k \geq 1$, we have
    \[ q(k) = \Ec{\sum_{n=\sttm_k}^{\sttm_{k+1}-1} r(n) }. \]
    In view of \cref{lem:q_and_sigma} and \eqref{eq:for-fut-ref} say, we expect $\sttm_{k+1}-\sttm_k$ to be typically of order $\frac{\sttm_{k+1}}{k}$. 
    Therefore, we split the sum in two, according as $n \geq \sttm_{k+1}- \lfloor K\frac{\sttm_{k+1}}{k} \rfloor$ or not. If $n \geq \sttm_{k+1}- \lfloor K\frac{\sttm_{k+1}}{k} \rfloor$, note that by the asymptotics in \cref{lem:prelim_GW}, 
    \[
    r(n)=\left( 1+O_K \left( \frac{1}{k} \right)\right) r(\sttm_{k+1}), 
    \]
    that is, there is a constant $C_K$ depending on $K$ (but not on $k$) such that $|r(n)-r(\sttm_{k+1})| \le \frac{C_K}{k}r(\sttm_{k+1})$. This implies 
    \begin{equation}\label{eqn:decom_qk_small_and_big}
        q(k) = \left( 1+O_K \left( \frac{1}{k}\right) \right) \Ec{ r(\sttm_{k+1}) \min\left( \sttm_{k+1}-\sttm_k, \Big\lfloor K \frac{\sttm_{k+1}}{k} \Big \rfloor \right) }  + \Ec{ \sum_{n \geq \sttm_k} \mathbbm{1}_{n < \sttm_{k+1}- \lfloor K\frac{\sttm_{k+1}}{k} \rfloor} \, r(n) }.
    \end{equation}
    Roughly speaking, we will estimate the first term by applying the strategy of the lower bound of \cref{lem:q_and_sigma} in a more accurate way, where the additional precision will come from \cref{thm:convergence_size_intersection}. The second term will be bounded crudely by the same proof as the upper bound of \cref{lem:q_and_sigma}. Our main goal will be to show that all the error terms that will appear are small relative to $q(k)$.

    \medskip
    
\noindent\emph{\underline{Step 2}: Neglecting the second term of \cref{eqn:decom_qk_small_and_big}.}
    Since it is easier, let us start with the second term.
    We note that if $n < \sttm_{k+1}-\lfloor K\frac{\sttm_{k+1}}{k} \rfloor$, then $\sttm_{k+1}>n+ \lfloor K\frac{\sttm_k}{k} \rfloor$. 
    Using the monotonicity of $r(\cdot)$ from~\cref{lem:prelim_GW}, we can bound $r(n)$ by $r(\sttm_k)$ and condition on $\sttm_k$ to get
    \begin{equation} \label{eq:2nd_term_cond_tau_k}
    \Ec{ \sum_{n \geq \sttm_k} \mathbbm{1}_{n < \sttm_{k+1}-\lfloor K\frac{\sttm_{k+1}}{k} \rfloor} \, r(n) } \leq \Ec{ r(\sttm_k) \sum_{n \geq \sttm_k} \Ppsq{\sttm_{k+1}>n+ \left\lfloor K\frac{\sttm_k}{k} \right\rfloor}{\sttm_k} }.
    \end{equation}
    Using the change of variables $j=n-\sttm_k+\lfloor K\frac{\sttm_k}{k} \rfloor$, we can now use~\eqref{eqn:bound_tail_increment_sigma} to obtain, almost surely,
    \begin{align*}
        \sum_{n \geq \sttm_k} \Ppsq{\sttm_{k+1}>n+ \Big\lfloor K\frac{\sttm_k}{k} \Big\rfloor}{\sttm_k} 
        &=
        \sum_{j \geq \lfloor K\frac{\sttm_k}{k}\rfloor} \Ppsq{\sttm_{k+1}>\sttm_k + j}{\sttm_k} \\
        &\stackrel{\mathclap{\eqref{eqn:bound_tail_increment_sigma}}}{\leq} \sum_{j \geq \lfloor K\frac{\sttm_k}{k}\rfloor} \left( 1+\frac{j}{\sttm_k} \right)^{-p(k-1)} \\
        &\leq \int_{\lfloor K \sttm_k/k \rfloor -1}^{+\infty} \left( 1+\frac{t}{\sttm_k} \right)^{-p(k-1)} \mathrm{d}t.
    \end{align*}
    Now for $k$ large enough so that the integral is finite, we find that
    \begin{align*}
    \int_{\lfloor K \sttm_k/k \rfloor -1}^{+\infty} \left( 1+\frac{t}{\sttm_k} \right)^{-p(k-1)} \mathrm{d}t
    &=
    \frac{\sttm_k}{p(k-1)-1} \left( 1+\frac{\lfloor K\frac{\sttm_k}{k} \rfloor-1}{\sttm_k} \right)^{1-p(k-1)} \\
    &\leq
    \frac{\sttm_k}{p(k-1)-1} \exp\left( \left( -pk+1+p \right) \log\left(1+\frac{K-2}{k}
    \right)  \right),
    \end{align*}
    where for the bound in the $\log$ we used that $\sttm_k\geq k$.
    We conclude that there exists a constant $C>0$ (independent of $K$) and a threshold $A_K$ such that for $k\ge A_K$, 
    \[
    \sum_{n \geq \sttm_k} \Ppsq{\sttm_{k+1}>n+ \Big\lfloor K\frac{\sttm_k}{k} \Big\rfloor}{\sttm_k} 
    \le
    C \frac{\sttm_k}{k} \mathrm{e}^{-pK}.
    \]
    By integrating in $\sttm_k$, we finally obtain that for $k\ge A_K$,
    \begin{equation}\label{eqn:q_accurate_bound_error_term}
        \Ec{ \sum_{n \geq \sttm_k} \mathbbm{1}_{n < \sttm_{k+1}-\lfloor K\frac{\sttm_{k+1}}{k}\rfloor}  \, r(n) } 
        \stackrel{\eqref{eq:2nd_term_cond_tau_k}}{\le} 
        C \frac{\mathrm{e}^{-pK}}{k}\Ec{r(\sttm_k)\sttm_k}
        \,\,
        \stackrel{\mathclap{\eqref{eq:asympt r(m)}}}{\le}
        \,\,
        C \frac{\mathrm{e}^{-pK}}{k} \Ec{\sttm_k^{-1/2}}
        \stackrel{\eqref{eqn:q_sigma_lower_and_upper}}{\le} C \mathrm{e}^{-p K}q(k),
    \end{equation}
    for some other constant $C$. This bounds the second term of~\eqref{eqn:decom_qk_small_and_big}.

\medskip

\noindent\emph{\underline{Step 3}: Estimating the first expectation in  \cref{eqn:decom_qk_small_and_big}.}
For the first expectation in~\eqref{eqn:decom_qk_small_and_big}, we will be more precise and rely on \cref{thm:convergence_size_intersection}. Recall the notation $\mathcal{L}_{\cap}^{\max}(t)$ for the set of leaves contained in all maximal positive subtrees of $t$. Let $\lambda\in (0,1)$ be the constant appearing in \cref{thm:convergence_size_intersection}. We now take $\eps \in (0,1)$ (to be thought of as a small constant to be tuned at the end of the proof in Step 5) and denote by $\teps$ the set of finite decorated trees $t$ such that
	\[ \left| \frac{\# \mathcal{L}_{\cap}^{\max}(t)}{\lis(t)}-\lambda \right| \geq \eps. \]
	In particular, \cref{thm:convergence_size_intersection} means that $\Ppsq{ T \in \teps}{\lis(T)=k} \to 0$ as $k \to \infty$. For future reference, observe that by definition of $\sttm_{k}$ and since $T$ conditioned to have size $n$ is distributed as $T_n$,  
    this implies that as $k\to\infty$,
	\begin{equation} \label{eq:A_sum_sigma_k}
	\Ec{\sum_{n={\sttm_k}}^{\sttm_{k+1}-1} r(n) \mathds{1}_{T_n \in \teps} }
    = \sum_{n\ge 1} r(n) \Pp{T_n \in \teps, X_n=k }
    =  \Pp{T \in \teps, \LIS(T)=k}
    = o(q(k)).
    \end{equation}
    Since our goal is to estimate the term 
    \[\Ec{ r(\sttm_{k+1}) \min\left( \sttm_{k+1}-\sttm_k, \Big\lfloor K \frac{\sttm_{k+1}}{k} \Big\rfloor \right) }=\Ec{ \Ecsq{r(\sttm_{k+1}) \min\left( \sttm_{k+1}-\sttm_k, \Big\lfloor K \frac{\sttm_{k+1}}{k} \Big\rfloor \right)}{\sttm_{k+1}}}\]
    in \cref{eqn:decom_qk_small_and_big}, we will first be interested in estimating 
    \begin{equation}\label{eq:goal-to-est}
        \Ppsq{\sttm_k < n- m}{\sttm_{k+1}=n},
    \end{equation}
    which will be achieved in \eqref{eq:final-est-proba}.
    For this, we mimic (in a more precise way) the proof of the lower bound on $q(k)$ in \cref{lem:q_and_sigma}. 

    For the rest of this step of the proof we fix $k \geq 2K/\eps$.
    We start by looking at the event we condition on in \eqref{eq:goal-to-est}. Let $n > k$.
    We recall from \cref{sect:positive-subtrees} that conditionally on $\revF_{n}$, the decorated tree $T_{n-1}$ is obtained by removing a uniformly chosen leaf $\uleaf{n}$ of $T_{n}$. 
    Therefore, if $X_{n}=k$, we have $\sttm_k=n$ if and only if $\uleaf{n}$ belongs to the intersection of all maximal positive subtrees of $T_{n}$, \emph{i.e.}\ almost surely, 
    \begin{align*} 
    \Ppsq{\sttm_k=n}{\revF_{n}}
    &= \frac{\# \mathcal{L}_{\cap}^{\max}(T_{n})}{n}\mathds{1}_{X_{n}=k}.  
    \end{align*}
    By distinguishing whether $T_n \in \teps$ or not, and in the case $T_n \in \teps$, using  the crude bound $0 \leq \# \mathcal{L}_{\cap}^{\max}(T_{n}) \le X_{n}$  and that $\lambda \in (0,1)$, we obtain for all $n>k$ and all $\eps \in (0,1)$, that
    \begin{equation} \label{eq:cond_sigma_k_A}
    \left| \Ppsq{\sttm_k=n}{\revF_{n}} - \lambda \frac{k}{n} \mathds{1}_{X_{n}=k} \right| 
    \le
    \frac{k}{n} \left(\eps + \mathbbm{1}_{T_{n}\in \teps}\right).
    \end{equation}

    We now turn to \eqref{eq:goal-to-est}. For all $i,n\geq 1$ such that $i \leq n-2$ and $n-i > k$, the events $\{ \sttm_{k+1}=n \}$ and $\{ \sttm_k \leq n-i \}$ belong to $\revF_{n-i}$ and have positive probability, so 
    \begin{align*}
        \Ppsq{\sttm_k<n-i}{\sttm_{k+1}=n, \sttm_k \leq n-i} 
        &= 1 - \Ecsq{\Ppsq{\sttm_k=n-i}{\revF_{n-i}}}{\sttm_{k+1}=n, \sttm_k \leq n-i}.
    \end{align*}
    In addition, note that on the event $\{\sttm_{k+1}=n, \sttm_k \leq n-i\}$, we have that $X_{n-i}=k$. Therefore, 
    using the bound in~\eqref{eq:cond_sigma_k_A} with $n-i$ in place of $n$, for all  $i, n\geq 1$ such that $i \leq n-2$ and $ n-i > k$,
    \begin{multline}
    \left|\Ppsq{\sttm_k<n-i}{\sttm_{k+1}=n, \sttm_k \leq n-i} - 1 + \lambda\frac{k}{n-i}\right| \nonumber \\
    \le \frac{k}{n-i} \Big( \eps +  \Ppsq{T_{n-i} \in \teps}{\sttm_{k+1}=n, \sttm_k \leq n-i} \Big).
    \end{multline}
    We now also assume $i \leq K \frac{n}{k}$ and define, for all  $i, n\geq 1$ such that $i \leq  K \frac{n}{k}$ and $ n-i > k$,
    \[ x_i^{k,n}\coloneqq 
    \Ppsq{\sttm_k<n-i}{\sttm_{k+1}=n, \sttm_k \leq n-i} - 1 + \lambda\frac{k}{n}.
    \] 
    Recalling that we fixed $k\geq 2K/\eps $, our last bound and the triangle inequality entail that if $1 \leq i \leq K \frac{n}{k}$ and $n-i > k$,
    \begin{align}
        |x_i^{k,n}| &\le
        \frac{k}{n-i} \left( \eps +  \Ppsq{T_{n-i} \in \teps}{\sttm_{k+1}=n, \sttm_k \leq n-i} \right)+\lambda\left|\frac{k}{n-i}-\frac{k}{n}\right|\notag\\
        &\leq
        \frac{3k}{n} \Big(\eps +  \Ppsq{T_{n-i} \in \teps}{\sttm_{k+1}=n, \sttm_k \leq n-i}\Big),\label{eqn:defn_x_i}
    \end{align}
    where for the last inequality we used, for the first term,  that $n-i\geq n(1-K/k)\geq n(1-\eps/2)\geq n/2$ since $\eps \in (0,1)$, and for the second term, that $\frac{k}{n-i}-\frac{k}{n}=\frac{k}{n}\frac{i}{n-i}\leq \frac{k}{n}\eps$ since $i\leq K \frac{n}{k}\leq \frac{\eps}{2}n$.
    By writing a telescopic product, for $j \geq 0$ and $ n-j > k$, we have 
    \begin{equation} \label{eq:telescopic_sigma_k}
    \Ppsq{\sttm_k < n-j}{\sttm_{k+1}=n} = \prod_{i=1}^{j} \left( 1-\lambda\frac{k}{n} + x_i^{k,n} \right) = \left( 1-\lambda\frac{k}{n} \right)^j \exp\left(\sum_{i=1}^{j} \log\left( 1+\frac{x_i^{k,n}}{1-\lambda k/n} \right) \right),
    \end{equation}
    where, when $j=0$, we interpret the empty product as $1$ and the empty sum as $0$.
    Therefore, for $0 \leq j \leq K \frac{n}{k}$ and $n-j>k$, we write
    \begin{equation}\label{eq:defn-y}
         y_j^{k,n} \coloneqq 
    \sum_{i=1}^{j}  \frac{|x_i^{k,n}|}{1-\lambda k/n}.
    \end{equation}
    For later convenience, we also set $y_j^{k,n} \coloneqq 0$ for all $0 \leq j \leq K \frac{n}{k}$ with $n-j \leq k $.
    
    Since $n>n-j>k$, the denominators $1-\lambda k/n$ are bounded away from $0$ (recall that $\lambda<1$). We deduce from \eqref{eqn:defn_x_i} that for $0 \leq j \leq K \frac{n}{k}$ and $n-j>k$, the sum $y_j^{k,n}$ is bounded by a constant that depends only on $K$, and thus so is the sum
    \[
    \bigg|\sum_{i=1}^{j} \log \bigg( 1+\frac{x_i^{k,n}}{1-\lambda k/n} \bigg)\bigg|
    \le
    y_j^{k,n}.
    \]
    Therefore, by the mean value inequality, we can find a constant $C_K>0$ that depends only on $K$ such that for all $0 \leq j \leq K \frac{n}{k}$ and $n-j>k$,
    \[
    \left|\exp\left(\sum_{i=1}^{j} \log\left( 1+\frac{x_i^{k,n}}{1-\lambda k/n} \right) \right)-1\right|
    \le C_K y_j^{k,n}
    \le C_K y_{\lfloor Kn/k \rfloor}^{k,n}.
    \]
    We conclude that for all $0 \leq j \leq K \frac{n}{k}$, 
    \begin{equation}\label{eq:final-est-proba}
        \Ppsq{\sttm_k < n-j}{\sttm_{k+1}=n} = \mathds{1}_{n-j > k}\left( 1-\lambda \frac{k}{n} \right)^j  \left( 1+O_K\left(y_{\lfloor Kn/k \rfloor}^{k,n}\right)\right),
    \end{equation}
    since the probability on the left-hand side is zero when $n-j \leq k$.
    The latter estimate gives us that for $k\geq 2K/\eps$, almost surely
    \begin{align*}
        \Ecsq{\min\left( \sttm_{k+1}-\sttm_k, \Big\lfloor K \frac{\sttm_{k+1}}{k} \Big\rfloor \right) }{\sttm_{k+1}}
        &=\sum_{j=0}^{\infty}\Ppsq{\min\left( \sttm_{k+1}-\sttm_k, \Big\lfloor K \frac{\sttm_{k+1}}{k} \Big\rfloor \right)>j}{\sttm_{k+1}}\\
        &=\sum_{j=0}^{\big\lfloor K \frac{\sttm_{k+1}}{k} \big\rfloor-1}
        \Ppsq{ \sttm_{k}<\sttm_{k+1}-j}{\sttm_{k+1}}
        \\
        &=\sum_{j=0}^{J_{K}-1}
       \left(\left( 1-\lambda \frac{k}{\sttm_{k+1}} \right)^j  \left( 1+O_K\left(y_{\lfloor K\sttm_{k+1}/k \rfloor}^{k,\sttm_{k+1}}\right)\right)\right),
    \end{align*}
    with $J_{K}\coloneqq \min\left(\Big\lfloor K \frac{\sttm_{k+1}}{k} \Big\rfloor,\sttm_{k+1}-k\right)$. 
    This is a geometric sum, so we end up with
    \[
    \Ecsq{\min\left( \sttm_{k+1}-\sttm_k, \Big\lfloor K \frac{\sttm_{k+1}}{k} \Big\rfloor \right) }{\sttm_{k+1}}
    =
    \frac{\sttm_{k+1}}{\lambda k} \left( 1- \left( 1-\lambda \frac{k}{\sttm_{k+1}} \right)^{J_{K}} \right) \left(1+O_K\left(y_{\lfloor K\sttm_{k+1}/k \rfloor}^{k,\sttm_{k+1}}\right)\right).
    \]
    We now simplify the exponent $J_K$ above. Notice that if $\tau_{k+1} -k < K \frac{\tau_{k+1}}{k} $, then  we must have that 
    $\tau_{k+1} < \frac{k}{1- \frac{K}{k}}$. Since $k \ge 2K/\eps$ and $\eps\in(0,1)$, this would imply $\tau_{k+1} < 2k$. Hence, if $\tau_{k+1} \ge 2k$, then $J_K = \Big\lfloor K \frac{\sttm_{k+1}}{k} \Big\rfloor$. Besides, if $J_K = \Big\lfloor K \frac{\sttm_{k+1}}{k} \Big\rfloor$, then 
    $\left( 1-\lambda \frac{k}{\sttm_{k+1}} \right)^{J_{K}} = O(\mathrm{e}^{-\lambda K}).$
    Therefore, we can combine these two observations into
    \begin{multline*}
        \Ecsq{\min\left( \sttm_{k+1}-\sttm_k, \Big\lfloor K \frac{\sttm_{k+1}}{k} \Big\rfloor \right)}{\sttm_{k+1}}  \mathds{1}_{\sttm_{k+1} \ge 2k}\\ 
        = \frac{\sttm_{k+1}}{\lambda k} \left( 1-O \left( \mathrm{e}^{-\lambda K} \right)\right)
        \left(1+O_K\left(y_{\lfloor K\sttm_{k+1}/k \rfloor}^{k,\sttm_{k+1}}\right)\right) \mathds{1}_{\sttm_{k+1} \ge 2k}.      
    \end{multline*}
    Integrating against $r(\tau_{k+1})$, we conclude that there is an absolute constant $C>0$ (that may change from one inequality to the next) and a constant $C_K>0$ depending on $K$ such that for all $ k\geq 2K/\eps$, 
    \begin{multline} \label{eq:bound_2nd_term_indicator}
    \left|\Ec{ r(\sttm_{k+1}) \min\left( \sttm_{k+1}-\sttm_k, \Big\lfloor K \frac{\sttm_{k+1}}{k} \Big\rfloor \right) \mathds{1}_{\sttm_{k+1} \ge 2k}} - \Ec{ r(\sttm_{k+1}) \frac{\sttm_{k+1}}{\lambda k} \mathds{1}_{\sttm_{k+1} \ge 2k}} \right| \\
    \le
    C \mathrm{e}^{-\lambda K} \Ec{ r(\sttm_{k+1}) \frac{\sttm_{k+1}}{\lambda k} \mathds{1}_{\sttm_{k+1} \ge 2k}}
    +
    C_K \Ec{ r(\sttm_{k+1}) \frac{\sttm_{k+1}}{\lambda k} y_{\lfloor K\sttm_{k+1}/k \rfloor}^{k,\sttm_{k+1}} \mathds{1}_{\sttm_{k+1} \ge 2k}} \\
    \stackrel{\eqref{eq:asympt r(m)}}{\le}
    C \frac{\mathrm{e}^{-\lambda K}}{\lambda k} \Ec{ \sttm_{k+1}^{-1/2}}
    +
    C_K \Ec{ r(\sttm_{k+1}) \frac{\sttm_{k+1}}{\lambda k} y_{\lfloor K\sttm_{k+1}/k \rfloor}^{k,\sttm_{k+1}} \mathds{1}_{\sttm_{k+1} \ge 2k}} \\
    \stackrel{\eqref{eqn:q_sigma_lower_and_upper}}{\le}
    C \mathrm{e}^{-\lambda K}q(k)
    +
    C_K \Ec{ r(\sttm_{k+1}) \frac{\sttm_{k+1}}{\lambda k} y_{\lfloor K\sttm_{k+1}/k \rfloor}^{k,\sttm_{k+1}} \mathds{1}_{\sttm_{k+1} \ge 2k}}.
    \end{multline}
    Although the indicator $\mathds{1}_{\sttm_{k+1} \geq 2k}$ is $1$ with high probability as $k \to \infty$, we decide to keep it, as it will be useful in the next step of the proof.
    
    We now estimate the first term on the right-hand side of~\eqref{eqn:decom_qk_small_and_big} using \eqref{eq:bound_2nd_term_indicator}: there is an absolute constant $C>0$ and a (possibly different) constant $C_K>0$ depending on $K$ such that for all $ k\geq 2K/\eps$,
    \begin{multline} \label{eq:2nd_term_error_interm}
    \left|\Ec{ r(\sttm_{k+1}) \min\left( \sttm_{k+1}-\sttm_k, \Big\lfloor K \frac{\sttm_{k+1}}{k} \Big\rfloor \right) } - \Ec{ r(\sttm_{k+1}) \frac{\sttm_{k+1}}{\lambda k} } \right| \\
    \le
    C \mathrm{e}^{-\lambda K}q(k)
    +
    C_K \Ec{ r(\sttm_{k+1}) \frac{\sttm_{k+1}}{\lambda k} y_{\lfloor K\sttm_{k+1}/k \rfloor}^{k,\sttm_{k+1}}\mathds{1}_{\sttm_{k+1} \ge 2k}} 
    + C_K\Ec{ r(\sttm_{k+1})  \frac{\sttm_{k+1}}{k} \mathds{1}_{\sttm_{k+1} < 2k}}.
    \end{multline}
    Finally, we may further bound the last term as follows:
    \[
    \Ec{ r(\sttm_{k+1})  \frac{\sttm_{k+1}}{k} \mathds{1}_{\sttm_{k+1} < 2k}}
    \stackrel{\eqref{eq:asympt r(m)}}{\le}
    C \Ec{ \frac{\sttm_{k+1}^{-1/2}}{k } \mathds{1}_{\sttm_{k+1} < 2k}}
    \le
    C \Ec{ \frac{\sttm_{k+1}^{-1/2}}{k } \frac{2k}{\sttm_{k+1}}}
    \le
    2C \Ec{ \sttm_{k+1}^{-3/2}}.
    \]
    By \eqref{eq:tail_sigma-2} and \eqref{eqn:q_sigma_lower_and_upper}, we deduce that the last term of \eqref{eq:2nd_term_error_interm} is $o_K(q(k))$ as $k \to\infty$. We have reached the following bound: for all $k \geq 2K/\eps$, 
    \begin{multline} \label{eq:2nd_term_error}
    \left|\Ec{ r(\sttm_{k+1}) \min\left( \sttm_{k+1}-\sttm_k, \Big\lfloor K \frac{\sttm_{k+1}}{k} \Big\rfloor \right) } - \Ec{ r(\sttm_{k+1}) \frac{\sttm_{k+1}}{\lambda k} } \right| \\
    \le
    C \mathrm{e}^{-\lambda K}q(k)
    + o_K(q(k)) +
    C_K \Ec{ r(\sttm_{k+1}) \frac{\sttm_{k+1}}{\lambda k} y_{\lfloor K\sttm_{k+1}/k \rfloor}^{k,\sttm_{k+1}}\mathds{1}_{\sttm_{k+1} \ge 2k}}.
    \end{multline}
    The next step will be to bound the final error term in the above expression.

\medskip

\noindent\emph{\underline{Step 4}: Estimating the final error term in \eqref{eq:2nd_term_error}.}
Most of the remaining work is to show that the last term above is small compared to $q(k)$. By definition of $y^{k,n}_j$ in \eqref{eq:defn-y} and the bound \eqref{eqn:defn_x_i}, this ultimately boils down to bounding the probability that $T_n\in \teps$ for various values of $n$. Therefore, we will rewrite this error term in such a way that we can apply~\eqref{eq:A_sum_sigma_k}.

We start by fixing $k \geq  2K/\eps$.
By definition of $y^{k,n}_j$ in \eqref{eq:defn-y}, recalling in particular that $y_j^{k,n} =0$ for all $0 \leq j \leq K \frac{n}{k}$ and $n-j \leq k$, the last term of~\eqref{eq:2nd_term_error} can be written 
\begin{align*}
\Ec{ r(\sttm_{k+1}) \frac{\sttm_{k+1}}{\lambda k} y_{\lfloor K\sttm_{k+1}/k \rfloor}^{k,\sttm_{k+1}} \mathds{1}_{\sttm_{k+1} \ge 2k}}  
&= \sum_{n=2k}^\infty \Pp{\sttm_{k+1}=n}r(n) \frac{n}{\lambda k} \frac{1}{1-\lambda k/n} \mathds{1}_{n-\lfloor Kn/k \rfloor > k}\sum_{i=1}^{\lfloor Kn/k \rfloor} |x_i^{k,n}| \\
&\leq \frac{1}{\lambda(1-\lambda)}\sum_{n=2k}^\infty \Pp{\sttm_{k+1}=n}r(n) \frac{n}{k} \mathds{1}_{n-\lfloor Kn/k \rfloor > k} \sum_{i=1}^{\lfloor Kn/k \rfloor} |x_i^{k,n}|.
\end{align*}
First, we note that the above indicator is always $1$. Indeed, using $K/k\leq \eps/2<1/2$ and $n\geq 2k$ we get $n-\lfloor Kn/k \rfloor > k$.
Now, using the upper bound \eqref{eqn:defn_x_i}, we get that 
\begin{multline}\label{eq:error_sum_x_i_2}
    \Ec{ r(\sttm_{k+1}) \frac{\sttm_{k+1}}{\lambda k} y_{\lfloor K\sttm_{k+1}/k \rfloor}^{k,\sttm_{k+1}} \mathds{1}_{\sttm_{k+1} \ge 2k}} 
    \leq\frac{3\eps K}{\lambda(1-\lambda)}
\sum_{n=2k}^\infty \Pp{\sttm_{k+1}=n}r(n) \frac{n}{k}\\
+
\frac{3}{\lambda(1-\lambda)}
\sum_{n=2k}^\infty \Pp{\sttm_{k+1}=n}r(n) \sum_{i=1}^{\lfloor Kn/k \rfloor} \Ppsq{T_{n-i} \in \teps}{\sttm_{k+1}=n, \sttm_k \leq n-i}.
\end{multline}
We now take care of the two sums separately.
For the first sum we have that
\begin{equation}\label{eq:error_1st_term_22}
\sum_{n=2k}^\infty \Pp{\sttm_{k+1}=n}r(n) \frac{n}{k} 
= \Ec{r(\sttm_{k+1}) \frac{\sttm_{k+1}}{k} \mathds{1}_{\sttm_{k+1}\geq 2k}} \stackrel{\eqref{eq:asympt r(m)}}{\le} 
\,\,
\frac{C}{k}\Ec{\sttm_{k+1}^{-1/2}}
\stackrel{\eqref{eqn:q_sigma_lower_and_upper}}{\le} 
C q(k)
\end{equation}
for some absolute constant $C>0$ that may change from one inequality to the next.

We now turn to the second sum in \eqref{eq:error_sum_x_i_2}.
We write it as
\begin{multline}
 \sum_{n=2k}^\infty \Pp{\sttm_{k+1}=n}r(n)  \sum_{i=1}^{\lfloor Kn/k \rfloor}  \Ppsq{T_{n-i} \in \teps}{\sttm_{k+1}=n, \sttm_k \leq n-i} \\ 
=\sum_{n=2k}^\infty r(n) \sum_{i=1}^{\lfloor Kn/k \rfloor} \frac{\Pp{T_{n-i} \in \teps, \sttm_k \leq n-i,\sttm_{k+1}=n}}{\Ppsq{\sttm_k \leq n-i}{\sttm_{k+1}=n}}. \label{eqn:decomposition_error_term_A_2}
\end{multline}
Let us now bound the denominators from below when $n \geq 2k$ 
and $1 \leq i \leq \lfloor Kn/k \rfloor$, recalling that we fixed $k \geq  2K/\eps$. These bounds guarantee that $\frac{k}{n-i} \leq \frac{1}{1-\tfrac{\eps}{2}} \frac{k}{n} < 1$. Therefore, we can use the crude bound~\eqref{eqn:reverse_Remy_geometric_comparison} to obtain
\[ \Ppsq{\sttm_k \leq n-i}{\sttm_{k+1}=n} \geq \left( 1-\frac{k}{n-i} \right)^i \geq \left( 1-\frac{1}{1-\tfrac{\eps}{2}} \frac{k}{n}\right)^i. \]
Using $\log(1-x)\geq -\frac{x}{1-x}$ for $0<x<1$, we obtain
\begin{align*}
    \Ppsq{\sttm_k \leq n-i}{\sttm_{k+1}=n} 
    \geq
    \exp\left(-\frac{i\tfrac{k}{n}}{(1-\tfrac{\eps}{2})-\tfrac{k}{n}}\right)
    \geq 
    \exp\left(-\frac{2K}{1-\eps}\right),
\end{align*}
where, in the last inequality, we used that $ i \leq \lfloor Kn/k \rfloor$ and $n\geq 2k$. 
Then, setting $C_{K,\eps} \coloneqq \exp\left(\frac{2K}{1-\eps}\right)$, \eqref{eqn:decomposition_error_term_A_2} yields 
\begin{align*}
& \sum_{n=2k}^\infty \Pp{\sttm_{k+1}=n}r(n) \sum_{i=1}^{\lfloor Kn/k \rfloor}  \Ppsq{T_{n-i} \in \teps}{\sttm_{k+1}=n, \sttm_k \leq n-i} \nonumber \\
&\leq  C_{K,\eps} \cdot \sum_{n=2k}^\infty r(n) \sum_{i=1}^{\lfloor Kn/k \rfloor} \Pp{T_{n-i} \in \teps, \sttm_k \leq n-i,\sttm_{k+1}=n}\\
&\leq C_{K,\eps} \cdot \Ec{\sum_{n=k+1}^\infty \indicator{\sttm_{k+1}=n} \, r(n) \sum_{i=1}^{\lfloor Kn/k \rfloor} \indicator{\sttm_k \leq n-i} \indicator{T_{n-i} \in \teps}}\\
&= C_{K,\eps} \cdot \Ec{ r(\sttm_{k+1}) \sum_{i=1}^{\lfloor K \sttm_{k+1}/k \rfloor} \indicator{\sttm_k \leq \sttm_{k+1}-i} \indicator{T_{\sttm_{k+1}-i} \in \teps}}\\
&\leq  C_{K,\eps}  \cdot \Ec{r(\sttm_{k+1})\sum_{j=\sttm_k}^{\sttm_{k+1}-1}  \indicator{T_{j} \in \teps}}.
\end{align*}
Since $r$ is decreasing by \cref{lem:prelim_GW}, the last expectation is $o_{\eps} \left( q(k) \right)$ as $k \to \infty$ by~\eqref{eq:A_sum_sigma_k}.
Hence, the second sum in \eqref{eq:error_sum_x_i_2} is $o_{K,\eps}(q(k))$ as $k \to \infty$. 
Since the first sum in \eqref{eq:error_sum_x_i_2} was bounded by $C q(k)$ 
in \eqref{eq:error_1st_term_22}, we conclude (taking into account the additional constants in \eqref{eq:error_sum_x_i_2}) that
\begin{equation} \label{eq:concl_step4}
\Ec{ r(\sttm_{k+1}) \frac{\sttm_{k+1}}{\lambda k} y_{\lfloor K\sttm_{k+1}/k \rfloor}^{k,\sttm_{k+1}} \mathds{1}_{\sttm_{k+1} \geq 2k} }
\leq
\eps C_K  q(k) + o_{K,\eps}(q(k)) 
\quad \text{as } k\to\infty.
\end{equation}

\medskip

\noindent\emph{\underline{Step 5}: Conclusion.}
Let us come back to the formula~\eqref{eqn:decom_qk_small_and_big} for $q(k)$. Recall the parameters $K\geq 1$ and $\eps\in (0,1)$ that we have introduced before. We bound the second term on the right-hand side of \eqref{eqn:decom_qk_small_and_big} using~\eqref{eqn:q_accurate_bound_error_term}, and estimate the first one with \eqref{eq:2nd_term_error} (with its last error term bounded by~\eqref{eq:concl_step4}). We obtain for $k$ large enough (where ``large enough'' depends on $K$ and $\eps$):
\begin{equation}
\label{eq:q(k)_error_E(eps,K,k)}
q(k) = \left( 1+O_K \left( \frac{1}{k}\right) \right) \Ec{ r(\sttm_{k+1}) \frac{\sttm_{k+1}}{\lambda k} } + \left( 1+O_K \left( \frac{1}{k}\right) \right) E(\eps,K,k),
\end{equation}
where the error term $E(\eps,K,k)$ is bounded by
\begin{equation} 
\label{eq:E(eps,K,k)_bound}
|E(\eps,K,k)|
\le
(C \mathrm{e}^{-CK} 
+ \eps C_K)   q(k)
+ o_{K,\eps}(q(k))
\end{equation}
for some absolute constant $C>0$ and some constant $C_K>0$ depending on $K$. 
We stress that the estimate \eqref{eq:q(k)_error_E(eps,K,k)} and the bound \eqref{eq:E(eps,K,k)_bound} hold for all $\eps\in (0,1)$ and $K\ge 1$.
We now tune the parameters $\eps$ and $K$. Let $\delta>0$. There exist $K=K(\delta)$ (large enough) and $\eps=\eps(\delta)\in(0,1)$ (small enough) such that the first term on the right-hand side of \eqref{eq:E(eps,K,k)_bound} is bounded by $\delta q(k)$. The values of $K$ and $\eps$ are now fixed (and only depend on $\delta$). These values being fixed, we can now take $k_0=k_0(\delta)$ large enough so that by  \eqref{eq:E(eps,K,k)_bound}, for all $k\geq k_0$,
\[
|E(\eps,K,k)|
\le
2 \delta q(k).
\]
This estimate holds for all $\delta>0$. Together with \eqref{eq:q(k)_error_E(eps,K,k)}, this implies that
\[
q(k) = \Ec{ r(\sttm_{k+1}) \frac{\sttm_{k+1}}{\lambda k} } + o(q(k))
\quad \text{as } k\to\infty,
\]
which means that $q(k) \sim \Ec{ r(\sttm_{k+1}) \frac{\sttm_{k+1}}{\lambda k} }$ as $k\to\infty$. 
We conclude, using once more \cref{lem:prelim_GW} and \eqref{eq:k-k1-equival}, that 
\[
q(k) \sim \frac{1}{2\sqrt{\pi}\lambda k} \Ec{\sttm_k^{-1/2}} 
\quad \text{as } k\to\infty,
\]
which establishes \cref{lem:equiv_q(k)}.
\end{proof}

\begin{remark}\label{rmk:lambda=alpha}
    From now on, we know that $\lambda=\alpha$ in~\cref{thm:convergence_size_intersection}, as identified in the proof of~\cref{thm:good_regularity} on page~\pageref{eq:cons-alpha-lambda}; see in particular the discussion below~\eqref{eq:cons-alpha-lambda}.
\end{remark}

 \section{The value of the exponent $\alpha$ from the good regularity estimates for $q$ and $Q$}
 \label{sec:value exponent}

The goal of this section is to complete the proof of~\cref{thm:main2}, determining the exact value of the exponent $\alpha$. We stress that so far $\alpha$ is only defined through \cref{thm:main}, but we also know that the good regularity estimates in~\cref{thm:good_regularity} hold for \emph{the same} exponent $\alpha$.

The latter estimates ensure that there are two slowly varying functions $\svfq, \svfbQ$ and an exponent 
\begin{equation}\label{eq:gamma-alpha-rel7}
\gamma \stackrel{\eqref{eq:gamma-alpha-rel}}{=}  \frac{1}{2\alpha}+1 \in \Big(\frac{3}{2},2\Big)
\end{equation}
such that
\begin{align}\label{eq:nec-estimates}
	q(k)=k^{-\gamma}\svfq(k) \qquad \text{ and } \qquad 
    Q(k)=k^{1-\gamma} \svfbQ(k).
\end{align}
We will use these estimates to deduce the relation of \cref{thm:main2} between the values of $\alpha$ and $p$.

\subsection{The exact value of the exponent $\alpha$}

To identify the exact value of $\alpha$, we first rewrite the recursive relation in \cref{lem:q_rel} as
\begin{align}\label{eq:equation q}
q(k) = \frac{1}{1+\frac{1-p}{p}Q(k)} \cdot \left(\sum_{i=1}^{(k-1)/2} q(i) q(k-i) + \frac{1-p}{2 p } q(k)^2 +\mathds{1}_{k\text{ even}}\frac 1 2 q(k/2)^2\right).
\end{align}
Because we are not able to tackle this equation on the sequence $q$ directly, we take interest in some related \emph{inequations} that can be satisfied for a test sequence $\hat{q}$, namely
\begin{align}\label{eq:inequation hat q geq}
	\hat{q}(k) \geq  \frac{1}{1+\frac{1-p}{p}Q(k)} \cdot \left(\sum_{i=1}^{(k-1)/2} q(i) \hat{q}(k-i) + \frac{1-p}{2 p } q(k)^2 +\mathds{1}_{k\text{ even}}\frac 1 2 q(k/2)^2\right),
\end{align}
 and 
 \begin{align}\label{eq:inequation hat q leq}
 	\hat{q}(k) \leq  \frac{1}{1+\frac{1-p}{p}Q(k)} \cdot \left(\sum_{i=1}^{(k-1)/2} q(i) \hat{q}(k-i)\right).
 \end{align}
We will deduce \cref{thm:main2} from the following two key ingredients:
 \begin{enumerate}
 \item A comparison lemma for test sequences (\cref{lem:comparison q hat q}) which ensures that if we can check that one of the inequalities \eqref{eq:inequation hat q geq} or \eqref{eq:inequation hat q leq} holds for $k$ large enough for some sequence $\hat{q}$, then we can compare the values of $q$ and $\hat{q}$. This ingredient is partially inspired by~\cite{auffinger2017minplus}.
 \item The good regularity estimates in~\cref{thm:good_regularity} (see~\eqref{eq:nec-estimates} for a closer reference) and their consequences in~\cref{lem:asymptotics q hat q} below.
 \end{enumerate} 

We start by stating our main tool to compare the values of $q$ and $\hat{q}$.

\begin{lem}[Comparison lemma for test sequences]\label{lem:comparison q hat q}
Let $(\hat{q}(k))_{k\geq 1}$ be a sequence of positive numbers.
\begin{enumerate}
	\item\label{it:lem:comparison q hat q:geq} If $\hat{q}(k)$ satisfies \eqref{eq:inequation hat q geq} for $k$ large enough, then there exists a constant $C>0$ such that
	\begin{align*}
		 \forall k\geq 1, \qquad  q(k)\leq C \hat{q}(k).
	\end{align*} 
	\item\label{it:lem:comparison q hat q:leq} If $\hat{q}(k)$ satisfies \eqref{eq:inequation hat q leq} for $k$ large enough, then there exists a constant $c>0$ such that 
	\begin{align*}
          \forall  k\geq 1, \qquad c \hat{q}(k) \leq  q(k).
	\end{align*} 
\end{enumerate}
\end{lem}

We proceed by stating the consequences of the good regularity estimates that we need here. Recall the exponent $\gamma\in (3/2,2)$ related to $\alpha$ through \eqref{eq:gamma-alpha-rel7}, and the slowly varying function $\svfbQ$ appearing in~\eqref{eq:nec-estimates}.

\begin{lem}[Asymptotic expansion lemma for test sequences]\label{lem:asymptotics q hat q}
	For $\hat{q}(k)=k^{-\beta}$ with $\beta \in \intervalleoo{\frac32}{2}$, we have as $k\to\infty$,
	\begin{itemize}
		\item \(\displaystyle
			\left( \sum_{i=1}^{(k-1)/2} q(i) \hat q(k-i) \right) - \hat q(k) = k^{1-\gamma - \beta} \svfbQ(k) \left(-2^{\gamma + \beta -1} +  \beta \int_{0}^{1/2}t^{1-\gamma}(1-t)^{-1-\beta} \, \mathrm{d}t + o(1) \right). 
		\)
		\item \(\displaystyle
			\frac{1-p}{p} Q(k) \hat{q}(k) \sim  \frac{1-p}{p} k^{1-\gamma - \beta} \svfbQ(k).
		\)
		\item \(\displaystyle
			\frac{1-p}{2 p } q(k)^2 +\mathds{1}_{k \,\mathrm{even}}\frac 1 2 q(k/2)^2 = o(k^{1-\gamma - \beta} \svfbQ(k)).
		\)
	\end{itemize}
\end{lem}

We will also need the following simple technical fact.

\begin{lem}[Monotonicity]\label{lem:monotonicity}
 	The function $f:\gamma \mapsto \frac{4^{\gamma-1}\sqrt{\pi}\,\Gamma(2-\gamma)}{\Gamma(\frac32-\gamma)}$ is (strictly) decreasing on $(\frac{3}{2},2)$. Moreover, for any $x\in(3/2,2)$, we have $\frac{4^{x-1}\sqrt{\pi}\,\Gamma(2-x)}{\Gamma(\frac32-x)}=2 ^{2x-1} 
	- x  \cdot  \int_{0}^{1/2}t^{1-x}(1-t)^{-1-x} \, \mathrm{d}t$.
\end{lem}

We now prove \cref{thm:main2} from Lemmas~\ref{lem:comparison q hat q}, \ref{lem:asymptotics q hat q}~and~\ref{lem:monotonicity}.

\begin{proof}[Proof of \cref{thm:main2} (assuming Lemmas~\ref{lem:comparison q hat q}, \ref{lem:asymptotics q hat q}~and~\ref{lem:monotonicity})]
	For $\beta,\nu\in \intervalleoo{\frac{3}{2}}{2}$ we write 
	\begin{align*}
		g(\beta,\nu)= -2^{\nu + \beta -1} +  \beta \int_{0}^{1/2}t^{1-\nu}(1-t)^{-1-\beta} \, \mathrm{d}t.
		\end{align*}
  Note that $g(\beta,\beta)=-f(\beta)$ for $f$ defined in \cref{lem:monotonicity}. 
  Recall the exponent $\gamma\in (3/2,2)$ in \eqref{eq:gamma-alpha-rel7} and suppose, for the sake of contradiction, that $g(\gamma,\gamma)> \frac{1-p}{p}$. 
	Then by continuity, we also have $g(\beta,\gamma)>\frac{1-p}{p}$ for some $\beta \in (3/2, \gamma)$.
	By the asymptotics of \cref{lem:asymptotics q hat q}, taking $\hat{q}(k) = k^{-\beta}$ with this choice of $\beta$ we then have, for $k$ large enough,
	\begin{align*}
		 \frac{1-p}{p} Q(k) \hat{q}(k) \leq \left( \sum_{i=1}^{(k-1)/2} q(i) \hat q(k-i) \right) - \hat q(k), 
	\end{align*}
so that \eqref{eq:inequation hat q leq} is satisfied.

Thanks to \cref{it:lem:comparison q hat q:leq} in \cref{lem:comparison q hat q}, there exists $c>0$ such that $c k^{-\beta} = c \hat{q}(k) \leq q(k)$ for all $k\geq 1$. Yet by~\cref{thm:good_regularity}, $q(k)=k^{-\gamma}\svfq(k)$. 
 This is absurd because $\beta < \gamma$. 

Symmetrically, if $g(\gamma,\gamma)=-f(\gamma)<\frac{1-p}{p}$, then by continuity, we also have $g(\beta,\gamma)<\frac{1-p}{p}$ for some $\beta > \gamma$. 
The same argument as above shows that \eqref{eq:inequation hat q geq} is satisfied for $k$ large enough if $\hat{q}(k)=k^{-\beta}$, 
so that by \cref{lem:comparison q hat q}, we have $C k^{-\beta}= C \hat{q}(k) \geq  q(k)$ for all $k\geq 1$, for some constant $C>1$.
Because $\beta>\gamma$, this again contradicts the behavior of $q(k)=k^{-\gamma}\svfq(k)$ established in~\cref{thm:good_regularity}.

In the end, we proved that necessarily $g(\gamma,\gamma) =- f(\gamma)= \frac{1-p}{p}$. 
Since by \cref{lem:monotonicity} the function $f$ is strictly decreasing on $\intervalleoo{\frac32}{2}$, this ensures that this equality determines $\gamma$ uniquely as we wanted. Substituting $\alpha\stackrel{\eqref{eq:gamma-alpha-rel}}{=}\frac{1}{2(\gamma-1)}$, we get that the equivalent equation $1/f(\gamma)= \frac{p}{p-1}$ is exactly the equation  in~\eqref{eq:eq-to-defn-alpha} for $\alpha$.
\end{proof}

In the following subsections, we give the missing proofs of Lemmas~\ref{lem:comparison q hat q}, \ref{lem:asymptotics q hat q}~and~\ref{lem:monotonicity}.

\subsection{The comparison lemma for test sequences}

We give here the missing proof of \cref{lem:comparison q hat q}.

\begin{proof}[Proof of \cref{lem:comparison q hat q}]
	We start with the proof of \cref{it:lem:comparison q hat q:geq}. 
	Assume that a sequence of positive numbers $(\hat{q}(k))_{k\geq 1}$ satisfies \eqref{eq:inequation hat q geq} for $k$ larger than some value $k_0\geq 1$.
	Let $C>1$ be a constant chosen large enough so that 
	\begin{align}\label{eq:inequality geq to prove}
		C\hat{q}(k) \geq q(k)
	\end{align}
	 for all $1\leq k \leq k_0$, which exists because we assumed $\hat{q}$ to be positive. 
	We now prove by induction that this inequality also holds for all values of $k\geq 1$.
	Let $k\geq k_0+1$, and assume the bound $C \hat{q}(j) \geq q(j)$ holds for all $1\le j\leq k-1$. 
	Using sequentially the inequality \eqref{eq:inequation hat q geq} that holds for any $k> k_0$, the induction hypothesis, and the recursion equation \eqref{eq:equation q} for $q$, we get
	\begin{align*} 
	C\hat{q}(k) 
	 &\geq	\frac{1}{1+\frac{1-p}{p}Q(k)} \cdot \left(\sum_{i=1}^{(k-1)/2} q(i) \cdot C\hat{q}(k-i) + C \frac{1-p}{2 p } q(k)^2 + C\mathds{1}_{k\text{ even}}\frac 1 2 q(k/2)^2\right)\\
	&\geq 	\frac{1}{1+\frac{1-p}{p}Q(k)} \cdot \left(\sum_{i=1}^{(k-1)/2} q(i) q(k-i) + \frac{1-p}{2 p } q(k)^2 +\mathds{1}_{k\text{ even}}\frac 1 2 q(k/2)^2\right)= q(k).
	\end{align*}
Hence, the induction step is proved, so \eqref{eq:inequality geq to prove} holds for all $k\geq 1$, and \cref{it:lem:comparison q hat q:geq} is proved. 

\medskip

The proof of \cref{it:lem:comparison q hat q:leq} follows the exact same steps: Assume that a sequence of positive numbers $(\hat{q}(k))_{k\geq 1}$ satisfies \eqref{eq:inequation hat q leq} for $k$ larger than some value $k_0\geq 1$.
Let $c>0$ be a constant chosen small enough so that $c\hat{q}(k) \leq q(k)$ for all $1\leq k \leq k_0$, which exists because $q$ is positive.
We now prove by induction that this inequality in fact holds for all values of $k\geq 1$.
Let $k> k_0$, and assume the bound $c \hat{q}(j) \le q(j)$ holds for all $1\le j\leq k-1$. 
Using sequentially the inequality \eqref{eq:inequation hat q leq} that holds for any $k\geq k_0$, the induction hypothesis, and the recursion equation \eqref{eq:equation q} for $q$, we get
\begin{align*} 
	c\hat{q}(k) &\leq	 \frac{1}{1+\frac{1-p}{p}Q(k)} \cdot \left(\sum_{i=1}^{(k-1)/2} q(i) \cdot c 
      \hat{q}(k-i) \right)\\
	&\leq	\frac{1}{1+\frac{1-p}{p}Q(k)} \cdot \left(\sum_{i=1}^{(k-1)/2} q(i) \cdot c\hat{q}(k-i) + \frac{1-p}{2 p } q(k)^2 + \mathds{1}_{k\text{ even}}\frac 1 2 q(k/2)^2\right)\\
	&\leq 	\frac{1}{1+\frac{1-p}{p}Q(k)} \cdot \left(\sum_{i=1}^{(k-1)/2} q(i) q(k-i) + \frac{1-p}{2 p } q(k)^2 +\mathds{1}_{k\text{ even}}\frac 1 2 q(k/2)^2\right)= q(k).
\end{align*}
This proves \cref{it:lem:comparison q hat q:leq}.
\end{proof}

\subsection{The asymptotic expansion lemma for test sequences}\label{sect:proof-expansion-lemma}

We give here the missing proof of \cref{lem:asymptotics q hat q}.

\begin{proof}[Proof of \cref{lem:asymptotics q hat q}]
The second estimate is in fact an equality. 
The third one is immediate from the equality $q(k)^2=k^{-2\gamma}\svfq(k)^2$ shown in \cref{thm:good_regularity}, and the fact that $\gamma>\beta-1$ (since $\beta, \gamma \in (3/2,2)$). Therefore, the rest of the proof is devoted to the proof of the first estimate.

By Abel summation, the sum in the first estimate is
\begin{align*}
	&\sum_{i=1}^{(k-1)/2} q(i) \hat{q}(k-i) \\
	&=
	\sum_{i=1}^{(k-1)/2} Q(i) \hat{q}(k-i) - \sum_{i=2}^{(k+1)/2} Q(i) \hat{q}(k-i+1) \\
	&=
	Q(1) \hat{q}(k-1)
	-Q\Big(\Big\lfloor \frac{k+1}{2}\Big\rfloor\Big) \hat{q}\Big(\Big\lceil \frac{k+1}{2}\Big\rceil\Big)  + \sum_{i=2}^{(k-1)/2} Q(i) (\hat{q}(k-i)-\hat{q}(k-i+1)).
\end{align*}
Since $\hat{q}(k) = k^{-\beta}$, we have
\begin{equation*}
	\hat{q}(k-i)-\hat{q}(k-i+1) = \beta(k-i)^{-1-\beta}+O((k-i)^{-2-\beta}).
\end{equation*}
Using this estimate, together with the fact that $Q(1)=1$ and $Q \left( \lfloor \frac{k+1}{2} \rfloor \right) \sim  \left( \frac{k}{2} \right)^{1-\gamma} \svfbQ(k)$ thanks to \cref{thm:good_regularity},
we can rewrite the above sum as
\begin{multline}\label{eq:estim}
	\sum_{i=1}^{(k-1)/2} q(i) \hat{q}(k-i)
	= k^{-\beta}(1+O(k^{-1}))
	-
	\svfbQ(k)   \left(\frac{k}{2}\right)^{1-\gamma - \beta}(1+o(1))\\
	+ \beta  \svfbQ(k) k^{1-\gamma-\beta} \left(\frac{1}{k}\sum_{i=2}^{(k-1)/2} \frac{\svfbQ(i)}{\svfbQ(k)} \left(\frac{i}{k}\right)^{1-\gamma}  \left(1-\frac{i}{k}\right)^{-1-\beta} \right)+O(k^{-1-\beta}),
\end{multline}
where for the last error term we used that 
\[\sum_{i=2}^{(k-1)/2} Q(i) (k-i)^{-2-\beta} \leq \sum_{i=2}^{(k-1)/2} (k-i)^{-2-\beta}=O(k^{-1-\beta}).\]
We now focus on the sum on the right-hand side of 
\eqref{eq:estim}. 
We split it into two parts: one for small values of $i$ and one for large values of $i$. Fix $\eps>0$. For the values of $i\geq \eps k$, by the uniform convergence theorem recalled in \cref{item:slow_variation1} of Lemma~\ref{lem:slow_variation}, we have that $\frac{\svfbQ(i)}{\svfbQ(k)}=1+o(1)$ as $k\to\infty$, uniformly for all $i\in \intervalleentier{\eps k}{(k-1)/2}$. Hence, by Riemann sum approximation,
\begin{align*}
	\frac{1}{k}\sum_{i=\lfloor \eps k\rfloor }^{(k-1)/2} \frac{\svfbQ(i)}{\svfbQ(k)} \left(\frac{i}{k}\right)^{1-\gamma}  \left(1-\frac{i}{k}\right)^{-1-\beta} 
	&=(1+o(1))\frac{1}{k}\sum_{i=\lfloor \eps k\rfloor }^{(k-1)/2}  \left(\frac{i}{k}\right)^{1-\gamma}  \left(1-\frac{i}{k}\right)^{-1-\beta} \\
	&=(1+o(1))\int_{\eps}^{1/2}t^{1-\gamma}(1-t)^{-1-\beta} \, \mathrm{d}t.
\end{align*}
We now focus on the values of $i < \eps k$. Let $\delta>0$ be a small parameter to be fixed later.
Using the Potter bounds recalled in \cref{item:slow_variation2} of Lemma~\ref{lem:slow_variation}, there exists a constant $C>1$ such that 
\begin{align*}
    \frac{\svfbQ(i)}{\svfbQ(k)} \leq C\left(\frac{i}{k}\right)^{-\delta} \qquad\text{for all } i, k \geq 1 \text{ such that $i < \eps k$}.
\end{align*}
Therefore, again by Riemann sum approximation,
\begin{align*}
	\frac{1}{k}\sum_{i=2}^{\lfloor \eps k\rfloor-1} \frac{\svfbQ(i)}{\svfbQ(k)} \left(\frac{i}{k}\right)^{1-\gamma}  \left(1-\frac{i}{k}\right)^{-1-\beta} 
	&\leq C \frac{1}{k}\sum_{i=2}^{\lfloor \eps k\rfloor-1} \left(\frac{i}{k}\right)^{1-\gamma-\delta}  \left(1-\frac{i}{k}\right)^{-1-\beta} \\
	&= C (1+o(1)) 
	\int_0^{\eps} t^{1-\gamma-\delta}(1-t)^{-1-\beta}\, \mathrm{d}t.
\end{align*}
Since $\gamma<2$, we can always choose $\delta>0$ small enough so that $1-\gamma-\delta>-1$, ensuring that the integral above is finite.

Putting together the estimates for small and large values of $i$ and sending $\eps$ to zero, we get 
\begin{align*}
	\frac{1}{k}\sum_{i=2}^{(k-1)/2} \frac{\svfbQ(i)}{\svfbQ(k)} \left(\frac{i}{k}\right)^{1-\gamma}  \left(1-\frac{i}{k}\right)^{-1-\beta} 
	=
	(1+o(1))\int_{0}^{1/2}t^{1-\gamma}(1-t)^{-1-\beta} \, \mathrm{d}t.
\end{align*}
Substituting this last estimate in \eqref{eq:estim}, and noting that 
$O(k^{-1-\beta})=o(k^{1-\gamma-\beta} \svfbQ(k))$ since $\gamma<2$ and $k^{\gamma-2}=o(\svfbQ(k))$, we obtain the first estimate of \cref{lem:asymptotics q hat q}.
\end{proof}

\subsection{The monotonicity lemma}

We finally give the missing proof of \cref{lem:monotonicity}.

 \begin{proof}[Proof of \cref{lem:monotonicity}]
    We first prove that the function $f:\gamma \mapsto \frac{4^{\gamma-1}\sqrt{\pi}\,\Gamma(2-\gamma)}{\Gamma(\frac32-\gamma)}$ is strictly decreasing on $(\frac{3}{2},2)$.
 	Since $\frac{4^{\gamma-1}\sqrt{\pi}\,\Gamma(2-\gamma)}{\Gamma(3/2-\gamma)}<0$  for all $\gamma\in(3/2,2)$, it is enough to show that 
 	\begin{align}\label{eq:neg-deriv}
 		\frac{\partial}{\partial \gamma} \log\left( -f(\gamma)\right)>0\qquad \text{for all  $\gamma\in(3/2,2)$}.
 	\end{align}
 	We have that 
 	\begin{align*}
 		\frac{\partial}{\partial \gamma} \log\left( -f(\gamma)\right)
 		&=
 		\frac{\partial}{\partial \gamma} \left( \log(4^{\gamma -1})+\log(\sqrt{\pi})+\log(\Gamma(2-\gamma))-\log(-\Gamma(3/2-\gamma))\right)\\
 		&=\log(4)-\psi(2-\gamma)+\psi(3/2-\gamma),
 	\end{align*}
 	where $\psi(\cdot)$ denotes the Digamma function $\psi=\Gamma'/\Gamma$. Using that $\psi(x)=\psi(x+1)-1/x$ and that for all $x>0$, $\log(x)-1/x\leq \psi(x)\leq \log(x)-1/(2x)$ (see e.g.\ \cite[Chap.\ 1.7.1]{bateman1953higher} and \cite{alzer1997some} respectively), we get the bound
 	\[
 	\frac{\partial}{\partial \gamma} \log\left( -f(\gamma)\right)\geq
 	\log\left( \frac{4(5/2-\gamma)}{2-\gamma}\right)+\frac{1}{4-2\gamma}-\frac{1}{3/2-\gamma}-\frac{1}{5/2-\gamma}.
 	\]
 	Then a standard analysis of the function on the right-hand side of the above equation shows that \eqref{eq:neg-deriv} holds.

    Finally, the second part of the lemma statement can be directly proved using basic properties of incomplete beta functions. More precisely, the integral that we want to express is nothing but the incomplete beta function
    \[
    B_{1/2}(2-x,-x)
    \coloneqq 
    \int_0^{1/2} t^{1-x} (1-t)^{-1-x} \mathrm{d}t, \quad x<2.
    \]
    The integration by parts relations for incomplete beta functions (see e.g.\ \cite[Chap.\ 2.5.3]{bateman1953higher}) give
    \[
    I_{1/2}(a,b) = 
    I_{1/2}(a,b+1) - \frac{x^a (1-x)^b}{bB(a,b)},
    \]
    where $B$ denote the usual beta function and $I_{1/2}(a,b) \coloneqq  \frac{B_{1/2}(a,b)}{B(a,b)}$. Using twice these relations, we are able to express $B_{1/2}(2-x,-x)$ in terms of (usual) beta functions and $B_{1/2}(2-x,2-x)$. But by symmetry $B_{1/2}(a,a) = \frac12 B(a,a)$, so we are able to express $B_{1/2}(2-x,-x)$ in terms of beta functions, and hence ultimately of gamma functions. 
    The formula claimed in \cref{lem:monotonicity} then comes from standard bookkeeping.
\end{proof}

\section{Scaling limit results}\label{sec:sc_limit_Xn}

The main goal of this section is to establish the scaling limit result for $X_n = \lis(T_n)$ as $n\rightarrow \infty$, stated in~\cref{thm:main3}.
This will be carried out in Sections~\ref{sec:proof_remaind_rn_1},~\ref{sec:proof_remaind_rn_2}~and~\ref{sec:proof_scaling_Xn}. 
It is worth emphasizing that this result does not require any knowledge of the precise value of the exponent $\alpha$, which was determined in the previous section.
Then, in \cref{sect:measurability}, we present an argument that completes the proof of \cref{thm:main_permutations}, and in  \cref{subsec:fragmentation_tree}, we establish a scaling limit result for the tree $T^{k,\lmax}$ as $k \rightarrow \infty$, proving~\cref{thm:fragmentation_tree}.

\subsection{The scaling limit of the LIS: proof outline}
\label{sec:sc_limit_Xn_outline}

We start with an outline of the proof of \cref{thm:main3}.
Recall the notation $\revF_n=\sigma(T_m, \ m\geq n)$ for the $\sigma$-algebra generated by the tree process after time $n$, as defined by Rémy's algorithm. 
We recall that for every $n\geq 1$, the tree $T_n$ is obtained by removing a uniform leaf $L_{n+1}$ from $T_{n+1}$.
We write 
\[X_{n}=X_{n+1}- \mathds{1}_{L_{n+1} \in \mathcal{L}^{\max}_{\cap}(T_{n+1})},\] 
so that for $n\ge 1$,
\begin{align*}
\Ecsq{X_n}{\revF_{n+1}}
&= X_{n+1} - \frac{\#\mathcal{L}^{\max}_{\cap}(T_{n+1})}{n+1}\\ 
&= X_{n+1} \left(1 - \frac{\alpha}{n+1}\right) + \frac{1}{n+1} \left(\alpha X_{n+1} - \#\mathcal{L}^{\max}_{\cap}(T_{n+1}) \right),
\end{align*}
with the idea that the second term on the right-hand side of the last display should be negligible in some sense, as suggested by  \cref{thm:convergence_size_intersection}. 
From there, it is quite natural to introduce the sequence $(c_n)_{n\geq 1}$ defined as 
\begin{align}\label{eq:definition cn}
c_n := \prod_{i=1}^n \left(1-\frac{\alpha}{i}\right) = \prod_{i=1}^n \frac{i- \alpha}{i} = \frac{\Gamma(n+1-\alpha)}{\Gamma(n+1) \Gamma(1-\alpha)}, 
\end{align}
so that the rescaled process  
\begin{equation}\label{eq:resc-proc}
    \resproc_n := c_n X_n
\end{equation}
satisfies
\begin{equation}\label{eq:def_rn}
	\Ecsq{\resproc_n }{\revF_{n+1}} 
	= \resproc_{n+1} +R_{n+1},\quad\text{with}\quad
R_{n+1}:=\frac{c_{n}}{n+1}\left( \alpha X_{n+1} - \#\mathcal{L}^{\max}_{\cap}(T_{n+1})\right).
\end{equation}
Now, if the remainder terms $R_n$ were identically zero, then the process $(\resproc_n)_{n\geq 1}$ would be a backward martingale and would thus converge almost surely as $n\rightarrow \infty$.
Since, by properties of the Gamma function, it is easy to check that
\begin{align}\label{eq:asymptotic behaviour cn}
c_n\underset{n\rightarrow \infty}{\sim} \frac{1}{\Gamma(1-\alpha)}n^{-\alpha},
\end{align}
this would show that $n^{-\alpha}X_n$ converges almost surely, as stated in \cref{thm:main3}.

Following this intuition, the strategy of our proof consists of showing that the remainder terms $R_n$ are small in some quantifiable way.
For this, we rely on 
the quantitative version of  \cref{thm:convergence_size_intersection} stated below, where we recall that $T^k$ is a $p$-signed critical binary Bienaymé--Galton--Watson tree $T$ conditioned on the event $\{\LIS(T)=k\}$.
We recall that we identified at this point the constant $\lambda$ in \cref{thm:convergence_size_intersection} with the exponent $\alpha$ in \cref{thm:main}; see \cref{rmk:lambda=alpha}. 
\begin{thm}[Quantitative version of \cref{thm:convergence_size_intersection}] \label{thm:quantitative_moments}
There exists some $\eps>0$ such that
\begin{align}\label{eq:polynomial convergence to c0}
\mathbb{E}\left[ \left| \frac{\#\mathcal{L}^{\max}_{\cap}(\TLis{k})}{k} -  \alpha \right| \right]  
&=O(k^{-\eps}).
\end{align}
\end{thm}
This theorem ensures that for a $p$-signed critical binary Bienaymé--Galton--Watson tree $T$ conditioned on having $\LIS(T)=k$,
the number of leaves $\#\mathcal{L}^{\max}_{\cap}(T)$ in the intersection of all its maximal positive subtrees is close to $\alpha \LIS(T)=\alpha k$. 
In our case, we can see from \eqref{eq:def_rn} that $R_n$ is rather defined as (a constant times) the corresponding quantity for $T_n$, where we recall that $T_n$ has the distribution of $T$ conditioned on $\{|T|=n\}$. 
Nevertheless, with some additional work, we manage to translate the information of  \cref{thm:quantitative_moments} under this different conditioning into some usable result on the remainders $R_n$. 

Below, we first prove \cref{thm:quantitative_moments} in \cref{sec:proof_remaind_rn_1}. We then turn it into \cref{lem:the error terms are summable as}, which is a quantitative result about the remainders $R_n$, through a series of technical lemmas (Lemmas~\ref{lem:a reinterpretation of thm quantitative moments},~\ref{lem:control sum remainders by sum of controlled terms},~\ref{lem:remainder series in k is polynomially small}~and~\ref{lem:as lower bound for X_n}) in \cref{sec:proof_remaind_rn_2}. 
Based on that, we will be able in \cref{sec:proof_scaling_Xn} to make rigorous sense of the martingale arguments sketched above, and to conclude the proof of \cref{thm:main3}.

\subsection{A quantitative version of the law of large numbers}\label{sec:proof_remaind_rn_1}

This section is devoted to the proof of \cref{thm:quantitative_moments}. Recall the definition of  good scales from~\cref{defn:good-scales}.
The proof of \cref{thm:quantitative_moments} relies on the fact that the good regularity estimates from \cref{thm:good_regularity} guarantee (as explained in \cref{rk:all_scales_good}) the existence of a constant $C > 0$ such that every integer satisfies the good scale condition from \cref{item:good-scale-condition} in \cref{defn:good-scales}.

Thus, in what follows, we set 
\begin{align*}
\mathrm{GoodScales} \coloneqq  \enstq{10^m}{m\geq 0}
\end{align*}
to be our new set of good scales.

\begin{proof}[Proof of \cref{thm:quantitative_moments}] 
From \cref{lem:cvg_first_mmt} and \cref{lem:cvg_2nd_mmt} we know that the first and second moments of $\frac{\#\mathcal{L}^{\max}_{\cap}(\TLis{k})}{k}$ converge towards respectively $\alpha$ and $\alpha^2$ (recall that $\lambda = \alpha$). We hence write 
\begin{align}\label{eq:first and second moment as limit plus error term}
\mathbb{E}\left[\frac{\#\mathcal{L}^{\max}_{\cap}(\TLis{k})}{k} \right] 
= \alpha  + \delta_1(k)
\qquad \text{and} \qquad 
\mathbb{E}\left[\left(\frac{\#\mathcal{L}^{\max}_{\cap}(\TLis{k})}{k}\right)^2 \right] 
=  \alpha^2 + \delta_2(k),
\end{align}
with $\delta_1(k),\delta_2(k) \rightarrow 0 $ as $k\rightarrow \infty$.
This ensures that
\begin{align*}
\mathbb{E}\left[ \left( \frac{\#\mathcal{L}^{\max}_{\cap}(\TLis{k})}{k} -  \alpha\right)^2 \right]  
&= \mathbb{E}\left[\left(\frac{\#\mathcal{L}^{\max}_{\cap}(\TLis{k})}{k}\right)^2 \right] - 2  \alpha\mathbb{E}\left[\frac{\#\mathcal{L}^{\max}_{\cap}(\TLis{k})}{k} \right] +  \alpha^2\\
&= \delta_2(k) - 2\alpha \delta_1(k), 
\end{align*}
and hence by Jensen's inequality, we get 
\begin{align}\label{eq:L1 distance bounded by sqrt of error terms}
\mathbb{E}\left[ \left| \frac{\#\mathcal{L}^{\max}_{\cap}(\TLis{k})}{k} -  \alpha \right| \right] 
\leq \sqrt{\mathbb{E}\left[ \left( \frac{\#\mathcal{L}^{\max}_{\cap}(\TLis{k})}{k} -  \alpha\right)^2 \right]} = \sqrt{\delta_2(k) - 2\alpha \delta_1(k)}. 
\end{align}
From there, it is clear that if we obtain that $\delta_1(k),\delta_2(k)=O(k^{-\eps})$ for some $\eps>0$, \emph{i.e.}\ that the convergence \eqref{eq:first and second moment as limit plus error term} of the first and second moments happens at least with polynomial speed, then we get the result claimed in the theorem. 
We fix $k\geq 1$ and take care of the two terms separately.

First, we bound $\delta_1(k)$. \cref{item:for-later-1}~in~\cref{lem:cvg_first_mmt} ensures that for any $k'\geq k$ and any $A>0$, we have
\begin{align}\label{eq:difference first moment k kprime}
\left\vert\mathbb{E}\left[\frac{\#\mathcal{L}^{\max}_{\cap}(\TLis{k})}{k} \right] - \mathbb{E}\left[\frac{\# \mathcal{L}^{\max}_{\cap}(\TLis{k'})}{k'} \right]\right\vert 
\le
2 \sup_{\mathsf{k}\ge 1}\Ecp{\mathsf{k}}{F_A} + \Pppt{k,k'}{\xmer \le A}.
\end{align}
We fix $A=k^{1/4}$ and estimate the two terms on the right-hand side of \eqref{eq:difference first moment k kprime} in turn.
First, by \cref{lem:forks_mmt}, for some constant $C>0$ we have 
\begin{align}\label{eq:expected number of forks is small}
\sup_{\mathsf{k}\ge 1}\Ecp{\mathsf{k}}{F_A} \leq C\cdot Q(A)=C \cdot Q(k^{1/4})= k^{-\frac{1}{8\alpha}+o(1)},
\end{align}
where the last equality comes from the rough estimate on $Q$ in~\cref{prop:rough_regularity}. 
Second, the number of good intervals included in 
$\intervalleentier{k^{1/4}}{k}$ is $s(k^{1/4},k)=\frac{3\log k}{4\log 10} (1+o(1))$. Therefore, by \cref{item:full_coupling_quantitative}~of~\cref{cor:full_coupling}, there exists a constant $c>0$ such that
\begin{align*}
\Pppt{k,k'}{\xmer \le k^{1/4}}\leq (1-c)^{s(k^{1/4},k)} = (1-c)^{\frac{3\log k}{4\log 10} (1+o(1))} \leq k^{\frac{3\log (1-c)}{4\log 10} (1+o(1))} \leq k^{-\eps+o(1)}
\end{align*}
as $k\rightarrow \infty$, for some $\eps>0$. 
Since we have upper bounded the two terms on the right-hand side of \eqref{eq:difference first moment k kprime} by quantities that do not depend on $k'$, we can then take $k'\rightarrow \infty$ in that inequality and get, for $\eps'>0$ small enough,
\begin{align}\label{eq:first moment converges polynomially fast}
|\delta_1(k)| = \left\vert\mathbb{E}\left[\frac{\#\mathcal{L}^{\max}_{\cap}(\TLis{k})}{k} \right] - \alpha \right\vert 
\le
k^{-\frac{1}{8\alpha}+o(1)} +k^{-\eps+o(1)} = O(k^{-\eps'}).
\end{align}

Second, we bound $\delta_2(k)$. We use \cref{item:2nd_mmt_1} of \cref{lem:cvg_2nd_mmt}: for any $k\geq A>1$, we have 
\begin{align*}
|\delta_2(k)|= \left|\mathbb{E}\left[\left(\frac{\#\mathcal{L}^{\max}_{\cap}(\TLis{k})}{k}\right)^2 \right] - \lambda^2  \right| \leq \mathbb{E}_k[F_A] + 3 \sup_{\mathsf{k}\geq A} \left| \delta_1(\mathsf{k})\right| + 7\left(\frac{A^2}{k} +\frac{C}{A}\right).
\end{align*}
We again take $A=k^{1/4}$ and use \eqref{eq:expected number of forks is small} and \eqref{eq:first moment converges polynomially fast} to obtain
\begin{align}\label{eq:second moment converges polynomially fast}
|\delta_2(k)| \leq  k^{-\frac{1}{8\alpha}+o(1)} + O(k^{-\frac{\eps}{8}}) + O(k^{-\frac{1}{4}}) = O(k^{-\eps'})
\end{align}
for some $\eps'>0$ small enough. 

Finally, substituting the bounds \eqref{eq:first moment converges polynomially fast} and \eqref{eq:second moment converges polynomially fast} into \eqref{eq:L1 distance bounded by sqrt of error terms} yields the claimed result. 
\end{proof}

\subsection{Control of the remainders}
\label{sec:proof_remaind_rn_2}
The goal of this subsection is to prove the following lemma, which allows us to control the sum of the remainder terms $R_n$ introduced in~\eqref{eq:def_rn}. 
\begin{lem}[Control on the remainders]\label{lem:the error terms are summable as}
The random sum $\sum_{n=2}^{\infty} |R_n|$
is almost surely finite.
Furthermore, there exists $\eps>0$ such that almost surely, for $n_0$ large enough,
\begin{align*}
\sum_{n=n_0}^{\infty} |R_n| \leq n_0^{-\eps}.
\end{align*}
\end{lem}
We point out that in this section, as in~\cref{lem:the error terms are summable as}, whenever we write ``almost surely, for $n$ large enough,'' the threshold for ``large enough'' may be random.
The proof will rely on four intermediate lemmas, which we will state and prove as we proceed.  
The proof of \cref{lem:the error terms are summable as} will then be provided at the end of this section.
Recall from \cref{subsec:rough regularity} the notation $\sttm_k$ for the hitting time of $k$ by $X_n$ introduced in~\eqref{eq:defn-stp-time}, as well as the probability $r(n)= \Pp{|T|=n}$ from \cref{lem:prelim_GW}. 
For clearer expressions, we introduce the following notation: for a finite sign-decorated binary tree $t$, we write 
\[f(t):=\lvert \alpha \LIS(t)-\#\mathcal{L}_\cap^{\max}(t) \rvert.\]
Note that since \( |R_{n}| =
    \frac{c_{n-1}}{n} \left|  \alpha X_{n} - \#\mathcal{L}^{\max}_{\cap}(T_{n}) \right| 
   =\frac{c_{n-1}}{n} f(T_n)
\), we can write for any $n_0 \ge 2$,
\begin{equation} \label{eq:sum of the error terms to control}
\sum_{n=n_0}^{\infty}\left\lvert R_n\right\rvert
     =\sum_{n=n_0}^{\infty} \frac{c_{n-1}}{n} f(T_{n}).
\end{equation}
Looking at the form of the expression \eqref{eq:sum of the error terms to control}, we can already see some similarity with the quantities appearing (in expectation) in the following rephrasing of \cref{thm:quantitative_moments}.
\begin{lem}[Rephrasing of \cref{thm:quantitative_moments}]
\label{lem:a reinterpretation of thm quantitative moments}
There is $\eps>0$ such that we have 
\begin{align} \label{eq:control sum pn ftn}
\Ec{\sum_{n=\sttm_k}^{\sttm_{k+1}-1} r(n) \cdot f(T_n)} =   O(k^{1-\eps} q(k))
\quad
\text{as } k\to\infty.
\end{align}
\end{lem}
\begin{proof}
First, \cref{thm:quantitative_moments} tells us that there exists $\eps>0$ such that
\begin{align*}
\Ec{f(T^k)}=\Ec{\lvert \alpha \LIS(T^k)-\#\mathcal{L}_\cap^{\max}(T^k) \rvert}=k \mathbb{E}\left[ \left| \frac{\#\mathcal{L}^{\max}_{\cap}(\TLis{k})}{k} -  \alpha \right| \right]   =O(k^{1-\eps}).
\end{align*}
Then, recalling that $r(n)= \Pp{|T|=n}$, we can rewrite the left-hand side as 
\begin{align*}
\Ec{f(T^k)} 
= \frac{1}{\Pp{|T|=k}} \cdot \Ec{f(T) \cdot \indicator{\LIS(T)=k}}
&= \frac{1}{q(k)} \sum_{n=1}^\infty r(n) \cdot \Ec{f(T_n) \indicator{\LIS(T_n)=k}}\\
&=\frac{1}{q(k)}  \Ec{\sum_{n=\sttm_k}^{\sttm_{k+1}-1} r(n) \cdot f(T_n)}.
\end{align*}
Combining the last two displays finishes the proof.
\end{proof}
The next lemma aims to upper bound the sum of the remainders $R_n$ in terms of the quantities controlled (in expectation) by the previous result.  
\begin{lem}[Upper bound on the sum of remainders]
\label{lem:control sum remainders by sum of controlled terms}
There exists a constant $C>0$ such that, almost surely, for any $n_0\geq 2$, we have
\begin{align*}
\sum_{n=n_0}^{\infty} |R_n| \leq C \cdot \sum_{k=X_{n_0}}^\infty (\sttm_k)^{\frac{1}{2}-\alpha }\left(\sum_{n=\sttm_k}^{\sttm_{k+1}-1} r(n) f(T_n)\right).
\end{align*}
\end{lem}
\begin{proof}
Using \eqref{eq:sum of the error terms to control}, we write
\begin{align}\label{eq:rewr-sums}
\sum_{n=n_0}^{\infty} |R_n| 
=
\sum_{n=n_0}^{\infty} \frac{c_{n-1}}{n} f(T_{n})
&\leq \sum_{n=\tau_{X_{n_0}}}^\infty \frac{c_{n-1}}{n} f(T_{n}) \notag\\
&= 
\sum_{k=X_{n_0}}^\infty \left(\sum_{n=\sttm_k}^{\sttm_{k+1}-1} \frac{c_{n-1}}{n} f(T_{n})\right).
\end{align}
Since $c_n\sim \frac{1}{\Gamma(1-\alpha)}n^{-\alpha}$ by \eqref{eq:asymptotic behaviour cn} and $r(n)\sim c n^{-3/2}$ thanks to \cref{lem:prelim_GW}, 
we bound the term appearing in front of $f(T_n)$ in the sum as
\begin{align*}
\frac{c_{n-1}}{n} \leq C n^{\frac{1}{2}-\alpha }  r(n),
\end{align*}
for some constant $C>0$.
We substitute this into \eqref{eq:rewr-sums} and additionally use the fact that if $n\geq \sttm_k$ then $n^{\frac{1}{2}-\alpha} \leq (\sttm_k)^{\frac{1}{2}-\alpha}$ because $\frac{1}{2}-\alpha<0$ (recall \cref{thm:prev-bound}). We get 
\begin{align*}
\sum_{n=n_0}^{\infty} \frac{c_{n-1}}{n} f(T_n) \leq C \cdot \sum_{k=X_{n_0}}^\infty (\sttm_k)^{\frac{1}{2}-\alpha }\left(\sum_{n=\sttm_k}^{\sttm_{k+1}-1} r(n) f(T_n)\right),
\end{align*}
which finishes the proof.
\end{proof}

In the next lemma, we use the estimates from \cref{lem:a reinterpretation of thm quantitative moments} to further control the upper bound in \cref{lem:control sum remainders by sum of controlled terms}. 
\begin{lem}[Polynomial tail control on the random sum]
\label{lem:remainder series in k is polynomially small}
The random sum 
\[
\sum_{k=1}^\infty (\sttm_k)^{\frac{1}{2}-\alpha }\left(\sum_{n=\sttm_k}^{\sttm_{k+1}-1} r(n) f(T_n)\right)
\]
is almost surely finite.
Furthermore, there exists $\eps>0$ such that almost surely, for $k_0$ large enough,
\begin{align*}
\sum_{k=k_0}^\infty (\sttm_k)^{\frac{1}{2}-\alpha }\left(\sum_{n=\sttm_k}^{\sttm_{k+1}-1} r(n) f(T_n)\right) \leq k_0^{-\eps}.
\end{align*}
\end{lem}

\begin{proof}[Proof of \cref{lem:remainder series in k is polynomially small}]
Let $\delta>0$ be a small parameter to be fixed later. By \eqref{eqn:alpha_as_limsup} in \cref{thm:main}, there is a constant $C>0$ such that, for all $n\ge 1$, $\Ec{X_{n}} \le C n^{\alpha+\delta/2}$.
Using Markov's inequality for $X_{2^i}$, for all $i\ge 1$, we have
\[
\Pp{X_{2^i} \geq (2^{i-1})^{\alpha +\delta}}
\le
(2^{i-1})^{-\alpha -\delta} \Ec{X_{2^i}}
\le
C 2^{-i\delta/2}
\]
for some (other) constant $C>0$. Therefore
\begin{align*}
\sum_{i\geq 1} \Pp{X_{2^i} \geq (2^{i-1})^{\alpha +\delta}} 
<\infty. 
\end{align*}
By the Borel--Cantelli lemma and the fact that $X_n$ is non-decreasing, we conclude that, almost surely, $X_n\leq n^{\alpha+\delta}$ for all $n$ large enough, so that $\sttm_k \geq k^{\frac{1}{\alpha+\delta}}$ for $k$ large enough.
We let 
\[K= 1+\sup \enstq{k\geq 1}{\sttm_k <  k^{\frac{1}{\alpha +\delta}}},\] 
and note that the argument above proves that $K<\infty$ almost surely.  

The definition of $K$ ensures that on the event $\{k\geq K\}$ we  have $\sttm_k \geq k^{\frac{1}{\alpha +\delta}}$ so that, 
for any $k\geq 1$, noting that $1/2-\alpha<0$ by \cref{thm:prev-bound},
\begin{align*}
\Ec{\indicator{k\geq K} (\sttm_k)^{\frac{1}{2}-\alpha }\left(\sum_{n=\sttm_k}^{\sttm_{k+1}-1} r(n) f(T_n)\right) } 
&\leq \Ec{\indicator{k\geq K} k^{\frac{1-2\alpha}{2\alpha +2\delta}}\left(\sum_{n=\sttm_k}^{\sttm_{k+1}-1} r(n) f(T_n)\right)}\\
&\leq k^{\frac{1-2\alpha}{2\alpha +2\delta}}\Ec{\sum_{n=\sttm_k}^{\sttm_{k+1}-1} r(n) f(T_n)}.
\end{align*}
By 
\cref{lem:a reinterpretation of thm quantitative moments}
this can be bounded as
\[
\Ec{\indicator{k\geq K}(\sttm_k)^{\frac{1}{2}-\alpha }\left(\sum_{n=\sttm_k}^{\sttm_{k+1}-1} r(n) f(T_n)\right) } 
\leq 
C k^{\frac{1-2\alpha}{2\alpha +2\delta}} \cdot q(k)\cdot k^{1-\eps}
\]
for some constants $C, \eps>0$.
Now, using the fact that $q(k)=k^{-\frac{1}{2\alpha}-1+o(1)}$, thanks to~\cref{prop:rough_regularity}, we find that by taking $\delta$ small enough, the last expression is 
smaller than $Ck^{-1-\eps/2}$.
Summing over $k\geq 2^i$ we get that 
\begin{align*}
\Ec{\sum_{k=2^i}^\infty\indicator{k\geq K}(\sttm_k)^{\frac{1}{2}-\alpha }\left(\sum_{n=\sttm_k}^{\sttm_{k+1}-1} r(n) f(T_n)\right)} =O(2^{-i\eps/2}) 
\quad \text{as } i\to\infty,
\end{align*}
which is summable. 
Hence, by Markov's inequality, 
\[
\sum_{i\ge 1} \Pp{\sum_{k=2^i}^\infty\indicator{k\geq K}(\sttm_k)^{\frac{1}{2}-\alpha }\left(\sum_{n=\sttm_k}^{\sttm_{k+1}-1} r(n) f(T_n)\right) > 2^{-i\eps/4}} 
< \infty.
\]
Therefore, by the Borel--Cantelli lemma and monotonicity, with probability one we have 
\begin{align*}
    \sum_{k=k_0}^\infty\indicator{k\geq K}(\sttm_k)^{\frac{1}{2}-\alpha }\left(\sum_{n=\sttm_k}^{\sttm_{k+1}-1} r(n) f(T_n)\right) \leq 2 k_0^{-\eps/4}
\end{align*}
for $k_0$ large enough. Finally, by taking $k_0$ large enough, we can assume that $k_0\ge K$, which yields the claim of \cref{lem:remainder series in k is polynomially small}.
\end{proof}

Lastly, because the sum appearing on the right-hand side of \cref{lem:control sum remainders by sum of controlled terms} starts from the random variable $X_{n_0},$ which can, in principle, be small, we state the next lemma to ensure that this quantity grows at least polynomially fast in $n_0$ almost surely.

\begin{lem}[Almost sure polynomial lower bound on $X_n$]\label{lem:as lower bound for X_n}
There exists an almost surely positive random variable $M$ such that for all $n\geq 1$,
\begin{align*}
X_n \geq M n^{\frac{p}{2}} \quad \text{almost surely}.
\end{align*}
\end{lem}

\begin{proof}[Proof of \cref{lem:as lower bound for X_n}]
We use a greedy construction to get a positive subtree $\bunderline{T}_n$ of $T_n$.
The way this is done is reminiscent of the proof of \cref{prop:rough_regularity}.
We let $\bunderline{T}_1=T_1$. 
Then at every step $n\geq 1$ of the forward Rémy's algorithm (\cref{sect:remy-algo}):
\begin{itemize}
 \item[(i)] If the edge $\uedge{n+1}$ that is picked is adjacent to one of the leaves of $\bunderline{T}_{n}$ \emph{and} the created vertex is a $\oplus$, then we add the newly added leaf to $\bunderline{T}_{n}$ to get $\bunderline{T}_{n+1}$. 
\item[(ii)]  Otherwise we let $\bunderline{T}_{n+1}\coloneqq \bunderline{T}_{n}$. 
\end{itemize}
Now let us consider the dynamics of the number $Y_n$ of leaves in $\bunderline{T}_n$. 
Conditionally on $Y_1,\dots , Y_{n}$ we have for all $n\ge 1$,
\begin{equation*}
Y_{n+1} = 
\left\lbrace
\begin{aligned}
Y_{n} + 1  \qquad & \text{with probability } \frac{p Y_{n}}{2n-1},\\
Y_{n} \qquad &\text{with probability } \frac{2n-1 -p Y_{n}}{2n-1}.\\
\end{aligned} \right.
\end{equation*}
From these transitions, we note that the sequence of vectors $(pY_n ,2n-1 - pY_n)_{n\geq 1}$
follows
the same dynamics as the (non-necessarily integer) amount of black and white balls in a generalized P\'olya urn where:
\begin{itemize}
\item At time $1$, the urn contains $p$ black balls and $1-p$ white balls.
\item At every step, a color is drawn at random from the urn with a distribution corresponding to its current composition. Then if the sampled color is black, we add $p$ black balls and $2-p$ white balls to the urn. Otherwise, we add $2$ white balls to the urn.
\end{itemize}
Such an urn is called \emph{balanced} because the number of balls added at each step is constant, and \emph{triangular} because the corresponding replacement matrix is triangular.
From general results on triangular balanced P\'olya urns
(e.g.\ \cite[Theorem 2]{aguech2009limit}),
we get that $n^{-\frac{p}{2}} \cdot Y_n \rightarrow L$ almost surely as $n \rightarrow \infty$, where $L>0$ almost surely. 
Hence, for $M\coloneqq \min_{n\geq 1} Y_n n^{-\frac{p}{2}}$ we get the result stated in the lemma, because by definition of $X_n$ we have $X_n\geq Y_n$ for all $n\geq 1$.
\end{proof}

We can now prove \cref{lem:the error terms are summable as} by combining all of the above. 
\begin{proof}[Proof of \cref{lem:the error terms are summable as}]
Using consecutively the results of Lemmas~\ref{lem:control sum remainders by sum of controlled terms},~\ref{lem:remainder series in k is polynomially small}~and~\ref{lem:as lower bound for X_n}, we get
\begin{align*}
\sum_{n=n_0}^{\infty} |R_n| 
\leq C \cdot \sum_{k=X_{n_0}}^\infty (\sttm_k)^{\frac{1}{2}-\alpha }\left(\sum_{n=\sttm_k}^{\sttm_{k+1}-1} r(n) f(T_n)\right) 
\leq C X_{n_0}^{-\eps} 
\leq C M^{-\eps} n_0^{-p\eps/2}.
\end{align*}
Taking $n_0=2$ shows the first claim of the lemma.
Then taking $\eps' = \frac{p\eps}{4}$ ensures that for $n_0$ large enough (the threshold being random), we have 
\begin{align*}
\sum_{n=n_0}^{\infty} |R_n| \leq n_0^{-\eps'}.
\end{align*}
This concludes the proof of \cref{lem:the error terms are summable as}.
\end{proof}

\subsection{The scaling limit result}
\label{sec:proof_scaling_Xn}

We can now establish the scaling limit result of \cref{thm:main3} for the sequence $(X_n)_{n\geq 1}$.

\begin{proof}[Proof of \cref{thm:main3}]
Recall the notation $\revF_n=\sigma(T_m, \ m\geq n)$ and denote $\revF_\infty = \bigcap_{n\geq 1} \revF_n$.
Recall from \cref{sec:sc_limit_Xn_outline} that we defined, for all $n\geq 1$,
\begin{align*}
c_n= \prod_{i=1}^n \left(1-\frac{\alpha}{i}\right) \quad \text{and} \quad \resproc_n = c_n X_n \quad \text{and} \quad R_{n+1}=\frac{c_{n}}{n+1}\left( \alpha X_{n+1} - \#\mathcal{L}^{\max}_{\cap}(T_{n+1})\right) .
\end{align*}
We can now introduce the  
auxiliary $(\revF_n)$-adapted process $(\auxproc_n)_{n\geq 1}$ as
\begin{equation}\label{eq:Bn_def}
\auxproc_n\coloneqq \resproc_n - \sum_{i=n+1}^{\infty}R_i, \quad n\ge 1,
\end{equation}
which is almost surely finite thanks to \cref{lem:the error terms are summable as}.

\medskip 
\noindent
\emph{\underline{Step 1:} The process $(\auxproc_n)_{n\geq 1}$ has martingale increments.}
Using $X_{n}= X_{n+1} - \mathds{1}_{L_{n+1} \in \mathcal{L}^{\max}_{\cap}(T_{n+1})}$ and $c_{n+1}= \left(1-\frac{\alpha}{n+1}\right) c_n$ along with the expression $R_{n+1} = \frac{c_n}{n+1} \big(\alpha X_{n+1} - \#\mathcal{L}^{\max}_{\cap}(T_{n+1})\big)$, we can write the increment of the auxiliary process as 
\begin{align*}
\auxproc_n - \auxproc_{n+1} 
= (\resproc_n - \resproc_{n+1}) - R_{n+1} 
&= c_n\left(\frac{\#\mathcal{L}^{\max}_{\cap}(T_{n+1})}{n+1} - \mathds{1}_{L_{n+1} \in \mathcal{L}^{\max}_{\cap}(T_{n+1})}\right).
\end{align*}
From this expression, it is clear that 
\begin{equation}\label{eq:first moment increment auxproc}
\Ecsq{\auxproc_n - \auxproc_{n+1} }{\revF_{n+1}}=0,
\end{equation}
and 
\begin{align}\label{eq:second moment increment auxproc}
\Ecsq{(\auxproc_n - \auxproc_{n+1})^2}{\revF_{n+1}}
&=(c_n)^2 \cdot \frac{\#\mathcal{L}^{\max}_{\cap}(T_{n+1})}{n+1} \cdot \Big(1- \frac{\#\mathcal{L}^{\max}_{\cap}(T_{n+1})}{n+1} \Big) \notag \\
&\leq  (c_n)^2 = O(n^{-2\alpha}),
\end{align}
where, in the inequality, we applied the deterministic bound $\#\mathcal{L}^{\max}_{\cap}(T_{n+1})\leq n+1$, and then we used the asymptotics \eqref{eq:asymptotic behaviour cn} for the sequence $(c_n)_{n\geq 1}$. 

We have just proved that the \emph{increments} of the process $(\auxproc_n)_{n\geq 1}$ are those of an $L^2$ backward martingale. 
Unfortunately, we actually do not know at this point whether $\auxproc_n$ is in $L^1$ or not, 
so we cannot claim that it is a backward martingale.

\medskip
\noindent
\emph{\underline{Step 2:} Introducing related backward supermartingales.}
We introduce 
\begin{equation}\label{eq:defn-stpmg}
    \stpmg\coloneqq \sup \enstq{ n \geq 2}{ \sum_{i \ge n} |R_i| > n^{-\eps} },
\end{equation}
where $\eps>0$ is as in \cref{lem:the error terms are summable as} and we take the convention that $\stpmg=+\infty$ if no such $n$ exists. 
By the second part of~\cref{lem:the error terms are summable as}, we have $\stpmg<+\infty$ almost surely. 
Then the negative part $\auxproc^{-}_{n\vee \stpmg}\coloneqq -\min(\auxproc_{n\vee \stpmg},0)$ is bounded for all $n\ge 1$, since
\begin{equation} \label{eq:Bntau_bounded}
\auxproc_{n\vee \stpmg}^{-}
\le 
\sum_{i>n\vee \stpmg} |R_i|
\le
({n\vee \stpmg})^{-\eps}
\le 
n^{-\eps} 
\leq 1 \qquad \text{almost surely.}
\end{equation}
Moreover, using that $\{\stpmg\le n\}$ is $\revF_{n+1}$-measurable, we have
\begin{align*}
\Ecsq{\auxproc_{n\vee \stpmg} - \auxproc_{(n+1)\vee \stpmg}}{\revF_{n+1}}
&=
\Ecsq{\mathds{1}_{\stpmg\le n} (\auxproc_{n\vee \stpmg} - \auxproc_{(n+1)\vee \stpmg})}{\revF_{n+1}} \\
&=
\mathds{1}_{\stpmg\le n} \Ecsq{\auxproc_{n} - \auxproc_{n+1}}{\revF_{n+1}} = 0.
\end{align*}
Hence, the increments of $(\auxproc_{n\vee \stpmg})_{n\ge 1}$ have vanishing (conditional) expectation.  
Now for any $K>0$, the process $(\auxproc_{n\vee \stpmg} \wedge K)_{ n\ge 1}$ is bounded between $-1$ and $K$ so it is in $L^1$. 
Noting that for any $x\in \mathbb{R}$, the function 
$y\mapsto (x + y)\wedge K$
is concave,
we use Jensen's inequality to get 
\begin{align*}
\Ecsq{\auxproc_{n\vee \stpmg} \wedge K}{\revF_{n+1}}
&=
\Ecsq{(\auxproc_{(n+1)\vee \stpmg} + (\auxproc_{n\vee \stpmg} - \auxproc_{(n+1)\vee \stpmg}))\wedge K}{\revF_{n+1}} \\
&\leq \left(\auxproc_{(n+1)\vee \stpmg}+ \Ecsq{\auxproc_{n\vee \stpmg} - \auxproc_{(n+1)\vee \stpmg}}{\revF_{n+1}}\right) \wedge K \\
&= \auxproc_{(n+1)\vee \stpmg} \wedge K,
\end{align*}
which ensures that the process $(\auxproc_{n\vee \stpmg} \wedge K)_{ n\ge 1}$ is a backward supermartingale.

\medskip
\noindent 
\emph{\underline{Step 3:} Proving the almost sure convergence of $(\auxproc_n)_{n\geq 1}$ and hence of $(n^{-\alpha}X_n)_{n\geq 1}$. }

For any $K>0$, the process $(\auxproc_{n\vee \stpmg} \wedge K)_{ n\ge 1}$ is a bounded backward supermartingale.
By standard results on backward supermartingales (see e.g.\ \cite[Section~5.6]{durrett2019probability}\footnote{Although technically the reference only deals with backward \emph{martingales}, the argument readily extends to supermartingales since it only relies on Doob's upcrossing lemma.}), this implies that $(\auxproc_{n\vee \stpmg} \wedge K)_{n\ge 1}$ converges almost surely as $n\to\infty$ towards a limit $\auxproc_\infty^K$ and that for all $n\ge 1$,
\[
\Ecsq{\auxproc_{n \vee \stpmg}\wedge K}{\revF_{\infty}} \leq \auxproc_{\infty}^K.
\]
By monotonicity in $K$, we get that $(\auxproc_{n\vee \stpmg})_{ n\ge 1}$ converges almost surely towards a (possibly infinite) limit $\auxproc_{\infty} = \lim_{K\rightarrow \infty} \auxproc_{\infty}^K$. 
Now since for all $n\ge 1$ we have $\auxproc_{n \vee \stpmg} \geq -1$, we can use the monotone convergence theorem when we pass to the limit $K\rightarrow \infty$ in the previous display and get 
\begin{equation}\label{eq:cond-exp-lim}
  \Ecsq{\auxproc_{n \vee \stpmg}}{\revF_{\infty}} = \lim_{K\rightarrow \infty} \Ecsq{\auxproc_{n \vee \stpmg}\wedge K }{\revF_{\infty}} 
   \leq  \lim_{K\rightarrow \infty} \auxproc_{\infty}^K = \auxproc_{\infty}.
\end{equation}
Since $\stpmg<\infty$ almost surely, we deduce that $(\auxproc_n)_{n\ge 1}$ also converges almost surely to the same $\auxproc_\infty$.
Note that the definition \eqref{eq:Bn_def} of $(\auxproc_n)_{n\geq 1}$, the asymptotic $c_n \sim n^{-\alpha}/\Gamma(1-\alpha)$ and its a.s.\ convergence towards $\auxproc_\infty$ ensure that we have the almost sure convergence
\begin{align}
n^{-\alpha} X_n \xrightarrow[n \rightarrow \infty]{} \lrv  \qquad \text{where} \qquad \lrv \coloneqq \Gamma(1-\alpha) \auxproc_\infty. 
\end{align}

\medskip
\noindent 
\emph{\underline{Step 4:} Proving that $\lrv<\infty$ almost surely.}
First, we prove that $\lrv < \infty$ almost surely relying on $L^2$ estimates.
By \eqref{eq:second moment increment auxproc}, the random variables $(\auxproc_{n+1}-\auxproc_n)_{n\ge 1}$ are all in $L^2$ and by the martingale property \eqref{eq:first moment increment auxproc}, they are uncorrelated, hence orthogonal in $L^2$. 
By the Pythagorean theorem, for any $1\leq N_0\leq N$ we have, now using the second moment inequality \eqref{eq:second moment increment auxproc}, 
\[
\Ec{(\auxproc_{N_0} - \auxproc_{N})^2} 
= \sum_{n=N_0}^{N-1}\Ec{(\auxproc_{n} - \auxproc_{n+1})^2} 
\stackrel{\eqref{eq:second moment increment auxproc}}{=} \sum_{n=N_0}^{N-1} O(n^{-2\alpha})
\leq \sum_{n=N_0}^{\infty} O(n^{-2\alpha}) 
= O((N_0)^{1-2\alpha}), 
\]
where we used that $\alpha>1/2$, which is guaranteed by \cref{thm:prev-bound}. 
This ensures that the sequence $(\auxproc_n-\auxproc_1)_{n\geq 1}$ is a Cauchy sequence in $L^2$ so it converges in $L^2$ (hence in probability) towards some almost surely finite random variable $Y$. 
Thus $\auxproc_n \rightarrow \lrv=\auxproc_1+Y$ in probability, which is then almost surely finite.

\medskip

\noindent 
\emph{\underline{Step 5:} Proving that $\lrv$ is almost surely positive.}
Secondly, we argue that $\Ec{\lrv}>0$, or equivalently $\Ec{\auxproc_\infty}>0$.
By \eqref{eq:cond-exp-lim}, 
\begin{equation}\label{eq:E[B_infty]}
\Ec{\auxproc_\infty} 
\geq
\Ec{\auxproc_{n\vee \stpmg}}
=
\Ec{\resproc_{n\vee \stpmg}} - \Ec{\sum_{i > n\vee \stpmg} |R_i|}
\ge 
\Ec{\resproc_{n\vee \stpmg}} - n^{-\eps},
\end{equation}
where the last inequality follows by definition of $\stpmg$ in~\eqref{eq:defn-stpmg}.
By the convergence in probability in \cref{thm:main}, 
we know that $\Pp{\resproc_n \geq n^{-\eps/2}} \geq 1/2$ for $n$ large enough. 
In addition, since $\stpmg$ is almost surely finite, we also have $\Pp{\stpmg>n}<1/3$ for $n$ large enough. 
This implies that for $n$ large enough,
\begin{align*}
\Ec{\resproc_{n\vee\stpmg}}
\geq
\Ec{\resproc_{n\vee\stpmg} \indicator{\resproc_n\ge n^{-\eps/2}} \indicator{\stpmg \le n}}
&\geq
n^{-\eps/2} \Pp{\resproc_n\ge n^{-\eps/2},\stpmg \le n} \\
&\geq
n^{-\eps/2} \Big(\Pp{\resproc_n\ge n^{-\eps/2}} - \Pp{\stpmg > n}\Big) \\
&\geq \frac16 n^{-\eps/2}.
\end{align*}
Therefore we deduce from \eqref{eq:E[B_infty]} that $\Ec{\lrv}>0$.

Now let us prove that $\lrv> 0$ almost surely.
Picking two independent uniform leaves $L_n^{(1)}$ and $L_n^{(2)}$ of $T_n$, we break the tree into three regions: $T_n^{(1)}$ and $T_n^{(2)}$ are the two subtrees above the highest common ancestor of $L_n^{(1)}$ and $L_n^{(2)}$ containing the leaves $L_n^{(1)}$ and $L_n^{(2)}$ respectively, and $T_n^{(3)}$ is the rest of the tree. 
According to \cref{lem:decomp-tree2} we have $(|T_n^{(1)}|-1, |T_n^{(2)}|-1, |T_n^{(3)}|) \sim \DirM(n-2; 1/2,1/2,1/2)$ and conditionally on those numbers, the trees $T_n^{(1)}$ and $T_n^{(2)}$ are independent $p$-signed uniform binary trees with these sizes. 
We then write
\begin{equation}  \label{eq:cvg_An_split}
\frac{X_n}{n^\alpha} \geq  \max
\left( 
\left(\frac{|T_n^{(1)}|}{n}\right)^\alpha \cdot \frac{\LIS(T_n^{(1)})}{|T_n^{(1)}|^\alpha}, \left(\frac{|T_n^{(2)}|}{n}\right)^\alpha \cdot \frac{\LIS(T_n^{(2)})}{|T_n^{(2)}|^{\alpha}}
\right).
\end{equation}
By \cref{lem:discrete-multinomial-distrd-rv},  $\frac{1}{n}(|T_n^{(1)}|, |T_n^{(2)}|, |T_n^{(3)}|) \xlongrightarrow[]{\mathrm{d}} (D_1, D_2, D_3)$ where $(D_1,D_2,D_3) \sim \Dir(1/2,1/2,1/2)$. Besides, Step 3 above implies that, jointly with the last convergence, we have
\begin{equation} \label{eq:cvg_LIS_1_2}
\left(\frac{\LIS(T_n^{(1)})}{|T_n^{(1)}|^\alpha},\frac{\LIS(T_n^{(2)})}{|T_n^{(2)}|^\alpha}\right) \xlongrightarrow[]{\mathrm{\mathrm{d}}} (\lrv^{(1)}, \lrv^{(2)})
\end{equation}
as $n \to \infty$, where $\lrv^{(1)}, \lrv^{(2)}$ are copies of $\lrv$,  independent from each other and from $(D_1,D_2,D_3)$.
Taking $n\rightarrow \infty$ in \eqref{eq:cvg_An_split}, we get that 
\begin{align*}
\lrv \succeq \max\left(D_1^\alpha \cdot \lrv^{(1)}, D_2^\alpha \cdot \lrv^{(2)}\right).
\end{align*}
Thus, 
\begin{align*}
\Pp{\lrv = 0} \leq \Pp{\lrv^{(1)} = \lrv^{(2)}=0} =\Pp{\lrv = 0} ^2, 
\end{align*}
so that $\Pp{\lrv = 0} \in \{0,1\}$. 
We know since $\Ec{\lrv}>0$ that this probability cannot be $1$, so we must have $\Pp{\lrv = 0}=0$.

\medskip

\noindent 
\emph{\underline{Step 6:} Proving that $\lrv$ is not deterministic.}
Finally, we justify using similar arguments that $\lrv$ cannot be a finite constant almost surely.
Assume by contradiction that $\lrv = b\in \intervalleoo{0}{\infty}$ almost surely.  
Then, considering the same decomposition as above, we write $S_n^{(1,2)}$ for the sign in $\{\ominus, \oplus\}$ of the highest common ancestor of the leaves $L_n^{(1)}$ and $L_n^{(2)}$. 
We have 
\begin{align}\label{eq:stochastic domination when node has a plus}
\frac{\LIS(T_n)}{n^\alpha} \geq 
\indicator{\{S_n^{(1,2)} = \oplus\}} \cdot  \left(\left(\frac{|T_n^{(1)}|}{n}\right)^\alpha \cdot \frac{\LIS(T_n^{(1)})}{|T_n^{(1)}|^\alpha} + \left(\frac{|T_n^{(2)}|}{n}\right)^\alpha \cdot \frac{\LIS(T_n^{(2)})}{|T_n^{(2)}|^{\alpha}}\right)
.
\end{align}
Now, from our assumptions and using again Step 3 above, the left-hand side converges in law towards a constant $b$. 
The right-hand side converges by \eqref{eq:cvg_LIS_1_2} towards a random variable 
\begin{align*}
\mathrm{Ber}(p) \cdot b \cdot \left(D_1^\alpha +D_2^\alpha \right),
\end{align*}
where $(D_1,D_2,D_3)\sim \Dir(1/2,1/2,1/2)$ and $\mathrm{Ber}(p)$ denotes a Bernoulli random variable with parameter $p$ independent of everything else. Hence, $b \succeq \mathrm{Ber}(p) \cdot b \cdot \left(D_1^\alpha +D_2^\alpha \right)$.
Since $\alpha<1$, the random variable on the right takes values strictly greater than $b$ with positive probability. 
This is a contradiction, so $\lrv$ is not deterministic.
\end{proof}

\subsection{The limiting random variable $X$ is a measurable function of the Brownian separable permuton}\label{sect:measurability}

We prove the only remaining claim of~\cref{thm:main_permutations}, \emph{i.e.}\ that $X$ is a deterministic measurable function of the Brownian separable permuton. The same explanation works for the companion remaining claim in the case of the Brownian cographon in~\cref{thm:main_graphs}.

Recall from \cref{remk:coupled-trees} that the permutations $\sigma_n$ in~\cref{thm:main_permutations} can be obtained from the $p$-signed uniform binary trees $\widetilde{T}_n=t(\efrak,\sfrak,p;U_1,\dots,U_n)$, where $(U_i)_{i\geq 1}$ are i.i.d.\ uniform random variables in $[0,1]$ and $(\efrak,\sfrak,p)$ is the signed Brownian excursion used to construct the Brownian separable permuton $\bm{\mu}_p$. 
In particular, the trees $(\widetilde{T}_n)_{n\geq 1}$ are coupled as in Rémy's algorithm.

Fix $m\geq 1$ and introduce, for all $n\geq 1$, the tree $\widetilde{T}^{(m)}_n= t(\efrak,\sfrak,p;U_{m+1},\dots U_{m+n})$.
Note that, since the two collections of points $\{U_1,\dots, U_n\}$ and $\{U_{m+1},\dots, U_{m+n}\}$ have all but $2m$ points in common,  
by definition of $\LIS(\cdot)$, we have for all $n\geq 1$ the deterministic bound
\[\left|\LIS(\widetilde{T}^{(m)}_n) - \LIS(\widetilde{T}_n)\right| \leq 2m.\]
As a consequence, from \cref{thm:main3}, we have the almost sure convergence
\begin{align}\label{eq:quenched}
X= 
\lim_{n\rightarrow \infty} n^{-\alpha} \LIS(\widetilde{T}_n) =  \lim_{n\rightarrow \infty} n^{-\alpha} \LIS(\widetilde{T}^{(m)}_n).
\end{align}
Hence, the limiting random variable $X$ can be written as a deterministic measurable function of $((\efrak,\sfrak,p),(U_i)_{i> m})$. 
Since this is true for any $m\geq 1$, we get that $X$ is $\mathcal{F}$-measurable for $\mathcal{F}=\bigcap_{m\geq 1}\sigma((\efrak,\sfrak,p),(U_i)_{i> m})$.
Since in addition the random variables $U_i$ for $i\ge 1$ are i.i.d., by Kolmogorov's $0$-$1$ law, the $\sigma$-field $\mathcal{G}=\bigcap_{m> 1}\sigma((U_i)_{i\geq m})$ is trivial -- in the sense that it only contains events of probability $0$ or $1$. Since  $\mathcal{F}=\sigma(\mathcal{G}\cup\sigma(\efrak,\sfrak,p))$, this ensures that $X$ is almost surely a deterministic measurable function of $(\efrak,\sfrak,p)$, \emph{i.e.}\ of the Brownian separable permuton $\bm{\mu}_p$.

\subsection{Scaling limit of the leftmost maximal positive subtree of $T^k$}\label{subsec:fragmentation_tree}

We prove in this final section the only remaining result, \emph{i.e.}\ \cref{thm:fragmentation_tree}. Recall from \eqref{eq:defn-tklmax} that we introduced the more compact notation $T^{k,\lmax}=(T^k)^{\lmax}$.

\begin{proof}[Proof of~\cref{thm:fragmentation_tree}]
    We first notice that \cref{lem:proba-comp} provides a recursive construction of $T^{k,\lmax}$ as follows. 
    Let $\mathcal{L}^{k,\lmax}$ be the set of leaves of $T^{k,\lmax}$.
    Let $v_1, v_2$ be the two children of the root of $T^{k}$ and let $K_1, K_2$ be their respective number of descendants in $\mathcal{L}^{k,\lmax}$. 
    We choose to order $v_1$ and $v_2$ in such a way that $K_1 \geq K_2$. We also set the convention $K_2=0$ if the root has only one child in $T^{k,\lmax}$, and $K_1=K_2=0$ if it has none (which implies $k=1)$. Then conditionally on $(K_1, K_2)=(k_1, k_2)$, the two trees of descendants of $v_1$ and $v_2$ with leaves in $\mathcal{L}^{k,\lmax}$ are independent, each distributed as $T^{k_1,\lmax}$ and $T^{k_2,\lmax}$ (taking the convention that $T^{0,\lmax}$ is empty). Moreover, the law\footnote{The notation $q_k$ is chosen to be consistent with~\cite{haasmiermont12fragmentation}.} $q_k$ of $(K_1, K_2)$ is given for $k\geq 1$ by
    \begin{alignat*}{3}
        &q_k \left( (k_1,k_2) \right) &&= p \frac{q(k_1) q(k_2)}{q(k)} &&\mbox{if $k_1+k_2=k$ and $1 \leq k_2 < \frac{k}{2}$},\\
        &q_k \left( (k_1,k_2) \right) &&= \frac{1}{2} p \frac{q(k_1) q(k_2)}{q(k)} &&\mbox{if $k_1=k_2=\frac{k}{2}$},\\
        &q_k \left( (k,0) \right) &&= (1-p) \left( \frac{q(k)}{2} + \sum_{i=1}^{k-1} q(i) \right) &&,\\
        &q_k(0,0) &&= \frac{1}{2q(k)} \mathbbm{1}_{k=1}. &&
    \end{alignat*}
    In the terminology of~\cite[Section 1.2.1]{haasmiermont12fragmentation}, this means that $T^{k,\lmax}$ is a Markov branching tree with splitting distributions $q_k$ (more precisely, we are in the case where the splitting is binary, \emph{i.e.}\ $q_k$ is supported on partitions of $k$ with $1$ or $2$ parts). Therefore, in order to conclude by~\cite[Theorem 5]{haasmiermont12fragmentation}, it is sufficient to check that the splitting distributions $(q_k)$ satisfy the assumption (H) of~\cite{haasmiermont12fragmentation}, with the measure $\nu=\nu_{\gamma}$ on binary partitions of $1$ given by~\eqref{eqn:dislocation_measure}.

    Let us now state this assumption (H) explicitly. We first define the normalized version $\overline{q}_k$ of $q_k$ as the measure on $[0,1]^2$ such that, for any measurable function $f:[0,1]^2 \to \bbR_{>0}$, we have
    \[ \int f \,\,\mathrm{d} \overline{q}_k = \sum_{k_1,k_2} q_k \left( (k_1, k_2) \right) f \left( \frac{k_1}{k}, \frac{k_2}{k} \right). \]
    Then assumption (H) with the measure $\nu_{\gamma}$ of~\eqref{eqn:dislocation_measure} states that there is a slowly varying function $\ell(\cdot)$ such that for any continuous function $f: [0,1]^2 \to \bbR_{>0}$, we have
    \begin{equation}\label{eqn:assumption_H_hassmiermont}
        k^{\gamma-1} \ell(k) \int (1-s_1) f(s_1, s_2) \,\, \overline{q}_k (\mathrm{d}s_1, \mathrm{d}s_2) \xrightarrow[k \to \infty]{} \int_0^{1/2} f(1-x,x) x^{1-\gamma} (1-x)^{-\gamma} \, \mathrm{d}x.
    \end{equation}
    Note that the parameter $\gamma$ of~\cite{haasmiermont12fragmentation} corresponds to our $\gamma-1$.  
    It turns out that we can choose $\ell$ to be $1/\varphi$ with $\varphi$ as in Theorem~\ref{thm:good_regularity}, \emph{i.e.}\ $\varphi$ is slowly varying and $q(k)=k^{-\gamma} \varphi(k)$. Using the explicit probabilities $q_k$ given above, the left-hand side of \eqref{eqn:assumption_H_hassmiermont} then rewrites for $k\ge 2$ as
    \[ \frac{k^{\gamma-1}}{\varphi(k)} p \left( \sum_{i=1}^{\lfloor \frac{k-1}{2} \rfloor} \frac{q(k-i) q(i)}{q(k)} \frac{i}{k} f \left( \frac{k-i}{k}, \frac{i}{k} \right) + \mathbbm{1}_{\text{$k$ even}} \frac{1}{2} \frac{q(k/2)^2}{q(k)} \cdot \frac{1}{2} f \left( \frac{1}{2}, \frac{1}{2} \right)\right). \]
    Note first that $\frac{q(k/2)^2}{q(k)}=4^{\gamma}k^{-\gamma} \varphi(k)(1+o(1))$, so the second term is negligible. After replacing $q(k)$ by $k^{-\gamma} \varphi(k)$, the last equation becomes
    \[ p \frac{1}{k}  \sum_{i=1}^{\lfloor \frac{k-1}{2} \rfloor} \left( 1-\frac{i}{k} \right)^{-\gamma} \left( \frac{i}{k} \right)^{1-\gamma} \frac{\varphi(k-i) \varphi(i)}{\varphi(k)^2} f \left( \frac{k-i}{k}, \frac{i}{k} \right) +o(1). \]
    We now fix $\eps>0$, and split the sum according to whether $i<\eps k$ or not. First, by \cref{item:slow_variation1} of Lemma~\ref{lem:slow_variation}, the ratio $\frac{\varphi(k-i) \varphi(i)}{\varphi(k)^2}$ converges to $1$ as $k \to \infty$ uniformly in $i \in \left[ \eps k, \frac{k}{2} \right]$. By Riemann sum approximation, this implies
    \begin{equation}\label{eqn:convergence_disloc_measure_truncated}
        p \frac{1}{k}  \sum_{i=\eps k}^{\lfloor \frac{k-1}{2} \rfloor} \left( 1-\frac{i}{k} \right)^{-\gamma} \left( \frac{i}{k} \right)^{1-\gamma} \frac{\varphi(k-i) \varphi(i)}{\varphi(k)^2} f \left( \frac{k-i}{k}, \frac{i}{k} \right) \xrightarrow[k \to \infty]{} p \int_{\eps}^{1/2} f(1-x,x) x^{1-\gamma} (1-x)^{-\gamma} \, \mathrm{d}x.
    \end{equation}
    
    Now let us now handle the terms where $i<\eps k$. We fix $\delta>0$ small enough so that $\gamma+\delta<2$ (recall from \eqref{eq:gamma-alpha-rel} that $\gamma<2$). Again, by \cref{item:slow_variation1} of Lemma~\ref{lem:slow_variation}, the ratio $\frac{\varphi(k-i)}{\varphi(k)}$ is bounded by an absolute constant which does not depend on $\eps$. So are the values of $f$ and the factors $\left( 1-\frac{i}{k} \right)^{-\gamma}$. On the other hand, \cref{item:slow_variation2} of Lemma~\ref{lem:slow_variation} shows $\frac{\varphi(i)}{\varphi(k)}=O \Big( \Big( \frac{i}{k} \Big)^{-\delta} \Big)$ uniformly in $i<k$, so we can write
    \[ p \frac{1}{k}  \sum_{i=1}^{\eps k} \left( 1-\frac{i}{k} \right)^{-\gamma} \left( \frac{i}{k} \right)^{1-\gamma} \frac{\varphi(k-i) \varphi(i)}{\varphi(k)^2} f \left( \frac{k-i}{k}, \frac{i}{k} \right) = O \left( \frac{1}{k} \sum_{i=1}^{\eps k} \left( \frac{i}{k} \right)^{1-\gamma-\delta} \right) = O \left( \eps^{2-\gamma-\delta} \right).\]
    Combined with~\eqref{eqn:convergence_disloc_measure_truncated}, this shows that~\eqref{eqn:assumption_H_hassmiermont} is satisfied and concludes the proof of Theorem~\ref{thm:fragmentation_tree}.
\end{proof}

\appendix

\section{Numerical simulations for the limiting random variable $X$}\label{sect:num-sim}

We present in~\cref{fig-simulations}~and~\cref{tab:mean_variance} some numerical simulations for the limiting random variables $X(p)$ appearing in \cref{thm:main_permutations}. Recall the notation $\sigma_n(p)\coloneqq \Perm[\bm{\mu}_p,n]$ from \eqref{eq:finite_perm_brownian}.

For the diagram on the left-hand side of~\cref{fig-simulations}, we sampled $100\,000$ independent copies of $\sigma_n(p_i)\coloneqq\Perm[\bm{\mu}_{p_i},n]$ of size $n = 1\,000\,000$, for each $p_i = \frac{i}{10}$ with $i \in \intervalleentier{1}{9}$. We then estimated the density function of (recall \cref{thm:main_permutations})
\[\frac{\LIS(\sigma_n(p_i))}{n^{\alpha(p_i)}} \approx X(p_i)
\qquad \text{for all }i\in \intervalleentier{1}{9},\] 
using the function SmoothKernelDistribution from Mathematica.

For the diagram on the right-hand side of~\cref{fig-simulations}, we sampled $100\,000$ independent uniform separable permutations $\overline{\sigma}_n$ of the same size $n = 1\,000\,000$, and estimated the density function (recall~\cref{conj:sep-and-cograph}) of 
\[\frac{\LIS(\overline{\sigma}_n)}{\symcst \cdot n^{\alpha(1/2)}} \approx X(1/2),\] 
again using SmoothKernelDistribution. The constant $c=\symcst$ was calibrated so that the two estimated densities on the right-hand side of~\cref{fig-simulations} would be as close as possible. More precisely, we chose the constant $c$ as the minimizer of 
\[\int_{0}^{100} |f_c(x)-g(x)|\, \mathrm{d}x,\] 
where $f_c(x)$ is the approximate density function of 
$\frac{\LIS(\overline{\sigma}_n)}{c \cdot n^{\alpha(1/2)}}$ and $g(x)$ is the approximate density function of $\frac{\LIS(\sigma_n(1/2))}{n^{\alpha(1/2)}}$. At $c=\symcst$, the above integral takes the value $0.055$.

Finally, in \cref{tab:mean_variance} we show the approximated values for the mean and variance of the random variables $X(p_i)$.

\begin{figure}[ht]
	\begin{center}
		\includegraphics[width=0.47\textwidth]{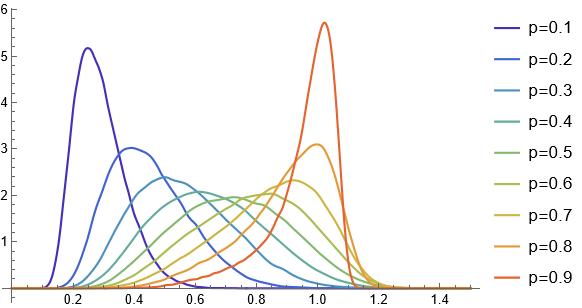} 
        \includegraphics[width=0.47\textwidth]{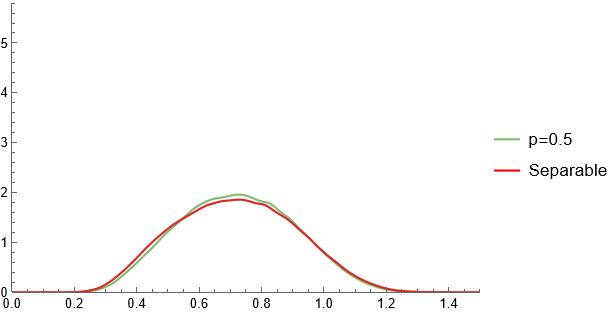} 
		\caption{
        \textbf{Left:} Approximate densities of the random variables $X(p)$. \textbf{Right:} Comparison between the approximate density of the limiting random variable for  $\frac{\LIS(\sigma_n(1/2))}{n^{\alpha(1/2)}}$ when the permutations $\sigma_n$ are sampled from the Brownian separable permuton with parameter $1/2$ and the approximate density of the limiting random variable for $\frac{\LIS(\overline{\sigma}_n)}{\symcst \cdot n^{\alpha(1/2)}}$, when the permutations $\overline{\sigma}_n$ are exactly uniform separable permutations.  
        \label{fig-simulations}}
	\end{center}
	\vspace{-3ex}
\end{figure}

\begin{table}[h!]
\centering
\begin{tabular}{|c|c|c|}
\hline
$p$ & $\Ec{X(p)}$ & $\Var{X(p)}$ \\
\hline
0.1 & 0.288 & 0.007 \\
0.2 & 0.449 & 0.017 \\
0.3 & 0.560 & 0.026 \\
0.4 & 0.648 & 0.032 \\
0.5 & 0.722 & 0.034 \\
0.6 & 0.787 & 0.033 \\
0.7 & 0.846 & 0.029 \\
0.8 & 0.901 & 0.022 \\
0.9 & 0.952 & 0.012 \\
\hline
\end{tabular}
\caption{Approximate mean and variance of $X(p)$ for different values of $p$.\label{tab:mean_variance}}
\end{table}

%
%

\clearpage

\section*{Index of notation}
\addcontentsline{toc}{section}{Index of notation}\label{sect:Index-of-notation}

\begin{flushleft}
  \begin{longtable}{cl}
    $\bm{\mu}_p$ & Brownian separable permuton of parameter $p\in (0,1)$ \\
    $\sigma_n$ & Permutation of size $n$ sampled from $\bm{\mu}_p$ \\
    $\lis(\sigma)$ & Length of the longest increasing sequence of the permutation $\sigma$ \\
    $\bm{W}_p$ & The Brownian cographon of parameter $p\in(0,1)$ \\
    $G_n$ & Random graph on $n$ vertices sampled from the Brownian cographon \\
    $\lcl(G)$, $\lin(G)$ & Size of the largest clique, independent set, of the graph $G$ \\
    $\alpha=\alpha(p)$ & Explicit exponent (depending on the parameter $p$), see \eqref{eq:eq-to-defn-alpha} \\
    $\gamma=\gamma(p)$ & Regular variation exponent related to $\alpha$, see \eqref{eq:gamma-alpha-rel} \\
    $t|_{\mathcal{L}}$ & Subtree of $t$ induced by the set of leaves $\mathcal{L}$, see \cref{sect:positive-subtrees} \\
    $|t|$ & Size (\emph{i.e.}\ number of leaves) of the tree $t$ \\ 
    $\oplus$, $\ominus$ & Signs carried by the nodes of a sign-decorated tree \\
    $\LIS(t)$ & Maximal size of a positive subtree of $t$ \\
    $\mathcal{L}^{\max}_{\cap}(t)$ & Set of leaves in the intersection of all maximal positive subtrees of $t$ \\
    $\uedge{n}, \usign{n}, \uleaf{n}$ & Edge, sign and leaf at step $n$ of Rémy's algorithm, see \cref{sect:remy-algo} \\
    $T_n$ & $p$-signed uniform binary tree with $n$ leaves, see \cref{sect:perm-main} \\
    $T$ & $p$-signed critical binary Bienaymé--Galton--Watson tree \\
    $T^k$ & $T$ conditioned on $\LIS(T) = k$ \\
    $q$, $Q$ & Law and tail function of $\LIS(T)$, see \eqref{eq:def_q_Q} \\
    $X_n$ & $X_n = \LIS(T_n)$, see \eqref{eq:def_X_X_m} \\
    $\svfq$, $\svfbQ$ & Slowly varying functions in \cref{thm:good_regularity} \\
    $r(n)$ & $r(n) =\mathbb{P}(|T|=n)$, see \cref{lem:prelim_GW} \\
    $\Ff_n$, $\revF_n$ & Filtrations associated with the forward, backward, Rémy's algorithm \\
    $\sttm_k$ & Stopping time $\sttm_k = \min \{ n \geq 1 \mid X_n=k\}$, see \eqref{eq:defn-stp-time} \\
    $t^{\lmax}$ & Leftmost maximal subtree of $t$, see \cref{sect:mark-desc} \\
    $L^k$ & Uniformly chosen leaf of $T^{k,\lmax}$ \\
    $Z^k=(S,D,\Mst{},W)$ & Four-component Markov chain, see \eqref{eq:def_MC_Z} \\
    $\eta$ & Death time of $\Mst{}$ \\
    $\mathbb{P}_k$ & Law of $Z^k$ started at $k$, see \cref{rmk:start-notation} \\
    $\pi((\ast, \ast, \ast, \ast) ; (\ast, \ast, \ast, \ast) )$ & Transition probabilities of $Z^k$, see \cref{sect:trans-prob2} \\
    $\pi(\ast ; \ast)$ & Transition probabilities of $\DMC$, see \eqref{eq:transitions decreasing chain} \\
    $\stht{t}$ & Time-change in $\Mst{}$, see \eqref{eq:defn-stopping-times} \\
    $\DMC$ & Decreasing Markov chain, $(\stht{t})$ time-change of $\Mst{}$ \\
    $\mathtt{T}$ & Death time of $\DMC$ \\
    $\theta(m)$ & Parameters of geometric random variable in \eqref{eq:trans_Mtau} \\
    $(B_m)_m$ & Black instructions, see \eqref{eq:blue-distrib} \\
    $(G_m)_m$ & Golden instructions, see \eqref{eq:red-law} \\
    $\mathbb{P}_{k,k'}$ & Law of the coupling of $M, M'$ started from $k,k'$, see \cref{defn:black-golden coupling} \\
    $\widetilde{\mathbb{P}}_{k,k'}$ & Law of the coupling of $Z, Z'$ started from $k,k'$, see \cref{cor:full_coupling} \\
    $\xmer$ & Merging value of $\DMC$ and $\DMC'$, see \cref{prop:coupling_merging} \\
    $\stmer$, $\stmerpr$ & Merging times of $\DMC$ and $\DMC'$, see \cref{prop:coupling_merging} \\
    $\mathrm{GoodScales}$ & A set of good scales, as defined in \cref{defn:good-scales} \\
    $\stt_i$ & Hitting time of $\intervalleentier{0}{i}$ by $\DMC$, see \eqref{eq:stt-defn} \\
    $\Gf_{\ell}$ & $\sigma$-field recording information strictly above $7\ell$, 
    see \eqref{eq:def_Ff} \\
    $\lambda$ & Scaling limit of $\mathcal{L}^{\max}_{\cap}(T^k)$, see \cref{thm:convergence_size_intersection}; later equated with $\alpha$ \\
    $\stdivv$ & Divergence height between $L^k$ and $\widetilde{L}^k$, see \eqref{eq:definition stdivv} \\
  \end{longtable}
\end{flushleft}

\addcontentsline{toc}{section}{References}
\bibliography{cibib,cibib2}
\bibliographystyle{hmralphaabbrv}

\end{document}